\documentclass[11pt,letterpaper]{amsart}

\usepackage{amssymb,latexsym, amsmath, amsxtra, mathrsfs, bm}
\usepackage[dvips]{graphics}
\usepackage[T1]{fontenc}

\usepackage[all]{xy}
\usepackage{xcolor}
\usepackage{subcaption}
\usepackage{verbatim}
\usepackage[abs]{overpic}
\usepackage{hyperref}
\usepackage{mathtools}
\usepackage{enumitem}

\DeclareMathAlphabet{\mathbbold}{U}{bbold}{m}{n}

\allowdisplaybreaks[3]

\makeatletter
\def\th@plain{%
	\thm@notefont{}
	\itshape 
}
\def\th@definition{%
	\thm@notefont{}
	\normalfont 
}
\makeatother

\theoremstyle{plain}
        \newtheorem{theorem}{Theorem}[section]
        \newtheorem*{theorem*}{Theorem}
        \newtheorem{thmx}{Theorem}
        
        \newtheorem{thml}{Theorem}

        \newtheorem{lemma}[theorem]{Lemma}
        \newtheorem{prop}[theorem]{Proposition}
        \newtheorem{cor}[theorem]{Corollary}

\theoremstyle{definition}
        \newtheorem{definition}[theorem]{Definition}
        \newtheorem{rem}[theorem]{Remark}
         
         \newtheorem{ex}[theorem]{Example}

\theoremstyle{remark}
        \newtheorem*{remark}{Remark}
        
        \newtheorem*{notation}{Notation}
        \newtheorem{question}{Question}

\numberwithin{equation}{section}
\numberwithin{theorem}{section}
\numberwithin{table}{section}
\numberwithin{figure}{section}

\providecommand{\defn}[1]{\emph{#1}}

\renewcommand{\le}{\leqslant}
\renewcommand{\leq}{\leqslant}
\renewcommand{\ge}{\geqslant}
\renewcommand{\geq}{\geqslant}

\newcommand{\diam}  {\operatorname{diam}}

\newcommand{\id} {\operatorname{id}}

\newcommand{\card} {\operatorname{card}}
\newcommand{\supp}{\operatorname{supp}}
\newcommand{\Per}{\mathrm{Per}}
\newcommand{\ve}{\varepsilon}
\newcommand{\R}{\mathbb{R}}

\newcommand{\N}{\mathbb{N}}      
\newcommand{\Z}{\mathbb{Z}}      

\providecommand{\abs}[1]{\lvert#1\rvert}
\providecommand{\Absbig}[1]{\bigl\lvert#1\bigr\rvert}

\providecommand{\norm}[1]{\|#1\|}
\providecommand{\Normbig}[1]{\bigl\|#1\bigr\|}

\renewcommand{\:}{\colon}

\newcommand{\LIP}{\operatorname{LIP}}

\newcommand{\Lip}{\operatorname{Lip}}

\newcommand{\CCC}{C}

\newcommand{\PPP}{\mathcal{P}}

\newcommand{\MMM}{\mathcal{M}}

\newcommand{\Holder}[1] {\CCC^{0,#1}}

\newcommand{\Hseminorm}[2] {\abs{#2}_{#1}}

\newcommand{\Hnorm}[2] {\norm{#2}_{#1}}

\newcommand{\Hnormbig}[2] {\Normbig{#2}_{#1}}

\newcommand{\RR}{\mathcal{L}}

\newcommand{\wt}[1]{\widetilde{#1}}

\renewcommand{\=}{\coloneqq}

\renewcommand{\O}{\mathcal{O}}

\newcommand{\Mmax}{\MMM_{\operatorname{max}}}

\newcommand{\Lock}{\operatorname{Lock}}

\newcommand{\QG}{\operatorname{QG}}

\newcommand{\CQG}{\operatorname{CQG}}

\newcommand{\mpe}{Q}

\newcommand{\nbhd}{B}

\newcommand{\cA}{\mathcal{A}}
\newcommand{\cB}{\mathcal{B}}

\newcommand{\cE}{\mathcal{E}}
\newcommand{\cF}{\mathcal{F}}
\newcommand{\cG}{\mathcal{G}}
\newcommand{\cH}{\mathcal{H}}

\newcommand{\cK}{\mathcal{K}}
\newcommand{\cL}{\mathcal{L}}
\newcommand{\cM}{\mathcal{M}}

\newcommand{\cO}{\mathcal{O}}
\newcommand{\cP}{\mathcal{P}}

\newcommand{\cS}{\mathcal{S}}
\newcommand{\cT}{\mathcal{T}}
\newcommand{\cU}{\mathcal{U}}

\newcommand{\fL}{\mathfrak{L}}

\newcommand{\fS}{\mathfrak{S}}

\newcommand{\sP}{\mathscr{P}}

\newcommand{\oB}{\overline{B}}

\newcommand{\tphi}{\widetilde{\phi}}

\newcommand{\tpsi}{\widetilde{\psi}}

\newcommand{\Sim}{\mathrm{Sim}}

\newcommand{\RLS}{\mathcal{R}}

\newcommand{\Crit}{\mathrm{Crit}}

\newcommand{\emergent}{E}

\newcommand{\myepsilon}{\varepsilon}

\begin{document}
	\title[Joint typical periodic optimization]{Joint typical periodic optimization}

	\author{Zelai~Hao \and Yinying~Huang \and Oliver~Jenkinson \and Zhiqiang~Li}
	
	\address{Zelai~Hao, School of Mathematical Sciences, Peking University, Beijing 100871, China}
	\email{2100010625@stu.pku.edu.cn}
	\address{Yinying~Huang, School of Mathematical Sciences, Peking University, Beijing 100871, China}
	\email{miaoyan@stu.pku.edu.cn}
	\address{Oliver~Jenkinson, School of Mathematical Sciences, Queen Mary, University of London, Mile End Road, London E1 4NS, United Kingdom}
	\email{o.jenkinson@qmul.ac.uk}
	\address{Zhiqiang~Li, School of Mathematical Sciences \& Beijing International Center for Mathematical Research, Peking University, Beijing 100871, China}
	\email{zli@math.pku.edu.cn}

	\dedicatory{Dedicated to the memory of Gonzalo Contreras}
	
	\subjclass[2020]{Primary: 37A99; Secondary: 37A05, 37B25, 37B65, 37C20, 37C50, 37D20, 37D35, 37E05, 37A44.}
	
	\keywords{Typical periodic optimization, ergodic optimization,   maximizing measure, expanding map, beta-transformation, beta-expansion}

	\thanks{
    Z.~H., Y.~H., and Z.~L.\ were partially supported by Beijing Natural Science Foundation (JQ25001 and 1214021) and National Natural Science Foundation of China (12471083, 12101017, 12090010, and 12090015).
	}

	\begin{abstract}
    We prove a generalised Yuan--Hunt--Ma\~n\'e Conjecture:
if $\cF$ is the Banach space of $\alpha$-H\"older functions,
and $\cT$ is either a space of Lipschitz
expanding maps, or of Anosov diffeomorphisms, 
or the family of beta-transformations
on the interval, there is an open dense subset of 
 $\cT\times\cF$ consisting of map-function
 pairs
 whose maximizing invariant measure is unique and supported on a periodic orbit.
	\end{abstract}
	
	\maketitle

    \setcounter{tocdepth}{2}

\enlargethispage{\baselineskip}
    \vspace*{-\baselineskip}
    \thispagestyle{empty}

	\tableofcontents

\section*{Introduction} \label{s:mini-introduction}

In this article we prove a generalised Yuan--Hunt--Ma\~n\'e Conjecture,
in the spirit of related work
of Hunt \& Ott \cite{HO96a, HO96b}, Yang, Hunt \& Ott \cite{YHO00}, and Ma\~n\'e \cite{Man96, Man97}.

Contreras \cite{Co16} proved the Yuan--Hunt Conjecture \cite{YH99} that ground states are generically periodic,
more precisely:

\begin{theorem*}
If $X$ is a compact metric space, and $T\:X\to X$ is an open Lipschitz expanding map, then
there exists an open dense subset $U$ of the space $\Lip(X)$ of real-valued Lipschitz functions,
such that if $\phi\in U$ then the unique $(T,\phi)$-maximizing measure is supported on a periodic orbit.
\end{theorem*}

In other words, Contreras' theorem asserts that for any fixed expanding map $T$, periodic 
optimization is a \emph{typical} property of the potential function.
In this article we go further,
in a direction proposed by Yang, Hunt \& Ott \cite{YHO00},
and establish that this typicality persists 
even when $T$ and $\phi$ are varied \emph{jointly}: 
we refer to this as \emph{joint typical periodic optimization} (Joint TPO).
If $\cE(X)$ denotes the space of all open Lipschitz expanding maps on $X$, we prove:

\begin{theorem*}
If $X$ is a compact locally connected metric space, then
there exists an open dense subset $U$ of $\cE(X) \times \Lip(X)$ 
such that if $(T,\phi) \in U$ then the unique $(T,\phi)$-maximizing measure is supported on a periodic orbit.
\end{theorem*}

The above theorem applies, in particular, to the case where $\cE(X)$ is the space of all Lipschitz expanding maps on a compact Riemannian manifold $X$. For such $X$ we also prove an analogous theorem for the space of all  Anosov diffeomorphisms on $X$.
Going beyond the archetypal uniformly hyperbolic context, we consider the prototypical nonuniformly hyperbolic setting of beta-transformations
$T_\beta(x)=\beta x \pmod 1$ with $\beta>1$ on the unit interval $I$, and prove:

\begin{theorem*}
There is an open dense subset $U$ of $(1,+\infty) \times \Lip(I)$ such that if $(\beta,\phi)\in U$ then
the unique $(T_\beta,\phi)$-maximizing measure is supported on a periodic orbit.
\end{theorem*}

The proof of these results centres on the principle of 
\emph{joint perturbation}, a novel framework allowing for 
the synchronised variation of both map and potential function.
The principle is most naturally realised in the uniformly 
hyperbolic setting, but its most technically demanding 
form is in the beta-transformation case, where 
distinct parts of the argument exploit differing structural 
characteristics of simple and non-simple beta-numbers; 
it is this case that occupies the majority of the article.
The same principle, once established in these foundational 
settings, can be formulated axiomatically and applied to 
further classes of hyperbolic systems: in \cite{HHJL26} it 
is used to establish Joint TPO for Axiom~A diffeomorphisms 
of all regularity classes, hyperbolic rational maps on the 
Riemann sphere, and smooth maps in one-dimensional settings.

\section{Statement of results}   \label{s:introduction}

\subsection{Background}

An important strand in dynamical systems theory is the understanding of those properties that are in some sense typical, or robust under perturbation of the system.
Investigations along these lines include classical works of Oxtoby \& Ulam \cite{OU41}, Halmos \cite{Ha44}, and Rokhlin \cite{Rok48}, on genericity in the context of ergodicity, mixing, and weak mixing,
and the considerable literature on structural stability and generic properties of diffeomorphisms 
(see e.g.~\cite{ACW16, BC04, BCW09, BD96, BDP03, GPS94, Man82, Man88, MV93, New79, Pu67, PS00, Rob71, SW00}), and on
typical properties of one-dimensional systems (see e.g.~\cite{AMM03, GrSw97, GrSw98, Ly97, Ly98, Ly02});
for notions of typicality in dynamical systems, see 
e.g.~the survey by
Hunt \& Kaloshin \cite{HK10}.

Of particular importance, in view of physical applications to predictability and stability, 
are situations where the solution of a variational problem is typically a periodic orbit, 
and the selection of this orbit is robust under all perturbations.
One such problem is described by 
specifying a dynamical system $T\:X\to X$, and a (potential) function $\phi\:X\to\R$, 
and investigating those $T$-invariant probability measures that maximize (or minimize) the corresponding ergodic average.
First considered in the 1990s, numerical work on this \emph{ergodic optimization} problem suggested
that typically there is a simple solution: the maximizing measure seems to often be 
periodic, i.e., supported on a single periodic orbit.
More precisely, if $\phi$ is chosen from a space of suitably regular (e.g.~smooth, or H\"older continuous) functions, and $T$ is chosen from a space of suitably chaotic maps, it appears to be the case that for a typical pair $(T,\phi)$, there exists a $T$-periodic orbit $\cO=\cO_{T,\phi}$ such that the $\phi$-average 
$\frac{1}{\card\cO} \sum_{x\in\cO}\phi(x)$
is strictly larger than the $\phi$-average $\int\! \phi\, \mathrm{d}\mu$ with respect to any $T$-invariant probability measure 
$\mu$ not supported on $\cO$.

In 1996, Hunt \& Ott \cite{HO96a, HO96b} conjectured, on the basis of numerical experiments for various parametrised families of chaotic maps $T$ and smooth functions $\phi$,
that
the set of parameter values for which maximizing measures are periodic has full Lebesgue measure;
this conjecture has been confirmed in a number of specific
settings (see \cite{BZ16, DLZ24, GSZ25}\footnote{In these works, ``full Lebesgue measure'' is expressed in terms of \emph{prevalence}, see e.g.~Hunt \& Kaloshin \cite{HK10}.}). 
In 1999, Yuan \& Hunt \cite{YH99} conjectured a topological analogue, that for any uniformly hyperbolic dynamical system $T$ (specifically, an expanding map or an Axiom A diffeomorphism), the maximizing measure is periodic for an open and dense subset of Lipschitz functions $\phi$.
These two conjectures naturally raise the question of whether topological typicality 
persists when \emph{both} the dynamical system and the potential 
function are allowed to vary jointly; this question, which is the 
central concern of the present article, was raised in the 
continuous-time setting by Yang, Hunt \& Ott \cite{YHO00}.

The Yuan--Hunt Conjecture came to occupy a central role in the development of ergodic optimization, and stimulated a sustained period of work over the following years
(see \cite{Bou01, Bou08, BQ07, CLT01, Mo08, QS12}), and was eventually 
proved by Contreras \cite{Co16} for expanding maps\footnote{More precisely, Contreras' proof is for distance-expanding, Lipschitz, open mappings. Together with novel perturbation arguments, the proof of \cite{Co16} exploits prior work on the so-called Ma\~n\'e cohomology lemma, as pioneered notably by Bousch \cite{Bou00,Bou01,Bou02,Bou08,Bou11}, Contreras, Lopes \& Thieullen \cite{CLT01}, Conze \& Guivarc'h \cite{CG95}, Lopes \& Thieullen \cite{LT03}, and Savchenko \cite{Sa99}, a closing lemma of Bressaud \& Quas \cite{BQ07}, and work of Morris \cite{Mo08} on the genericity of maximizing measures with zero entropy.}, 
and by Huang, Lian, Ma, Xu \& Zhang \cite{HLMXZ25}
for the remaining uniformly hyperbolic cases\footnote{The approach of \cite{HLMXZ25} is more direct than that of \cite{Co16}, using a novel closing lemma combined with the one of Bressaud \& Quas \cite{BQ07}, but without the need for the intermediate generic zero entropy result of \cite{Mo08}; in the setting of expanding maps, this proof strategy has been summarised in the expository article of Bochi \cite{Boc19}.};
going beyond the setting of uniform hyperbolicity, Li \& Zhang \cite{LZ25} proved the analogous
result for expanding Thurston maps from complex dynamics.
Progress towards the Hunt--Ott Conjecture and the Yuan--Hunt Conjecture, and broader advances in the field of ergodic optimization (see e.g.~\cite{Boc18, Je06, Je19} for further details),
prompted the Typical Periodic Optimization (TPO) Conjecture (cf.~\cite[Section~7]{Je19}), that 
if $T$ is suitably hyperbolic, and $\cF$ is a Banach space consisting of suitably regular continuous functions, then
the set of functions with a periodic maximizing measure will contain
an open dense subset of $\cF$. 

A parallel conjecture of Ma\~n\'e (\cite[Problem III]{Man96}, \cite[p.~143]{Man97}) concerning the minimizing measures of Lagrangian flows, as introduced by Mather
\cite{Mat89, Mat91}, is that for a generic Lagrangian there is a single minimizing measure, and it is supported on a periodic orbit.
Ma\~n\'e's Conjecture has been influential in Aubry--Mather theory \cite{CFR15}, and in the study of the Hamilton--Jacobi equation
\cite{DFIZ16, FS04}, and has stimulated the cross-fertilisation of ideas with ergodic optimization
(notably, \cite{CLT01} is influenced by Ma\~n\'e's work on Aubry--Mather theory, while \cite{Co24} employs ergodic optimization techniques to tackle Ma\~n\'e's Conjecture).
In the Ma\~n\'e--Mather formulation, the Lagrangian whose action is to be optimized determines the flow itself, and this binding of potential function to dynamics has inspired workers in ergodic optimization, beginning with \cite{CLT01}, to study generic properties of measures that optimize Lyapunov exponents (see also \cite{GaSh24, JM08, MT13, ST20}).

\subsection{Joint typical periodic optimization}\label{JTPO_intro_subsection}
 
In the present article we 
undertake a programme investigating the most general notion of topological typicality in ergodic optimization, by allowing the independent variation of both dynamical system and potential function.
The conjectures of Hunt \& Ott \cite{HO96a, HO96b} concern 
\emph{measure-theoretic} typicality: that for a typical smooth 
performance function, in the sense of Lebesgue measure on 
parameter space, the maximizing orbit is periodic with low period.
The Yuan--Hunt conjecture \cite{YH99} 
concerns the complementary \emph{topological} notion: for a 
fixed uniformly hyperbolic map, the set of functions with a 
periodic maximizing measure contains an open dense subset.
Our aim is to address the natural further question of whether 
topological typicality persists when both the dynamical system 
and the potential function are varied jointly. This question 
was raised in the continuous-time setting by Yang, Hunt \& Ott 
\cite{YHO00}, whose numerical observations suggested that 
joint variation of system and performance function parameters 
can restore typical periodicity in cases where, for individual 
systems, it fails\footnote{The observation in 
\cite[p.~1951]{YHO00} concerns `enlarging the meaning of 
typicality to be with respect to both variation of system 
parameters and performance function parameters'; this is 
illustrated (see \cite[Section~IV]{YHO00}) in the context 
of the Lorenz equations, where, for standard parameter values, 
individual low-period optimality fails for many performance 
functions, yet joint variation restores it.}.
Our aim is to establish the robustness of this joint typicality 
across rather different dynamical structures, encompassing both 
the structurally stable uniformly hyperbolic systems, as well 
as the prototypical nonuniformly hyperbolic setting of 
beta-transformations.
This joint approach, already proposed in 
\cite{YHO00} and articulated in the ergodic optimization 
setting in \cite[pp.~2610--2611]{Je19}, is most natural from the perspective of 
physical applications, allowing perturbation of all laws 
governing the optimization.

Specifically, our interest is in the \emph{joint typical periodic optimization (Joint TPO)} problem,
allowing joint variation of pairs $(T,\phi)$ of map $T$ and function $\phi$
in suitable \emph{product spaces}, and seeking to establish that periodic optimization is a typical property with respect to the \emph{topology of the product space}.
More precisely,
for any Banach space 
$\cF$ consisting of suitable continuous
real-valued functions on a compact metric space $X$,
and a suitable space $\cT$ of self-maps of $X$,
we say that a pair 
$(T,\phi)\in\cT \times \cF$
has the
 \emph{periodic optimization property} if there exists some $T$-periodic orbit $\cO$,
with corresponding $T$-invariant probability measure 
\begin{equation}\label{periodic_measure_form}
\mu_{\cO} \= \frac{1}{\card \cO} \sum_{x\in \cO} \delta_x,
\end{equation}
such that
\begin{equation}\label{periodic_optimization_property}
\int\! \phi \, \mathrm{d}\mu_{\cO} 
> \int\! \phi \, \mathrm{d}\nu \quad\text{ for all }\nu\in\MMM(X,T)\smallsetminus\{\mu_{\cO}\}, 
\end{equation}
where $\MMM(X,T)$ denotes the set of $T$-invariant Borel probability measures on $X$. In other words,
if a measure $\mu\in\MMM(X,T)$ is called \emph{$(T,\phi)$-maximizing} whenever
$\int\! \phi\, \mathrm{d}\mu = \sup_{\nu\in\MMM(X,T)} \int\! \phi\, \mathrm{d}\nu \eqqcolon Q(T,\phi)$,
the periodic optimization property for $(T,\phi)$ means that there is a unique $(T,\phi)$-maximizing
measure, and it is of the form 
(\ref{periodic_measure_form}) for some $T$-periodic orbit $\cO$.

For a fixed $T\in\cT$, we say that 
$\{T\}\times\cF$
has the
\emph{typical periodic optimization (TPO)}
property if there is an open dense subset of
$\cF$ such that $(T,\phi)$ has the periodic optimization property for every element $\phi$ of the subset. Whenever the function space $\cF$ is clear from the context, we simply say that $T$ \emph{has TPO}.

We say that the pair $(\cT,\cF)$
(or indeed the space  $\cT\times\cF$, equipped with the product topology)
has the \emph{joint typical periodic optimization (Joint TPO)} property 
 if there is an open dense set of pairs $(T,\phi)$ in the product space $\cT \times \cF$ 
that have the periodic optimization property.\footnote{Although conceptually related to Ma\~n\'e's Lagrangian framework 
(where flow and Lagrangian are rigidly coupled variables), and
the Lyapunov optimization of \cite{CLT01},
the Joint TPO results in this article reveal a stronger stability phenomenon: periodic optimization, 
when achieved at a given system-potential pair in a certain open dense set,
persists under arbitrarily small \emph{independent} perturbations of both components.}
In order to distinguish between these notions,
it will sometimes be convenient to refer to TPO as
\emph{Individual TPO}.
As we explain below, typical properties in the product space $\cT \times \cF$ 
neither imply, nor are implied by, typical properties in typical
fibres $\{T\}\times \cF$, so a Joint TPO result is not a straightforward consequence of an Individual TPO result, and an Individual
TPO result is not a straightforward consequence of a Joint TPO result.\footnote{The situation is more subtle if we are content with \emph{generic} statements, 
but even here there is little connection between joint results and fibrewise results; a notable exception is Theorem~\ref{GPO_open_dense}, which is a straightforward consequence of a Joint TPO result.} 

In the initial part of this article 
(Section~\ref{sec_JTPO_distance_expanding}),
we consider joint typical periodic optimization in the foundational setting of uniformly hyperbolic systems, where Individual TPO is known to hold: 
here the members of $\cT$ are either expanding maps or Anosov diffeomorphisms\footnote{It is possible to prove Joint TPO results for more general hyperbolic dynamical systems, 
including Axiom A diffeomorphisms (cf.~\cite{Man88,Pal87,Sma67,Sma70}) and Axiom A systems in
one (real or complex) dimension (cf.~\cite{GrSw97, GrSw98, KSS07, Ly83, Ly86, Ly97, MSS83});
such proofs are beyond the scope of the present work and are presented elsewhere (see \cite{HHJL26}).}, 
and $\cF$ is 
either 
the H\"older space $\Holder{\alpha}(X)$, or 
(where appropriate) the $C^1$ function space. 
While of primary importance for establishing the conceptual framework of our approach, 
this uniformly hyperbolic setting allows for a particularly concise treatment, serving as a precursor to the more intricate analysis 
handled in the majority of the article, which is devoted to Joint TPO in the nonuniformly hyperbolic
setting of beta-transformations, and serves to confirm the robustness of the joint perturbation framework.

We first consider  
Lipschitz
expanding maps, which in particular includes the class of all smooth expanding maps on arbitrary compact connected Riemannian manifolds, as studied notably by Shub \cite{Sh69, Sh70}
(see also e.g.~\cite{ES68, FJ78, GoRo23, Gr81, Hirs70} and \cite[Chapter 6]{URM22}). More specifically, if $X$ is a compact locally connected metric space, and
$\cE(X)$ is the space of Lipschitz open mappings on $X$ that are distance-expanding, 
we prove:

\begin{thmx}[Joint TPO for expanding maps] \label{expandingJTPO}
Suppose $X$ is a compact locally connected metric space, and $\alpha \in (0,1]$.
There is an open dense subset of pairs 
$(T,\phi)\in \cE(X)\times \Holder{\alpha}(X)$ 
with the periodic optimization property.
\end{thmx}

Note that
typicality in $\cE(X)\times\Holder{\alpha}(X)$ is in general not a straightforward consequence of typicality in each $\{T\}\times\Holder{\alpha}(X)$:
indeed purely topological arguments do not even imply a weaker joint \emph{generic} periodic optimization\footnote{By \emph{joint generic periodic optimization} we mean that a \emph{residual} subset of 
the product space $\cE(X)\times\Holder{\alpha}(X)$ has the periodic optimization property.}
result as an automatic consequence of Individual TPO in each fibre.\footnote{To justify this assertion, note that provided $\cE(X)$ contains an embedded curve, there exists a subset $\cL\subseteq \cE(X)\times\Holder{\alpha}(X)$ that is not residual in  $\cE(X)\times\Holder{\alpha}(X)$, yet for each $T\in\cE(X)$, the set $\bigl\{\phi\in\Holder{\alpha}(X): (T,\phi)\in\cL \bigr\}$ contains an open dense subset of $\Holder{\alpha}(X)$; this is a consequence of the fact 
(see \cite[Theorem~15.5]{Ox71}, and cf.~\cite[p.~54]{Ke95}) that there exists a nonmeagre subset of $[0,1]^2$ such that no three of its elements are collinear (cf.~also footnote~\ref{nonmeagrecollinear}).}
We next consider Anosov diffeomorphisms:

\begin{thmx}[Joint TPO for Anosov diffeomorphisms]
\label{anosovJTPO}
Let $M$ be a smooth compact Riemannian manifold,
with distance function induced by the Riemannian metric,
and let $\cA(M)$ be the space of $C^1$ Anosov diffeomorphisms on $M$, equipped with the $C^1$ topology.
  For all $\alpha \in (0,1]$,
there is an open dense subset of pairs 
$(T,\phi)\in \cA(M)\times \Holder{\alpha}(M)$ 
with the periodic optimization property.
\end{thmx}

A more delicate context for joint typical periodic optimization results, beyond the paradigmatic uniformly hyperbolic setting, arises if
we take the space $\cT$ to be a classical one-parameter family of maps of the unit interval $I=[0,1]$, the 
\emph{beta-transformations} $T_\beta(x)= \beta x \pmod 1$, $\beta>1$. 
Unlike for Lipschitz open distance-expanding maps,
not every beta-transformation has TPO: for example if $\beta=2$ then the function $\phi$ given by $\phi(x)\=x$ has no maximizing measure, and it is readily checked that there is an open neighbourhood of $\phi$ in 
$\Holder{1}(I)$  consisting of functions with no maximizing measure. 
Moreover, it can be shown that there is a dense collection of parameters $\beta\in (1,+\infty)$ for which $T_\beta$ does not have TPO (cf.~Remark~\ref{r_counter_example_simple_beta} and \cite[Theorem~5]{Par60}).
Rather remarkably, despite this pervasive failure of Individual TPO, we are nonetheless
able to establish that Joint TPO \emph{does} hold for the class of beta-transformations, where 
$\cT=(1,+\infty)$ is equipped with its usual topology, and
$\cT\times\cF=(1,+\infty)\times\Holder{\alpha}(I)$ is equipped with the product topology:

\begin{thmx}[Joint TPO for beta-transformations]\label{jtpo_tbeta}
    Fix $\alpha\in(0,1]$.
    There is an open dense subset of pairs $(\beta,\phi)\in (1,+\infty)\times\Holder{\alpha}(I)$
    such that $(T_\beta,\phi)$ has the periodic optimization property.
\end{thmx}

There is a noteworthy parallel with the numerical observations of 
Yang, Hunt \& Ott: 
in \cite[Section~IV]{YHO00}, the Lorenz equations at 
standard parameter values exhibit failure of 
individual low-period optimality 
(the optimal orbit has anomalously high 
period), 
while in our beta-transformation
setting the failure is more severe, since for a dense 
set of parameters $\beta$ the map $T_\beta$ 
fails to have Individual TPO altogether.
In both cases, however, the joint setting restores the expected behaviour: 
Yang, Hunt \& Ott observed numerically that jointly varying 
the dynamical and performance function parameters yields typical 
low-period optimality, while Theorem~\ref{jtpo_tbeta} provides 
the first rigorous mathematical confirmation of the analogous 
joint phenomenon in the nonuniformly hyperbolic setting of 
beta-transformations.

From a technical perspective,
the proof of Theorem~\ref{jtpo_tbeta} is significantly
more challenging than for the uniformly hyperbolic analogues (Theorems~\ref{expandingJTPO}
and~\ref{anosovJTPO}): 
the strategy of proof is outlined in Subsection~\ref{subsection:heuristicoverview} below,
and is mediated by symbolic dynamics, exploiting 
amongst other things the fact that the orbit structure 
enjoys certain monotonicity properties as 
$\beta$ varies.
Despite the fact that not every $T_\beta$ has TPO,
we can
nevertheless prove Individual TPO results for \emph{most} parameters $\beta$, in the sense of the following two theorems:

\begin{thmx}[Individual TPO for generic parameters $\beta$]\label{typical_beta}
Fix $\alpha\in(0,1]$.
For a residual set of values $\beta>1$, the 
beta-transformation $T_\beta$ has TPO 
(i.e., there is an open dense subset  $V_\beta\subseteq\Holder{\alpha}(I)$ such that if $\phi\in V_\beta$ then $(T_\beta,\phi)$ has the periodic optimization property). 
\end{thmx}

\begin{thmx}[Individual TPO for almost every parameter $\beta$]\label{almost_every_beta}
Fix $\alpha\in(0,1]$.
For Lebesgue almost every $\beta>1$, the 
beta-transformation $T_\beta$ has TPO 
(i.e., there is an open dense subset  $V_\beta\subseteq\Holder{\alpha}(I)$ such that if $\phi\in V_\beta$ then $(T_\beta,\phi)$ has the periodic optimization property). 
\end{thmx}

\begin{remark}
    From a symbolic dynamics perspective, our Theorems~\ref{typical_beta} and~\ref{almost_every_beta} may be viewed as establishing TPO for a class of systems significantly broader than subshifts of finite type and sofic subshifts; the set of parameters $\beta$ for which the corresponding $\beta$-shift is sofic is merely countable
      (see e.g.~\cite[Proposition~4.2]{Bl89}), hence in particular a set that is both meagre and of zero Lebesgue measure.
\end{remark}

Here again,
a Joint TPO result is not a mere consequence of Individual TPO holding in most fibres\footnote{\label{nonmeagrecollinear}Specifically, the existence of a nonmeagre subset $A\subseteq[0,1]^2$ such that no three of its points are collinear (cf.~\cite[Theorem~15.5]{Ox71}, \cite[p.~54]{Ke95}) means that if $\phi\in\Holder{\alpha}(I)$ is not identically zero then the embedding $\iota\:[0,1]^2\to (1,+\infty)\times\Holder{\alpha}(I)$ given by $\iota(x,y)=(x+2,y\cdot\phi)$ is such that $\iota(A)$ is nonmeagre, yet the set
$\bigl\{\psi\in\Holder{\alpha}(I):(\beta,\psi)\in \iota(A) \bigr\}$ contains at most two points, hence is nowhere dense, for every $\beta\in(1,+\infty)$.}.
Conversely, the joint typical periodic optimization of Theorem~\ref{jtpo_tbeta} does 
not
automatically imply Theorem~\ref{typical_beta}: indeed the nonseparability of $\Holder{\alpha}(I)$ means there exists an open dense subset\footnote{Specifically, for each $x\in I$, set $f_x \coloneqq d(\cdot, x)^\alpha \in \Holder{\alpha}(I)$, then let $\psi \: I \to (1,+\infty)$ be a bijection, and define $\cL$ to be the complement of $\bigcup_{x\in I} \{ \psi(x) \} \times \overline{B}(f_x, 1/3)$.} $\cL$ 
of $(1,+\infty)\times\Holder{\alpha}(I)$ such that for all $\beta>1$, the fibre 
$\bigl\{\phi\in\Holder{\alpha}(I):(\beta,\phi)\in\cL \bigr\}$
is not even dense in $\Holder{\alpha}(I)$.
However, Theorem~\ref{jtpo_tbeta} does readily imply
(cf.~\cite[Lemma~8.42]{Ke95}) the following
Theorem~\ref{GPO_open_dense}, which can be considered
an analogue of Theorem~\ref{typical_beta}, 
with the roles of maps and (potential) functions interchanged.

\begin{thmx}[Individual TPO for generic potentials]\label{GPO_open_dense}

Fix $\alpha\in(0,1]$.
There is a residual subset $R\subseteq \Holder{\alpha}(I)$ such that for all $\phi\in R$,
there is an open and dense set of parameters
$B_\phi\subseteq(1,+\infty)$ such that if $\beta\in B_\phi$ then $(T_\beta,\phi)$ has the periodic optimization property. 
\end{thmx}

Theorems~\ref{expandingJTPO}, \ref{anosovJTPO}, \ref{jtpo_tbeta}, \ref{typical_beta}, \ref{almost_every_beta}, and~\ref{GPO_open_dense} of course all imply corresponding
\emph{typical uniqueness} results: an open dense subset 
for which the maximizing measure is \emph{unique}.
Such results are analogous to Ma\~n\'e's theorem
\cite[Theorem~A]{Man96} on uniqueness of the minimizing measure for generic Lagrangians, and (individual) generic uniqueness results in ergodic optimization (see e.g.~\cite[Theorem~6]{CLT01} and \cite[Theorem~3.2]{Je06}),
though these latter results assert uniqueness only on a residual subset, rather than the open dense subset implied by our theorems. 
One consequence of this \emph{joint typical uniqueness} 
is a \emph{large deviation principle} as $t\to+\infty$
for families of $(T,t\phi)$-equilibrium states,
 for an open dense subset of pairs $(T,\phi)\in \cE(X)\times \Holder{\alpha}(X)$
 (see Remark~\ref{LDPremark}).

\smallskip

\subsection{Joint perturbation formalism: heuristic overview of the proofs}\label{subsection:heuristicoverview}

The strategy for proving Joint TPO, in the various settings of expanding maps, Anosov diffeomorphisms, and beta-transformations, employs a novel framework that we refer to as 
\emph{joint perturbation}, involving simultaneous perturbation of both the dynamical system and the potential function, in order to yield periodic optimization.
For clarity, we summarise below the main ideas behind this approach: the joint perturbation formalism finds its most direct and streamlined form in the context of uniform hyperbolicity,  
while its application to beta-transformations,
requiring a substantially deeper and more exhaustive treatment, reveals joint typicality to be a robust phenomenon likely to persist across a wide range of dynamical settings.

\medskip

\begin{enumerate}
\smallskip
\item[\textbf{Case $\cH$:}] \emph{Uniform hyperbolicity. Joint perturbation formalism for Theorems~\ref{expandingJTPO} and~\ref{anosovJTPO}}. 

\smallskip

    \begin{enumerate}
\smallskip
\item[\textbf{Step~$\cH$.1:}] \emph{(Joint perturbation preparation)}. 
For open expanding maps,
our local connectedness hypothesis ensures suitable control of the surjectivity radius (Proposition~\ref{p.locallyconnected}), yielding the Locally Connected Shadowing Lemma (Lemma~\ref{l_perturbation_T_and_cO}),
guaranteeing a weak form of structural stability for perturbed systems, which is needed in Step~$\cH.2$ below.
For Anosov diffeomorphisms we directly use structural stability. 
 An enhanced \emph{Ma\~n\'e cohomology lemma} (Theorem~\ref{l.mane}) then gives a uniform $L=L(T,\alpha)>0$ such that  H\"older functions $\phi$ admit a H\"older sub-action $u$, with 
      $\Hseminorm{\alpha}{u} \le L \Hseminorm{\alpha}{\phi}$; the need for this additional control is a distinctive feature of Joint TPO (as compared to Individual TPO).
\smallskip
\item[\textbf{Step~$\cH$.2:}] \emph{(Joint perturbation realisation)}.  The joint perturbation theorems
(Theorems~\ref{t.criticaltheorem} and~\ref{jointperturbationuniformhyp}) are the key technical results towards 
Theorems~\ref{expandingJTPO} and~\ref{anosovJTPO}.
Fixing a map $T_0$, and letting $\cO_0$ be any of its periodic orbits, we then show that for all $T$ sufficiently close to $T_0$, there is a $T$-periodic orbit $\cO$ such that for suitable H\"older functions whose unique $T_0$-maximizing measure is supported by $\cO_0$, there is a nearby function whose unique $T$-maximizing measure is supported by $\cO$.
\smallskip
\item[\textbf{Step~$\cH$.3:}] \emph{(Joint periodic locking)}. 
For open expanding maps we define 
the \emph{joint periodic locking set}
\begin{equation*}
    \fL^\alpha (X)\=		
\big\{(T,\phi)\in \cE(X)\times \Holder{\alpha}(X) : \phi \in \Lock^\alpha(T)\big\}, 
\end{equation*}
where each individual locking set
$\Lock^\alpha(T)$
consists of
those $\alpha$-H\"older functions whose $T$-maximizing measure is unique, periodic, and stably maximizing under perturbations.
The Joint Perturbation Theorem
(Theorem~\ref{t.criticaltheorem}) is used to prove\footnote{The openness of $\fL^\alpha (X)$ itself 
(Theorem \ref{openjointperiodiclockingset}) is a noteworthy consequence of the proof method used here; by contrast the analogous joint periodic locking set for beta-transformations is \emph{not} an open subset of $(1,+\infty) \times \Holder{\alpha}(I)$,
though by Theorem~\ref{t.product.typical.periodic} it does \emph{contain} an open dense subset.
This surprising phenomenon for beta-transformations reflects the absence of structural stability \textit{per se},
and the partial compensation for this via the structural monotonicity (cf.~Step~$\beta.2$).}:

    \end{enumerate}
\end{enumerate}
        \smallskip

\begin{thmx}[Open joint periodic locking set]\label{openjointperiodiclockingset}
The joint periodic locking set
$\fL^\alpha (X)$ is open in $\cE(X)\times\Holder{\alpha}(X)$.
\end{thmx}

\begin{enumerate}
\smallskip \item[]
    \begin{enumerate}

To prove Theorem \ref{openjointperiodiclockingset} we consider subsets of the form $\{T\} \times B(\phi,\epsilon)$ such that the unique maximizing measure of
every $\psi \in B(\phi,\epsilon)$ is supported on the same periodic orbit. Using the joint perturbation theorem (Theorem \ref{t.criticaltheorem}), we prove that the piece $\{T\} \times B(\phi,\epsilon)$ can be promoted to a neighbourhood,
in the sense that there exists
$\delta \in (0,\epsilon)$ such that the open rectangle
$B_{\Lip}(T,\delta) \times B(\phi, \epsilon-\delta)$ is contained in the joint periodic locking set.
Contreras' Individual TPO Theorem can
then be used to show that 
$\fL^\alpha (X)$ is \emph{dense} in $\cE(X)\times\Holder{\alpha}(X)$, and Theorem~\ref{expandingJTPO} follows.
Notably, the approach used to prove
Theorem~\ref{expandingJTPO}
allows us to provide the following quantitative refinement of this result: 

   \end{enumerate}
\end{enumerate}
        \smallskip 

\begin{thmx}[Effective Joint TPO for expanding maps]\label{t.JTPO.explicit}
Suppose $X$ is a compact locally connected metric space, and $\alpha \in (0,1]$.
Suppose $(T_0, \phi)$
belongs to the joint periodic locking set $\fL^\alpha(X)$.
Then explicit constants $E,\, r>0$ can be given, such that if 
$(T,\psi) \in \cE(X) \times \Holder{\alpha}(X)$ satisfies $d_{\Lip}(T_0,T)<E$ and $\Hnorm{\alpha}{\phi-\psi}<r$, then $(T,\psi)\in \fL^\alpha(X)$.
\end{thmx}

\begin{enumerate}
\smallskip \item[]
    \begin{enumerate}

For Anosov diffeomorphisms an analogous argument is used, together with the joint perturbation theorem (Theorem~\ref{jointperturbationuniformhyp}), in order to prove Theorem~\ref{anosovJTPO};
an effective version 
of Theorem~\ref{anosovJTPO}
also follows, using the same method as for Theorem~\ref{t.JTPO.explicit}, but for brevity is omitted.
    \end{enumerate}
\bigskip

\end{enumerate}

\begin{rem}
The constituent parts of the joint perturbation formalism serve clearly distinguished purposes:
the preparatory step $\cH$.1
has (weaker but sufficient) versions
     $\beta$.2, $
   \beta$.3, and $\beta$.6 below,
   whereas the actual execution of joint perturbation
   (steps $\cH$.2 and $\cH$.3) corresponds to steps
   $\beta$.8 and $\beta$.9 below, where a more delicate joint perturbation argument revolves around
   non-simple beta-numbers.
   Note that steps $\beta$.1, $\beta$.2, $\beta$.3, $\beta$.4, $\beta$.5, $\beta$.7 can also be interpreted in a purely \emph{individual} sense, as their combined effect is to establish Individual TPO for $T_\beta$, for Lebesgue almost every $\beta>1$.
\end{rem}

\begin{enumerate}

\medskip

  \item[\textbf{Case $\beta$:}] \emph{Beta-transformations. Joint perturbation formalism for Theorems~\ref{jtpo_tbeta}, \ref{typical_beta}, \ref{almost_every_beta}, and~\ref{GPO_open_dense}.}
  \smallskip
    \begin{enumerate}
\smallskip
      \item[\textbf{Step~$\beta$.1:}] \emph{(Compactification)}.  A necessary first step is to develop a theory of ergodic optimization for the $T_\beta$,
accommodating the fact that these maps are not continuous, and that for certain $\beta$
the set $\MMM(I,T_\beta)$ of $T_\beta$-invariant probability measures is not compact.
The tool here is the
\emph{upper beta-transformation}
 $U_\beta\: I\to I$, 
defined 
by $U_\beta(0)\coloneqq 0$ and
$U_\beta (x )\coloneqq\beta x- \lfloor\beta x \rfloor'$
for $x\in (0,1]$,
where $\lfloor \beta x \rfloor' \coloneqq \max\{n\in\Z: n < \beta x\}$.
We show that
the set $\MMM(I,U_\beta)$ of $U_\beta$-invariant probability measures is weak$^*$ compact,
and equal to the closure of $\MMM(I,T_\beta)$,
so a $(U_\beta,\phi)$-maximizing measure exists whenever $\phi$ is continuous.
\smallskip
      \item[\textbf{Step~$\beta$.2:}] \emph{(Structural monotonicity)}. 
We exploit a fundamental property of beta-transformations: the structure of their corresponding symbolic dynamical systems (beta-shifts) has a certain monotonicity as $\beta$ varies (cf.~Lemma~\ref{l_apprioxiation_beta_shif_beta}). This property forms the basis for two steps: the first is simple beta-number perturbation (Step~$\beta.5$), the second is the Joint Perturbation Theorem for beta-transformations (Theorem~\ref{l_8_main_lemma}, cf.~Step~$\beta.9$). Corollary~\ref{l8.shadow}, as employed in Theorem~\ref{l_8_main_lemma}, represents an analogue of the Locally Connected Shadowing Lemma (Lemma~\ref{l_perturbation_T_and_cO}) of Step~$\cH.1$, 
and although the family of beta-transformations does not enjoy structural stability\footnote{As noted in Step~$\cH.1$, the structural stability of Anosov diffeomorphisms is exploited directly during the joint perturbation step. For open expanding maps, although it is possible to prove structural stability (a result that does not appear to be in the literature in the generality of open distance-expanding maps on compact locally connected spaces), it is more convenient to instead establish the Locally Connected Shadowing Lemma
(Lemma~\ref{l_perturbation_T_and_cO}).}, this weaker shadowing property is nevertheless suitable for exploitation in the proof of
Theorem~\ref{l_8_main_lemma}. 
\smallskip
      \item[\textbf{Step~$\beta$.3:}] \emph{(Beta-numbers)}. 
For certain $\beta$,
the \emph{critical orbit} (i.e.,~the orbit of $1$ under $U_\beta$) is finite: in this postcritically finite case $\beta$ is called a \emph{beta-number} (following \cite{Par60}).
A result underpinning the Joint TPO result (Theorem~\ref{jtpo_tbeta})   
is that, for \emph{upper} beta-transformations, Individual TPO holds for \emph{beta-numbers}
(note that the same is \emph{not} true for $T_\beta$, cf.~Remark~\ref{r_counter_example_simple_beta}):

    \end{enumerate}
\end{enumerate}
\smallskip

\begin{thmx}[Individual TPO for beta-numbers]\label{t_TPO_thm_beta_number}
	Fix $\alpha\in(0,1]$. If $\beta>1$ is a beta-number, then $U_\beta$ has TPO
    (i.e., there is an open dense subset  $V_\beta\subseteq\Holder{\alpha}(I)$ such that if $\phi\in V_\beta$ then $(U_\beta,\phi)$ has the periodic optimization property). 
\end{thmx}

\begin{enumerate}
\smallskip \item[]
    \begin{enumerate}

\smallskip 
\item[\textbf{Step~$\beta$.4:}] \emph{(Emergent parameters)}. 
To prove
    Theorems~\ref{typical_beta} and~\ref{almost_every_beta} we introduce 
    the notion of
    \emph{emergent} 
parameters $\beta$: 
for an emergent $\beta>1$, the symbolic dynamics 
for the critical orbit is essentially different from that witnessed in beta-shifts with parameter strictly smaller than $\beta$, so is considered to have newly \emph{emerged} at this particular $\beta$. 
For the complementary parameter set we prove:

    \end{enumerate}
\end{enumerate}
        \smallskip

\begin{thmx}[Individual TPO for non-emergent parameters]\label{t_TPO_thm_non_emergent}
Fix $\alpha\in(0,1]$.
If $\beta>1$  is non-emergent, then both
$T_\beta$ and $U_\beta$ have TPO
(i.e., there is an open dense subset  $V_\beta\subseteq\Holder{\alpha}(I)$ such that if $\phi\in V_\beta$ then both $(T_\beta,\phi)$ 
and $(U_\beta,\phi)$
have the periodic optimization property). 
\end{thmx}

\begin{enumerate}
\smallskip \item[]
    \begin{enumerate}
\smallskip\item[]

  The set of emergent parameters can be shown to be both topologically meagre and of zero Lebesgue measure, so
  Theorems~\ref{typical_beta} and~\ref{almost_every_beta}
  follow from Theorem~\ref{t_TPO_thm_non_emergent}.

\smallskip
\item[\textbf{Step~$\beta$.5:}] \emph{(Simple beta-number perturbation)}.
A key part of our overall strategy,
employed to prove Theorems~\ref{typical_beta}, \ref{almost_every_beta},
\ref{t_TPO_thm_beta_number},
and~\ref{t_TPO_thm_non_emergent},
is a perturbation argument in the space of parameters $\beta>1$. This exploits the fact that the dynamics for a given $\beta$ is approximable by sub-systems corresponding to
    so-called \emph{simple} beta-numbers in $(1,\beta)$.

\smallskip
\item[\textbf{Step~$\beta$.6:}] \emph{(A Ma\~n\'e lemma for beta-transformations)}. 
To prove Theorems~\ref{typical_beta}, \ref{almost_every_beta},
\ref{t_TPO_thm_beta_number},
and~\ref{t_TPO_thm_non_emergent},
the simple beta-number perturbation works in conjunction
with a new Ma\~n\'e cohomology 
lemma\footnote{The terminology \emph{Ma\~n\'e lemma}, first used by Bousch \cite{Bou00}, reflects the structurally similar result 
of Ma\~n\'e \cite{Man96} for Lagrangian flows, and the subsequent discovery of the unpublished \cite{CG95} justifies the
alternative terminology \emph{Ma\~n\'e--Conze--Guivarc'h lemma} \cite{Bou11,Mo09}; the result has also been called
a \emph{Bousch--Ma\~n\'e cohomology lemma} \cite{Je01}
and a \emph{normal form theorem} \cite{Je06}.} 
for beta-transformations
(Theorem~\ref{mane}).
Results of this kind
(cf.~Step~$\cH.1$ above)
have been 
known since the 1990s
(see \cite{Bou00,Bou01,Bou02,Bou08,Bou11, CLT01, CG95, LT03, Sa99}), 
and assert
that cohomology classes of suitably regular functions contain versions
(so-called \emph{revealed versions}, cf.~\cite{Je19})
 for which the maximizing measures are
readily apparent; in favourable settings the revealed version 
inherits the modulus of continuity of the original function
(see e.g.~\cite{Bou00,Bou01,BJ02,CLT01, LZ25}). 
Our Ma\~n\'e lemma
is for H\"older 
functions, 
and asserts the existence of \emph{two} revealed versions,
both enjoying one-sided continuity
(one is left-continuous, the other right-continuous): the critical orbit 
introduces discontinuities, but
away from this orbit both versions are locally H\"older.
In particular, when $\beta$ is a beta-number\footnote{In the case that $\beta$ is a \emph{simple} beta-number, it can be shown that there is a \emph{continuous} revealed version, cf.~Remark~\ref{post_mane_holder_remark}(ii).}, the revealed versions are piecewise H\"older (with finitely many pieces): this analogue of 
Step~$\cH.2$ is a key ingredient for the proof of Theorem~\ref{jtpo_tbeta}.
The 
proof of the beta-transformations Ma\~n\'e lemma
involves an operator fixed point
that is a Borel measurable function 
with one-sided limits everywhere, and locally H\"older away from the critical orbit: this fixed point can then be \emph{regularised}, yielding both a left-continuous and a right-continuous \emph{sub-action}, allowing the definition of left-continuous and right-continuous revealed versions.\footnote{Since beta-transformations are neither continuous nor open, but on the other hand they are transitive, the approach to proving this Ma\~n\'e lemma is rather different from, and more technical than, the one employed by \cite{LiSu26} to establish the
Ma\~n\'e lemma
used in Step~$\cH$.1; see Remark~\ref{Manetechnical} for more details.}
The one-sided continuity of these revealed versions allows us to deduce a Revelation Theorem (Theorem~\ref{l_subordination}),
asserting that 
any maximizing measure
is
supported within the union of the sets of maxima of these revealed versions; this serves as an important part of the following Step~$\beta.7$.

\smallskip
\item[\textbf{Step~$\beta$.7:}] \emph{(Critical-regular analysis)}.
The other key ingredient for proving
Theorems~\ref{typical_beta}, \ref{almost_every_beta},
\ref{t_TPO_thm_beta_number},
and~\ref{t_TPO_thm_non_emergent}
is the identification of
two fundamental subsets of $\Holder{\alpha}(I)$, 
the \emph{critical set} $\Crit^\alpha(\beta)$, and the \emph{regular set} $\RLS^\alpha(\beta)$.
The first of these sets consists of functions for which the critical orbit is maximizing, and the 
second consists of those functions enjoying good restrictions
to various
Cantor subsets on which $T_\beta$ acts as an open expanding map: by
exploiting Contreras' Individual TPO theorem \cite{Co16} for such maps,
we are able to
 prove the Dense Regular Functions Theorem
(Theorem~\ref{p_density_local_locking}), asserting that $\RLS^\alpha(\beta)$ is dense
in $ \Holder{\alpha}(I)$. This, together with the Revelation Theorem (Theorem~\ref{l_subordination}),
yields a proof of
Theorems~\ref{t_TPO_thm_beta_number} and
\ref{t_TPO_thm_non_emergent}.\footnote{Note that, by contrast with \cite{Co16,HLMXZ25, LZ25}, the method for proving our Individual TPO theorems does not make explicit use of shadowing (indeed the shadowing property does not hold for beta-transformations, cf.~\cite{BGS25}).}
  It is an open problem to determine whether $U_\beta$ has TPO for emergent parameters $\beta$: here
  a finer analysis  
leads to a Structural Theorem
 (Theorem~\ref{t_TPO_thm_emergent}),
identifying $\Crit^\alpha(\beta)$
as a potential obstacle to Individual TPO.

\smallskip
\item[\textbf{Step~$\beta$.8:}] \emph{(Joint TPO for beta-transformations)}.
Having established the above Individual TPO theorems for beta-transformations, the conversion of this into the Joint TPO result Theorem~\ref{jtpo_tbeta} is patterned, to some extent, on the approach used
in Steps $\cH.1$--$\cH.3$ to prove Theorems~\ref{expandingJTPO} and~\ref{anosovJTPO}.
For beta-transformations, however,
the technical analysis
required 
is significantly more delicate, and the fine
structure of this particular family of maps
must be exploited.
The strategy consists of proving
a more general Theorem~\ref{t.product.typical.periodic},
asserting that Joint TPO holds for upper beta-transformations as well as for beta-transformations:
in each case we show that the joint periodic locking set 
contains an open dense subset 
of $(1,+\infty)\times\Holder{\alpha}(I)$.

\smallskip
\item[\textbf{Step~$\beta$.9:}] \emph{(Joint perturbation realisation)}.
The proof of Theorem~\ref{t.product.typical.periodic} relies on the Joint Perturbation Theorem for beta-transformations
(Theorem~\ref{l_8_main_lemma}), the most technically challenging result of the article.
Although similar in spirit to the
analogous joint perturbation theorems for expanding maps
(Theorem~\ref{t.criticaltheorem})
and Anosov diffeomorphisms (Theorem~\ref{jointperturbationuniformhyp}),
the additional difficulties encountered in proving
Theorem~\ref{l_8_main_lemma} relate to the non-persistence of certain periodic orbits, the discontinuous variation of some points under perturbation of the parameter, the weaker Ma\~n\'e lemma (Theorem~\ref{mane}), and the consequent difficulties in controlling ergodic averages.\footnote{For further details see in particular the comments in footnotes~\ref{Claim3remarks}
and~\ref{subcase(ii)remarks} concerning Claim 3 and Subcase~(ii) in the proof of Theorem~\ref{l_8_main_lemma}.}
The proof of Theorem~\ref{l_8_main_lemma} 
centres on two key features: on the one hand
\emph{non-simple} beta-numbers $\beta$ (these are 
shown to be dense in parameter space, and their specific properties 
are systematically exploited throughout the proof),
and on the other hand a perturbation using the beta-transformations Ma\~n\'e lemma (Theorem~\ref{mane}).
Once Theorem~\ref{t.product.typical.periodic} is proved, 
Theorem~\ref{GPO_open_dense} follows readily.

    \end{enumerate}
\end{enumerate}

\begin{rem}
The joint perturbation methodology developed throughout this 
article is taken up in \cite{HHJL26}, where an axiomatic 
framework of hyperbolic maps is introduced and applied to 
establish Joint TPO for various further systems, including 
Axiom~A diffeomorphisms, hyperbolic rational maps, and 
one-dimensional $C^r$ maps.
It should be noted, however, that the class of beta-transformations treated here is not amenable to the methods of
\cite{HHJL26}, and in general the effective Joint TPO results such as Theorem~\ref{t.JTPO.explicit} require the approach given here.
\end{rem}

\subsection{Organisation of the article}

Some key notation and terminology used throughout the article is fixed in Subsection~\ref{sct_Notation} below,
including the notion of periodic locking set
$\Lock^\alpha(T)$,
consisting of
those $\alpha$-H\"older functions whose maximizing measure is unique, periodic, and stably maximizing under perturbations.
In Section~\ref{sec_JTPO_distance_expanding} we 
introduce the space $\cE(X)$ of open Lipschitz distance-expanding maps on compact locally connected spaces $X$,
and for all H\"older function spaces $\Holder{\alpha}(X)$
establish joint typical periodic optimization in the product space $\cE(X) \times \Holder{\alpha}(X)$.
This is proved as the slightly stronger Theorem~\ref{t.JTPO.general},
asserting that
$		
\big\{(T,\phi)\in \cE(X)\times \Holder{\alpha}(X) : \phi \in \Lock^\alpha(T)\big\}$ 
	is itself an open dense subset, where openness is proved in Theorem~\ref{openjointperiodiclockingset};
    the effective version of Joint TPO is proved as 
    Theorem~\ref{t.JTPO.explicit}.
The special case when $X$ is a compact Riemannian manifold follows (Theorem~\ref{expandingJTPOmanifolds}),
and in this setting we also establish Joint TPO results
for $\cT$ the space of 
Anosov diffeomorphisms\footnote{The presentation for Anosov diffeomorphisms, 
in Subsection~\ref{anosovsubsection}, is deliberately 
less detailed than for expanding maps and beta-transformations: 
certain arguments and estimates follow the same pattern as 
in the expanding map case, and a more 
comprehensive treatment of Joint TPO for general hyperbolic 
systems is given in \cite{HHJL26}.}: a slight strengthening of Theorem~\ref{anosovJTPO} is proved (as Theorem~\ref{anosovJTPO.moreprecise}), as well as a Joint TPO theorem for $C^1$ function spaces (Theorem~\ref{anosovJTPO_C1}).
Thereafter we focus on the more technically challenging setting of beta-transformations. Section~\ref{betashiftsection}, whose proofs are
deferred to Appendix~\ref{beta_proofs_section}, is preparatory in nature.
Firstly, we establish various
preliminary results about beta-transformations, beta-expansions,
and the closely related beta-shifts, being particularly careful to distinguish the commonalities and differences between these systems; in so doing, we take the opportunity to clarify and correct some aspects of the published literature.
Secondly, in order to develop ergodic optimization in this setting, we undertake
 an analysis of the set 
$\MMM(I,T_\beta)$ of $T_\beta$-invariant measures.
We prove the existence of $(U_\beta,\phi)$-maximizing measures
for all values $\beta>1$, and all continuous functions $\phi$; we also introduce the notion of limit-maximizing 
measure, and establish the equivalence between 
maximizing measures for the upper beta-transformation and the beta-shift, and 
limit-maximizing measures for beta-transformations.
 In Section~\ref{Sec_mane_lemma} we establish the important Ma\~{n}\'{e} cohomology lemma
for beta-transformations, and develop a revelation theorem as its consequence. 
Section~\ref{sec_proof_of_main_results} 
is devoted to proving our Individual TPO results for beta-transformations, and includes proofs of Theorems~\ref{typical_beta}, \ref{almost_every_beta},
\ref{t_TPO_thm_beta_number},
and~\ref{t_TPO_thm_non_emergent}.
To establish these results,
the class of emergent parameters $\beta$ is introduced,
the critical subset $\Crit^\alpha(\beta)$ and
regular subset $\RLS^\alpha(\beta)$ of $\Holder{\alpha}(I)$ are defined, and the Ma\~n\'e lemma of Section~\ref{Sec_mane_lemma} is a fundamental tool.
In Section~\ref{JTPO_beta_section} we establish joint typical periodic optimization in the setting of beta-transformations: Theorem~\ref{jtpo_tbeta} is proved, via the slightly stronger Theorem~\ref{t.product.typical.periodic}, showing 
that both 
$\bigl\{ (\beta,\phi)\in (1,+\infty)\times \Holder{\alpha}(I) : \phi\in \Lock^\alpha(T_\beta) \bigr\}$
and
$\bigl\{ (\beta,\phi)\in (1,+\infty)\times \Holder{\alpha}(I) : \phi\in \Lock^\alpha(U_\beta) \bigr\}$
contain 
    open dense subsets of $(1,+\infty) \times \Holder{\alpha}(I)$.
Theorem~\ref{t.product.typical.periodic} is itself a consequence of the key joint perturbation result
Theorem~\ref{l_8_main_lemma}.
Theorem~\ref{GPO_open_dense} follows readily
from Theorem~\ref{t.product.typical.periodic}.
We conclude Section~\ref{JTPO_beta_section} with a question concerning a variant notion of Joint TPO in the setting of symbolic dynamics.
	Appendix~\ref{beta_proofs_section} is devoted to the proofs of the results stated in Section~\ref{betashiftsection}.

	\subsection{Notation}\label{sct_Notation}

We follow the convention that $\N \= \{1, \, 2, \, 3, \, \dots\}$ and $\N_0 \= \{0\} \cup \N$. For $x\in\R$, we define the floor function $\lfloor x\rfloor$ as the largest integer $\leq x$, and the
strict floor function $\lfloor x \rfloor'$ as the largest integer $<x$.
The cardinality of a set $A$ is denoted by $\card A$.

The collection of all maps from a set $X$ to a set $Y$ is denoted by $Y^X$. The constant zero function $\mathbbold{0} \: X \to \R$ maps each $x\in X$ to $0$.

Let $(X,d)$ be a metric space. For subsets $A,B\subseteq X$, we set $d(A,B) \= \inf \{d(x,y) : x\in A,\,y\in B\}$, and $d(A,x)=d(x,A) \= d(A,\{x\})$ for each $x\in X$. For each subset $Y\subseteq X$, we denote the diameter of $Y$ by $\diam(Y) \= \sup\{d(x,y) : x, \, y\in Y\}$,
for each $\myepsilon>0$, denote the $\myepsilon$-neighbourhood of $Y$ (in $X$) by
$\nbhd (Y,\myepsilon )\= \{x\in X:d (x,Y )<\myepsilon \}$,
and the closure of $\nbhd (Y,\myepsilon )$ by $\overline\nbhd(Y,\myepsilon)$.
For each $y\in X$ and each $\myepsilon>0$, write $\nbhd(y,\myepsilon)\=\nbhd(\{y\},\myepsilon)$.

For a finite nonempty set $F \subseteq X$, let $\Delta(F)$ denote its \emph{minimum interpoint distance}, i.e., $\Delta(F) \= \min \{ d(x,y) : x, \, y \in F, \, x\neq y\}$ if $\card F\ge 2$ and $\Delta(F) \= +\infty$ if $\card F=1$.

Let $\CCC(X)$ denote the space of continuous functions from $X$ to $\R$,
and 
$\PPP(X)$ the set of Borel probability measures on $X$. 
For $\phi \: X\to \R$, we write 
$\Hnorm{\infty}{\phi}=
\Hnorm{\infty, X}{\phi}\=\sup\{\abs{\phi(x)}:x\in X\}$.

Let $(X,d )$ be a compact metric space and $\alpha\in(0,1]$. A function $\phi\:X\to \R$ is called  
$\alpha$-H\"older if
\begin{equation*}
	\Hseminorm{\alpha,X}{\phi} \= \sup \{ \abs{\phi(x) - \phi(y)} / d(x,y)^\alpha : x, \, y \in X, \, x\neq y \}<+\infty.
\end{equation*}
Denote by $\Holder{\alpha} (X )$ the set of real-valued $\alpha$-H\"older functions $\phi$ on $X$, equipped
with the H\"older norm $\Hnorm{\alpha, X}{\cdot}$ given by
\begin{equation*}
	\Hnorm{\alpha, X}{\phi}\=\Hseminorm{\alpha,X}{\phi}+ \Hnorm{\infty, X}{\phi},
\end{equation*} 
which makes $\Holder{\alpha} (X )$ a Banach space. 
It is often convenient to write $\Hnorm{\alpha}{\cdot}$
instead of $\Hnorm{\alpha, X}{\cdot}$:
this will be done throughout Section~\ref{sec_JTPO_distance_expanding}, 
and is usually done in the case that $X=[0,1]$ (the exception being in Section~\ref{sec_proof_of_main_results}, where $X$ is chosen to be various different subsets of $[0,1]$).

For a Borel measurable map $T\: X \to X$ on a compact metric space $X$, let $\MMM(X,T)$ denote the set of $T$-invariant Borel probability measures on $X$, and
	define
    the \emph{ergodic supremum} of a bounded Borel measurable function $\psi \: X\to \R$ to be 
	\begin{equation}\label{e_ergodicmax}
		\mpe(T,\psi)\=\sup\limits_{\nu\in\MMM(X,T)}\int\! \psi \, \mathrm{d}\nu.
	\end{equation}
	Any measure $\mu \in \MMM(X,T)$ that 
attains the supremum in (\ref{e_ergodicmax}) is called a \emph{$(T,\psi)$-maximizing measure}, and
it will also be convenient to refer to $\mu$ as a
 \emph{$\psi$-maximizing measure} (for the map $T$),
 and as a \emph{$T$-maximizing measure} (for the function $\psi$).
The set of $(T,\psi)$-maximizing measures is denoted by
	\begin{equation}\label{e_setofmaximizngmeasures}
		\Mmax(T,  \psi) \=\biggl\{ \mu\in\MMM(X,T) : \int \! \psi \,\mathrm{d} \mu = \mpe(T,\psi) \biggr\}.
	\end{equation}

 The orbit of a point $x\in X$ is called a \emph{maximizing orbit} for $(T,\psi)$ (or \emph{$(T,\psi)$-maximizing}) if 
\begin{equation}\label{maximizing_orbit_eq}
   	\lim\limits_{n\to +\infty}
\frac{1}{n} \sum\limits_{i=0}^{n-1}\psi \bigl( T^i(x) \bigr)
=\mpe(T,\psi) .
\end{equation}

For a map $T\: X\to X$ and a real-valued function $\phi\: X\to \R$, define
\begin{equation*}
S_n^T \phi(x) \= \sum_{i=0}^{n-1} \phi \bigl(T^i(x) \bigr) \quad \text{ for } x\in X, n\in \N .
\end{equation*}
Note that by definition $S_0^T \phi \equiv0$. 
In the particular case where $T$ is some beta-transformation $T_\beta$, we will write $S_n \phi = S_n^{T_\beta} \phi$ whenever there is no possibility of confusion.

We write $I \= [0,1] $. In this article, we equip every subset of $I$ with the usual Euclidean metric, denoted by $d$.

For a nonempty set $\cA$ equipped with the discrete topology, 
$\cA^\N=\{ \{a_n\}_{n=1}^{+\infty}: a_n\in \cA\text{ for all }n\in\N\}$ will be equipped with the product topology.
For $t>1$,  the metric $d_t$ on $\cA^\N$ defined
by $d_t(\{a_n\}_{n=1}^{+\infty},\{b_n\}_{n=1}^{+\infty}) \= t^{-p}$, where $p$ is the smallest positive integer with $a_p \neq b_p$, and $d_t(\{a_n\}_{n=1}^{+\infty},\{b_n\}_{n=1}^{+\infty}) \= 0$ if $a_n=b_n$ for all $n\in \N$,
generates the product topology on $\cA^\N$.

Infinite sequences will be written as  $A= a_1 a_2 \ldots = \{a_n\}_{n\in\N}$, and finite sequences as $B = b_1 b_2 \dots b_k = \{b_n\}_{n=1}^{k}$. Denote $(b_1 b_2 \dots b_k)^{\infty} \= b_1 b_2 \dots b_k b_1 b_2 \dots b_k b_1 b_2 \dots$ and write $(b_1 b_2 \dots b_k)^m$ for the first $km$ terms of $(b_1 b_2 \dots b_k)^{\infty}$, for $m\in \N$.

If $\cA\subseteq\R$ is equipped with the order induced by $\R$, and 
$A, B\in \cA^\N$, write $A\prec B$ when $A$ has strictly
smaller lexicographic order than  $B$, i.e.,~$a_i=b_i$ for $1\le i\le n-1$, and $a_n<b_n$, for some $n\in \N$.
Write $A\preceq B$ to mean $A\prec B$ or $A=B$.

Define the (left) shift map 
\begin{equation*}
\sigma \: \cA^\N \to \cA^\N
\end{equation*}
by $\sigma  ( A  )\= \{ a_{n+1}\}_{n\in \N} $
for all
$A = \{ a_n\}_{n\in \N} \in \cA^\N$.

Let $X$ be a topological space. For a map $T\:X\to X$ and $x\in X$, denote the orbit of $x$ by
\begin{equation*}
	\cO^T (x )\= \{T^n (x ) : n\in \N_0\} .
\end{equation*}
In this article, we write $\cO_\beta(x)\=\cO^{T_\beta}(x)$ and $\cO^*_\beta(x)\=\cO^{U_\beta}(x)$ whenever there is no possibility of confusion. If there exists $n\in \N$ with $T^n (x )=x$, then  
$\cO (x )$ is called a periodic orbit and $x$ a periodic point. We denote the set of periodic points of $T$ by $\Per(T)$. There is a unique  $T$-invariant Borel probability measure $\mu_{\cO}$ supported on a periodic orbit $\cO$,
given by
\begin{equation*}
	\mu_{\cO} \= \frac{1}{\card \cO} \sum_{x\in \cO} \delta_x.
\end{equation*}
For a topological space $X$, a point $a \in X$, and a function $f \: X \to \R$, we write $\lim_{y \to a} f(y) = x^+$ if $\lim_{y \to a} f(y) = x$ and there exists a neighbourhood $U$ of $a$ such that $f|_{U \smallsetminus \{a\}}\geq x$. We define $\lim_{y \to a} f(y) = x^-$ similarly. We say that a sequence of real numbers $\{ x_n \}_{n\in \N}$ converges to a real number $x^+ $ (written as $\lim_{n\to +\infty} x_n= x^+$) if $x_n \geq x$ for each $n\in \N$ and $\lim_{n\to +\infty} x_n= x$. We define $\lim_{n\to +\infty} x_n= x^-$ similarly. Moreover, for a real number $a$, and a function $g \: \R \to \R$, we denote $\lim_{y \searrow a} g(y) \= \lim_{y \to a} g|_{(a, +\infty)}(y)$, that is, the right-hand limit of $g$ at $a$. We denote the left-hand limit $\lim_{y \nearrow a} g(y) \= \lim_{y \to a} g|_{(-\infty,a)}(y)$ similarly. More generally, if $\R$ is replaced by a topological well-ordered set $Y$, the above notions and notation are defined similarly.

Fix a constant $\alpha \in (0,1]$, a compact metric space $X$, and a measurable map $T\: X \to X$. We define subsets 
\begin{equation*}
	\sP^\alpha(T)   \quad \text{ and } \quad
	\Lock^\alpha(T) 
\end{equation*}
of the set $\Holder{\alpha}(X)$ as follows: $\sP^\alpha(T)$ is the set of those $\phi \in \Holder{\alpha}(X)$ with a $\phi$-maximizing measure supported on a periodic orbit of $T$. If a function $\phi \in \sP^\alpha(T)$ satisfies $\card \Mmax(T,  \phi) = 1$ and $\Mmax(T,  \phi) = \Mmax(T,  \psi)$ for all $\psi \in  \Holder{\alpha}(X)$ sufficiently close to $\phi$ in $\Holder{\alpha}(X)$, we say that $\phi$ has the \emph{(periodic) locking property}\footnote{The terminology follows \cite{Boc19, BZ15}, and is somewhat inspired by~\cite{Bou00, Je00}.} in $\Holder{\alpha}(X)$ (with respect to $T$). The set $\Lock^\alpha(T)$ is defined to consist of all $\phi \in \sP^\alpha(T)$ satisfying the (periodic) locking property in $\Holder{\alpha}(X)$.

\section{Joint TPO: Uniform hyperbolicity}\label{sec_JTPO_distance_expanding}

In this section, we establish joint typical periodic optimization in the context of uniformly hyperbolic systems (the space of Lipschitz open distance-expanding maps, and the space of $C^1$ Anosov diffeomorphisms).

\subsection{Expanding maps}

Recall the following notion (see e.g.~\cite[Chapter~4]{PU10}):
\begin{definition}[Distance-expanding map]\label{distance_expanding_map_defn}
	For $ (X,d )$ a compact metric space, $T\:X\to X$ is called a \emph{distance-expanding map} if there exist constants $\lambda>1$ and $\gamma>0$ such that if $x, \, y \in X$ with $d (x,y ) \in (0, \gamma)$, then
	$d (T (x ),T (y ) ) > \lambda d (x,y )$.
    
    We refer to $(\gamma, \lambda)$ as a pair of \emph{expanding constants} for the distance-expanding map $T$.
\end{definition}

We shall be interested in the class of distance-expanding maps that are, moreover, Lipschitz and open (i.e.,~map open sets to open sets):

	\begin{notation}
	Suppose $(X,d)$ is a compact 
    metric space. A map $T\:X\rightarrow X$ is \emph{Lipschitz (continuous)} if its \emph{Lipschitz constant}
    \begin{equation}\label{lipschitzconstant}
		\LIP(T) \= \sup_{x,\, y \in X, \, x\neq y} \frac{d(T(x),T(y))}{d(x,y)}
	\end{equation}
    is finite. The set of Lipschitz continuous open distance-expanding maps $T\:X\rightarrow X$ will be denoted by $\cE(X)$, and equipped with the metric\footnote{It is readily checked that $d_{\Lip}$ defines a complete metric.} 
   	\begin{equation}\label{e.Lip.metric}
     \begin{aligned}
		&d_{\Lip}(T_1,T_2)\\
        &\qquad \= \sup_{x\in X} d(T_1(x),T_2(x))+\sup_{x,y\in X,x\neq y} \frac{\abs{d(T_1(x),T_1(y))-d(T_2(x),T_2(y))}}{d(x,y)} 
     \end{aligned}
	\end{equation}
	for $T_1,\,T_2\in \cE(X)$. For each $T \in \cE(X)$ and each $\alpha\in (0,1]$, define the \emph{joint (periodic) locking set} $\fL^\alpha (X)$ by
	\begin{equation}
		\fL^\alpha (X)\=\big\{(T,\phi)\in \cE(X)\times \Holder{\alpha}(X) : \phi \in \Lock^\alpha(T)\big\} .
	\end{equation}
\end{notation}

At this level of generality, open distance-expanding maps
were investigated by Ruelle\footnote{Emanating from his theory of Smale spaces, Ruelle \cite[p.~149]{Ru78} notes that this theory of open distance-expanding maps is more general (and thus less rich) than the theory of smooth expanding maps given by Shub \cite{Sh69} and Hirsch \cite{Hirs70}. In particular, in this generality there is no classification of those $X$ admitting an open distance-expanding map.} \cite[\S 7.26]{Ru78}, 
and systematically studied by
Przytycki \& Urba\'nski \cite[Chapter~3]{PU10} (see also 
\cite[Chapter~4]{URM22}, \cite[Chapter~11]{VO16});
it is an open problem to determine which 
compact metric spaces admit open distance-expanding maps, though partial results abound
(see e.g.~\cite{CM88, Gr81, Hira90a, Hira90b, Hirs70, Nek14, Nek20, Sh70}).
Our consideration of the space $\cE(X)$, equipped with the complete metric $d_{\Lip}$, appears to be new, as is the exploitation of local connectedness (see Proposition~\ref{p.locallyconnected} and Lemma~\ref{l_perturbation_T_and_cO}).

\smallskip

We wish to
establish the following result, which in particular implies 
Theorem~\ref{expandingJTPO}:

\begin{thml}[Joint TPO for expanding maps] \label{t.JTPO.general}
	If $X$ is a compact locally connected metric space and $\alpha \in (0,1]$, then the joint periodic locking set $\fL^\alpha(X)$ is an open and dense subset of $\cE(X) \times \Holder{\alpha}(X)$.
\end{thml}

\subsection{Joint perturbation preparation: local connectedness, shadowing, and a Ma\~n\'e lemma}

To prove Theorem~\ref{t.JTPO.general}, we require a number of preparatory results.	
Firstly, we 
show that local connectedness of $X$ affords a degree of control over the surjectivity radius $\delta$, a result which will ultimately lead to the key Locally Connected Shadowing Lemma (Lemma~\ref{l_perturbation_T_and_cO}) on the existence of shadowing orbits for perturbed maps.

\begin{prop}\label{p.locallyconnected}
	Let $X$ be a compact locally connected metric space. Given $\gamma>0$, there exists $\delta \in (0,\gamma)$ such that for all $\lambda>1$, if $T\in \cE(X)$ has expanding constants $(\gamma, \lambda)$, then $B(T(x) , \delta) \subseteq T(B(x,  \delta/\lambda))$ for all $x\in X$.
\end{prop}

\begin{proof}
	As $X$ is compact and locally connected, there exists a finite open cover $\cU$ of $X$ consisting of connected open subsets, each with diameter smaller than $\gamma$. Let $\lambda>1$, and suppose $T\in \cE(X)$ has a pair of expanding constants $(\gamma,\lambda)$. For each $y\in X$ and $U\in \cU$ containing $x \= T(y)$, define $V \= T(B(y, \gamma) ) \cap U$. As $T$ is an open map, $V$ is an open subset of $X$, therefore also an open subset of $U$. Since $T$ is distance-expanding, if $z\in B(y,\gamma) \smallsetminus \oB (y,  \gamma/\lambda)$ then $d(x, T(z)) > \gamma$, so $T \bigl( B(y,\gamma) \smallsetminus \oB(y,  \gamma/\lambda) \bigr) \cap U= \emptyset$, therefore $V = T\bigl( \oB(y, \gamma/\lambda) \bigr) \cap U$. Since $T$ is continuous and $\oB(y, \gamma/\lambda)$ is compact, $T \bigl( \oB(y, \gamma/\lambda) \bigr)$ is also compact. As $X$ is compact and Hausdorff, $T \bigl( \oB(y, \gamma/\lambda) \bigr)$ is a closed subset of $X$, so $V$ is also a closed subset of $U$. As $U$ is connected and $V \neq \emptyset$, we have $V=U$. 
	
	Now let $\delta>0$ be a Lebesgue number for $\cU$; clearly $\delta <\gamma$, since the diameters of all members of $\cU$ are smaller than $\gamma$. For each $z\in X$, let $x \= T(z)$. For each $y\in B(x,\delta)$, there exists $U \in \cU$ such that $x, \, y \in U$. As $U \subseteq T(B(z,\gamma))$, there exists $w \in T^{-1}(y) \cap B(z,\gamma)$, and the distance-expanding property of $T$ means that $w\in B(z,\delta/\lambda) $. 
    That is, $B(T(z), \delta) \subseteq T(B(z, \delta/\lambda))$ for any $z\in X$, as required. 
\end{proof}

\begin{lemma}\label{l.dLip}
Let $X$ be a compact metric space, and
suppose $T_0\in \cE(X)$ has expanding constants $(\gamma, \lambda)$. If $T\in \cE(X)$ with $d_{\Lip}(T_0,T)< \lambda-1$, then $(\gamma, \lambda- d_{\Lip}(T_0,T))$ is a pair of expanding constants for $T$.
\end{lemma}

\begin{proof}
	For all $x, \, y\in X$ with $d(x,y)<\gamma$, 
    from the definition of $d_{\Lip}$ (cf.~(\ref{e.Lip.metric})) we see that
	\begin{equation*}
		\begin{aligned}
			d(T(x), T(y)) &\ge d(T_0(x), T_0(y)) - \abs{d(T_0(x), T_0(y)) - d(T(x), T(y))} \\
			&> \lambda d(x,y) - d_{\Lip}(T_0,T) d(x,y)
            =(\lambda-d_{\Lip}(T_0,T))d(x,y),
		\end{aligned}
	\end{equation*}
	and $\lambda - d_{\Lip}(T_0,T) >\lambda -(\lambda-1) =1$, so 
indeed $(\gamma, \lambda- d_{\Lip}(T_0,T))$ is a pair of expanding constants for $T$, as required.
\end{proof}

Following \cite{PU10}, for any $\alpha>0$ a sequence 
$\{x_n\}_{n\in \N}$ of points in $X$ will be called an 
\emph{$\alpha$-pseudo-orbit} if $d(x_{n+1}, T(x_n) ) \le \alpha$ for all $n\in \N$. 
We say that $\{x_n\}_{n=0}^{k-1}$ is a
\emph{periodic $\alpha$-pseudo-orbit} 
if $d(x_{n+1 \pmod k},T(x_n)) \le \alpha$ for all $0\le n\le k-1$.
Given an $\alpha$-pseudo-orbit $\{x_n\}_{n\in \N}$, and $\beta>0$, a $T$-orbit $\{y_n\}_{n\in \N}$ 
satisfying  $d(x_n,y_n)\le \beta$ for all $n\in \N$
will be called a 
\emph{$\beta$-shadowing orbit} of $\{x_n\}_{n\in \N}$; in this case we also say that $\{y_n\}_{n\in \N}$ 
\emph{$\beta$-shadows} $\{x_n\}_{n\in \N}$.

We require the following result regarding the existence of shadowing orbits.\footnote{Note that the key hypothesis that $B(T(x),\delta) \subseteq T(B(x, \delta/\lambda))$ in our Lemma~\ref{l.shadowing} is implicit rather than explicit in the statement of
\cite[Lemma~4.2.3]{PU10}, having been chosen earlier
(in \cite[Section~4.1]{PU10}).}

\begin{lemma}[{\cite[Lemma~4.2.3]{PU10}}]\label{l.shadowing}
	Let $X$ be a compact metric space, and suppose $T \in \cE(X)$ has a pair of expanding constants $(\delta, \lambda)$ satisfying $B(T(x),\delta) \subseteq T(B(x, \delta/\lambda))$ for all $x\in X$.  
    If $\beta\in (0,\delta)$ and $\tau\in (0, \min \{(\lambda-1)\beta , \, \delta\}]$, then every $\tau$-pseudo-orbit has a $\beta$-shadowing orbit.
\end{lemma}

The following shadowing result, exploiting local connectedness of $X$, is fundamental to our 
method of proof in Theorem~\ref{t.criticaltheorem}.

\begin{lemma}[Locally Connected Shadowing Lemma]\label{l_perturbation_T_and_cO}
	Suppose $X$ is a compact locally connected metric space, $T_0 \in \cE(X)$ has expanding constants $(\gamma, \lambda)$, and $\cO_0$ is a periodic orbit for $T_0$. There exists $D=D(T_0,\cO_0)>0$ 
    such that if $T \in \cE(X)$ with $d_{\Lip}(T_0,T)<D$, then the following hold:
		\begin{enumerate}[label=\rm{(\roman*)}]
		\smallskip
		\item $\LIP(T) \le 2\LIP(T_0)$.
		\smallskip
		
		\item $(\gamma, \lambda- d_{\Lip}(T_0,T))$ is a pair of expanding constants for $T$.
		\smallskip
		
		\item $\cO_0$ is a periodic $d_{\Lip}(T_0,T)$-pseudo-orbit for $T$. 
		\smallskip
		
		\item There exists a $T$-periodic orbit $\cO$ that 
        $\bigl(\frac{d_{\Lip}(T_0,T)}{\lambda-d_{\Lip}(T_0,T)-1}\bigr)$-shadows $\cO_0$, and whose minimum interpoint distance\footnote{Recall (cf.~Subsection~\ref{sct_Notation}) that $\Delta(F) \coloneqq \min \{ d(x,y) : x, \, y \in F, \, x\neq y\}$ if $\card F\ge 2$ and $\Delta(F) \coloneqq +\infty$ if $\card F=1$.}
        satisfies $\Delta(\cO)\geq \frac{1}{2}\Delta(\cO_0)$.  
	\end{enumerate}
    Moreover, if $\card \cO_0 >1$ we can choose 
    	\begin{equation}\label{e.def.D}
					D  \= \min \{ \delta,  \, \LIP(T_0), \,  (\lambda-1) \delta / (1+ \delta) , \,  (\lambda-1) \Delta(\cO_0) / (4+ \Delta(\cO_0)) \},
		\end{equation}
        and if $\card \cO_0 =1$ we can choose 
        \begin{equation}\label{e.def.D.singleton}
					D  \= \min \{ \delta,  \, \LIP(T_0), \,  (\lambda-1) \delta / (1+ \delta)   \},
		\end{equation}
    where $\delta\in (0,\gamma)$ is a surjectivity radius as in Proposition~\ref{p.locallyconnected}, depending only on $T_0$.
\end{lemma}

\begin{proof}
Suppose $T \in \cE(X)$ is such that $d_{\Lip}(T_0,T)<D$.
	
	(i) As $D \le \LIP(T_0)$ (cf.~(\ref{e.def.D}), (\ref{e.def.D.singleton})), by (\ref{e.Lip.metric}), for all $x, \, y\in X$ we have
	\begin{equation*}
		\begin{aligned}
			d(T(x), T(y)) &\le d(T_0(x), T_0(y)) + \abs{d(T_0(x), T_0(y))-d(T(x), T(y))} \\
			&\le (\LIP(T_0) +d_{\Lip}(T_0,T))  d(x,y) \le 2\LIP(T_0) d(x,y),
		\end{aligned}
	\end{equation*}
	and thus $\LIP(T) \le 2\LIP(T_0)$, as required.

  (ii) \& (iii)	Now $d_{\Lip}(T_0,T)<D  < \lambda-1$ (cf.~(\ref{e.def.D}), (\ref{e.def.D.singleton})), so applying Lemma~\ref{l.dLip} gives that $(\gamma, \lambda- d_{\Lip}(T_0,T))$ is a pair of expanding constants for $T$, as required in (ii). 
  Now  $d(T_0(x), T(x)) \le d_{\Lip}(T_0,T) $
  for all $x \in X$,
  so $\cO_0$ is a $d_{\Lip}(T_0,T)$-pseudo-orbit for $T$, as required in (iii).
  
  (iv) Define $\lambda' \= \lambda- d_{\Lip}(T_0,T)$, so that Proposition~\ref{p.locallyconnected} gives 
  \begin{equation}\label{e.openbound}
  	B(T(x), \delta) \subseteq T(B(x, \delta/\lambda'))
  \end{equation} 
  for all $x\in X$. Now $D \le \delta$ and $D \le \frac{(\lambda-1)\delta}{1+\delta}$ (cf.~(\ref{e.def.D}), (\ref{e.def.D.singleton})), so $d_{\Lip}(T_0,T)<D \le \delta$ and $d_{\Lip}(T_0,T)<D \le (\lambda-D-1) \delta < (\lambda' -1)\delta$. Thus by (\ref{e.openbound}), we can apply Lemma~\ref{l.shadowing} to $T$ and $\cO_0$, so as to obtain the required $T$-periodic orbit $\cO$ that $\frac{d_{\Lip}(T_0,T)}{\lambda'-1}$-shadows  $\cO_0$, with $\card \cO \le \card \cO_0$. If
  $\card \cO =1$ then the minimum interpoint distance $\Delta(\cO)$ is by definition equal to $+\infty$ (cf.~Subsection~\ref{sct_Notation}), so the inequality $\Delta(\cO)\geq \frac{1}{2}\Delta(\cO_0)$ is immediate.
  If $\card \cO >1$, then  
  \begin{equation}\label{shadowing_quality}
  \frac{d_{\Lip}(T_0,T)}{\lambda'-1} < \frac{D}{\lambda-D-1} \le \frac{1}{4} \Delta(\cO_0),
  \end{equation}
  using the fact that
  $D \le \frac{(\lambda-1) \Delta(\cO_0)}{ 4+\Delta(\cO_0) }$ (cf.~(\ref{e.def.D})).
    Now
  $\cO_0$ is $\frac{d_{\Lip}(T_0,T)}{\lambda'-1}$-shadowed by $\cO$, so (\ref{shadowing_quality}) 
  implies that it is $\frac{1}{4}\Delta(\cO_0)$-shadowed by $\cO$.
  Hence, for distinct points $x, \, y \in \cO$, there exist $x' , \, y' \in \cO_0$ with $d(x,x') \le (1/4) \Delta(\cO_0)$ and $d(y,y') \le (1/4) \Delta(\cO_0)$. Now $d(x',y') \ge \Delta(\cO_0)$ by definition, so $d(x,y) \ge d(x',y') - d(x,x') - d(y,y') \ge (1/2)\Delta(\cO_0) $, but $x, \, y$
  are arbitrary distinct members of $\cO$, therefore $\Delta(\cO) \ge (1/2) \Delta(\cO_0)$, as required.
\end{proof}

Versions of the following Theorem~\ref{l.mane}, often referred to as the
\emph{Ma\~n\'e Lemma}\footnote{The term \emph{Ma\~n\'e lemma} originated in \cite{Bou00}, in view of the resemblance to a result of Ma\~n\'e \cite[Theorem~B]{Man96} in the context of Lagrangian flows.}, are well known; 
in the generality of our setting (i.e.,~open distance-expanding Lipschitz maps on arbitrary compact metric spaces),
the semi-norm bound\footnote{The bound in (\ref{e.psi<0_subnormbound}) on the semi-norm of the sub-action $u$, by a fixed constant multiple of the semi-norm of $\phi$, will be important in our subsequent proof of Joint TPO, where both the map $T$ and function $\phi$ vary; in the proofs of Individual TPO, as in \cite{Boc19,Co16,HLMXZ25}, this additional control is not required.}
(\ref{e.psi<0_subnormbound}) appears
as \cite[Proposition~3.6]{LiSu26} (see also \cite[Theorem~3.1]{Bou11}, \cite[Theorem~1.3]{STY24}). It is well known that if $T\in \cE(X)$ has expanding constants $(\gamma, \lambda)$, then there exists a constant $\kappa(T)>0$ such that $B(T(x), \kappa(T)) \subseteq T( B(x, \gamma/2))$ for all $x\in X$ (see e.g.\ \cite[Lemma~4.1.2]{PU10}).

\begin{theorem}[\bf Ma\~n\'e Lemma]\label{l.mane}
	Let $X$ be a compact metric space. Suppose $T\in \cE(X)$ has expanding constants $(\gamma, \lambda)$, and $\alpha\in (0,1]$.
    The constant $L=L(T,\alpha)>0$ defined by
    \begin{equation}\label{Lmaxdefn}
		L \= \max \bigl\{  (\lambda^\alpha-1)^{-1}, \,   2  N \kappa(T)^{-\alpha}(\diam X)^\alpha   +  N  (\lambda^\alpha-1)^{-1}\bigr\}
	\end{equation}
    is
    such that
    for each $\phi\in \Holder{\alpha} (X)$, there exists $u \in \Holder{\alpha}(X)$ satisfying 
	\begin{equation}
		\psi \= \phi - \mpe(T, \phi) + u - u \circ T \le 0  \quad \text{ and }\quad \label{e.psi<0_subnormbound}  
		\Hseminorm{\alpha}{u} \le L \Hseminorm{\alpha}{\phi} 	,    
	\end{equation}
    where $N$ is the maximum cardinality of a $\min\{ (\lambda -1) \min\{ \gamma/2,\, \kappa(T)\},\, \kappa(T)  \}$-separated subset of $X$.
\end{theorem}

\subsection{Joint perturbation}

The following Theorem~\ref{t.criticaltheorem} is the key result that will allow us to prove 
the
openness of
     $\cE(X) \times \Holder{\alpha}(X)$
(Theorem~\ref{openjointperiodiclockingset}),
from which both
Theorem~\ref{t.JTPO.general}, and
Theorem~\ref{expandingJTPO}, will follow. 
The underlying \emph{joint perturbation} idea will be a motif throughout the article, re-occurring in different forms in the context of Anosov diffeomorphisms (see Theorems~\ref{jointperturbationuniformhyp}
and \ref{anosovJTPO_C1}), and beta-transformations
(see Theorem~\ref{l_8_main_lemma}).
A feature of the proof is that all constants in the perturbation are made completely explicit, an approach that also facilitates our proof of the Effective Joint TPO theorem (Theorem~\ref{t.JTPO.explicit}).

\begin{theorem}[Joint Perturbation: expanding maps]\label{t.criticaltheorem}
	Let $X$ be a compact locally connected metric space. Suppose $T_0 \in \cE(X)$ has expanding constants $(\gamma, \lambda)$, and $\alpha\in (0,1]$. Let $\cO_0$ be a $T_0$-periodic orbit. Let $D=D(T_0,\cO_0)>0$ be as in Lemma~\ref{l_perturbation_T_and_cO}, and assume
    that $D< \min\{ 1 , \, \lambda -1 \}$.	
	Then there exists $C>0$ such that for all $T \in\cE(X)$ with $d_{\Lip}(T_0,T)<D$, there exists a $T$-periodic orbit $\cO$ such that for all $\phi \in \Holder{\alpha}(X)$ with $\Mmax(T_0,\phi) = \{\mu_{\cO_0}\}$, the unique $T$-maximizing measure for the function $\phi - 2C \Hseminorm{\alpha}{\phi} d_{\Lip}(T_0,T)^{\alpha/2} d( \cdot, \cO)^\alpha$	is $\mu_{\cO}$. 
\end{theorem}

\begin{proof}
    Let $\delta\in (0,\gamma)$ be as in Proposition~\ref{p.locallyconnected}.  Let $L=L(T_0,\alpha)>0$ be as in the Ma\~n\'e Lemma (Theorem~\ref{l.mane}). Define constants
	\begin{align}
		p & \= \card \cO_0, \label{e.def.p} \\
		r & \= \min  \{  \Delta(\cO_0) / 4, \, \gamma \}, \label{e.def.r} \\
    	L_1 & \= 1+   ((\lambda-D)^\alpha -1 )^{-1} + 2L \, (2 \LIP(T_0))^\alpha, \label{e.def.X} \\
		L_2 & \= L + 1 + (\lambda-D-1)^{-\alpha}, \label{e.def.Y} \\
		C & \= \max  \{  1, \, L_2(1+p+L_1) ( 2 \LIP(T_0)/r )^\alpha \}. \label{e.def.C}
	\end{align}

    Fix $\phi \in \Holder{\alpha}(X)$ with $\Mmax(T_0,\phi) = \{\mu_{\cO_0}\}$. Then $\Hseminorm{\alpha}{\phi} \neq 0$.
	 By the Ma\~n\'e Lemma (Theorem~\ref{l.mane}) applied to the map $T_0$, there exist $u, \psi\in\Holder{\alpha}(X)$ satisfying (\ref{e.psi<0_subnormbound}), in other words,
    \begin{equation}\label{e.psi<0_subnormbound_reiterated}
    	\psi \= \overline{\phi} + u - u \circ T_0 \le 0 \quad \text{ and } \quad
        \Hseminorm{\alpha}{u} \le L \Hseminorm{\alpha}{\phi},
        \end{equation}
        where we recall that $\overline{\phi}$ denotes $ \phi - \mpe (T_0,\phi)$.

Now suppose $T \in \cE(X)$ is such that $d_{\Lip}(T_0,T)<D$, and
let $\cO$ be the periodic orbit from Lemma~\ref{l_perturbation_T_and_cO}~(iv).
Without loss of generality we assume $T\neq T_0$. Combining (\ref{e.def.r}) and Lemma~\ref{l_perturbation_T_and_cO}~(iv), we see that
	\begin{equation}\label{e.gapbound}
         r\leq  \Delta(\cO) /2 .        
	\end{equation}
	Now define $\psi' \= \overline{\phi} + u - u \circ T$ and 
	\begin{equation}\label{e.def.tau}
		\tau \= ( L + 1 ) \Hseminorm{ \alpha }{\phi} d_{\Lip}(T_0,T)^\alpha > 0,
	\end{equation}
    so that (\ref{e.psi<0_subnormbound_reiterated}) gives
	\begin{equation}\label{e.psi'<tau}
		\psi'(x) \le \psi(x) + \abs{u(T_0(x)) - u(T(x))} \le L \Hseminorm{ \alpha}{\phi} d_{\Lip}(T_0,T)^\alpha < \tau
        \quad\text{	for all } x\in X.
	\end{equation}
	Define functions
	\begin{align}
		\xi &\= \overline{\phi} - C \Hseminorm{\alpha}{\phi} d_{\Lip}(T_0,T)^{\alpha/2} d(\cdot, \cO)^\alpha \quad\text{ and} \label{e.def.xi} \\
		\xi' &\= \xi + u - u \circ T = \psi' - C \Hseminorm{\alpha}{\phi} d_{\Lip}(T_0,T)^{\alpha/2} d(\cdot, \cO)^\alpha, \label{e.def.xi'} 
	\end{align}
	where the second equality in (\ref{e.def.xi'}) follows from (\ref{e.def.xi}) and the definition of $\psi'$. 
    Define $\eta \= \int\! \xi' \,\mathrm{d}\mu_{\cO}$, so that (\ref{e.def.xi'}), (\ref{e.def.xi}),  Lemma~\ref{l_perturbation_T_and_cO}~(iv), 
    and the fact that $\int\! \overline{\phi} \,\mathrm{d}\mu_{\cO_0}=\mpe(T_0,\overline{\phi})=0$,
    give
	\begin{equation}\label{e.eta}
		\begin{aligned}
			\eta &= \int\! \xi' \,\mathrm{d}\mu_{\cO} = \int\! \xi \,\mathrm{d}\mu_{\cO} =\int\! \overline{\phi} \,\mathrm{d}\mu_{\cO} \\
			&\ge \int\! \overline{\phi} \,\mathrm{d}\mu_{\cO_0} - \Hseminorm{\alpha}{\phi} \biggl( \frac{d_{\Lip}(T_0,T)}{\lambda-d_{\Lip}(T_0,T)-1} \biggr)^\alpha \ge -\Hseminorm{\alpha}{\phi} \biggl( \frac{d_{\Lip}(T_0,T)}{\lambda-D-1} \biggr)^\alpha.
		\end{aligned}
	\end{equation} 
	Combining (\ref{e.def.tau}), 
    (\ref{e.eta}), and recalling that 
    $L_2 = L + 1 + (\lambda-D-1)^{-\alpha}$
    (cf.~(\ref{e.def.Y})), we obtain
	\begin{equation}\label{e.tau-eta}
		\tau-\eta \le L_2 \Hseminorm{\alpha}{\phi} d_{\Lip}(T_0,T)^\alpha.
	\end{equation}
	We wish to prove that $\mu_{\cO} \in \Mmax(T,\xi)$, which is equivalent to showing that $\mu_{\cO} \in \Mmax(T, \xi')$, and by \cite[Proposition~2.2]{Je19} it suffices to establish that 
	\begin{equation}\label{e.sufcon}
		\liminf_{n\to +\infty} \frac{1}{n} S_n^{T} \xi'(x) \le \eta \quad \text{ for all } x\in X. 
	\end{equation}
	By (\ref{e.eta}), (\ref{e.def.xi'}), and (\ref{e.psi'<tau}), we see that
    \begin{equation*}
        \eta
        = \int\! \xi' \,\mathrm{d}\mu_{\cO}
        = \int\! \psi' \,\mathrm{d}\mu_{\cO}
        < \tau. 
    \end{equation*}
    Combined with the fact that $\Hseminorm{\alpha}{\phi} \neq 0$ and $T\neq T_0$, this allows us to define
	\begin{equation}\label{e.def.rho}
		\rho \=  \bigl(   C \Hseminorm{\alpha}{\phi} d_{\Lip}(T_0,T)^{\alpha/2}  \big/ (  \tau-\eta )  \bigr)^{- 1/\alpha} > 0.
	\end{equation}
	By (\ref{e.psi'<tau}) and (\ref{e.def.xi'}), 
	\begin{equation}\label{e.xi'<eta}
		\xi'(x) \le \eta \quad \text{ if } x\notin B(\cO, \rho).
	\end{equation}
	
	Fixing $x\in X$, we will recursively construct two sequences $\{x_t\}_{t\in \N}$ in $X$ and $\{n_t\}_{t\in \N}$ in $\N$, with the properties that $x_{t+1} = T^{n_t} (x_t)$ and 
	\begin{equation}\label{e.recursiveaim}
		S_{n_t}^{T} \xi' (x_t) \le n_t \eta \quad \text{ for all } t\in \N.
	\end{equation}  
	
	\smallskip
	\emph{Base step.} Define $x_1 \= x$.
	
	\smallskip
	\emph{Recursive step.} Assume that for some $t\in \N$, $\{x_i\}_{i=1}^{t}$ and $\{n_i\}_{i=1}^{t-1}$ are defined. Consider the following three cases.
	
	\smallskip
	\emph{Case A.} Assume $x_t \in \cO$. Define $n_t \= p $ and $x_{t+1} \= T^p(x_t)$. Then (\ref{e.recursiveaim}) follows from the definition of $\eta$.
	
	\smallskip
	\emph{Case B.} Assume $x\notin B(\cO, \rho)$. Define $n_t \= 1$ and $x_{t+1} \= T(x_t)$, so that (\ref{e.recursiveaim})  follows from (\ref{e.xi'<eta}).
	
	\smallskip
	\emph{Case C.} Assume $x\in B(\cO, \rho) \smallsetminus \cO$. 
    Using (\ref{e.def.rho}), (\ref{e.tau-eta}),
    the fact that $C \ge L_2 D^{\alpha/2} ( 2 \LIP(T_0) / r )^\alpha$ (cf.~(\ref{e.def.C})), the fact that $d_{\Lip}(T_0,T)<D$,
     and Lemma~\ref{l_perturbation_T_and_cO}~(i), we obtain 
	\begin{equation}\label{e.rho<r}
		\rho \le  r / (2 \LIP(T_0)) \le  r / \LIP(T) .
	\end{equation}
	Let $y\in \cO$ be such that $d(x_t, \cO) = d(x_t, y)$, and define
	\begin{align}
		N &\= \min \bigl\{ i\in \N_0 : d\bigl(T^{i+1}(x_t), T^{i+1}(y)\bigr) \ge r \bigr\} , \label{e.def.N} \\
		m & \= \min\bigl\{ i \in \N_0 : d\bigl(T^i(x_t), T^i(y)\bigr) \ge \rho \bigr\}. \label{e.def.m}
	\end{align}
	The existence of $N$ follows from the fact that $r\le \gamma$ (cf.~(\ref{e.def.r})) and the existence of $m$ follows from $\rho<r$ (by (\ref{e.rho<r}) and $\LIP(T)>1$). By (\ref{e.def.r}), (\ref{e.gapbound}), and Lemma~\ref{l_perturbation_T_and_cO}~(ii), for $0 \le i \le N$, 
	\begin{equation}\label{e.expanding}
     \begin{aligned}
		d\bigl(T^i(x_t), \cO\bigr) 
        &= d\bigl(T^i(x_t), T^i(y)\bigr) \\
        &\le \frac{ d( T^{i+1}(x_t) , T^{i+1}(y) )}{\lambda-d_{\Lip}(T_0,T)} 
        < \frac{ d( T^{i+1}(x_t) , T^{i+1}(y) )}{\lambda-D} . 
     \end{aligned}
	\end{equation} 
	By (\ref{e.rho<r}), $m$ exists and satisfies
	\begin{equation}\label{e.1<m<N}
		1 \le m \le N.
	\end{equation}
	In this case, we define $n_t \= N+1$ and $x_{t+1} \= T^{N+1}(x_t)$.
	
	Next, we estimate $S_{n_t}^{T} \xi'(x_t)$ so as to deduce (\ref{e.recursiveaim}). Noting that $\xi'\le \psi'$, direct calculation gives
	\begin{equation}\label{e.4terms}
		S_{n_t}^{T} \xi'(x_t) \le S_m^{T} \psi' (y) + \Absbig{S_m^{T} \psi' (y)- S_m^{T} \psi' (x_t) } + S_{N-m}^{T} \xi' (T^m(x_t)) + \xi'\bigl(T^N(x_t)\bigr).
	\end{equation}
	We now estimate the four terms on the right-hand side of (\ref{e.4terms}). For the first term, write $m = qp+l$ for some $q \in \N_0$ and $0\le l \le p-1$, recalling that $p=\card \cO_0$ (cf.~(\ref{e.def.p})). Using (\ref{e.psi'<tau}), 
	\begin{equation}\label{e.ineq1}
		S_m^{T} \psi' (y) = q S_p^{T} \psi' (y) + S_l^{T} \psi' (y) \le qp \eta + l \tau = m \eta +l(\tau-\eta) \le m \eta + p(\tau-\eta).
	\end{equation}
	For the second term on the right-hand side of (\ref{e.4terms}), note that (\ref{e.1<m<N}), Lemma~\ref{l_perturbation_T_and_cO}~(i), and (\ref{e.def.m}) give
	\begin{equation*}
		d(x_t,y) < d(T^m(x_t), T^m(y)) \le \LIP(T) d\bigl(T^{m-1}(x_t), T^{m-1}(y)\bigr) < 2\LIP(T_0) \rho.
	\end{equation*}
	Thus, by (\ref{e.def.xi}), (\ref{e.expanding}), the semi-norm bound in (\ref{e.psi<0_subnormbound_reiterated}), and (\ref{e.def.X}), we estimate
		\begin{align} \label{e.ineq2}
			&\Absbig{S_m^{T} \psi' (y)- S_m^{T} \psi' (x_t) } \\
            &\qquad \le \Absbig{S_m^{T} \overline{\phi} (y)- S_m^{T} \overline{\phi} (x_t) } + \abs{u(x_t) - u(y)} + \abs{u(T^m(x_t)) - u(T^m(y))} \notag\\
			&\qquad \le \Hseminorm{\alpha}{\phi}  \sum_{i=0}^{m-1} d\bigl(T^i(x_t), T^i(y)\bigr)^\alpha + 2\Hseminorm{\alpha}{u} (2\rho \LIP(T_0))^\alpha \notag\\
			&\qquad \le \Hseminorm{\alpha}{\phi} \rho^\alpha \sum_{i=0}^{m-1} (\lambda-D)^{-i\alpha} + 2L \Hseminorm{\alpha}{\phi} (2\rho \LIP(T_0))^\alpha \notag\\
			&\qquad \le \Hseminorm{\alpha}{\phi} \rho^\alpha \bigl( 1 + ((\lambda-D)^\alpha -1)^{-1} + 2L (2 \LIP(T_0))^\alpha \bigr) \notag\\
			&\qquad = L_1\Hseminorm{\alpha}{\phi} \rho^\alpha . \notag
		\end{align}
	For the third term on the right-hand side of (\ref{e.4terms}), by (\ref{e.def.m}), (\ref{e.expanding}), and (\ref{e.xi'<eta}) we have
	\begin{equation}\label{e.ineq3}
		S_{N-m}^{T} \xi' (T^m(x_t)) \le (N-m)\eta,
	\end{equation}
	while for the fourth term, by (\ref{e.def.xi'}), (\ref{e.psi'<tau}), (\ref{e.def.N}), and Lemma~\ref{l_perturbation_T_and_cO}~(i), we have
	\begin{equation}\label{e.ineq4}
		\xi'\bigl(T^N(x_t)\bigr) \le \tau - C \Hseminorm{\alpha}{\phi} d_{\Lip}(T_0,T)^{\alpha/2} ( 2\LIP(T_0) /r )^{-\alpha}.
	\end{equation}
	Finally, note that $d_{\Lip}(T_0,T) < D < 1$, as this was a hypothesis. Combining (\ref{e.4terms}), (\ref{e.ineq1}), (\ref{e.ineq2}), (\ref{e.ineq3}), and (\ref{e.ineq4}), using $C = \max\{ 1, \, L_2(1+p+L_1) (2 \LIP(T_0) /r)^\alpha \}$ (cf.~(\ref{e.def.C})), (\ref{e.def.rho}), and (\ref{e.tau-eta}), and writing $L_3\= d_{\Lip}(T_0,T)$ for simplicity, we have
	\begin{equation*}
		\begin{aligned}
			S_{n_t}^{T} \xi' (x_t) - n_t \eta &\le (p+1)(\tau-\eta) +L_1 \Hseminorm{\alpha}{\phi} \rho^\alpha  -C \Hseminorm{\alpha}{\phi} L_3^{\alpha/2} (2 \LIP(T_0) / r)^{-\alpha} \\
			&\le (\tau-\eta) \bigl(1+p + L_1 L_3^{-\alpha/2}\bigr) - C \Hseminorm{\alpha}{\phi} L_3^{\alpha/2} (2 \LIP(T_0) / r)^{-\alpha} \\
			& \le (\tau-\eta) L_3^{-\alpha/2} (1+p+L_1) -C \Hseminorm{\alpha}{\phi} L_3^{\alpha/2} (2 \LIP(T_0) / r)^{-\alpha} \\
			& \le L_2 \Hseminorm{\alpha}{\phi} L_3^{\alpha/2} (1+p+L_1) - C \Hseminorm{\alpha}{\phi} L_3^{\alpha/2} (2\LIP(T_0)/r)^{-\alpha}\\
			& \le 0.
		\end{aligned}
	\end{equation*}
	So the required inequality (\ref{e.recursiveaim}) holds, and therefore the recursive step is complete.	
	
	\smallskip
	For each $t\in \N$, set $N_t \= \sum_{i=1}^{t} n_i$. By (\ref{e.recursiveaim}), 
	\begin{align*}
		\liminf_{n\to +\infty} \frac{1}{n} S_n^{T} \xi' (x) 
        &\le \liminf_{t\to +\infty} \frac{1}{N_t} S_{N_t}^{T} \xi' (x) \\
        &= \liminf_{t\to +\infty} \frac{1}{N_t} \sum_{i=1}^{t} S_{n_i}^{T} \xi' (x_t) 
        \le \frac{\eta}{N_t} \sum_{i=1}^{t} n_i = \eta.
	\end{align*}
	Thus, (\ref{e.sufcon}) holds, so $\mu_{\cO} \in \Mmax(T, \xi)$. Then by (\ref{e.def.xi}), $\mu_{\cO}$ is the unique $T$-maximizing measure for the function
	$
		\phi - 2C \Hseminorm{\alpha}{\phi} d_{\Lip}(T_0,T)^{\alpha/2} d(\cdot, \cO)^\alpha = \xi + \mpe(T,\phi) - C \Hseminorm{\alpha}{\phi} d_{\Lip}(T_0,T)^{\alpha/2} d(\cdot, \cO)^\alpha$,  
	and so the proof of the theorem is complete.
\end{proof}

The  Joint Perturbation Theorem
(Theorem~\ref{t.criticaltheorem}) can now be used to
prove Theorem~\ref{openjointperiodiclockingset},
asserting that $\fL^\alpha(X)$ is an \emph{open} subset of $\cE(X) \times \Holder{\alpha}(X)$.

\begin{proof}[Proof of 
Theorem~\ref{openjointperiodiclockingset}]
	Let $T_0 \in \cE(X)$, $\alpha \in (0,1]$, and $\phi \in \Lock^\alpha(T_0)$. 
    In this proof, we write $B(\phi, r) \= \bigl\{\psi \in \Holder{\alpha} (X) : \Hnorm{\alpha}{\phi- \psi} < r \bigr\}$ and $B_{\Lip}(T_0,r) \= \{T \in \cE(X) : d_{\Lip}(T_0, T) < r \}$ for $r \geq 0$.
    
    Let $\cO_0$ be the $T_0$-periodic orbit such that $\Mmax(T_0,\phi) = \{\mu_{\cO_0}\}$. Then there exists $\theta>0$ such that if $\psi\in \Holder{\alpha}(X) $ with $\Hnorm{\alpha}{\psi-\phi}<\theta$ then $\Mmax(T_0,\psi) = \{\mu_{\cO_0}\}$. Now $\Hseminorm{\alpha}{ \, \cdot \, } \le \Hnorm{\alpha}{\cdot}$, and $\Hseminorm{\alpha}{\, \cdot \,}$ is sub-additive, so we can assume without loss of generality that $\theta$ is small enough (i.e., $\theta \le 2^{-1} \Hseminorm{\alpha}{\phi}$) such that 
	\begin{equation}\label{e.neighbornormbound}
		\frac{1}{2} \Hseminorm{\alpha}{ \phi } \le \Hseminorm{\alpha}{ \psi } \le \frac{3}{2} \Hseminorm{\alpha}{ \phi } \quad \text{ for all } \psi\in B(\phi, \theta).
	\end{equation} 
	Applying Theorem~\ref{t.criticaltheorem} (to $T_0$ and $\cO_0$), let $C, \, D>0$ be as in that theorem.  Define 
	\begin{equation}\label{e.def.E}
		E \= \min \bigl\{ D, \, 1, \, \lambda - 1,\,  (  18C (1+ \diam (X)^\alpha) \Hseminorm{\alpha}{\phi} / \theta )^{ - 2/\alpha} \bigr\}.
	\end{equation}
    
	Now fix $T \in B_{\Lip}(T_0,E)$ and $\psi \in B(\phi, \theta/2)$. Now $d_{\Lip}(T_0,T) <E \le \min\{D, \, 1, \, \lambda - 1\}$, so let $\cO$ be the $T$-periodic orbit whose existence is guaranteed by Theorem~\ref{t.criticaltheorem}. 
    Set 
	\begin{equation}\label{e.def.psi'}
		\psi' \= \psi + 6C \Hseminorm{\alpha}{\psi} d_{\Lip}(T_0,T)^{\alpha/2} d(\cdot, \cO)^\alpha.
	\end{equation}
	Note that $\Hnorm{\alpha}{d(\cdot, \cO)} = \Hnorm{\infty}{d(\cdot, \cO)} + \Hseminorm{\alpha}{d(\cdot, \cO)} \le 1+ \diam(X)^\alpha$.
		
	Thus, by (\ref{e.def.psi'}), (\ref{e.neighbornormbound}), and (\ref{e.def.E}), we obtain
	\begin{equation*}
		\Hnorm{ \alpha }{\psi-\psi'} \le 6C (1+ \diam(X)^\alpha) \Hseminorm{\alpha}{\psi} E^{\alpha/2} \le 9C (1+ \diam(X)^\alpha) \Hseminorm{\alpha}{\phi} E^{\alpha/2} \le \theta/2,
	\end{equation*}
	and together with the fact that $\psi\in B(\phi, \theta/2)$, we deduce that $\psi'\in B(\phi, \theta)$. So both $\psi$ and $\psi'$ belong to $B(\phi,\theta)$, thus by (\ref{e.neighbornormbound}) we deduce that 
    \begin{equation}\label{e.neighbornormbound.deduction}
    6C \Hseminorm{\alpha}{\psi} 
    \ge 3C \Hseminorm{\alpha}{\phi} 
    \ge 2C \Hseminorm{\alpha}{\psi'}.
    \end{equation}
    Now Theorem~\ref{t.criticaltheorem} asserts that $\mu_{\cO}$ is the unique $T$-maximizing measure for the function 
	\begin{equation*}
		\psi' - 2C \Hseminorm{\alpha}{\psi'} d_{\Lip}(T_0,T)^{\alpha/2} d(\cdot, \cO)^\alpha,
	\end{equation*}
	so by (\ref{e.neighbornormbound.deduction}), together with the fact that $\mu_{\cO}$ is the unique $T$-maximizing measure for $-d(\cdot,\cO)^\alpha$,
    we see that $\mu_{\cO}$
    is also the unique $T$-maximizing measure for 
    $
    	\psi = \psi' - 6C \Hseminorm{\alpha}{\psi} d_{\Lip}(T_0,T)^{\alpha/2} d(\cdot, \cO)^\alpha$.
	Therefore we have established that the open neighbourhood 
    $B_{\Lip}(T_0,E) \times B(\phi, \theta/2)$
    of $(T_0,\phi)$ is contained in
    $\fL^\alpha(X)$. 
    So $\fL^\alpha(X)$ is an open
    subset of $\cE(X) \times \Holder{\alpha}(X)$, as required. 
\end{proof}	

The remaining ingredient needed to prove Theorem~\ref{t.JTPO.general} is
the following individual typical periodic optimization theorem due to
Contreras \cite[Theorem~A]{Co16} (the statement below is closer to the reformulation of Bochi \cite{Boc19}):

\begin{theorem}[Individual TPO for expanding maps \cite{Co16}]\label{co}
	If $X$ is a compact metric space and $T\in\cE(X)$,
	and $\alpha\in(0,1]$,
	then $\Lock^\alpha(T)$ is an open and dense subset of $\Holder{\alpha} (X )$.
\end{theorem}

\begin{rem}
	The definition of \emph{expanding} used in \cite{Co16} corresponds precisely to our definition of \emph{open distance-expanding and Lipschitz continuous}. 
\end{rem}

We are now able to prove Theorem~\ref{t.JTPO.general}:

\begin{proof}
[Proof of Theorem~\ref{t.JTPO.general}]

The joint periodic locking set
$\fL^\alpha(X)$ is open in
     $\cE(X) \times \Holder{\alpha}(X)$, by
 Theorem~\ref{openjointperiodiclockingset}.
If $W$ is any nonempty open set in $\cE(X) \times \Holder{\alpha}(X)$,
 there exist nonempty open sets $U \subseteq \cE(X)$ and $V \subseteq \Holder{\alpha}(X)$ such that
$U \times V \subseteq W$.

Let $T_0 \in U$.
Since $\Lock^\alpha(T_0)$ is dense in $\Holder{\alpha}(X)$, by Theorem~\ref{co},
the set $\Lock^\alpha(T_0)$ must intersect the nonempty open set $V$; 
so there exists $\phi_0 \in \Lock^\alpha(T_0) \cap V$.
We claim that $(T_0,\phi_0) \in W \cap \fL^\alpha(X)$. 
To see this, it suffices to note 
that $\phi_0 \in \Lock^\alpha(T_0)$, so $(T_0,\phi_0) \in \{T_0\} \times \Lock^\alpha(T_0)\subseteq \fL^\alpha(X)$. 
Now $W$ was arbitrary, so $\fL^\alpha(X)$ intersects every nonempty open subset of $\cE(X) \times \Holder{\alpha}(X)$,
in other words $\fL^\alpha(X)$ is indeed a dense subset.

So $\fL^\alpha(X)$ is both open and dense in
     $\cE(X) \times \Holder{\alpha}(X)$, 
as required.
\end{proof}

We can now prove the effective version of Joint TPO
(Theorem~\ref{t.JTPO.explicit}):

\begin{proof}[Proof of Theorem~\ref{t.JTPO.explicit}]
Since $(T_0, \phi)$
belongs to the joint periodic locking set $\fL^\alpha(X)$,
there exists a $T_0$-periodic orbit $\cO_0$ such that
$\Mmax(T_0,\phi)=\{\mu_{\cO_0}\}$. 
Let $p  \= \card \cO_0$ (cf.~(\ref{e.def.p})), 
        let $(\gamma, \lambda)$
       be expanding constants for $T_0$, and let
        $\delta\in (0,\gamma)$ be a surjectivity radius, as in Proposition~\ref{p.locallyconnected}.
        Define $D>0$
    as in Lemma~\ref{l_perturbation_T_and_cO}, by
    \begin{equation*}
					D  \= 
                    \begin{cases}
                    \min \bigl\{ \delta,  \, \LIP(T_0), \,  \frac{(\lambda-1) \delta}{1+ \delta} , \,  \frac{\lambda-1}{2} , \, \frac{(\lambda-1) \Delta(\cO_0)}{4+ \Delta(\cO_0)} \bigr\}&\text{ if $p>1$,}\cr
                    \min \bigl\{ \delta,  \, \LIP(T_0), \,  \frac{(\lambda-1) \delta}{1+ \delta} , \,  \frac{\lambda-1}{2}  \bigr\} &\text{ if $p=1$.}
                    \end{cases}
                    \end{equation*}
Let $N$ be the maximum cardinality of a $\min\{ (\lambda -1) \min\{ \gamma/2,\, \kappa(T_0)\},\, \kappa(T_0)  \}$-separated subset of $X$, 
        and define $L>0$ (cf.~(\ref{Lmaxdefn}))
          by
\begin{equation*}\label{Lmaxdefn.again}
		L \= \max \bigl\{  (\lambda^\alpha-1)^{-1}, \,   2  N \kappa(T_0)^{-\alpha}(\diam X)^\alpha   +  N  (\lambda^\alpha-1)^{-1}\bigr\}
	\end{equation*}
 Now set $r  \= \min  \{  \Delta(\cO_0) / 4, \, \gamma \}$
 (cf.~(\ref{e.def.r})), 
 $L_1  \= 1+   ((\lambda-D)^\alpha -1 )^{-1} + 2L \, (2 \LIP(T_0))^\alpha$
 (cf.~(\ref{e.def.X})), 
  $L_2  \= L + 1 + (\lambda-D-1)^{-\alpha}$
 (cf.~(\ref{e.def.Y})), and
$C  \= \max  \{  1, \, L_2(1+p+L_1) ( 2 \LIP(T_0)/r )^\alpha \}$ (cf.~\ref{e.def.C}).

To find $\theta>0$ such that
\begin{equation}\label{gperturbationtheta}
       \MMM_{\max}(T_0,\psi)=
        \{\mu_{\cO_0}\} \ \text{ for all }\
        \psi\in\Holder{\alpha}(X) \ 
        \text{ with }\
        \|\psi-\phi\|_{ \alpha } < \theta,
        \end{equation}
we let $F(\cO_0)$ denote the $(p-1)$-dimensional 
quotient vector space consisting of all real-valued functions defined on $\cO_0$, where functions are identified if they differ by a constant, and
    let $F_0(\cO_0)$ denote the $(p-1)$-dimensional  vector space consisting of all functions $\phi:\cO_0\to\R$ such that $\int\!  \phi\, \mathrm{d}\mu_{\cO_0}=0$, noting that
    the semi-norm $|\cdot|_\alpha$
    gives a norm on both $F(\cO_0)$ and $F_0(\cO_0)$.
    Let $c_0:F(\cO_0)\to F_0(\cO_0)$ be the invertible linear map defined by $c_0(\phi)=\phi\circ T-\phi$,
    and let $\Normbig{c_0^{-1}}$ be the operator norm of its inverse.
    By \cite[Proposition~5.1]{Bou08}, if
    $\theta_0\=\bigl( 1+ \Normbig{c_0^{-1}}(1+\LIP_\alpha(T_0))\bigr)^{-1}$, where
    $\LIP_\alpha(T_0)$
     denotes the Lipschitz constant (cf.~(\ref{lipschitzconstant})) of $T_0$ with respect to the metric $d^\alpha$ defined by $d^\alpha(x,y)\=d(x,y)^\alpha$,
      then 
      \begin{equation}\label{bouschprop5.1}
        \mu_{\cO_0} \in \MMM_{\max}(T_0,\psi)\ \text{ for all }
   \psi\in\Holder{\alpha}(X)\ \text{ with }\
   \|\psi+d(\cdot,\cO_0)^\alpha\|_\alpha\le \theta_0.
   \end{equation}
        Choosing $l>0$ such that $\mu_{\cO_0}$ is the unique maximizing measure for 
        $\phi + l d(\cdot,\cO_0)^\alpha$, we then claim that $\theta\=\theta_0 l$ satisfies (\ref{gperturbationtheta}).
To prove (\ref{gperturbationtheta}) with $\theta\=\theta_0 l$, note that if $\nu\in\MMM(X,T)\smallsetminus\{\mu_{\cO_0}\}$, since $\mu_{\cO_0}$ is the unique maximizing measure for $\phi + l d(\cdot,\cO_0)^\alpha$, then
        $\mpe(T,\phi)=\int\!  \phi\, \mathrm{d}\mu_{\cO_0}
        =\int\!  (\phi+l d(\cdot,\cO_0)^\alpha)\, \mathrm{d}\mu_{\cO_0}> \int\!  (\phi+l d(\cdot,\cO_0)^\alpha)\, \mathrm{d}\nu$,
        in other words,
        \begin{equation}\label{doubstar}
            \mpe(T,\phi)-\int\!  \phi\, \mathrm{d}\nu >
            l\int\!  d(\cdot,\cO_0)^\alpha\, \mathrm{d}\nu.
        \end{equation}
        If $\|\psi-\phi\|_{\alpha} < \theta_0 l$ then, writing
        $\xi\=\psi-\phi$, we have  
        $\|\xi/l\|_{\alpha} < \theta_0 $,
        so (\ref{bouschprop5.1}) implies that
        $\mu_{\cO_0}$ is a maximizing measure for
        $-d(\cdot,\cO_0)^\alpha + \xi/l$,
        so
        $\int\!  (-d(\cdot,\cO_0)^\alpha + \xi/l)\,\mathrm{d}\nu \le \int\! (-d(\cdot,\cO_0)^\alpha + \xi/l)\,\mathrm{d}\mu_{\cO_0}$, in other words,
        \begin{equation}\label{singstar}
            \int\!  \xi\, \mathrm{d}\nu - \int\!  \xi\, \mathrm{d}\mu_{\cO_0} \le l \int\!  d(\cdot,\cO_0)^\alpha\, \mathrm{d}\nu.
        \end{equation}
        Combining (\ref{doubstar}) and (\ref{singstar}) gives
      $\mpe(T,\phi)-\int\!  \phi\, \mathrm{d}\nu > \int\!  \xi\, \mathrm{d}\nu - \int\!  \xi\, \mathrm{d}\mu_{\cO_0}$,
      and since $\mpe(T,\phi)=\int\! \phi\,\mathrm{d}\mu_{\cO_0}$ then
      $\int\!  (\phi+\xi)\, \mathrm{d}\mu_{\cO_0}> \int\!  (\phi+\xi)\, \mathrm{d}\nu$.
      Since $\nu\in\MMM(X,T)\smallsetminus\{\mu_{\cO_0}\}$ was arbitrary, it follows that $\mu_{\cO_0}$ is the unique maximizing measure for
        $\phi+\xi=\psi$, so (\ref{gperturbationtheta}) is proved with $\theta\=\theta_0 l$.

With this choice of $\theta$, we finally
set $r\=\theta/2$, and
define $E>0$ (cf.~(\ref{e.def.E})) by
    \begin{equation*}\label{e.def.E.again}
		E \= \min \bigl\{ D, \, 1, \, \lambda - 1,\,  (  18C (1+ \diam (X)^\alpha) \Hseminorm{\alpha}{\phi} / \theta )^{ - 2/\alpha} \bigr\},
	\end{equation*}
    and then the argument used in the proof of Theorem~\ref{openjointperiodiclockingset}
    shows that
    the product of nonempty open balls 
    $B_{\Lip}(T_0,E) \times B(\phi, r)$
    of $(T_0,\phi)$ is contained in
    $\fL^\alpha(X)$, as required. 
    \end{proof}

As mentioned in Section~\ref{s:introduction}, a class
of expanding maps of particular interest consists of
those on compact Riemannian manifolds (see e.g.~the exposition in \cite[Chapter 6]{URM22}), as investigated initially by
 Shub \cite{Sh69, Sh70}, and later notably by Gromov
\cite{Gr81}. 
Our Joint TPO result for expanding maps has an immediate corollary in the setting of compact Riemannian manifolds:

\begin{theorem}[Joint TPO for expanding maps on manifolds] \label{expandingJTPOmanifolds}
Suppose $M$ is a compact Riemannian manifold, with distance function induced by the Riemannian metric, and $\alpha \in (0,1]$.
There is an open dense subset of pairs 
$(T,\phi)\in \cE(M)\times \Holder{\alpha}(M)$ 
with the periodic optimization property.
\end{theorem}

\begin{rem}\label{LDPremark}
As mentioned in Subsection~\ref{JTPO_intro_subsection},
the fact that Theorem~\ref{expandingJTPO} implies
\emph{joint typical uniqueness} of the maximizing measure
has an interpretation in terms of a large deviation
principle for zero-temperature limits of equilibrium states.
Specifically, if $\mu_{T,t\phi}$ denotes an equilibrium state for the map $T$ and potential $t\phi$, then for
an open dense subset of pairs $(T,\phi)\in \cE(X)\times \Holder{\alpha}(X)$, the family
$\{\mu_{T,t\phi}\}_{t\in(1,+\infty)}$
satisfies the large deviation principle as $t\to+\infty$,
in other words there exists a lower semi-continuous function $I \: X \rightarrow [0,+\infty]$ 
such that
$\liminf_{t\to+\infty}\frac{1}{t}\log\mu_{T,t\phi}(\cG)\geq-\inf_{x\in \cG}I(x)$
if $\cG\subseteq X$ is open,
		and $\limsup_{t\to+\infty}\frac{1}{t}\log\mu_{T,t\phi}(\cK)\leq-\inf_{x\in \cK}I(x)$
if $\cK\subseteq X$ is closed.
     This follows from results in \cite{LiSu26}
(cf.~\cite{Ana04, Wan19} for related phenomena in the 
zero-temperature limit of, respectively, Lagrangian dynamics and SLE). 
\end{rem}

\subsection{Anosov diffeomorphisms}\label{anosovsubsection}

Here we will prove
joint typical periodic optimization 
for the space of $C^1$ Anosov diffeomorphisms.\footnote{The treatment in this subsection is more streamlined and less detailed than the proofs of Joint TPO for expanding maps and beta-transformations, on the one hand so as to minimise repetition of similar arguments, on the other hand in view of the more comprehensive study 
\cite{HHJL26}
of Joint TPO for general hyperbolic systems.}
Recall that for a smooth compact  Riemannian manifold $M$,
if a diffeomorphism $f:M\to M$ has a hyperbolic structure on all of $M$ then it is said to be Anosov 
(see e.g.~\cite{KH95, Ni71, Robi95, Wen16}).
We require the following joint perturbation result: 

\begin{theorem}[Joint Perturbation: Anosov diffeomorphisms]\label{jointperturbationuniformhyp} 
Let $M$ be a smooth compact  Riemannian manifold, with  distance function $d$ induced by the Riemannian metric, and let
$\cA(M)$ be the space of $C^1$ Anosov diffeomorphisms on $M$, equipped with the $C^1$ topology.
Let $\alpha\in(0,1]$.
Let $f\in\cA(M)$, 
and suppose $\cO$ is an $f$-periodic orbit.
There exists a neighbourhood $U\subseteq \cA(M)$ of $f$,
and $C>0$, such that for all $g\in U$, 
and all
$\phi\in \Holder{\alpha}(M)$ with 
$\MMM_{\max}(f,\phi)=\{\mu_{\cO}\}$, there exists a topological conjugacy $h_g$ with $h_g \circ f = g \circ h_g$, 
and if $\cO_g\= h_g(\cO)$ then
$\mu_{\cO_g}$ is
 the unique $g$-maximizing measure for the function
$\phi-2C\Hseminorm{\alpha}{\phi} d_\infty(h_g,\id)^{\alpha/2} d(\cdot,\cO_g)^\alpha$,
where
$d_\infty(h_g,\id)\=\max_{x\in M} d(h_g(x),x)$.
\end{theorem}
\begin{proof} 
The proof of this result includes aspects that are similar in spirit to those for
the joint perturbation theorem for expanding maps
(Theorem~\ref{t.criticaltheorem})
and beta-transformations (Theorem~\ref{l_8_main_lemma}), and we reflect this by wherever possible adopting analogous notation.
Since the proofs of Theorems~\ref{t.criticaltheorem}
and~\ref{l_8_main_lemma} are given in full detail, and
in order to minimise repetition, we
therefore only indicate the general strategy of the proof here,
omitting calculations which closely resemble those given in the proofs of Theorems~\ref{t.criticaltheorem}
and~\ref{l_8_main_lemma}.

Using e.g.~\cite[Theorem~4.6, Lemmas~4.8, and~4.11]{Wen16} it can be shown that for all $f\in\cA(M)$, there exist $K,\delta>0$, $\lambda>1$, and a neighbourhood $U$ of $f$ in $\cA(M)$, such that if $g\in U$, $n\in\N$, and
$\max_{0\le i\le n} d\bigl(g^i(x),g^i(y)\bigr)<\delta$, then
\begin{equation}\label{uniformexponentialinstability}
d\bigl(g^i(x),g^i(y)\bigr) \le K \lambda^{-\min\{i,\,n-i\}}( d(x,y)+d(g^n(x),g^n(y) )\quad\text{ for all }0\le i\le n.
\end{equation} 
Moreover $\cA(M)$ has structural stability 
(see e.g.~\cite[Theorems~4.19 and~4.20]{Wen16}):
the above $U$ can be chosen such that if
$g\in U$
then
$\LIP(g) \le 2\LIP(f)$,
and there is a homeomorphism $h_g:M\to M$ with $ h_g\circ f = g\circ h_g$, and 
$d_\infty(h_g,\id)
<\min\{\Delta(\cO)/4,1\}$, which implies $\Delta(\cO_g) \ge \Delta(\cO)/2$.
For Anosov diffeomorphisms there is also a Ma\~n\'e lemma with semi-norm control
(cf.~\cite[Theorem~3.1 and Section~4]{Bou11}):
there exists $L>0$ (depending on $f$) such that given 
$\phi\in\Holder{\alpha}(M)$, there exists $u\in\Holder{\alpha}(M)$ with
$\Hseminorm{\alpha}{u} \le L \Hseminorm{\alpha}{\phi}$
and
$\psi\= \overline{\phi}+u-u\circ f \le 0$.
Defining $r\=\min\{\Delta(\cO)/(8\LIP(f)), \delta\}$, we have
$r\le \Delta(\cO)/(8\LIP(f)) \le \Delta(\cO_g)/(4\LIP(f))\le \Delta(\cO_g)/(2\LIP(g))$.

Defining $\psi_g\=\overline{\phi}+u-u\circ g$, we estimate $\psi_g\le \tau \=  \operatorname{O}(\Hseminorm{\alpha}{\phi} d_\infty(h_g,\id)^\alpha)$
(cf.~(\ref{e.psi'<tau})); note, however, that unlike in (\ref{e.psi'<tau}), here we estimate $\psi_g$ by using that
\begin{align*}
    d_\infty(f,g) 
    &\le d_\infty(f,h_g\circ f) + d_\infty(g,g\circ h_g) \\
    &\le (1+\LIP(g))d_\infty(h_g,\id) 
    \le (1+2\LIP(f)) d_\infty(h_g,\id),
\end{align*}
and therefore $\Hnorm{\infty}{u\circ f - u\circ g} \le L\Hseminorm{\alpha}{\phi} (1+2\LIP(f))^\alpha d_\infty(h_g,\id)^\alpha $.
Defining $\eta\=\int \! \psi\, \mathrm{d}\mu_{\cO_g}$ we see that 
$\eta = -\operatorname{O}(\Hseminorm{\alpha}{\phi} d_\infty(h_g,\id)^\alpha)$ (cf.~(\ref{e.eta})).
For $C>0$, define $\phi_C\=\phi-C\Hseminorm{\alpha}{\phi} d_\infty(h_g,\id)^{\alpha/2} d(\cdot,\cO_g)^\alpha$ and
$\psi_C\=\psi_g-C\Hseminorm{\alpha}{\phi} d_\infty(h_g,\id)^{\alpha/2} d(\cdot,\cO_g)^\alpha$, and note that $\eta=\int\! \psi_C\, \mathrm{d}\mu_{\cO_g}$.

We wish to show that $\mu_{\cO_g}$ is
$(g,\psi_C)$-maximizing, i.e.,~(cf.~\cite[Proposition~2.2]{Je19}) that for $x\in M$,
\begin{equation}\label{gpsiliminf}
\liminf_{n\to+\infty} \frac{1}{n} S_n^g\psi_C(x)\le\eta.
\end{equation}
Define $\rho\=\bigl(\frac{\tau-\eta}{C\Hseminorm{\alpha}{\phi} d_\infty(h_g,\id)^{\alpha/2}}\bigr)^{1/\alpha}$, to make $\psi_C(x)\le\eta$ if
$x\notin B(\cO_g,\rho)$.
For $x\in M$ we recursively define sequences $\{x_t\}_{t\in\N}$ and $\{n_t\}_{t\in\N}$,
such that $x_{t+1}=g^{n_t}(x_t)$ for all $t$.
Set $x_1\=x$, then as a recursive step assume, for $t\in\N$, that $\{x_i\}_{i=1}^t$ and $\{n_i\}_{i=1}^{t-1}$ are defined, and consider the following three
(exhaustive and mutually exclusive) cases.
As a first case, if $x_t\notin B(\cO_g,\rho)$ then let $n_t\=1$ and $x_{t+1}\= g(x_t)$.
As a second case (in the purely expanding context of Theorem~\ref{t.criticaltheorem}, this corresponds to the trivial case $x_t \in \cO_g$), suppose that $\cO^g(x_t)\subseteq B(\cO_g,r)$.
Let $y\in\cO_g$ be such that $d(x_t,\cO_g)=d(x_t,y)\le r$.
Since $r\le \Delta(\cO_g)/(2\LIP(g))$ then $d(g(x_t),g(y))\le \Delta(\cO_g)/2$, thus $d(g(x_t),\cO^g(y))=d(g(x_t),\cO_g)\le r$.
By induction it can then be shown that for all $n\in\N$, $d(g^n(x_t),g^n(y))=d(g^n(x_t),\cO_g)\le r$.
Now $r\le \delta$, so 
\begin{align*}
&    \liminf_{n\to+\infty} \frac{1}{n} S_n^g\psi_C(x_t) \\
&\qquad\le \liminf_{n\to+\infty}\frac{1}{n} S_n^g\psi_g(x_t) 
= \liminf_{n\to+\infty}\frac{1}{n}S_n^g \overline{\phi}(x_t) \\
&\qquad\le \lim_{n\to+\infty}\frac{1}{n}S_n^g\overline{\phi}(y) 
+\liminf_{n\to+\infty}\frac{1}{n}\Hseminorm{\alpha}{\phi} \sum_{i=0}^{n-1} d\bigl(g^i(x_t),g^i(y)\bigr)^\alpha \\
&\qquad\le \eta + \liminf_{n\to+\infty} \frac{2}{n}K\Hseminorm{\alpha}{\phi} \frac{\lambda^\alpha}{\lambda^\alpha-1}\bigl(d(x_t,y)+d\bigl(g^{n-1}(x_t),g^{n-1}(y)\bigr)\bigr)^\alpha,
\end{align*}
therefore $\liminf_{n\to+\infty} \frac{1}{n} S_n^g\psi_C(x_t)
\le \eta+  \liminf_{n\to+\infty}\frac{2K}{n} (2r)^\alpha \Hseminorm{\alpha}{\phi} \frac{\lambda^\alpha}{\lambda^\alpha-1} 
=\eta$, so (\ref{gpsiliminf}) follows. 

The third case is where $\cO(x_t)$ is not contained in $B(\cO_g,r)$, but $x_t\in B(\cO_g,\rho)$.
Let $y\in \cO_g$ with $d(x_t,y)=d(x_t,\cO_g)$.
Define 
$$N\=\min\bigl\{i\in\N_0:d\bigl(g^{i+1}(x_t),g^{i+1}(y)\bigr)\ge r \bigr\}$$
and
$$m\=\max\bigl\{i\in\N_0: i\le N, \, d\bigl(g^{i-1}(x_t),g^{i-1}(y)\bigr)< \rho\bigr\}$$
(cf.~the similar definitions (\ref{e.def.N}) and (\ref{e.def.m}) for expanding maps), where $r\le \Delta(\cO_g) /(2\LIP(g))$ implies $d \bigl(g^i(x_t), \cO_g \bigr) = d\bigl(g^i(x_t), g^i(y)\bigr)$ if $0\le i \le N$. We define
$n_t\=N+1$ and $x_{t+1}\=g^{N+1}(x_t)$.
Then $S_{n_t}^g \psi_C(x_t) \le 0$ can be obtained via
an argument analogous to the one used in the proof of Theorem~\ref{t.criticaltheorem},
by separately estimating four terms, and choosing $C$ sufficiently large. The only significant difference is that in the inequality (\ref{e.ineq2}), the term $\Absbig{S_m^T \psi'(x_t) - S_m^T \psi'(y)}$ is estimated via the distance-expanding property, whereas here we estimate $\Absbig{S_m^g \psi_g(x_t)-S_m^g \psi_g(y)}$ using (\ref{uniformexponentialinstability}). 
This proves (\ref{gpsiliminf}),
and completes the recursive step.

Having shown that $\mu_{\cO_g}$ is
$(g,\psi_C)$-maximizing,
it follows that $\mu_{\cO_g}$ is the \emph{unique} $g$-maximizing measure for 
$\phi-2C\Hseminorm{\alpha}{\phi} d_\infty(h_g,\id)^{\alpha/2} d(\cdot,\cO_g)^\alpha$,
as required.
\end{proof}

The above joint perturbation result allows us to deduce the following slightly stronger version of Theorem~\ref{anosovJTPO}:

\begin{thml}[Joint TPO for Anosov diffeomorphisms]
\label{anosovJTPO.moreprecise}
Let $M$ be a smooth compact Riemannian manifold,
with  distance function induced by the Riemannian metric,
and let $\cA(M)$ be the space of $C^1$ Anosov diffeomorphisms on $M$, equipped with the $C^1$ topology.
  For all $\alpha \in (0,1]$, the set
$\bigl\{(f,\phi)\in\cA(M)\times\Holder{\alpha}(M):\phi\in\Lock^\alpha(f)\bigr\}$ is open and dense in $\cA(M)\times \Holder{\alpha}(M)$.
\end{thml}
\begin{proof}
This follows from Theorem
\ref{jointperturbationuniformhyp}
by an argument analogous to the one used to prove Theorem~\ref{t.JTPO.general} from
Theorem~\ref{t.criticaltheorem}, and the fact that
every Anosov diffeomorphism has TPO,  
by \cite{HLMXZ25}.
\end{proof}

In fact Joint TPO can also be proved in the context of $C^1$ function spaces:
if $C^1(M)$ is the space of $C^1$ functions on $M$, equipped with its usual topology,
then there is an open dense subset of pairs 
$(T,\phi)\in \cA(M)\times C^1(M)$ 
with the periodic optimization property.
More specifically, if $\Lock^{C^1}(f)$
denotes the set of those $\phi\in C^1(M)$
whose $f$-maximizing measure is unique and periodic,
with $\Mmax(f,  \phi) =\Mmax(f,  \psi)$ for all $\psi \in  C^1(M)$ sufficiently close to $\phi$ in $C^1(M)$,
then we have:

\begin{theorem}[Joint TPO for Anosov diffeomorphisms, $C^1$ function space]\label{anosovJTPO_C1}
    Let $M$ be a smooth compact Riemannian manifold, with distance function induced by the Riemannian metric,
and let $\cA(M)$ be the space of $C^1$ Anosov diffeomorphisms on $M$, equipped with the $C^1$ topology.
The set
$\bigl\{(f,\phi)\in\cA(M)\times C^1(M):\phi\in\Lock^{C^1}(f)\bigr\}$ is open and dense in $\cA(M)\times C^1(M)$.
\end{theorem}
\begin{proof}
We will first
establish a $C^1$ joint perturbation result.
In other words, we will show that
for $f\in\cA(M)$, 
and $\cO$ an $f$-periodic orbit,
for all $g\in\cA(M)$ sufficiently close to $f$, 
setting $\cO_g \= h_g(\cO)$,
there exists a function $v_g \in C^1(M)$ such that (a) $\Mmax(g, v_g) = \{ \mu_{\cO_g} \}$,
(b) for all
$\phi\in C^1(M)$ with 
$\MMM_{\max}(f,\phi)=\{\mu_{\cO}\}$,
the periodic measure
$\mu_{\cO_g}$
is the unique
$(g,\phi + \Hseminorm{1}{\phi} v_g)$-maximizing measure, and (c) $\lim_{g\to f}\Hnorm{\infty}{\mathrm{D}v_g} =0$.

Let us fix $f\in\cA(M)$, and the $f$-periodic orbit $\cO$, and for $\alpha=1$ let the neighbourhood $U$, and constant $C>0$,
be as in Theorem~\ref{jointperturbationuniformhyp}.
By \cite[Proposition~2.5 and~(2.26)]{HLMXZ25}, 
choosing $u= \mathbbold{0}$ to be the function that is identically zero, we see that
    there exists $\delta>0$ such that for every $g\in U$, if $\xi\in\Holder{1}(M)$ with
    $\Hseminorm{1}{\xi} < 3$ and $\Hnorm{\infty}{\xi} < \delta$, then
    for  
    $d_{g,\xi} \= -d(\cdot, \cO_g) + \xi$,
     the unique $(g,d_{g,\xi})$-maximizing measure is $\mu_{\cO_g}$. 

     Now fix $g\in U$, and let $\phi\in C^1(M)$ with 
$\MMM_{\max}(f,\phi)=\{\mu_{\cO}\}$.
Since $\phi$ is $C^1$, and therefore Lipschitz,
the case $\alpha=1$ of
Theorem~\ref{jointperturbationuniformhyp}  gives that $\Mmax(g, \phi - 2C_g \Hseminorm{1}{\phi} d(\cdot, \cO_g)) = \bigl\{\mu_{\cO_g} \bigr\}$, where
$C_g \= C d_\infty(h_g, \id)^{1/2}$.
By \cite[Theorem~2.7]{HLMXZ25}, there exists $w\in C^1(M)$ with $\Hnorm{\infty}{\mathrm{D}w} < 3/2$ and $\Hnorm{\infty}{w+d(\cdot, \cO_g)} < (5/6)\delta$.
Define the $C^1$ function
$\psi \= \phi + 12  C_g \Hseminorm{1}{\phi} w$, and write
\begin{equation}\label{psi_3term}
\psi = ( \phi- 2C_g \Hseminorm{1}{\phi} d(\cdot, \cO_g) )  +
( 12C_g \Hseminorm{1}{\phi} w + 2C_g \Hseminorm{1}{\phi} d(\cdot, \cO_g) ).
\end{equation}
The above inequalities
can be used
to estimate 
\begin{align*}
        12C_g \Hseminorm{1}{\phi} \Hseminorm{1}{w+ d(\cdot, \cO_g)} &\le 12C_g \Hseminorm{1}{\phi} ( \Hseminorm{1}{w} + \Hseminorm{1}{d(\cdot, \cO_g)} ) < 30 \Hseminorm{1}{\phi} C_g \quad \text{ and }\\
        12C_g \Hseminorm{1}{\phi}\Hnorm{\infty}{w+ d(\cdot, \cO_g)} &< 10 C_g \Hseminorm{1}{\phi}\delta,
    \end{align*}
    so taking $\xi\=\frac{6}{5}(w+d(\cdot,\cO_g))$ gives
    that $\mu_{\cO_g}$ is the unique $g$-maximizing measure
    for the function $-d(\cdot,\cO_g)+\frac{6}{5}(w+d(\cdot,\cO_g))$,
    hence for its positive multiple
    $-10C_g \Hseminorm{1}{\phi} d(\cdot,\cO_g)+12C_g \Hseminorm{1}{\phi}(w+d(\cdot,\cO_g))=12 C_g \Hseminorm{1}{\phi} w + 2 C_g \Hseminorm{1}{\phi} d(\cdot, \cO_g)$,
    and hence also for the function $w$.
    So both 
    $\phi- 2C_g \Hseminorm{1}{\phi} d(\cdot, \cO_g)$
    and
$12C_g \Hseminorm{1}{\phi} w + 2C_g \Hseminorm{1}{\phi} d(\cdot, \cO_g)$
have $\mu_{\cO_g}$ as their unique $g$-maximizing measure,
therefore by (\ref{psi_3term}), 
$\mu_{\cO_g}$ is the unique $g$-maximizing measure
for $\psi$, and hence the desired joint perturbation result follows by taking $v_g \= 12C_g w$.

The above $C^1$ joint perturbation result can then be used to prove Joint TPO, by an argument analogous to the
one used in the proof of Theorem~\ref{t.JTPO.general}:
more precisely, the ideas used in the proof of Theorem~\ref{t.JTPO.general} can be adapted for the
present $C^1$ case if 
the term 
$
-Cd_{\Lip}(T_0, T)^{\alpha/2} d(\cdot,\cO)^\alpha
=
-Cd_{\Lip}(T_0, T)^{1/2} d(\cdot,\cO)
$ 
is replaced 
by functions $v_g$ satisfying
conditions (a), (b), and (c), for $g$ sufficiently close to $f$.
\end{proof}

\section{Beta-transformations and maximizing measures}\label{betashiftsection}

We now turn our attention to typical periodic optimization, and joint typical periodic optimization, for a specific one-parameter family of maps on the unit interval.
This family of \emph{beta-transformations} has been studied since
 the foundational papers of R\'enyi \cite{Re57} and Parry \cite{Par60}, motivated in particular by connections with aspects of number theory, in view of the link with
\emph{beta-expansions} of the form $\varepsilon_1/\beta + \varepsilon_2/\beta^2+\varepsilon_3/\beta^3+\cdots$
(see e.g.~\cite{AB07,Be86,CK04,DK02, DK03,FS92,Ka15,Sck80, Si03}). Beta-transformations have also been studied from the point of view of symbolic dynamics (see e.g.~\cite{AJ09,Bl89,IT74, LiSc05, Scj97,Si76}) 
and of ergodic theory
(see e.g.~\cite{Hof78,Sm73,Wal78}).

The classical nature of this subject means that
certain preparatory results in this section are either known, or resemble known results, though the literature is somewhat scattered; for ease of exposition, proofs are 
deferred until Appendix~\ref{beta_proofs_section}. 
In Subsection~\ref{betaintrosubsection} we recall the definitions of beta-transformations, beta-expansions, and beta-shifts, and the fundamental relations between these objects.
Certain monotonicity and approximation properties as a function of the parameter $\beta$
are considered in Subsection~\ref{subsec_approximation_properties},
along with notation and results concerning cylinder sets. 
In Subsection~\ref{sct_ext_maximizing_measures}, 
we develop a theory of ergodic optimization for discontinuous maps such as beta-transformations, and relations between various sets of 
invariant measures are established.

\subsection{Beta-transformations, beta-expansions, and beta-shifts}\label{betaintrosubsection}

We begin by recalling the definitions and basic properties of beta-transformations, as well as the related beta-expansions and beta-shifts.

\begin{definition}[Beta-transformations]\label{d_T_beta_and_U_beta}
	Given a real number $\beta> 1$, the \emph{beta-trans\-form\-ation} $T_{\beta}\:I \to I$ is defined by 
	\begin{equation}\label{e_def_T_beta}
		T_{\beta}  ( x  )\=\beta x- \lfloor \beta x  \rfloor , \quad x\in I  .
	\end{equation}
	Recall that $\lfloor x\rfloor'=\max\{n\in \Z: n<x\}$ for $x\in \R$. The \emph{upper beta-transformation} $U_\beta \: I \to I$ is defined by $U_\beta(0)\=0$ and
	\begin{equation}\label{e_def_U_beta}
		U_\beta (x )\=\beta x- \lfloor\beta x \rfloor'  , \quad x\in I\smallsetminus\{0\}   .
	\end{equation}
Note that Kalle and Steiner \cite[Definition~2.4]{KS12} refer to the upper beta-transformation as the left-continuous beta-transformation. 
\end{definition}

\begin{definition}[Beta-expansions]\label{d_beta_expansions}
	Given a real number $\beta> 1$,
	write 
	\begin{equation*}
		\cB\=\{ 0, \,1, \,\dots, \, \lfloor \beta \rfloor  \}.
	\end{equation*}
	Define the \emph{$\beta$-expansion} of $x\in I$ to be the sequence 
	\begin{equation*}
		\underline{\ve}  ( x,\beta  ) = \{ \ve_n ( x,\beta  )  \}_{n\in\N} 
	\in \cB^\N 
	\end{equation*}
	given by
	\begin{equation}\label{en}
		\ve_n ( x,\beta  )\=\bigl \lfloor \beta T_{\beta}^{n-1} ( x )  \bigr \rfloor \quad \text{ for all }  n\in \N ,
	\end{equation}		
	and define the \emph{upper $\beta$-expansion}\footnote{Blanchard \cite[p.~136]{Bl89} refers to the upper $\beta$-expansion
	$\underline{\ve}^*  ( x,\beta  )$
	as a kind of \emph{incorrect} $\beta$-expansion.}  of $x\in I$ to be the sequence 
	\begin{equation*}
		\underline{\ve}^*  ( x,\beta  ) = \{ \ve^*_n ( x,\beta  )  \}_{n\in \N}
	\in \cB^\N 
	\end{equation*} 
	given by 
	\begin{equation}\label{en*}
		\ve^*_n ( x,\beta  )\= \bigl \lfloor \beta U_{\beta}^{n-1} ( x )  \bigr \rfloor'\quad \text{ for all }   n\in \N.
	\end{equation}	
\end{definition}

\begin{rem}\label{r_after_def_beta_expansions}
	For $\beta>1$, the beta-transformation and upper beta-transfor\-mation are related by
	\begin{equation*}
		U_\beta(x) =  \limsup_{y\to x} T_\beta(y).
	\end{equation*}
	The set $D_\beta$ of points of discontinuity for $T_\beta$ is 
	\begin{equation}  \label{e_D_beta}
		D_\beta \= T_\beta^{-1}(0) \smallsetminus \{0\}
		= U_\beta^{-1}(1)
		=\{ j/\beta:j\in \Z \} \cap (0,1]  ,
	\end{equation}
	and this is precisely the set of points at which $T_\beta$ and $U_\beta$ differ,
	with $T_\beta(x) = 0$ and $U_\beta(x) = 1$ for all $x\in D_\beta$.
\end{rem}

\begin{lemma}\label{l_continuity_T_U_on_x}
	If $\beta>1$, $n\in\N$, $a\in[0,1)$, and $b\in (0,1]$, then
	\begin{enumerate}[label=\rm{(\roman*)}]
		\smallskip		
		\item $\lim_{x\searrow a}T^n_\beta(x)=T_\beta^n(a)^+$ and $\lim_{x\nearrow b}U^n_\beta(x)=U_\beta^n(b)^-$,
		\smallskip
		
		\item $\ve_n(\cdot,\beta)$ is right-continuous on $[0,1)$ and $\ve^*_n(\cdot,\beta)$ is left-continuous on $(0,1]$,
		\smallskip
		
		\item $T^n_\beta(0)=\ve_n(0,\beta)=U^n_\beta(0)=\ve^*_n(0,\beta)=0$,
		\smallskip
		
		\item $\lim_{x\nearrow b}T_\beta^n(x)= U_\beta^n(b)^{-}$ and $\lim_{x\nearrow b}\ve_n(x,\beta)= \ve^*_n(b,\beta)$.
	\end{enumerate}
\end{lemma}

\begin{lemma}\label{l_continuity_T_U_on_beta}
	If $\beta>1$, $n\in\N$, and $x\in I$, then 
	\begin{enumerate}[label=\rm{(\roman*)}]
		\smallskip		
		\item $\lim_{\gamma\searrow \beta}T^n_\gamma(x)=T_\beta^n(x)^+$ and $\lim_{\gamma\nearrow \beta}U^n_\beta(x)=U_\beta^n(x)^-$,
		\smallskip
		
		\item $\ve_n(x,\cdot)$ is right-continuous and $\ve^*_n(x,\cdot)$ is left-continuous,
		\smallskip
		
		\item $\lim_{\gamma\nearrow \beta}T_\gamma^n(x)= U_\beta^n(x)^{-}$ and $\lim_{\gamma\nearrow \beta}\ve_n(x,\gamma)= \ve^*_n(x,\beta)$.
	\end{enumerate}
\end{lemma}

\begin{definition}[Beta-shifts]\label{beta shifts}
	Given $\beta> 1$, 
	define $\pi_\beta \: I \to \cB^\N$ 
	by
\begin{equation*}
	\pi_\beta(x)\=\underline{\ve}(x,\beta)
	=\{ \ve_n ( x,\beta  ) \}_{n\in \N},
\end{equation*}
	and define $\pi^*_\beta \: I \to \cB^\N$ by
\begin{equation*}
	\pi^*_\beta(x)\=\underline{\ve}^*(x,\beta)
	=\{ \ve^*_n ( x,\beta  ) \}_{n\in \N}.
\end{equation*}
	Define the \emph{beta-shift} $\cS_\beta$ to be the closure in $\cB^\N$ of the image 
under $\pi_\beta$ of the half-open interval $[0,1)$, in other words,
\begin{equation}\label{e_original_def_S_beta}
    \cS_\beta\=\overline{\pi_\beta([0,1))},
\end{equation}
	 where $\cB^\N$ is equipped with the product topology. 
     Note that Lemma~\ref{l_continuity_T_U_on_beta}~(iii) implies that $\pi_\beta^*(I) \subseteq \cS_\beta$.
\end{definition}

\begin{definition}\label{d_X_beta_and_tilde_X_beta}
	Given $\beta>1$, let $X_\beta$ denote the closure in $\cB^\N$
	of the image $\pi_\beta(I)$.
	Define $h_\beta \: X_\beta \to I$ by
\begin{equation}\label{hbeta}
	h_\beta(\{z_i\}_{i\in \N})\=\sum_{i =1}^{+\infty}z_i\beta^{-i}.
\end{equation}
For each $x\in I$, define $i_x \: (1,+\infty) \to \N_0^\N$ and $i_x^* \: (1,+\infty) \to \N_0^\N$ by
\begin{equation}\label{ix beta}
	i_x(\beta) \= \pi_\beta(x) \quad \text{ and } \quad i_x^*(\beta) \= \pi_\beta^*(x).
\end{equation}
\end{definition}

The following lemma means that our definition of upper $\beta$-expansion is equivalent to the definition of incorrect $\beta$-expansion in \cite{IT74} and \cite{YT21}. 

\begin{lemma}\label{l_equivalent_def_pi_*}
	Fix $\beta>1$. Then $\pi_\beta(0)=\pi_\beta^*(0)=(0)^\infty$ and $\pi^*_\beta(a)=\lim_{x\nearrow a}\pi_\beta(x)$ for all $a\in(0,1]$.
\end{lemma}

Proposition~\ref{p_relation_of_coding} below collects a number of basic properties of beta-trans\-formations and beta-expansions that will be 
required later; the majority of the results can be found in the existing literature (specifically,
in \cite{Bl89, IT74,Par60,Re57,YT21}), and for the remainder we provide proofs. 

\begin{prop}\label{p_relation_of_coding}
	If $\beta>1$, then the following hold:
	\begin{enumerate}[label=\rm{(\roman*)}]
		\smallskip		
		\item\label{p_relation_of_coding__i} We have
		\begin{equation*}
			\pi^*_\beta(1) = \begin{cases}
				(	z_1z_2\dots(z_n-1))^\infty & \text{ if } \pi_\beta(1)=z_1z_2\dots z_n(0)^\infty,z_n>0 , \\
				\pi_\beta(1)  & \text{ if $\pi_\beta(1)$ has infinitely many nonzero terms}. 
			\end{cases}
		\end{equation*}
		
		\item\label{p_relation_of_coding__ii} For each $x\in (0,1]$,
		\begin{equation*}
			\pi^*_\beta(x) = \begin{cases}
				z_1z_2\dots(z_n-1)\pi_\beta^*(1) & \text{ if } \pi_\beta(x)=z_1z_2\dots z_n (0)^\infty,z_n>0 , \\
				\pi_\beta(x)  & \text{ if $\pi_\beta(x)$ has infinitely many nonzero terms}. 
			\end{cases}
		\end{equation*}
		
		\item\label{p_relation_of_coding__iii} $\sigma\circ \pi_\beta=\pi_\beta \circ T_\beta$ and $\sigma\circ \pi^*_\beta=\pi^*_\beta \circ U_\beta$ on $I$.
		\smallskip
		
		\item\label{p_relation_of_coding__iv} $(h_\beta\circ \pi_\beta) (x) = x$ and $(h_\beta\circ \pi^*_\beta) (x) = x$ for each $x \in I$.
		\smallskip
		
		\item\label{p_relation_of_coding__v} $h_\beta \circ \sigma = T_\beta \circ h_\beta$ on $\pi_\beta(I)$ and $h_\beta \circ \sigma = U_\beta \circ h_\beta$ on $\pi^*_\beta(I)$.
		\smallskip
		
		\item\label{p_relation_of_coding__vi} $\pi_\beta$ and $\pi^*_\beta$ are strictly increasing, i.e.,~$x<y$ implies $\pi_\beta(x)\prec \pi_\beta(y)$ and $\pi^*_\beta(x)\prec \pi^*_\beta(y)$.
		\smallskip
		
		\item\label{p_relation_of_coding__vii} $\pi_\beta(x)\prec \pi_\beta^*(y)$ if $0\leq x<y\leq 1$.
		\smallskip
		
		\item\label{p_relation_of_coding__viii} $\bigl\{ \omega\in X_\beta: \pi^*_\beta(x)\prec \omega\prec \pi_\beta(x) \bigr\}=\emptyset$ for all $x\in I$.
		\smallskip
		
		\item\label{p_relation_of_coding__ix} $\pi_\beta$ is right-continuous on $[0,1)$ and $\pi^*_\beta$ is left-continuous on $(0,1]$.
		\smallskip
		
		\item\label{p_relation_of_coding__x} $h_\beta$ is a continuous surjection and is nondecreasing, i.e.,~$\omega \prec \omega'$ implies $h_\beta(\omega)\leq h_\beta (\omega')$.
		\smallskip
		
		\item\label{p_relation_of_coding__xi} The inverse image $h_\beta^{-1}(x)$ of $x\in (0,1]$ consists either of one point $\pi_\beta(x)$ or of two distinct points $\pi_\beta(x)$ and $\pi^*_\beta(x)$. The latter case occurs only when $T^n_\beta(x)=0$ for some $n\in\N$. Moreover, $h_\beta^{-1}(0)=\{(0)^\infty\}$.
		\smallskip
		
		\item\label{p_relation_of_coding__xii} The function $h_\beta \: (X_\beta,d_\beta)\to (I,d)$ is Lipschitz.
		\smallskip 
		
		\item\label{p_relation_of_coding__xiii} For each $x \in (0,1]$, the functions $i_x$ and $i_x^*$ are both strictly increasing functions. Moreover, $i_0(\beta)=i_0^*(\beta)=(0)^\infty$ for all $\beta>1$.
		\smallskip
		
		\item\label{p_relation_of_coding__xiv} For each $x\in I$, the function $i_x$ is right-continuous and the function $i_x^*$ is left-continuous.
	\end{enumerate}
\end{prop}

The following classification of values $\beta>1$, and the interpretation in terms of dynamical behaviour, will be required in our subsequent investigations.

\begin{definition}[Classification of $\beta >1$]\label{d_classification_beta}
	A real number $\beta>1$ is said to be
	\begin{enumerate}[label=\rm{(\roman*)}]
		\smallskip
		\item a \emph{simple beta-number} if $\underline{\ve} (1,\beta )$ has only finitely many nonzero terms;
		
		\smallskip
		\item a \emph{non-simple beta-number} if 
		$\underline{\ve} (1,\beta )$ is preperiodic  (i.e.,~there exists $n\in \N$ such that $\sigma^n (\underline{\ve} (1,\beta) )$ is periodic), but $\beta$ is not a simple beta-number;
		
		\smallskip
		\item \emph{non-preperiodic} if $\beta$ is not a beta-number (i.e.,~$\beta$ satisfies neither (i) nor (ii) above). 
	\end{enumerate}
\end{definition}
\begin{rem}\label{r_after_classification_of_beta}
	The terminology \emph{beta-number}, as well as \emph{simple beta-number}, was introduced by Parry \cite{Par60},  who proved (see \cite[Theorem~5]{Par60}) that the set of simple beta-numbers is dense in $(1,+\infty)$. Some authors refer to simple beta-numbers as \emph{Parry numbers} (see e.g.~\cite{Ka15}).
	It is readily seen that $\beta$ is a simple beta-number if and only if $1$ is a periodic point of $U_\beta$. 
\end{rem}

The following proposition summarises the relation between periodic points and invariant measures of $T_\beta$ and $U_\beta$.

\begin{prop}\label{p_relation_T_beta_and_wt_T_beta}
	If $\beta>1$, then the following hold: 
	\begin{enumerate}[label=\rm{(\roman*)}]
		\smallskip
		\item $T_\beta^{-1}(0)=\{0\} \cup D_\beta$, 
        where $\{0\} \cap D_\beta =\emptyset$. Moreover,
        $T_\beta^{-1}(1)=\emptyset$, $U_\beta^{-1}(0)=\{0\}$, and $U_\beta^{-1}(1)=D_\beta$.  The maps $T_\beta$ and $U_\beta$ coincide when restricted to $I \smallsetminus D_\beta$.
		
		\smallskip
		\item $\Per(T_\beta)\subseteq \Per(U_\beta)$. If $\cO_\beta^*$ is a periodic orbit for $U_\beta$, then $\cO_\beta^*$ is a periodic orbit for $T_\beta$ if and only if $1\notin \cO_\beta^*$.
		
		\smallskip
		\item $\MMM(I,T_\beta)\subseteq \MMM(I,U_\beta)$. If $\mu\in \MMM(I,U_\beta)$, then $\mu\in \MMM(I,T_\beta)$ if and only if $\mu(\{1\})=0$.
		
		\smallskip
		\item If $\beta$ is not a simple beta-number, then $\mathrm{Per}(T_\beta)=\mathrm{Per}(U_\beta)$ and $\MMM(I,T_\beta) \allowbreak =\MMM(I,U_\beta)$.  
		
		\smallskip
		\item If $\beta$ is a simple beta-number, then $\Per(U_\beta) = \Per(T_\beta) \cup \cO_\beta^*(1)$ and $\cM(I,U_\beta)$ is the convex hull of $\bigl\{\mu_{\cO_\beta^*(1)}\bigr\}\cup\cM(I,T_\beta)$. 
		
		\smallskip
		\item $T_\beta$ and $U_\beta$ are distance-expanding.
		Specifically, 
		if
		$x, \, y \in I$ with $\abs{x-y} < 1/(2\beta)$, then $\abs{T_\beta(x) - T_\beta(y)} \ge \beta \abs{x-y}$ and $\abs{U_\beta(x) - U_\beta(y)} \ge \beta \abs{x-y}$.    
	\end{enumerate}
\end{prop}

While the support of any $T$-invariant probability measure $\mu$ is such that $\supp \mu = T(\supp \mu)$ in the case where $T$ is continuous (see e.g.~\cite[p.~156]{Ak93}), the same is not true for the discontinuous maps $T_\beta$ and $U_\beta$; nevertheless we do have the following result.

\begin{lemma}\label{support_invariant_set}
	Suppose $\beta>1$ and $\mu \in \MMM(I,U_\beta)$. Then
	$U_\beta(\supp\mu)=\supp \mu$
	if $0\notin \supp \mu$, and $T_\beta(\supp\mu)=\supp \mu$
	if $1\notin \supp \mu$.
\end{lemma}

\subsection{Monotonicity and approximation properties in parameter space}\label{subsec_approximation_properties}

Here we recall some monotonicity and approximation properties for the one-parameter family of beta-shifts.

The following proposition characterises those sequences on the alphabet $\cB=\{ 0, \, 1,\, \dots, \, \lfloor \beta  \rfloor  \} $ that arise as the $\beta$-expansion of a real number $x\in [0,1]$.

\begin{prop}\label{2.11}
	If $\beta>1$, then the following hold: 
	\begin{enumerate}[label=\rm{(\roman*)}]
		\smallskip
		\item $\pi_\beta([0,1))=\bigl\{A \in \cB^\N: \sigma^n(A)\prec \pi^*_\beta(1)\text{ for all }n\in \N_0\bigr\}$.
		
		\smallskip
		\item $\cS_\beta$ 
		can also be expressed as
		\begin{equation}\label{e_def_S_beta}
			\cS_\beta
			= \bigl\{A \in \cB^\N: \sigma^n(A)\preceq \pi^*_\beta(1)\text{ for all }n\in \N_0\bigr\} .
		\end{equation}
	\end{enumerate}
\end{prop}

The representation (\ref{e_def_S_beta}) implies that the closed subset $\cS_\beta$ satisfies $\sigma(\cS_\beta)=\cS_\beta$;
	in other words it is a subshift of the full shift $\cB^\N$, so we may regard the shift map $\sigma$ as a self-map
	$\sigma \: \cS_\beta \to \cS_\beta$. 	

  \begin{rem}\label{r_after_S_beta}
	\begin{enumerate}[label=\rm{(\roman*)}]
		\smallskip
		\item Define 
		$\wt{X}_\beta\=\bigl\{A \in \cB^\N: \sigma^n(A)\preceq \pi_\beta(1)\text{ for all }n\in \N_0\bigr\}$.

         \smallskip
        \item If $\beta$ is not a simple beta-number, then $\pi_\beta(1)=\pi_\beta^*(1)$ (see Proposition~\ref{p_relation_of_coding}~\ref{p_relation_of_coding__i}), and hence $\cS_\beta=X_\beta=\wt{X}_\beta$ by (\ref{e_def_S_beta}), Definitions~\ref{beta shifts} and~\ref{d_X_beta_and_tilde_X_beta}. 

        \smallskip
		\item If $\beta$ is a simple beta-number, then $\pi_\beta^*(1)\prec \pi_\beta(1)$ (see Proposition~\ref{p_relation_of_coding}~\ref{p_relation_of_coding__i}), hence $X_\beta$ is the union of
		the beta-shift $\cS_\beta$ and the singleton set $\{\pi_\beta(1)\}$ 
		(which is disjoint from $\cS_\beta$),
		and 
		\begin{equation}\label{inclusions_proper}
			\cS_\beta\subseteq X_\beta\subseteq \wt{X}_\beta
		\end{equation} 
		by (\ref{e_def_S_beta}), Definitions~\ref{beta shifts} and~\ref{d_X_beta_and_tilde_X_beta}. 
		The inclusions in (\ref{inclusions_proper}) are proper:
		for example, when $\beta=2$, we have 
		$
			2(0)^\infty\in X_2\smallsetminus\cS_2$, $12(0)^\infty\in \wt{X}_2\smallsetminus X_2$.
        
        \smallskip
		\item If $\beta$ is a simple beta-number then
		$\sigma$ maps $\cS_\beta$ surjectively onto itself, and maps $\wt{X}_\beta$ surjectively onto itself,
		but $\sigma \: X_\beta\to X_\beta$ is not surjective.
		\smallskip
		
		\item A complement to Proposition~\ref{2.11}~(i) is that, for each $A\in\N_0^\N$, there exists $\beta>1$ with $A=\underline{\ve} (1,\beta )$ if and only if 
		$\sigma^n (A )\prec A \text{ for all } n\in\N;$
		and if such a number $\beta>1$ exists then it is unique (see \cite[Corollary~1]{Par60}). Consequently, each of the three classes in Definition~\ref{d_classification_beta} is readily seen to be nonempty.
		\smallskip
		
		\item Some authors define the beta-shift to be either $X_\beta$ or $\wt{X}_\beta$, instead of $\cS_\beta$.
		For example it is defined to be $X_\beta$ in \cite[p.~1696]{AB07}, \cite[Definition~2.2]{Scj97}, and \cite[p.~1438]{KQ22}, and defined to be $\wt{X}_\beta$ in \cite[p.~248]{Si76} and \cite[p.~179]{Wal82}.
	\end{enumerate}
\end{rem}

\begin{lemma}\label{l_apprioxiation_beta_shif_beta}
	If $\beta>1$, then the following hold: 
	\begin{enumerate}[label=\rm{(\roman*)}]
		\smallskip
		\item If $1<\beta'<\beta$, then $\cS_{\beta'} \subseteq \cS_\beta$.
		
		\smallskip
		\item 
		$\cS_\beta=\overline{\bigcup_{\gamma\in(1,\beta)}\cS_\gamma}.$
	\end{enumerate}		
\end{lemma}

\begin{rem}
	Lemma~\ref{l_apprioxiation_beta_shif_beta}
	is hinted at as part of \cite[Proposition~4.1]{IT74} (though in \cite{IT74} it is slightly mis-stated, and not proved,
	so for the convenience of the reader we include a proof in Appendix~\ref{beta_proofs_section}).
	We note that another part of \cite[Proposition~4.1]{IT74}  is 
	false: in general it is not the case that $\cS_\beta=\bigcap_{\gamma>\beta}\cS_\gamma$
	(for example if $\beta=2$ then $2(0)^\infty\in \cS_\gamma$ for all $\gamma>2$, but $2(0)^\infty\notin \cS_2$),
	however the intersection can be expressed as
	\begin{equation*}
		\wt{X}_\beta=\bigcap_{\gamma>\beta}\cS_\gamma.
	\end{equation*}
\end{rem}

\begin{definition}
	For $1<\gamma<\beta$, define
	\begin{equation} \label{Hbetagamma}
		H_{\beta}^{\gamma}\=h_\beta(\cS_{\gamma})=\bigg\{\sum_{i=1}^{+\infty}z_i\beta^{-i}: \{z_i\}_{i\in\N}\in \cS_{\gamma}\bigg\},
	\end{equation} 
	and if $\psi\in\Holder{\alpha}(I)$ then define the corresponding \emph{restricted ergodic supremum}
	\begin{equation}\label{e_def_beta_gamma}
     \begin{aligned}
		\mpe_{\beta, \, \gamma}(\psi)        
        \={}&\mpe\bigl(T_\beta|_{H^\gamma_\beta},\psi|_{H^\gamma_\beta}\bigr) \\
        ={}&\sup\biggl\{\int \! \psi \, \mathrm{d} \mu :\mu \in \mathcal{M}(I,T_\beta),\,\supp \mu \subseteq H_\beta^\gamma\biggr\}.
     \end{aligned}
	\end{equation}
\end{definition}

\begin{lemma}\label{H beta gamma}
	Suppose $\beta>1$. 
	If $\cK \subseteq I$ is a nonempty compact set with $1\notin \cK=T_\beta(\cK)$, then there exists $\beta' \in (1,\beta)$ such that $\cK \subseteq H_\beta^\gamma$ for each $\gamma \in (\beta', \beta)$. 
	
	Similarly,
	if  $\cK^*$ is a nonempty compact set with $1\notin \cK^*=U_\beta(\cK^*)$, then there exists $\beta' \in (1,\beta)$ such that $\cK^* \subseteq H_\beta^\gamma$ for each $\gamma \in (\beta' , \beta)$.
\end{lemma}

Recall that  a homeomorphism $g \: Y_1 \to Y_2$ between metric spaces $(Y_1,d_1)$ and $(Y_2,d_2)$
is \defn{bi-Lipschitz} if there exists a constant $C \geq 1$ such that for all $u,\, v \in Y_1$,
$
	C^{-1} d_1( u, v ) \leq d_2 ( g(u), g(v) ) \leq C d_1 ( u, v ) $.

\begin{lemma}\label{l_bi_lipschitz_S_gamma_H_gamma}
	For $1<\gamma<\beta$, the map $\pi_\beta|_{H_\beta^\gamma}\:\bigl(H_\beta^\gamma,d\bigr)\to (\cS_\gamma,d_\beta)$ is bi-Lipschitz.
\end{lemma}

\begin{lemma}\label{l_property_H_gamma}
	For each $\beta>1$ and each $\gamma \in (1,\beta)$, the set $H_\beta^{\gamma}$ is a closed subset of $I$ satisfying $T_\beta \bigl( H_\beta^\gamma \bigr)\subseteq H_\beta^\gamma$ and the restricted beta-transformation $T_\beta|_{H_\beta^{\gamma}} \: H_\beta^{\gamma} \to H_\beta^{\gamma}$ has the following properties:
	\begin{enumerate}[label=\rm{(\roman*)}]
		\smallskip
		\item $T_\beta|_{H_\beta^{\gamma}}$ is Lipschitz.
		
		\smallskip
		\item $T_\beta|_{H_\beta^{\gamma}}$ is distance-expanding.
		
		\smallskip
		\item If $\gamma$ is a simple beta-number, then $T_\beta|_{H_\beta^{\gamma}}$ is an open mapping.
	\end{enumerate}
\end{lemma}

\medskip
We conclude this subsection with some notation and results concerning cylinder subsets of $I$.

\begin{definition}\label{d_n_cylinders}
	Fix $\beta>1$ and $n\in \N$. A length-$n$ prefix $(\myepsilon_1,   \myepsilon_2, \dots, \myepsilon_n)$ is said to be \emph{$\beta$-admissible} 
	if $\myepsilon_1  \dots  \myepsilon_n  (0)^\infty  \in \pi_\beta([0,1))$.
	
	For each $\beta$-admissible length-$n$ prefix, we define the corresponding \emph{$n$-cylinder} to be
	\begin{equation}\label{e_def_n_cylinders}
		I (\myepsilon_1, \myepsilon_2, \dots, \myepsilon_n)= \{x\in[0, 1) :\ve_i(x,\beta) =\myepsilon_i \text{ for all } 1 \leq i \leq n\},
	\end{equation}
	and if $T^n_\beta(I(\myepsilon_1,\myepsilon_2,\dots,\myepsilon_n))=[0, 1)$
	we say that the cylinder $I(\myepsilon_1,\myepsilon_2,\dots,\myepsilon_n)$ is \emph{full}.
	Let $W^n$ denote the set of all $n$-cylinders, and let $W^n_0$ denote the set of all full $n$-cylinders.  
	
	Note that the $n$-cylinder $I(\myepsilon_1,\myepsilon_2,\dots,\myepsilon_n)$ is a left-closed and right-open interval, with the left endpoint 
	\begin{equation}\label{left_endpoint}
		\frac{\myepsilon_{1}}{\beta}+\frac{\myepsilon_{2}}{\beta^2}+\dots+\frac{\myepsilon_n}{\beta^{n}}.
	\end{equation}
\end{definition}

\begin{prop}\label{p_cylinders}
	Fix $\beta>1$, $n\in\N$, and $I^n \= I(\myepsilon_1,\dots, \myepsilon_n)\in W^n$. 
	\begin{enumerate}[label=\rm{(\roman*)}]
		
		\smallskip
		\item $[0,1) = \bigcup_{J^n \in W^n} J^n$, and the $n$-cylinders $J^n$ in $W^n$ are pairwise disjoint.
		
		\smallskip
		\item If $m \in \{1,\,\dots,\,n\}$ and $x,\, y \in I^n$, then $T_{\beta}^m(y) - T_{\beta}^m(x) = \beta^m(y-x)$. Consequently, $T_{\beta}^m|_{I^n}$ is continuous and strictly increasing.
		
		\smallskip
		\item If $\myepsilon_n > 0$, then $I(\myepsilon_1, \dots, \myepsilon_{n-1},  b)\in W^n_0$ for $b \in\{0,\,\dots,\,\myepsilon_n-1\} $ with the right endpoint $(b+1)\beta^{-n}+\sum_{i=1}^{n-1}\myepsilon_i \beta^{-i}$.
		
		\smallskip
		\item If $I^n\in W^n_0$ and the right endpoint of $I^n$ is not $1$, then there exists a $T_\beta^n$-fixed point in $I^n$.
		
		\smallskip
		\item There exists $m \in \{0,\,1,\,\dots,\,n\}$ such that $T_{\beta}^n (I^n) = [0, U_\beta^m(1))$.
		
	\end{enumerate}
\end{prop}

\begin{lemma}\label{l_Distortion}
	Fix $\beta>1$. Suppose $\alpha\in(0,1]$ and $\phi\in \Holder{\alpha}(I)$. For all $n\in\N$, $I^n\in W^n$, and $x,\,y\in I^n$, we have 
	$
		\abs{S_{n}\phi(x)-S_{n}\phi(y)}
		\leq   \frac{\Hseminorm{\alpha}{\phi}}{\beta^\alpha-1} \Absbig{T_{\beta}^n(x)-T_{\beta}^n(y)}^{\alpha}$.
\end{lemma}

\subsection{Maximizing measures}\label{sct_ext_maximizing_measures}

        Here we introduce the notion of \emph{limit-maxi\-mizing measure}, which will be useful for a dynamical system, such as $T_\beta$, whose set of invariant measures is not necessarily weak$^*$ compact. For $\beta>1$ and $\phi\in C(I)$, we first show that the existence of a maximizing measure for $(I,U_\beta,\phi)$ is equivalent to the existence of a maximizing measure for $(X_\beta,\sigma, \phi\circ h_\beta) .$ We then prove that a measure is limit-maximizing for $(I,T_\beta,\phi)$ if and only if it is maximizing for $(I,U_\beta,\phi)$.

\begin{definition}\label{d_limit_max_measure}
	Let $T \: X\to X$ be a Borel measurable map on a compact metric space $X$. For a Borel measurable function $\psi \: X\to \R$, a probability measure $\mu$ is called a \emph{$(T,\psi)$-limit-maximizing measure}, or simply a \emph{$\psi$-limit-maximizing measure}, if it is a weak$^*$ accumulation point of $\MMM(X,T)$ and $\int \!\psi \,\mathrm{d}\mu=\mpe(T,\psi)$. We denote the set of $(T,\psi)$-limit-maximizing measures by $\MMM_{\max}^*(T,\psi)$. 
\end{definition}

 Clearly, $\MMM_{\max}(T,\psi)\subseteq \MMM_{\max}^*(T,\psi)$. 
 For $\beta>1$, let us write 
\begin{equation}\label{e_def_Z_beta}
	Z_\beta \= \bigl\{ x\in I:\pi_\beta(x)\neq \pi_\beta^*(x) \bigr\}.
\end{equation}

The following lemma collects together some basic properties of $Z_\beta$.

\begin{lemma}\label{l_properties_z_beta}
	If $\beta>1$, then the following hold:
		\begin{enumerate}[label=\rm{(\roman*)}]
		\smallskip
		\item $Z_\beta=\bigl(\bigcup_{n\in\N}T_\beta^{-n}(0)\bigr)\smallsetminus\{0\}=\bigcup_{n\in\N}U_\beta^{-n}(1)$ and in particular $D_\beta\subseteq Z_\beta$.
		\smallskip
		
		\item $h_\beta^{-1}(W)=\pi_\beta^*(W)\cup\pi_\beta(W\cap Z_\beta)$ for each $W\subseteq I$ and $\pi_\beta(Z_\beta)\cap \pi_\beta^*(I)=\emptyset$.
		\smallskip
		
		\item If $x\in I\smallsetminus Z_\beta$, then $T_\beta^n(x)=U_\beta^n(x)$ for all $n\in \N$.
		\smallskip
		
		\item $\pi_\beta$ and $\pi_\beta^*$ are continuous on $I\smallsetminus Z_\beta$.
		\smallskip
		
		\item $\mu(\pi_\beta(Z_\beta))=0$
		for all $\mu\in \MMM(X_\beta,\sigma)$.
	\end{enumerate}
\end{lemma}

Now we consider the relation between $\cM ( I, U_\beta )$ and $\cM (X_\beta, \sigma)$.

\begin{notation}
Fix $\beta>1$. Define 
\begin{equation*}
	G_\beta \: \cM ( I, U_\beta ) \to \cM(X_\beta,\sigma)
\end{equation*}
to be the pushforward of $\pi_\beta^*$,
in other words,
\begin{equation}\label{e_def_G}
	G_\beta (\mu) (Y) \= \mu \bigl(\bigl(\pi_\beta^*\bigr)^{-1} (Y)\bigr)
\end{equation}
for each $\mu \in \cM ( I, U_\beta )$ and each Borel measurable subset $Y \subseteq X_\beta$. By Proposition~\ref{p_relation_of_coding}~\ref{p_relation_of_coding__iii}, it is straightforward to check that $G_\beta$ is well-defined. Define 
\begin{equation*}
	H_\beta \:\cM(X_\beta,\sigma)  \to \cP (I)
\end{equation*}
by
\begin{equation}\label{e_def_H}
	H_\beta (\nu) (W) \= \nu \bigl({h_\beta}^{-1} (W)\bigr)
\end{equation}
for each $\nu \in \cM(X_\beta,\sigma)$ and each Borel measurable subset $W \subseteq I$.
\end{notation}

\begin{prop}\label{p_coding_mpe_relation}
	If $\beta>1$ and $\phi \in C(I)$, then the following hold:
	\begin{enumerate}[label=\rm{(\roman*)}]
		\smallskip
		\item $H_\beta$ is a homeomorphism from $\MMM(X_\beta,\sigma)$ to $\MMM(I,U_\beta)$ with respect to the weak$^*$ topology, with $G_\beta^{-1}=H_\beta$.
		
		\smallskip
		\item $\MMM(I,U_\beta)$ is compact in the weak$^*$ topology.
		
		\smallskip
		\item $\mpe(U_\beta,\phi)=\mpe(\sigma|_{X_\beta},\phi\circ h_\beta)$ and $\MMM_{\max}( U_\beta, \phi)=H_\beta(\MMM_{\max}(\sigma|_{X_\beta},\phi\circ h_\beta))\neq \emptyset$.

        \smallskip
        \item If $\cO$ is an $(X_\beta,\sigma)$-periodic orbit, then $h_\beta(\cO)$ is an $(I,U_\beta)$-periodic orbit, with $\card \cO = \card h_\beta(\cO)$.
	\end{enumerate}
\end{prop}

\begin{prop}\label{mpe=}
	If $\beta>1$ and $\phi\in C(I)$, then the following hold:
	\begin{enumerate}[label=\rm{(\roman*)}]
		\smallskip
		\item $\MMM(I,U_\beta)$ is equal to the weak$^*$ closure of $\MMM(I, T_\beta)$.
		
		\smallskip
		\item $\mpe(T_\beta, \phi) = \mpe (U_\beta, \phi)$.
		
		\smallskip
		\item $\MMM^*_{\max}(T_\beta,\phi)=\MMM_{\max}(U_\beta,\phi)$.
	\end{enumerate}
\end{prop}

	\section{A Ma\~n\'e cohomology lemma for beta-transformations}\label{Sec_mane_lemma}

The purpose of this section is to prove a version of the \emph{Ma\~n\'e cohomology
lemma} (Theorem~\ref{mane}) in the context of beta-transformations,
and derive a new \emph{revelation theorem} (Theorem~\ref{l_subordination}), an important consequence regarding the support of a maximizing measure.
To establish this result, there are several differences and difficulties compared to the Ma\~n\'e lemma for open expanding maps (Theorem~\ref{l.mane}), notably the lack of continuity and openness of beta-transformations, so the method of proof here will be rather different (see Remark~\ref{Manetechnical} for further details).
A key tool is the introduction of an operator analogous to the one used by Bousch \cite{Bou00}, and showing (Proposition~\ref{p_calibrated_sub-action_exists}) that it has a fixed point (function) with certain regularity properties (following \cite{GLT09}, this fixed point can be referred to as a \emph{calibrated sub-action}).

For a Borel measurable map $T \: I \to I$, and a bounded Borel measurable function $\psi \: I\to\R$, to study the $(T,\psi)$-maximizing measures it is convenient, whenever possible, to consider a cohomologous function $\tpsi$ satisfying $\tpsi\leq \mpe(T,\psi)$.
We recall the following (cf.~\cite[p.~2601]{Je19}):

\begin{definition}
	Suppose $T\: I\to I$ is Borel measurable, and $\psi\: I\to\R$ is bounded and Borel measurable.
If $\psi\le Q(T,\psi)$ and $\psi^{-1}(Q(T,\psi))$ contains $\supp \mu$ for some $\mu\in\MMM(I,T)$, then $\psi$ is said to be \emph{revealed}.
If $Q(T,\psi)=0$ then $\psi$ is said to be \emph{normalised}; in particular, 
    a normalised function $\psi$ is  revealed if and only if $\psi\le 0$ and $\psi^{-1}(0)$ contains $\supp \mu$ for some $\mu\in\MMM(I,T)$.
\end{definition}

\begin{lemma}\label{normalised_cohomologous}
	Suppose $T\: I\to I$ is Borel measurable, $\phi \: I\to\R$ is bounded and Borel measurable, and 
	$\Mmax(T,  \phi)\neq\emptyset$.
	Denote $\overline{\phi}=\phi - Q(T,\phi)$, and suppose $\tphi=\overline{\phi}+u-u\circ T$
	for some bounded Borel measurable function $u \: I\to\R$.
	Then the following hold:
	\begin{enumerate}[label=\rm{(\roman*)}]
		\smallskip
		\item	$Q\bigl(T,\tphi\bigr)=Q\bigl(T,\overline{\phi}\bigr)=0$.
		\smallskip
		
		\item	$\Mmax(T,  \phi)=\Mmax\bigl(T,  \overline{\phi}\bigr)=\Mmax\bigl(T,  \tphi\bigr)$.
		\smallskip
		
		\item	If $\tphi\le0$ and if $x\in I$ is such that $\mathcal{O}^T(x)\subseteq \tphi^{-1}(0)$,
		then $\mathcal{O}^T(x)$ is a $(T,\phi)$-maximizing orbit.
	\end{enumerate}	
\end{lemma}
\begin{proof}
	(i) and (ii) follow from (\ref{e_ergodicmax}), (\ref{e_setofmaximizngmeasures}), and the fact that
	$
		\int \!\wt\phi\,\mathrm{d}\mu
        =\int \! \bigl(\overline{\phi}+u-u\circ T \bigr)\,\mathrm{d}\mu
        =\int \!\overline{\phi}\,\mathrm{d}\mu$ for all $\mu\in \MMM(I,T)$.
	
	If $\tphi\le0$ and $\mathcal{O}^T(x)\subseteq \tphi^{-1}(0)$, then $0=\frac{1}{n}S_n^T\wt\phi(x)=\frac{1}{n}S_n^T\overline{\phi}(x)+\frac{1}{n}(u(x)-u(T^n(x)))$
	for all $n\in \N$, and (iii) follows from the fact that $u$ is bounded.
\end{proof}

The following operator\footnote{Although nonlinear, the operator $\cL_\psi$ is \emph{tropical linear} (see e.g.~\cite{LiSu26} for further development of this tropical functional analysis viewpoint; see also \cite{BLL13}).}
$\cL_\psi$ is an analogue of the one used by Bousch in \cite{Bou00}.

\begin{definition}
Let $\psi \: [0,1) \to\R$ be bounded and Borel measurable.
For $\beta>1$,
define 
$\cL_\psi\:\R^{[0,1)}\to \R^{[0,1)}$
by
\begin{equation}\label{e_Def_BouschOp}
	\mathcal{L}_\psi(u)(x) \=
		\max \bigl\{(u+\psi)(y) : y\in T_\beta^{-1}(x) \smallsetminus \{1\} \bigr\}, \quad x\in [0,1).
\end{equation}
\end{definition}

Note that $\cL_\psi$ is well defined since $T_\beta  ([0,1)) = [0,1)$ (cf.~(\ref{e_def_T_beta})), and if $u \: [0,1) \to\R$ is bounded then so is $\cL_\psi (u)$. For a function $u\: I \to \R$ and a bounded Borel measurable function $\psi\: I \to \R$, we define $\cL_\psi(u) \= \cL_{\psi|_{[0,1)}} \bigl( u|_{[0,1)} \bigr)$.

\begin{lemma}\label{Property of L}
	If $\beta>1$ and $\psi \: I\to\R$ is bounded and Borel measurable, 
    and
	$\overline{\psi} \= \psi - \mpe (T_\beta, \psi)$,
    then the following hold:
	\begin{enumerate}[label=\rm{(\roman*)}]
		\smallskip
		\item If $x\in [0,1)$, $n\in\N$, and $u\: I \to \R$ is bounded, then	
		\begin{equation*}
		\mathcal{L}^n_{\overline{\psi}}(u)(x)+n\mpe(T_\beta,\psi)
		=\mathcal{L}^n_\psi(u)(x)
		=\max\bigl\{u(y)+S_{n}\psi(y) : y\in T_{\beta}^{-n}(x) \smallsetminus \{1\} \bigr\}.
		\end{equation*}

		\item $\mathcal{L}_\psi(\sup_{v\in\mathcal{A}} v)=\sup_{v\in\mathcal{A}}\mathcal{L}_\psi(v)$ for any collection $\mathcal{A}$ of bounded real-valued functions on $[0,1)$.
		\smallskip
		
		\item 
		If $\{u_n\}_{n\in \mathbb{N}}$ is a pointwise convergent sequence of bounded real-valued functions on $[0,1)$,
then $\lim_{n\to +\infty}\mathcal{L}_\psi(u_n)=\mathcal{L}_\psi(\lim_{n\to +\infty}u_n)$, 
 where $\lim_{n\to +\infty}$ denotes pointwise limit.
	\end{enumerate}		
\end{lemma}

\begin{proof}
	\smallskip
	(i) The first equality in (i) is immediate from (\ref{e_Def_BouschOp}) and 
    the fact that $\overline{\psi} = \psi - \mpe (T_\beta, \psi)$.
    As $T_\beta([0,1)) = [0,1)$, then
    \begin{equation}\label{e_Def_BouschOp*}
     \begin{aligned}
        &\max \bigl\{ u(y) + S_n\psi(y) : y\in T_\beta^{-n} (x) \smallsetminus \{1\} \bigr\} \\
        &\qquad= \max \bigl\{ u(y) + S_n\psi(y) : y\in \bigl(T_{\beta}|_{[0,1)} \bigr)^{-n}(x) \bigr\}
     \end{aligned}
    \end{equation}
    for all $n\in \N$ and $x\in [0,1)$.
	Then the second equality is easily proved by (\ref{e_Def_BouschOp}), (\ref{e_Def_BouschOp*}), and induction (cf.~e.g.~\cite{JMU06, JMU07}).
	
	(ii) follows readily from the fact that $\psi+\sup_{v\in\mathcal{A}} v = \sup_{v\in A}(\psi+v)$,
	and that for each $x\in [0,1)$,
	\begin{equation*}
	\max\limits_{y= T^{-1}_{\beta}(x) \smallsetminus \{1\}} \sup_{v\in\mathcal{A}}  (\psi+v )(y)
	=
	\sup_{v\in\mathcal{A}} \max\limits_{y= T^{-1}_{\beta}(x) \smallsetminus \{1\}}   (\psi+v )(y) .
	\end{equation*}

	(iii) Define $v \: [0,1) \to \R$ by $v(x)\=\lim_{n\to+\infty} u_n(x)$ for all $x\in [0,1)$.
	Fix arbitrary $x\in [0,1)$ and $\myepsilon > 0$. Then
	there exists $N=N(x,\myepsilon) \in \N$ such that if $n\geq N$ 
	then $\abs{ u_n (y) - v (y) } < \myepsilon$
	for each of the finitely many pre-images $y\in T_{\beta}^{-1}(x)$.
	
	Fix $n \geq N$. Let $y_1, \, y_2 \in T_{\beta}^{-1}(x) \smallsetminus \{1\}$ satisfy
	$\RR_\psi (u_n) (x) = (\psi+ u_n )(y_1)$ and $\RR_\psi (v) (x) = (\psi+ v )(y_2)$, so that
	\begin{equation*}
		\begin{aligned}
			 \bigl(\RR_\psi ( u_n )  - \RR_\psi (v)  \bigr)  (x)
			&\leq (\psi+ u_n) (y_1) - (\psi + v) (y_1)
			=     u_n (y_1) - v(y_1) < \myepsilon \text{ and }\\
			 \bigl(\RR_\psi ( u_n ) - \RR_\psi (v)  \bigr) (x)
			&\geq (\psi  + u_n) (y_2) - (\psi  + v) (y_2)  
			=     u_n (y_2) - v(y_2) >  - \myepsilon   .
		\end{aligned}
	\end{equation*}
	Then (iii) follows. 
\end{proof}

\begin{notation}
	For $\beta>1$ and $\alpha\in(0,1]$, we write
	\begin{equation}\label{kalphabetadefn}
	K_{\alpha,\beta}\= \frac{1}{\beta^\alpha -1} .
	\end{equation}
\end{notation}

\begin{lemma} \label{l_Bousch_Op_preserve_space}
	Suppose $\beta>1$, $\alpha\in (0,1]$, $\phi\in \Holder{\alpha}(I)$, and $n\in\N$. Then
	
	\begin{equation}\label{e_Bousch_Op_Holder_seminorm_bound_x<y}
		\cL_\phi^n(u)(x)-\cL_\phi^n(u)(y) 
		\geq -K_{\alpha,\beta} ( \Hseminorm{\alpha}{ \phi}+\Hseminorm{\alpha}{u} )\abs{x-y}^{\alpha}
	\end{equation}
	for all $u\in \Holder{\alpha}(I)$, and all $x,y\in [0,1)$ with $x \leq y$. 
	
	If, moreover, for all $1\le i\le n$ the interval $(x,y]$ does not contain $U_\beta^i(1)$, 
	then
	\begin{equation}\label{e_Bousch_Op_Holder_seminorm_bound}
		\Absbig{ \cL_\phi^n(u)(x)-\cL_\phi^n(u)(y)  }
		\leq K_{\alpha,\beta} ( \Hseminorm{\alpha}{ \phi}+\Hseminorm{\alpha}{u} )\abs{x-y}^{\alpha} .
	\end{equation}
\end{lemma}

\begin{proof}
	Suppose $u\in \Holder{\alpha}(I)$ and $x,\,y\in [0,1)$ with $x\leq y$. Without loss of generality, we assume $x<y$. 
	By Lemma~\ref{Property of L} (i), there exists $y'=T^{-n}_{\beta}(y) \smallsetminus \{1\}$
	such that
	\begin{equation}\label{starequality}  
		\mathcal{L}_{\phi}^n(u)(y)=u(y')+S_{n}\phi(y').
	\end{equation}
	By Proposition~\ref{p_cylinders}~(i), there exists $I^n\in W^n$ containing $y'$. Since $x<y$, Proposition~\ref{p_cylinders}~(v) guarantees that $x\in T^{n}_{\beta}(I^n)$ as well.
	Consider $x'\in T^{-n}_{\beta}(x) \cap I^n$, then $x'\neq 1$, so
	we have
	\begin{equation}\label{2starinequality}  
		\mathcal{L}_{\phi}^n(u)(x)\ge u(x')+S_{n}\phi(x').
	\end{equation}
	
	Combining (\ref{starequality}) and (\ref{2starinequality}) gives
	\begin{equation}\label{3starcombined}
		\mathcal{L}_{\phi}^n(u)(x)-\mathcal{L}_{\phi}^n(u)(y)\geq S_{n}\phi(x')+u(x')-S_{n}\phi(y')-u(y') .
	\end{equation}
	Since $x',y'\in I^n$,  Lemma~\ref{l_Distortion}
	gives
	\begin{equation}\label{4star}
		S_{n}\phi(x^\prime)-S_{n}\phi(y^\prime)\geq -K_{\alpha,\beta}  \Hseminorm{\alpha}{\phi} \abs{x-y}^{\alpha}.
	\end{equation}
	Now $u\in \Holder{\alpha}(I)$, so $u(x')-u(y') \ge - \Hseminorm{\alpha}{u} \abs{x'-y'}^\alpha$
	and $\abs{x'-y'}=\beta^{-n}\abs{x-y}$ by Proposition~\ref{p_cylinders}~(ii),
	so $u(x')-u(y') \ge - \Hseminorm{\alpha}{u} \beta^{-n\alpha}\abs{x-y}^\alpha$. 
	But $\beta^{-n\alpha}< 	K_{\alpha,\beta}$, so
	\begin{equation}\label{5star}
		u(x')-u(y') \ge -K_{\alpha,\beta} \Hseminorm{\alpha}{u} \abs{x-y}^\alpha .
	\end{equation}
	Combining (\ref{3starcombined}), (\ref{4star}), and (\ref{5star}) gives the required inequality
	(\ref{e_Bousch_Op_Holder_seminorm_bound_x<y}).
	
	A similar argument can be used to establish the bound (\ref{e_Bousch_Op_Holder_seminorm_bound}).
	Specifically, suppose that $x,\,y\in [0,1)$, $n\in\N$ are such that $(x,y] \cap\bigl\{U_\beta(1),\,\dots,\,U_\beta^n(1)\bigr\}=\emptyset$, so that, by Proposition~\ref{p_cylinders}~(v), if $I^n\in W^n$ then
	$x\in T^{n}_{\beta}(I^n)$ if and only if $y\in T^{n}_{\beta}(I^n)$.
	
	By Proposition~\ref{p_cylinders}~(i) and Lemma~\ref{Property of L}~(i) there exists some $I^n\in W^n$ and $x''\in I^n$ with $x''\in T^{-n}_{\beta}(x) \smallsetminus \{1\}$, 
	such that  
	$\mathcal{L}_{\phi}^n(u)(x)=S_{n}\phi(x'')+u(x'')$.
	
	Let $y'' \in T^{-n}_{\beta}(y) \cap I^n$. An argument analogous to the one above, using Lemma~\ref{l_Distortion} and Lemma~\ref{Property of L}~(i), then gives
	\begin{equation*}
		\begin{aligned}
		\mathcal{L}_{\phi}^n(u)(x)-\mathcal{L}_{\phi}^n(u)(y)
		&\leq S_{n}\phi(x'')+u(x'')-S_{n}\phi(y'')-u(y'')\\
		&\leq(\beta^\alpha-1)^{-1}\Hseminorm{\alpha}{ \phi}\abs{x-y}^\alpha+\Hseminorm{\alpha}{u} \abs{x''-y''}^\alpha\\
		&\leq(\beta^\alpha-1)^{-1}\Hseminorm{\alpha}{\phi}\abs{x-y}^\alpha+\Hseminorm{\alpha}{u}\beta^{-n\alpha}\abs{x-y}^\alpha\\
		&\leq K_{\alpha,\beta} (\Hseminorm{\alpha}{\phi}+\Hseminorm{\alpha}{u})\abs{x-y}^\alpha,
	\end{aligned}
	\end{equation*}
	and (\ref{e_Bousch_Op_Holder_seminorm_bound}) follows.
\end{proof}

Of particular interest 
will be the choice $u=\mathbbold{0}$, the function that is identically zero on $I$,
and evaluation of (\ref{e_Bousch_Op_Holder_seminorm_bound_x<y}) at the left endpoint of $I$, which we record as follows:

\begin{cor}\label{u=0corollary}
	Suppose $\beta>1$ and $\alpha\in (0,1]$. If $\phi\in \Holder{\alpha}(I)$, $n\in\N$, and $y\in [0,1)$, then 
	\begin{equation*}
		\cL_\phi^n(\mathbbold{0})(0) \geq \cL_\phi^n(\mathbbold{0})(y) 
		-K_{\alpha,\beta}  \Hseminorm{\alpha}{ \phi} .\label{e_Bousch_Op_Holder_seminorm_bound_x<y+corollary}
	\end{equation*}
\end{cor}

We are now able to find a fixed point  
$u_\phi$ of the 
operator $\cL_{\overline\phi}$:

\begin{prop}\label{p_calibrated_sub-action_exists}
	Suppose $\beta>1$ and $\alpha\in (0,1]$. If $\phi\in \Holder{\alpha}(I)$ 
	then the function $u_\phi \: [0,1) \to\R$ given by
	\begin{equation}  \label{e_calibrated_sub-action_exists}
		u_\phi(x)\=\limsup\limits_{n\to +\infty}\mathcal{L}_{\overline{\phi}}^n(\mathbbold{0})(x),\quad x\in [0,1),
	\end{equation}
    where $\overline{\phi} \= \phi - \mpe (T_\beta, \phi)$,
	satisfies the following properties:
	\begin{enumerate}[label=\rm{(\roman*)}]
		\smallskip
		\item $u_\phi$ is Borel measurable and $\abs{u_{\phi}(x)}\leq (2+3 K_{\alpha,\beta}) \Hseminorm{\alpha}{\phi}$ for each $x\in [0,1)$.
		\smallskip
		
		\item If $a\in(0,1]$ then
		$\lim_{x\nearrow a}u_\phi(x)$ exists,
		and provided $a\neq 1$, satisfies $\lim_{x\nearrow a}u_\phi(x)\geq u_\phi(a)$.
		If $a\in[0,1)$ then
		$\lim_{x\searrow a}u_\phi(x)$ exists,
		and satisfies $u_\phi(a)\geq \lim_{x\searrow a}u_\phi(x)$.
		In particular, if $a\in(0,1)$ then
		\begin{equation}\label{liminequality}
			\lim_{x\nearrow a}u_\phi(x)\geq u_\phi(a)\geq \lim_{x\searrow a}u_\phi(x).
		\end{equation}
		
		\item $\abs{u_{\phi}(x)-u_{\phi}(y)}  \le K_{\alpha,\beta} \Hseminorm{\alpha}{\phi}  \abs{x-y}^\alpha$ if $0\leq x<y< 1$ satisfy $(x,y]\cap \cO_\beta^*(1)=\emptyset$.
		\smallskip
		
		\item $\mathcal{L}_{\overline{\phi}}(u_\phi)=u_\phi$.
	\end{enumerate}	
\end{prop}

\begin{proof}
	For each $n\in\N$ and each $x\in [0,1)$, we write 
	\begin{equation}\label{pndefn}
		p_n(x)\=\mathcal{L}_{\overline{\phi}}^n(\mathbbold{0})(x)
        \quad \text{ and } \quad
        q_n(x)\=\sup_{m\ge n} p_m(x) .
	\end{equation}
	Note that, 
	for each $x\in [0,1)$, the sequence $\{q_n(x)\}_{n\in\N}$ is nonincreasing	and 
	\begin{equation*}
		u_\phi(x)=\lim\limits_{n\to +\infty}q_n(x)=\limsup\limits_{n\to +\infty}p_n(x).
	\end{equation*}
	
	(i) Fix $n\in \N$. By (\ref{e_Bousch_Op_Holder_seminorm_bound}), $p_n$ is continuous at all points except for $U_\beta(1)$, $\dots$, $U_\beta^{n}(1)$, and hence Borel measurable. Combining this with (\ref{e_calibrated_sub-action_exists}), $u_\phi$ is Borel measurable. 
	By Lemma~\ref{Property of L}~(i), there exists 
	$y_n \in T^{-n}_\beta(0) \smallsetminus \{1\}$ such that 
	\begin{equation}\label{pnzeroyn}
		p_n(0)= \mathcal{L}_{\overline{\phi}}^n(\mathbbold{0})(0)
		=S_{n}\overline{\phi}(y_n).
	\end{equation}
	By Proposition~\ref{p_cylinders}~(i), there exists $I^n = I(a_1,   a_2,  \dots,  a_n)$ containing $y_n$. By Proposition~\ref{p_cylinders}~(ii) and~(v), $y_n$ is the left-endpoint of $I^n$, then (\ref{left_endpoint}) gives
	\begin{equation}\label{yndefn}
		y_n=\frac{a_1}{\beta}+\dots+\frac{a_n}{\beta^n}.
	\end{equation}
	Define 
	\begin{equation}\label{kdefn}
		k\=\min\{i\in \N_0: a_j=0\text{ for all }i+1\le j\le n\} .
	\end{equation} 
	
	\emph{Case 1.} If $k=0$, we get that $y_n=0$ and since $0$ is a fixed point of $T_\beta$, 
	\begin{equation}\label{e_caes_k=0_1}
		p_n(0)=n\overline{\phi}(0)\leq 0 .
	\end{equation}	
	
	\emph{Case 2.} If $k>0$, $y_n (\neq 1)$ is the right endpoint of a $k$-full cylinder $I^k_0 \= I(a_1,\dots,a_k-1)$ by
	Proposition~\ref{p_cylinders}~(iii) and hence by Proposition~\ref{p_cylinders}~(iv) 
	there is a $T_{\beta}^k$-fixed point $z_n$ in $I^k_0$. Thus,
	\begin{equation}\label{znaveragenonpositive}
		S_k \overline{\phi}(z_n)\le k\mpe \bigl(T_\beta,\overline{\phi}\bigr)= 0 .
	\end{equation}
	Since $0$ is a fixed point of $T_\beta$,
	$\overline{\phi}(0)\le Q\bigl(T_\beta,\overline{\phi}\bigr)=0$, and hence combining this with the fact that $T^k_\beta(y_n)=0$ (see (\ref{yndefn}) and~(\ref{kdefn})), we obtain
	\begin{equation}\label{snsk}
		S_n\overline{\phi}(y_n)
		=S_k\overline{\phi}(y_n)+S_{n-k}\overline{\phi}\bigl(T^k_\beta(y_n)\bigr)
		=S_k\overline{\phi}(y_n)+(n-k)\overline{\phi}(0)\leq S_k\overline{\phi}(y_n) .
	\end{equation}
	If $k\ge 2$, define the $(k-1)$-cylinder $I^{k-1} \= I(a_1,  a_2,  \dots,  a_{k-1})$. As $I^n \subseteq I^{k-1}$ and $I^k_0 \subseteq I^{k-1}$ by (\ref{e_def_n_cylinders}), we have $y_n, \, z_n \in I^{k-1}$, then Lemma~\ref{l_Distortion} gives
	\begin{equation}\label{e_sumbound}
		\begin{aligned}
			&S_{k}\overline{\phi}(y_n)-S_{k}\overline{\phi}(z_n) \\
            &\qquad= S_{k-1}\overline{\phi}(y_n)-S_{k-1}\overline{\phi}(z_n) + \overline{\phi} \bigl(T_\beta^{k-1}(y_n)\bigr) - \overline{\phi}\bigl( T_\beta^{k-1}(z_n)\bigr)  \\
			&\qquad\le (1+K_{\alpha,\beta}) \Hseminorm{\alpha}{\phi} \Absbig{T_\beta^{k-1}(y_n) - T_\beta^{k-1}(z_n)}^\alpha \le (1+K_{\alpha,\beta}) \Hseminorm{\alpha}{\phi}.
		\end{aligned}
	\end{equation}
	Moreover, (\ref{e_sumbound}) also holds if $k=1$.
	
	Combining (\ref{pnzeroyn}), (\ref{snsk}), (\ref{e_sumbound}), and (\ref{znaveragenonpositive}) gives
	\begin{equation}\label{pn0upperbound}
		p_n(0)=S_n\overline{\phi}(y_n) \le S_k\overline{\phi}(y_n)
		= S_{k}\overline{\phi}(y_n)-S_{k}\overline{\phi}(z_n)+S_{k}\overline{\phi}(z_n)
		\leq (1+K_{\alpha,\beta})\Hseminorm{\alpha}{\phi}.
	\end{equation}
	Combining Corollary~\ref{u=0corollary}, (\ref{e_caes_k=0_1}), and (\ref{pn0upperbound}) gives
	\begin{equation} \label{pnupperbound}
	p_n(x) 	\le p_n(0) +  K_{\alpha,\beta} \Hseminorm{\alpha}{\phi} \le (1+2 K_{\alpha,\beta})\Hseminorm{\alpha}{\phi}
	\quad\text{ for all }x\in [0,1), n\in\N,
	\end{equation}
	so from (\ref{e_calibrated_sub-action_exists}) we deduce
	the upper bound
	\begin{equation}\label{uphiupperbound}
		u_\phi (x) \leq (1+2  K_{\alpha,\beta})\Hseminorm{\alpha}{\phi} \quad\text{ for all }x\in [0,1).
	\end{equation}

	We now seek to derive a lower bound on $u_\phi$, via a
	lower bound on $p_n(x)$. Fix $n\in \N$. First we would like to show there exists $w_n \in [0,1)$ satisfying
    \begin{equation}\label{xnnonnegative}
		S_{n}\overline{\phi}(w_n)\ge -\Hseminorm{\alpha}{\phi}.
	\end{equation}
    If $\Hseminorm{\alpha}{\phi} = 0$, in other words $\phi$ is 
    constant, (\ref{xnnonnegative}) holds for each $w_n\in [0,1)$. Now assume $\Hseminorm{\alpha}{\phi}>0$.
	Note that $Q\bigl(T_\beta,\overline{\phi}\bigr)=0$, so $Q\bigl(T_\beta, S_n\overline{\phi}\bigr)=0$. There exists $\mu \in \cM( I,T_\beta )$ with $\int\! \overline{\phi} \,\mathrm{d}\mu > -\Hseminorm{\alpha}{\phi}$. By Proposition~\ref{p_relation_T_beta_and_wt_T_beta}~(iii), $\mu(\{1\})=0$,
	so $\sup \bigl\{ S_n \overline{\phi}(x) : x\in [0,1) \bigr\} > -\Hseminorm{\alpha}{\phi}$. Consequently, there exists $w_n\in [0,1)$ satisfying (\ref{xnnonnegative}).
    
    Now choose $w_n\in [0,1)$ to satisfy (\ref{xnnonnegative}).
	If we define $y\=T_{\beta}^n(w_n)$, Lemma~\ref{Property of L}~(i) and (\ref{xnnonnegative}) give
	\begin{equation}\label{lphibarnineq}
		\mathcal{L}^n_{\overline{\phi}}(\mathbbold{0})(y)
		=\max\bigl\{S_{n}\overline{\phi}(w) : w= T_{\beta}^{-n}(y) \smallsetminus \{1\}\bigr\}
		\ge
		S_{n}\overline{\phi}(w_n)\ge -\Hseminorm{\alpha}{\phi} .
	\end{equation}
	Combining (\ref{pndefn}), Corollary~\ref{u=0corollary},
	and (\ref{lphibarnineq}) gives
	\begin{equation} \label{lphibarzerozero}
		p_n(0) = \cL_{\overline{\phi}}^n(\mathbbold{0})(0) \geq \cL_{\overline{\phi}}^n(\mathbbold{0})(y) 
		-K_{\alpha,\beta}  \Hseminorm{\alpha}{ \phi} \ge -(1 +K_{\alpha,\beta} ) \Hseminorm{\alpha}{ \phi} .
	\end{equation}

	Let $y_n$ and $k$ be as in (\ref{yndefn}) and (\ref{kdefn}). Fix $x\in [0,1)$. When $k=0$, we get $y_n=0$ and $p_n(0)=n\overline{\phi}(0)$. Notice that $0, \, x/\beta^n\in I(0,\dots,0) \in W^n$ and $x/\beta^n \in T_\beta^{-n}(x) \smallsetminus \{1\}$, using Lemma~\ref{l_Distortion} and (\ref{lphibarzerozero}) gives
	\begin{equation}  \label{pnk0lowerbound}
		p_n(x)
		\geq S_n\overline{\phi} (x/\beta^n )
		= S_n\overline{\phi} (x/\beta^n ) - S_n\overline{\phi}(0) + p_n(0)
		\geq-(1+2 K_{\alpha,\beta})\Hseminorm{\alpha}{\phi}.
	\end{equation}

	When $k>0$,	by Proposition~\ref{p_cylinders}~(iii) and (\ref{yndefn}),
	we have the full $n$-cylinder  
	\begin{equation*}
		J^n  \= I(0,\dots,0,a_1,a_2,\dots, a_k-1)\in W^n_0 ,
	\end{equation*}
	with the right endpoint equal to $y_n \big/ \beta^{n-k}$.
	Since $J^n$ is full, then there exists $z_n \in T_\beta^{-n}(x) \cap J^n$. Noting that $z_n \neq 1$, then by (\ref{pndefn}) and Lemma~\ref{Property of L}~(i),
	\begin{equation}\label{pnonegreaterthan}
		p_n(x)=  \cL_{\overline{\phi}}^n(\mathbbold{0})(x)
		=\max\bigl\{S_{n}\overline{\phi}(z) : z\in T_{\beta}^{-n}(x) \smallsetminus \{1\}\bigr\}
		\ge 
		S_{n}\overline{\phi}(z_n ) .
	\end{equation}
	If $n\ge 2$, denote $J^{n-1} \= I(0,  \dots,  0,  a_1,  a_2,  \dots,  a_{k-1}) \in W^{n-1}$, in particular, $J^{n-1} = I(0,  \dots,  0) \in W^{n-1}$ if $k=1$. By (\ref{yndefn}), we have $y_n \big/ \beta^{n-k} \in J^{n-1}$, and since $J^n\subseteq J^{n-1}$, we have $z_n\in J^{n-1}$. Applying Lemma~\ref{l_Distortion}, we have
	\begin{equation}\label{e_sumbound2}
		\begin{aligned}
			&S_n \overline{\phi} (z_n) - S_n \overline{\phi} \bigl( y_n\big/\beta^{n-k} \bigr) \\
			&\qquad=  S_{n-1} \overline{\phi} (z_n) - S_{n-1} \overline{\phi} \bigl( y_n\big/\beta^{n-k} \bigr)  
			+ \overline{\phi} \bigl( T_\beta^{n-1}(z_n) \bigr) - \overline{\phi} \bigl( T_\beta^{n-1} \bigl( y_n\big/\beta^{n-k} \bigr)\bigr)\\
			&\qquad\ge -(1+K_{\alpha,\beta}) \Hseminorm{\alpha}{\phi} \Absbig{ T_\beta^{n-1}(z_n)- T_\beta^{n-1} \bigl( y_n\big/\beta^{n-k} \bigr) }^\alpha \ge -(1+K_{\alpha,\beta}) \Hseminorm{\alpha}{\phi}.
		\end{aligned}
	\end{equation} 
	Moreover, (\ref{e_sumbound2}) also holds if $n=1$.
	
	Now by (\ref{e_def_T_beta}),
	\begin{equation}\label{sumsrewrite}
		S_{n}\overline{\phi}\Bigl( \frac{y_n}{\beta^{n-k}}\Bigr)
        =S_k\overline{\phi}(y_n)+S_{n-k}\overline{\phi}\Bigl(\frac{y_n} {\beta^{n-k}}\Bigr)
		=S_k\overline{\phi}(y_n)+\sum_{i=1}^{n-k}\overline{\phi}\Bigl(\frac{y_n}{\beta^{i}}\Bigr) .
	\end{equation}
	Note that $\overline{\phi}\bigl(y_n\big/\beta^i\bigr)-\overline{\phi}(0) \ge -\Hseminorm{\alpha}{\phi} \bigl(y_n\big/\beta^{i}\bigr)^\alpha$
	for $1\le i\le n-k$, so
	\begin{equation}\label{nminusk}
		\sum_{i=1}^{n-k}\overline{\phi}\bigl(y_n\big/\beta^{i}\bigr) 
        \ge  (n-k)\overline{\phi}(0)-\Hseminorm{\alpha}{\phi}\sum_{i=1}^{n-k}\bigl(y_n\big/\beta^{i}\bigr)^\alpha
        \ge  (n-k)\overline{\phi}(0)- K_{\alpha,\beta}\Hseminorm{\alpha}{\phi} .
	\end{equation}
	Combining (\ref{pnonegreaterthan}), (\ref{e_sumbound2}), (\ref{sumsrewrite}), and~(\ref{nminusk}) gives
	\begin{equation}\label{combiningthreegives}
		p_n(x)\ge (n-k)\overline{\phi}(0)-(1+2 K_{\alpha,\beta})\Hseminorm{\alpha}{\phi} + S_k\overline{\phi}(y_n) .
	\end{equation}
	However, (\ref{pnzeroyn}) and (\ref{snsk})
	together give
	\begin{equation}\label{pnzerorewrite}
		(n-k)\overline{\phi}(0) + S_k\overline{\phi}(y_n) = p_n(0) .
	\end{equation}
	So combining (\ref{combiningthreegives}), (\ref{pnzerorewrite}), and (\ref{lphibarzerozero}) gives
	\begin{equation}\label{pnonepnzero}
		p_n(x) \ge p_n(0) - (1+2K_{\alpha,\beta})\Hseminorm{\alpha}{\phi} \ge  -(2+3 K_{\alpha,\beta})\Hseminorm{\alpha}{\phi} .
	\end{equation}
	Therefore, (\ref{pnk0lowerbound}), (\ref{pnonepnzero}), (\ref{e_calibrated_sub-action_exists}), and (\ref{pnzeroyn}) give
	\begin{equation}\label{uphilowerbound}
		u_\phi(x) \ge -(2+3   K_{\alpha,\beta})\Hseminorm{\alpha}{\phi} .
	\end{equation}
	The bounds (\ref{uphiupperbound}) and (\ref{uphilowerbound}) together give the 
	required inequality 
	$\abs{u_{\phi}(x)}\leq (2+3 K_{\alpha,\beta}) \Hseminorm{\alpha}{\phi}$,
	so (i) is proved.

	\smallskip
	(\romannumeral2) If $0\le x<y<1$ then taking $u=\mathbbold{0}$
	in Lemma~\ref{l_Bousch_Op_preserve_space} and using Lemma~\ref{Property of L}~(i) give
	$
		\cL_{\overline{\phi}}^n(\mathbbold{0})(x) \geq \cL_{\overline{\phi}}^n(\mathbbold{0})(y) 
		-K_{\alpha,\beta}  \Hseminorm{\alpha}{ \phi} \abs{x-y}^{\alpha}$,
	and taking the limit supremum, together with (\ref{e_calibrated_sub-action_exists}), gives 
	\begin{equation}\label{e_Holder_u_phi_half}
		u_\phi(x)\geq u_\phi(y)-K_{\alpha,\beta}\Hseminorm{\alpha}{\phi}\abs{x-y}^\alpha.
	\end{equation}
	If $a\in[0,1)$ then in particular (\ref{e_Holder_u_phi_half}) holds for all $a<x<y<1$, so 
	taking $\liminf_{x\searrow a}$ gives
	\begin{equation}\label{e_Holder_u_phi_half_liminf}
		\liminf_{x\searrow a}	u_\phi(x)\geq u_\phi(y)-K_{\alpha,\beta}\Hseminorm{\alpha}{\phi}\abs{a-y}^\alpha,
	\end{equation}
	and taking  $\limsup_{y\searrow a}$ in (\ref{e_Holder_u_phi_half_liminf}) gives
	\begin{equation}\label{e_Holder_u_phi_half_liminf_limsup}
		\liminf_{x\searrow a}	u_\phi(x)\geq \limsup_{y\searrow a} u_\phi(y),
	\end{equation}
	so $\lim_{x\searrow a} u_\phi(x)$ exists, as required.
	Now setting $x=a$ in (\ref{e_Holder_u_phi_half}), and taking $\lim_{y\searrow a} u_\phi(y)$, gives that
	\begin{equation}\label{liminequality1}
		u_\phi(a)\ge \lim_{y\searrow a} u_\phi(y),
	\end{equation}
	as required.
	
	If $a\in(0,1]$,
	an analogous argument shows that 
	$\lim_{x\nearrow a}u_\phi(x)$ exists,
	and if moreover $a\neq 1$ then 
	\begin{equation}\label{liminequality2}
		\lim_{x\nearrow a}u_\phi(x)\geq u_\phi(a).
	\end{equation}
	If $a\in(0,1)$, the required inequality (\ref{liminequality}) is immediate from (\ref{liminequality1}) and (\ref{liminequality2}).

	\smallskip    
	(\romannumeral3) Now suppose $0\le x<y<1$ with $(x,y]\cap\cO_\beta^*(1)=\emptyset$.
	For each $\myepsilon>0$, by (\ref{pndefn}) and (\ref{e_calibrated_sub-action_exists}) there exists $N\in \mathbb{N}$ such that 
	$\abs{p_N(x)-u_\phi(x)}<\myepsilon$ and
	$q_N(y)-u_{\phi}(y)<\myepsilon$, so
	\begin{equation}\label{uphixydifference1}	
			u_\phi(x)-u_\phi(y) < p_N(x)-q_N(y)+2\myepsilon
			\leq p_N(x)-p_N(y)+2\myepsilon
			\leq K_{\alpha,\beta}\Hseminorm{\alpha}{\phi}\abs{x-y}^\alpha+2\myepsilon,
	\end{equation}
	where the final inequality uses (\ref{e_Bousch_Op_Holder_seminorm_bound}).
	Similarly, there exists $M\in \mathbb{N}$ such that
	$\abs{p_M(y)-u_\phi(y)}<\myepsilon$ and
	$q_M(x)-u_{\phi}(x)<\myepsilon$, and an analogous calculation gives
	\begin{equation}\label{uphixydifference2}
		u_\phi(x)-u_\phi(y) \ge  -K_{\alpha,\beta}\Hseminorm{\alpha}{\phi} \abs{x-y}^\alpha-2\myepsilon.
	\end{equation}
	Since $\myepsilon>0$ was arbitrary,  (\romannumeral3) follows from (\ref{uphixydifference1}) and  (\ref{uphixydifference2}).
	
	\smallskip
	(\romannumeral4) If $x\in [0,1)$, then by Lemma~\ref{Property of L}~(ii), (iii), (\ref{pndefn}), and the boundedness of $p_n(x)$ and $q_n(x)$ (cf.~(\ref{pnupperbound}) and (\ref{pnonepnzero})),
\begin{align*}
		\mathcal{L}_{\overline{\phi}}(u_\phi)(x)
		&=\mathcal{L}_{\overline{\phi}}\bigl(\lim\limits_{n\to +\infty}q_n\bigr)(x)
		=\lim\limits_{n\to+\infty}\mathcal{L}_{\overline{\phi}} \bigl( \sup_{m\ge n} \mathcal{L}_{\overline{\phi}}^m(\mathbbold{0}) \bigr) (x)\\
		&=\lim\limits_{n\to+\infty}  \sup_{m\ge n} \mathcal{L}_{\overline{\phi}}^{m+1}(\mathbbold{0})(x) 
		=\lim\limits_{n\to+\infty}q_{n+1}(x)
		=u_{\phi}(x).\qedhere
\end{align*}
\end{proof}

 The following construction of the regularisations of the function $u_\phi$ is key to our revelation theorem (Theorem~\ref{l_subordination}).  

\begin{definition}
Fix $\beta>1$ and $\alpha\in (0,1]$. For each $\phi\in\Holder{\alpha}(I)$, let $u_\phi$ be the function defined in (\ref{e_calibrated_sub-action_exists}).
Since $U_\beta$ is left-continuous and upper semi-continuous on $(0,1]$, we define a \emph{sub-action} for $(U_\beta,\phi)$ by
\begin{equation}\label{e_def_u_phi_-}
	u_{\beta, \phi}^-(x) \=\begin{cases}
		u_\phi(0) &\text{ if }x=0,\\
		\lim_{y\nearrow x} u_\phi(y) &\text{ if }x \in (0,1].
	\end{cases}
\end{equation}
Since $T_\beta$ is right-continuous and lower semi-continuous on $[0,1)$, we define a sub-action for $(T_\beta,\phi)$ by
\begin{equation}\label{e_def_u_phi_+}
	u_{\beta, \phi}^+(x) \=\begin{cases}
		\lim_{y \searrow x} u_\phi(y) &\text{ if }x\in[0,1),\\
		\lim_{y\searrow T_\beta(1)} u_\phi(y)-\overline{\phi}(1)&\text{ if }x=1.
	\end{cases}
\end{equation}
We define the \emph{left-continuous revealed version} $\tphi^-$ and the \emph{right-continuous revealed version} $\tphi^+$ by
\begin{align}
	\tphi^-&\=\overline{\phi} + u_{\beta, \, \phi}^- -  u_{\beta, \, \phi}^- \circ U_\beta,\label{e_def_tphi^-} \\
	\tphi^+&\=\overline{\phi} + u_{\beta, \, \phi}^+  - u_{\beta, \, \phi}^+  \circ T_\beta.\label{e_def_tphi^+}
\end{align}
\end{definition}

We are now able to prove a Ma\~n\'e 
lemma\footnote{Note that the statement of the Ma\~n\'e cohomology lemma for beta-transformations (Theorem~\ref{mane}) takes a rather different form from the analogous result for open expanding maps (Theorem~\ref{l.mane}): in Theorem~\ref{mane} the continuity properties are emphasised, as these are particularly delicate, whereas the cohomology properties are implicit in the definitions (\ref{e_def_tphi^-}), and (\ref{e_def_tphi^+}).}
for beta-transformations, involving the above sub-actions and revealed versions.

\begin{theorem}[Ma\~n\'e lemma for beta-transformations]\label{mane}
	If $\beta>1$, $\alpha \in (0,1]$, and $\phi \in \Holder{\alpha}(I)$, then the following hold:
	\begin{enumerate}[label=\rm{(\roman*)}]
		\smallskip
		\item $u_{\beta, \, \phi}^-$ is bounded, left-continuous, and upper semi-continuous, on $(0,1]$, and $u_{\beta, \, \phi}^+$ is bounded, right-continuous, and lower semi-continuous, on $[0,1)$.
		\smallskip
		
		\item $\tphi^- \le 0$ and $\tphi^+ \le 0$ on $I$. The function $\tphi^-$ is left-continuous on $(0,1]$, and $\tphi^+$ is right-continuous on $[0,1)$.
		\smallskip
		
		\item If the closed interval $[x,y] \subseteq I$ 
		is disjoint from the orbit $\cO_\beta^*(1)$, then
		\begin{align}
			\Absbig{u_{\beta, \, \phi}^- (x) - u_{\beta, \, \phi}^- (y)} &\le K_{\alpha,\beta} \Hseminorm{\alpha}{\phi}  \abs{x-y}^\alpha\quad\text{ and} \\
			\Absbig{u_{\beta, \, \phi}^+ (x) - u_{\beta, \, \phi}^+ (y)}
			&\le K_{\alpha,\beta} \Hseminorm{\alpha}{\phi}  \abs{x-y}^\alpha.
		\end{align}
	\end{enumerate}
\end{theorem}

\begin{proof}
	(i) follows immediately from (\ref{e_def_u_phi_-}), (\ref{e_def_u_phi_+}), and Proposition~\ref{p_calibrated_sub-action_exists}~(i) and~(ii).
	
	\smallskip
	(ii) By (\ref{e_Def_BouschOp}) and Proposition~\ref{p_calibrated_sub-action_exists}~(iv), 
	\begin{equation}\label{e_u_phi_x_neq_0}
		u_{\phi}(x)=\max\bigl\{\overline{\phi}(y)+u_{\phi}(y):y\in  T_\beta^{-1}(x) \smallsetminus \{1\} \bigr\} \quad \text{ for }x\in [0,1).
	\end{equation}	
	So for all $x\in [0,1)$, we have 
	\begin{equation}\label{e_ophi+u_phi_less_than}
		\tphi(x) \= \overline{\phi}(x)+u_{\phi}(x)- u_{\phi}(T_\beta(x)) \le 0.
	\end{equation}
	Combining this inequality with (\ref{e_def_tphi^+}), (\ref{e_def_u_phi_+}), and Lemma~\ref{l_continuity_T_U_on_x}~(i), for all $x\in [0,1)$, we obtain
	\begin{equation}\label{e_tphi_leq_0_x_neq_0}
		\tphi^+(x)=\overline{\phi}(x)+u_{\beta,\phi}^+(x)-u_{\beta,\phi}^+(T_\beta(x)) = \lim_{y\searrow x} \tphi(y)\leq 0 .
	\end{equation}
	By (\ref{e_def_tphi^+}) and (\ref{e_def_u_phi_+}),
	\begin{equation}\label{e_tphi+_1_eq_0}
		\tphi^+(1)=\overline{\phi}(1)+u_{\beta,\phi}^+(1)-u_{\beta,\phi}^+(T_\beta(1))=0 .
	\end{equation} 
	Combining (\ref{e_tphi_leq_0_x_neq_0}) and (\ref{e_tphi+_1_eq_0}) gives $\tphi^+ \le 0$ on $I$ and the right-continuity of $\tphi^+$ on $[0,1)$.  Combining (\ref{e_ophi+u_phi_less_than}), (\ref{e_def_tphi^-}), (\ref{e_def_u_phi_-}), and Lemma~\ref{l_continuity_T_U_on_x}~(iv), for all $x\in (0,1]$, we obtain
	\begin{equation}\label{e_tphi-le0}
		\tphi^-(x) = \overline{\phi} (x)  + u_{\beta,\phi}^-(x) - u_{\beta,\phi}^-(U_\beta(x))= \lim_{y\nearrow x}\tphi(y) \le 0.
	\end{equation}
	Since $U_\beta(0)=0$ (see (\ref{e_def_U_beta})), by (\ref{e_def_tphi^-}), 
	\begin{equation}\label{e_tphi_0_lessthan_0}
		\tphi^- (0)=\overline{\phi}(0)\leq 0  .
	\end{equation}
	Combining (\ref{e_tphi_0_lessthan_0}) and (\ref{e_tphi-le0}) gives $\tphi^-\le 0$ on $I$ and the left-continuity of $\tphi^-$ on $(0,1]$.
	
	\smallskip
	(iii) This follows immediately from (\ref{e_def_u_phi_-}), (\ref{e_def_u_phi_+}), and Proposition~\ref{p_calibrated_sub-action_exists}~(iii).
\end{proof}

\begin{rem}\label{post_mane_holder_remark}
\begin{enumerate}[label=\rm{(\roman*)}, leftmargin=*]
		\smallskip
		\item
If $\beta$ is a beta-number, then $\cO_\beta^*(1)$ is a finite set, and by Theorem~\ref{mane}~(iii), the functions $u_{\beta, \, \phi}^-$ and $u_{\beta, \, \phi}^+$ are piecewise $\alpha$-H\"older.
\item If $\beta$ is a \emph{simple} beta-number then an improved version of Theorem~\ref{mane} holds, stemming from the fact that the corresponding beta-shift is of finite type: it can be shown that there is a \emph{continuous} revealed version
(more precisely, if $\phi$ belongs to $\Holder{\alpha}(I)$
then there exists a revealed version in $\Holder{\alpha}(I)$).
Moreover, if $\beta$ is a \emph{non-simple} beta-number, then there exist $\phi\in\Holder{\alpha}(I)$ such that
there is no sub-action for $\phi$ in $\Holder{\alpha}(I)$.
 We omit the proofs of these results, as they are not required for establishing subsequent theorems.
\end{enumerate}
\end{rem}

\begin{rem}\label{Manetechnical}
    The approach to proving the Ma\~n\'e lemma for beta-transfor\-mations (Theorem~\ref{mane}) is quite different from the one used in \cite[Proposition~3.6]{LiSu26} to prove the analogous result for open distance-expanding maps (Theorem~\ref{l.mane}). Specifically, the transitivity of beta-transformations means that we are here able to construct a \emph{calibrated} sub-action, 
    something that is in general not possible for nontransitive systems. On the other hand, due to the lack of openness of the map, the function $u_\phi$ defined in (\ref{e_calibrated_sub-action_exists}) can only be shown to be H\"older continuous away from the critical orbit, rather than on the whole of $I$; however a special property of beta-transformations (namely that
    if $x<y$ then any inverse branch for $y$ is also an inverse branch for $x$), led to the inequality (\ref{e_Bousch_Op_Holder_seminorm_bound_x<y}), and ultimately to the proof that $u_\phi$ has both one-sided limits existing everywhere. Note that the lack of openness also means that there is no shadowing lemma analogue of Lemma~\ref{l.shadowing}, 
    therefore to prove that $-\infty < u_\phi < +\infty$ we were unable to pattern the approach on that used to prove Theorem~\ref{l.mane}, but instead employed a more technical argument depending on various special characteristics of beta-transformations.
\end{rem}

The importance of Theorem~\ref{mane} is in allowing us to establish a form of \emph{revelation theorem} (cf.~\cite[Section~5]{Je19}):
the following Theorem~\ref{l_subordination} localises the support of a maximizing measure as lying in the union of the zero sets of the revealed versions $\tphi^-$ and $\tphi^+$, and reveals that individual points in the support of such a measure have their full orbit, under either $U_\beta$ or $T_\beta$, contained in either
$\bigl(\tphi^+\bigr)^{-1} (0)$ 
or $\bigl(\tphi^-\bigr)^{-1} (0)$ respectively.
This latter fact will be 
exploited in our proof of Theorem~\ref{t_TPO_thm_non_emergent}, more specifically in establishing the key Lemma~\ref{l_Q_gamma_equal_Q_non_emergent}.

\begin{theorem}[Revelation theorem] \label{l_subordination} 
	If $\beta>1$, $\alpha \in (0,1]$, $\phi\in \Holder{\alpha}(I)$, and $\mu\in \Mmax(U_\beta,\phi)$, then the following hold:
	\begin{enumerate}[label=\rm{(\roman*)}]
		\smallskip
		\item If $x\in \supp \mu$, then either $\tphi^- \equiv 0$ on $\cO_\beta^*(x)$, or $\tphi^+ \equiv 0$ on $\cO_\beta(x)$. 
		
		\smallskip
		\item If $1\in \supp \mu$ then $\tphi^- \equiv 0$ on $\cO_\beta^*(1)$.
	\end{enumerate}
    In particular, $\supp \mu \subseteq \bigl(\tphi^+\bigr)^{-1} (0)  \cup  \bigl(\tphi^-\bigr)^{-1} (0)$.
\end{theorem}	

\begin{proof}
	(i) 
	Since
	$\tphi^- =\overline{\phi} + u_{\beta, \, \phi}^- -  u_{\beta, \, \phi}^- \circ U_\beta$ (see~(\ref{e_def_tphi^-}))
	where $ u_{\beta, \, \phi}^-$ is bounded and Borel measurable (see Proposition~\ref{p_calibrated_sub-action_exists}~(i) and (\ref{e_def_u_phi_-})), and
	$\overline{\phi}=\phi-Q(T_\beta,\phi)= \phi-Q(U_\beta,\phi)$ (cf.~Proposition~\ref{mpe=}~(ii)) is normalised, it follows from Lemma~\ref{normalised_cohomologous}~(i) and (ii) that $Q\bigl(U_\beta,\tphi^-\bigr)=Q\bigl(U_\beta,\overline{\phi}\bigr)=0$
	and $\Mmax\bigl(U_\beta,\phi\bigr) =  \Mmax\bigl(U_\beta,\overline{\phi}\bigr) =  \Mmax(U_\beta,\tphi^-)$.
	So
	\begin{equation}\label{tphi_zero_integral}
		\int\! \tphi^- \, \mathrm{d}\mu = 0.
	\end{equation}
	
	Suppose $x\in \supp \mu$, so that $\mu ((x-\myepsilon,x+\myepsilon) )>0$ 
	for all $\myepsilon>0$.
    Thus, at least one of the following two cases will occur: either $\mu ((x-\myepsilon, x]) > 0$ for all $\myepsilon>0$, or $\mu ((x,x+\myepsilon)) > 0$ for all $\myepsilon>0$.
	
	Firstly, let us suppose that 
	\begin{equation}\label{leftopenrightclosedpositivemeasurex}
		\mu ((x-\myepsilon, x]) > 0  \text{ for all }\myepsilon>0,
	\end{equation}
	and in this case we aim to show that
	\begin{equation}\label{ubetaorbit}
		\tphi^- \equiv 0 \text{ on }\cO_\beta^*(x).
	\end{equation}
	
    If $x= 0$, then (\ref{leftopenrightclosedpositivemeasurex}) implies that $\mu(\{0\}) > 0$, 
	and since $\tphi^- \le 0$, it follows that 
	$
		0 = \int \! \tphi^- \,\mathrm{d}\mu \le \mu(\{0\}) \tphi^-(0)$,
	so that $\tphi^-(0)=0$,
	in other words, $\tphi^- \equiv 0$ on $\{0\}=\cO_\beta^*(0)$, so (\ref{ubetaorbit}) holds.
	
	If $x \neq 0$,
	we first claim that 
	\begin{equation}\label{leftopenrightclosedpositivemeasurey}
		\mu ((y-\myepsilon, y]) > 0 \text{ for all }y \in \cO_\beta^*(x) \text{ and } \myepsilon>0.
	\end{equation}

	To prove (\ref{leftopenrightclosedpositivemeasurey}), note that
	$0 \notin \cO_\beta^*(x)$ since $x \neq 0$ (see Proposition~\ref{p_relation_T_beta_and_wt_T_beta}~(i)).
	Assume that $y\in\cO_\beta^*(x)$, so that $y=U_\beta^k(x)$ for some $k\in \N$. By Lemma~\ref{l_continuity_T_U_on_x}~(i), for all $\myepsilon>0$ there exists $\delta>0$ such that $U_\beta^k((x-\delta,x])\subseteq (y-\myepsilon,y]$. But $\mu\in \MMM(I,U_\beta)$, so
	(\ref{leftopenrightclosedpositivemeasurex}) implies
	that
	$\mu ((y-\myepsilon, y]) = \mu \bigl(U_\beta^{-k}(y-\myepsilon, y]\bigr)\ge \mu((x-\delta,x])>0$.
	So 
	(\ref{leftopenrightclosedpositivemeasurey}) holds.
	
	Now  $\tphi^-$ is  left-continuous by Theorem~\ref{mane}~(ii), so
	if $y\in \cO_\beta^*(x)$ were such that $\tphi^-(y)<0$, then there would exist $\rho,\myepsilon >0$ with $\tphi^-|_{(y-\myepsilon, y ]} < -\rho$, and 
	(\ref{leftopenrightclosedpositivemeasurey}) would imply that
	\begin{equation*}
		\int \! \tphi^- \,\mathrm{d}\mu \le \int_{(y-\myepsilon, y ]} \! \tphi^- \,\mathrm{d}\mu < -\rho \cdot 
		\mu((y-\myepsilon, y ]) < 0,
	\end{equation*}
	which contradicts (\ref{tphi_zero_integral}).
	So $\tphi^-(y)$ cannot be strictly negative, but on the other hand $\tphi^-\le 0$ by Theorem~\ref{mane}~(ii),
	thus $\tphi^-(y) = 0$. Since $y$ was an arbitrary point in $\cO_\beta^*(x)$, (\ref{ubetaorbit}) holds.

	Secondly, let us suppose that
	\begin{equation}\label{openpositivemeasurex}
		\mu ((x,x+\myepsilon)) > 0  \text{ for all }\myepsilon>0,
	\end{equation}
	and in this case we aim to show that
	$\tphi^+ \equiv 0 \text{ on }\cO_\beta (x)$.
	Note that (\ref{openpositivemeasurex}) gives $x \neq 1$, and consequently $1\notin \cO_\beta(x)$ by Proposition~\ref{p_relation_T_beta_and_wt_T_beta}~(i).
	By arguments analogous to the ones above,
	Lemma~\ref{l_continuity_T_U_on_x}~(i) first implies that 
	$\mu ([y,y+\myepsilon)) > 0$ for all $y\in\cO_\beta(x)$ and $\myepsilon>0$,
	and then the right-continuity of $\tphi^+$ (see Theorem~\ref{mane}~(ii)) implies
	that
	$\tphi^+(y)=0$ for each $y \in \cO_\beta(x)$, as required. 
	
	\smallskip
	(ii)	If $1\in\supp\mu$, then (\ref{leftopenrightclosedpositivemeasurex}) holds for $x=1$, and the argument in (i) above gives (\ref{ubetaorbit}), in other words, $\tphi^- \equiv 0$ on $\cO_\beta^*(1)$, as required.
\end{proof}

\begin{rem}\label{maximizing_set_remark}
A consequence of Theorem~\ref{mane}, and a counterpoint to Theorem~\ref{l_subordination}, is that $\phi$-maximizing measures can be characterised in terms of the support of their pushforward under $\pi_\beta^*$: specifically, for all $\beta>1$, $\alpha\in(0,1]$, and $\phi\in \Holder{\alpha}(I)$, there exists a closed subset $K\subseteq X_\beta$
such that a $U_\beta$-invariant measure $\mu$ belongs to $\Mmax(U_\beta,\phi)$ if and only if 
$\supp G_\beta(\mu) \subseteq K$.
To see this, recall from Proposition~\ref{p_coding_mpe_relation} that
$\mu\in\Mmax(U_\beta,\phi)$ if and only if $G_\beta(\mu)\in \Mmax\bigl(\sigma|_{X_\beta},\phi\circ h_\beta\bigr)$, and then the continuity of $\sigma|_{X_\beta}$ and $\phi\circ h_\beta$ means that the existence of such a $K$
(a so-called \emph{maximizing set} in the terminology of Morris \cite{Mo07}) is guaranteed
(using \cite[Theorem~1, Proposition~1]{Mo07})
 if 
\begin{equation}\label{e_subordination_sufficient_condition}
		\sup\limits_{n\in\N}\sup\limits_{A\in X_\beta} S_n^\sigma\bigl(\overline{\phi}\circ h_\beta\bigr)(A)<+\infty.
\end{equation}
The bound (\ref{e_subordination_sufficient_condition}) can be proved using 
$\tphi^-$ and $\tphi^+$,
since the sub-actions $u_{\beta,\phi}^-$, $u_{\beta,\phi}^+$ are bounded, 
together with the fact that $X_\beta=\pi_\beta(I)\cup\pi_\beta^*(I)$ implied by (\ref{e_def_Z_beta}) and Lemma~\ref{l_properties_z_beta}~(ii).
More precisely, if $A\in\pi_\beta^*(I)$, with $A=\pi_\beta^*(x)$, then
parts~\ref{p_relation_of_coding__iv} and~\ref{p_relation_of_coding__v} of
 Proposition~\ref{p_relation_of_coding}, together with (\ref{e_def_tphi^-}), give
\begin{equation*}
S_n^\sigma\bigl(\overline{\phi}\circ h_\beta\bigr)(A)
=S_n^{U_\beta}\overline{\phi}(x) 
=S_n^{U_\beta}\tphi^-(x)+u_\phi^-\bigl(U_\beta^n(x)\bigr)-u_\phi^-(x) \quad \text{ for all }n\in\N,
\end{equation*}
and using that $\tphi^- \le 0$ (by Theorem~\ref{mane}), together with the bound on the modulus of $u_\phi$ from Proposition~\ref{p_calibrated_sub-action_exists} (i) and the definition of $u_\phi^-$, yields
\begin{equation}\label{6K_bound}
S_n^\sigma\bigl(\overline{\phi}\circ h_\beta\bigr)(A) \le 2(3K_{\alpha,\beta}+2)\Hseminorm{\alpha}{\phi} \quad \text{ for all }n\in\N.
\end{equation}
If $A\in\pi_\beta(I)$ then a similar argument, using $\tphi^+$, gives
\begin{equation}\label{6K+2_bound}
S_n^\sigma\bigl(\overline{\phi}\circ h_\beta\bigr)(A) \le 2(3K_{\alpha,\beta}+3)\Hseminorm{\alpha}{\phi} \quad \text{ for all }n\in\N
\end{equation}
(the right-hand side of (\ref{6K+2_bound}) differing from that of (\ref{6K_bound}) due to the way $u_\phi^+$ is defined at the point 1), and (\ref{6K_bound}), (\ref{6K+2_bound}) together imply (\ref{e_subordination_sufficient_condition}).
\end{rem}

\begin{rem}
In view of Remark~\ref{maximizing_set_remark}, which characterises maximizing measures in terms of some support being contained in a maximizing set, a natural question is whether
the condition
that
$\supp \mu \subseteq \bigl(\tphi^+\bigr)^{-1} (0) \cup \bigl(\tphi^-\bigr)^{-1} (0)$,
which by Theorem~\ref{l_subordination}
is necessary for a $U_\beta$-invariant measure to be $\phi$-maximizing,
 is actually \emph{sufficient}.
The answer is no, as we explain below; note that this contrasts with the situation for 
open distance-expanding maps, where the
Ma\~n\'e lemma (Theorem \ref{l.mane})
guarantees that the sub-action (and hence the revealed version) is continuous. 

Let $\beta\approx 2.48119$ be the largest root of the cubic polynomial $\zeta^3-2\zeta^2-2\zeta+2$.
This is a non-simple beta-number, with $\pi_\beta(1)=2(10)^\infty$.

Define the fixed point
$
z\=h_\beta ((1)^\infty)=\frac{1}{\beta-1} \approx 0.675,
$
and the two period-2 points
$
x\= h_\beta ((10)^\infty)=\frac{\beta}{\beta^2-1}\approx 0.481,
$
$
y\= h_\beta((01)^\infty)=\frac{1}{\beta^2-1}\approx 0.194.
$
We will exhibit Lipschitz functions $\phi$ such that the periodic measure $\mu\=\frac{1}{2}(\delta_x+\delta_y)$
is not $(U_\beta,\phi)$-maximizing, yet its support $\{x,\,y\}$ is contained in
$\bigl(\tphi^+\bigr)^{-1} (0) \cup \bigl(\tphi^-\bigr)^{-1} (0)$.

Defining $\tau(s)\=(s+1)/\beta$,
let $x_1\= h_\beta(1(10)^\infty)=\tau(x)=\frac{\beta^2+\beta-1}{\beta^3-\beta}\approx 0.597$, and
define sequences
\begin{align*}
t_i &\= h_\beta \bigl((1)^{i}2(10)^\infty \bigr)
= h_\beta\bigl((1)^{i}\pi_\beta(1)\bigr)
=\tau^i(1)
=\beta^{-i}\biggl( 2 + \sum_{j=1}^{i-1} \beta^j\biggr) , \quad i\in\N, \\
y_i &\= h_\beta \bigl((1)^{i-1}0(01)^\infty \bigr)=\tau^{i-1}(y/\beta) = \frac{\beta^{i+1}+\beta^i-\beta^2-\beta+1}{\beta^i(\beta^2-1)}, \quad i\in\N.
\end{align*}
Note in particular that $\{t_i\}_{i\in \N}$ is a backwards orbit of $1$ (under $U_\beta$, but not $T_\beta$) converging to $z$, that  $\{y_i\}_{i\in \N}$ is a backwards orbit of $y$ (under both $T_\beta$ and $U_\beta$), also converging to $z$, 
 that $T_\beta^{-1}(x)=U_\beta^{-1}(x)=\{y,\,x_1,\,1\}$,
and that
the relative ordering of the various points is
\begin{align*}
0 < y_1<y<y_2<x<y_3<x_1<&y_4<y_5<y_6<\cdots \\
&< z <\cdots < t_3 < t_2 < t_1 < 1.
\end{align*}
For $\alpha\in(0,1]$, let $\phi\in \Holder{\alpha}(I)$ be nonpositive, with $\phi(y)=-1$ and $\phi(x_1)=-2$, and
identically zero on the points $x$, $1$, and all points in the two sequences $\{t_i\}_{i\in \N}$ and  $\{y_i\}_{i\in \N}$. 
(For concreteness, $\phi$ might be chosen to be the function that is identically zero on the intervals $[0,y_1]$, $[y_2,y_3]$, and $[y_4,1]$, and
affine on each of $[y_1,y]$, $[y,y_2]$, $[y_3,x_1]$, $[x_1,y_4]$, with $\phi(y)=-1$ and $\phi(x_1)=-2$, though our analysis does not assume this). 
Note that $\phi$ is normalised, 
 i.e.,~$Q(U_\beta,\phi)=0$,
since it attains its maximum value $0$ at the fixed point $z$.
On the other hand, $\int\! \phi\, \mathrm{d}\mu = -1/2 < Q(U_\beta,\phi)$, so $\mu\notin\Mmax(U_\beta,\phi)$.

For this value of $\beta$ (and indeed whenever $\beta$ is a beta-number),
it can be shown that the sub-actions $u_{\beta,\phi}^-$, $u_{\beta,\phi}^+$,
defined in (\ref{e_def_u_phi_-}), (\ref{e_def_u_phi_+}), satisfy
 $u_{\beta,\phi}^-(r) = \limsup_{n\to+\infty} \max \bigl\{S_n^{U_\beta}\phi(s):s\in U_\beta^{-n}(r) \bigr\}$ for all $r\in(0,1]$, and
$u_{\beta,\phi}^+(r) = \limsup_{n\to+\infty} \max \bigl\{S_n^{T_\beta}\phi(s):s\in T_\beta^{-n}(r) \bigr\}$
for all $r\in[0,1)$.
Now for each $n\in\N$, the point $y_n\in  U_\beta^{-n}(y) \cap T_\beta^{-n}(y)$, and
$S_n^{U_\beta}\phi(y_n)=0=S_n^{T_\beta}\phi(y_n)$, thus
$u_{\beta,\phi}^-(y)=0= u_{\beta,\phi}^+(y)$.
For each $n\ge 2$, the point $t_{n-1}\in U_\beta^{-n}(x)$,
and $S_n^{U_\beta}\phi(t_{n-1})=0$, thus $u_{\beta,\phi}^+(x)=0$.
Since the point $1$ does not have any pre-images under $T_\beta$, the points
$t_{n-1}$ do not belong to $T_\beta^{-n}(x)$.
Apart from 1, the other two pre-images of $x$ are $y$ and $x_1$, and the inequality $\phi(y)>\phi(x_1)$,
together with the fact that $\phi\equiv 0$ on the backwards orbit $\{y_n\}_{n\in \N}$ of $y$, means that
$\max \bigl\{S_n^{T_\beta}\phi(s):s\in T_\beta^{-n}(x) \bigr\}=\phi(y)=-1$ for all $n\in\N$, and hence
$u_{\beta,\phi}^+(x)=-1$.

Finally, we can evaluate the left-continuous and right-continuous revealed versions at, respectively, the points $x$ and $y$,
to find (cf.~(\ref{e_def_tphi^-}) and (\ref{e_def_tphi^+})) that
$
\tphi^-(x) =\overline{\phi}(x) + u_{\beta, \phi}^-(x) -  u_{\beta, \phi}^- ( U_\beta(x))
= \phi(x) + u_{\beta,  \phi}^-(x) -  u_{\beta, \phi}^- ( y) = 0+0+0=0$ and 
$\tphi^+(y) =\overline{\phi}(y) + u_{\beta, \phi}^+(y)  - u_{\beta,  \phi}^+ (T_\beta(y))
= \phi(y) + u_{\beta,  \phi}^+(y)  - u_{\beta, \phi}^+(x) = -1 +0+1 =0$,
so $\supp \mu =\{x,\,y\}\subseteq \bigl(\tphi^+\bigr)^{-1} (0) \cup \bigl(\tphi^-\bigr)^{-1} (0)$,
as claimed.
\end{rem}

\section{Individual TPO: beta-transformations}\label{sec_proof_of_main_results}

The primary purpose of this section is to prove, in the context of beta-transformations, the individual typical periodic optimization theorems stated in Section~\ref{s:introduction} (more precisely, Theorems~\ref{typical_beta}, \ref{almost_every_beta},
\ref{t_TPO_thm_beta_number},
and~\ref{t_TPO_thm_non_emergent}). Our proofs will exploit the properties of two fundamental subsets of $\Holder{\alpha}(I)$, 
the \emph{critical set} $\Crit^\alpha(\beta)$, and the \emph{regular set} $\RLS^\alpha(\beta)$. 
In Subsection~\ref{emergent_section} we introduce the notion of an \emph{emergent} parameter $\beta$, establish several alternative characterisations, and show that the set of emergent parameters constitutes a small subset of $(1,+\infty)$.
Some characterisations of $\Crit^\alpha(\beta)$,
consisting of those functions for which the critical orbit is maximizing,
are given in Subsection~\ref{subsec_E_alpha_beta}. In Subsection~\ref{subsec_U_alpha_beta}, we first define
the regular set $\RLS^\alpha(\beta)\subseteq \Holder{\alpha}(I)$, which consists of those $\phi\in \Holder{\alpha}(I)$ satisfying $\phi|_{H_\beta^\gamma}\in \Lock^\alpha\bigl(T_\beta|_{H_\beta^\gamma}\bigr)$ for all simple beta-numbers $\gamma\in (1,\beta)$, and then show that it is a dense subset of $\Holder{\alpha}(I)$ (Theorem~\ref{p_density_local_locking}). In Subsection~\ref{subsec_proof_of_emergent} we focus on emergent parameters $\beta$; in the absence of an Individual TPO theorem, we establish a structural theorem (Theorem~\ref{t_TPO_thm_emergent}) identifying the critical set $\Crit^\alpha(\beta)$ as the source of any possible failure of typical periodic optimization.  
Lastly, in Subsection~\ref{subsec:proofofTPOtheorems}, 
we prove the various Individual TPO theorems 
(Theorems~\ref{typical_beta}, \ref{almost_every_beta},
\ref{t_TPO_thm_beta_number},
and~\ref{t_TPO_thm_non_emergent}).

\subsection{Emergent parameters}\label{emergent_section}

\begin{definition}\label{d_emergent}
	A parameter $\beta>1$ will be called \emph{emergent} if 
	\begin{equation*}\label{emergentequation}
		\overline{\cO^\sigma(\pi_\beta^*(1))}\cap \cS_\gamma=\emptyset \quad \text{ for all }\gamma\in(1,\beta) .
	\end{equation*}
	The set of emergent parameters will be denoted by
    $\emergent$, and called the \emph{emergent set}.
\end{definition}

\begin{rem}\label{emergent_terminology}
	The intuition behind the terminology \emph{emergent}, as described in Section~\ref{s:introduction},
	and in view of Lemma~\ref{l_apprioxiation_beta_shif_beta}, is that
	the particular symbolic dynamics of the upper beta-expansion $\pi_\beta^*(1)$ has newly \emph{emerged} at the value $\beta$, in the sense that it does not resemble any that has been witnessed for $\gamma\in(1,\beta)$. 
\end{rem}

Recalling from (\ref{e_def_Z_beta}) that $Z_\beta=\bigl\{x\in I:\pi_\beta(x)\neq \pi_\beta^*(x)\bigr\}$, the following proposition gives two alternative characterisations of emergent parameters.

\begin{prop}\label{cE}
	For $\beta>1$, the following are equivalent:
		\begin{enumerate}[label=\rm{(\roman*)}]	
		\smallskip
		\item $\beta$ is emergent.
		
		\smallskip
		\item $\overline{\cO_\beta^*(1)} \cap H_\beta^\gamma \subseteq Z_\beta$ for each $\gamma \in (1,\beta)$.
		
		\smallskip
		\item $ \bigl(\overline{\cO^\sigma(\pi_\beta^*(1))},\sigma \bigr)$ is minimal.		
	\end{enumerate}
\end{prop}

\begin{proof}	
	(i) implies (ii): Assume that $\beta$ is emergent. Let us suppose, for a contradiction, that the result is false, i.e., that there exists $\gamma\in (1,\beta)$ and $x\in I$ satisfying
	\begin{equation*}
		x\in \bigl(\overline{\cO_\beta^*(1)} \cap H_\beta^\gamma\bigr) \smallsetminus Z_\beta.
	\end{equation*}
	By Proposition~\ref{p_relation_of_coding}~\ref{p_relation_of_coding__iii}, we have $\cO^{\sigma}\bigl(\pi_\beta^*(1)\bigr)=\pi_\beta^*\bigl(\cO_\beta^*(1)\bigr)$. Since the map $\pi_\beta^*$ is continuous at $x$ (see Lemma~\ref{l_properties_z_beta}~(iv)), we have $\pi^*_\beta(x)\in \overline{\cO^\sigma(\pi_\beta^*(1))}$. Since $\pi_\beta|_{H_\beta^\gamma}$ is a homeomorphism (see Lemma~\ref{l_bi_lipschitz_S_gamma_H_gamma}) and $x\in H_\beta^\gamma$, we have $\pi_\beta(x)\in \cS_\gamma$. Hence, since $x\notin Z_\beta$ we have $\pi_\beta(x)=\pi_\beta^*(x)\in \overline{\cO^\sigma(\pi_\beta^*(1))}\cap\cS_\gamma$, which contradicts (see  Definition~\ref{d_emergent}) the fact that $\beta$ is emergent.
	\smallskip
	
	(ii) implies (i): Assume that $\beta$ is non-emergent. Then there exists $\gamma\in (1,\beta)$ such that $\cK_\gamma\= \cS_\gamma\cap \overline{\cO^\sigma(\pi_\beta^*(1))}$ is a nonempty closed subset of $\cS_\gamma$ with $\sigma(\cK_\gamma)\subseteq \cK_\gamma$, by Definition~\ref{d_emergent}. Then by Lemma~\ref{l_bi_lipschitz_S_gamma_H_gamma} and Proposition~\ref{p_relation_of_coding}~\ref{p_relation_of_coding__v}, $h_\beta(\cK_\gamma)$ is a nonempty closed subset of $H_\beta^\gamma$ with $T_\beta(h_\beta(\cK_\gamma))\subseteq h_\beta(\cK_\gamma)$. Since $h_\beta$ is continuous and $\cO_\beta^*(1)=h_\beta \bigl(\cO^\sigma \bigl(\pi_\beta^*(1) \bigr) \bigr)$ (see Proposition~\ref{p_relation_of_coding}~\ref{p_relation_of_coding__x},~\ref{p_relation_of_coding__iii}, and~\ref{p_relation_of_coding__iv}), we obtain $h_\beta(\cK_\gamma)\subseteq \overline{\cO_\beta^*(1)}$. Hence $h_\beta(\cK_\gamma)\subseteq \overline{\cO_\beta^*(1)} \cap H_\beta^\gamma$.
	
	If $h_\beta(\cK_\gamma)\subseteq Z_\beta$, then from the fact that $T_\beta(h_\beta(\cK_\gamma))\subseteq h_\beta(\cK_\gamma)$, and Lemma~\ref{l_properties_z_beta}~(i), we have both that
	$0\in h_\beta(\cK_\gamma)$ and $0\notin Z_\beta$, which is a contradiction. 	
	If on the other hand $h_\beta(\cK_\gamma)\nsubseteq Z_\beta$ then there exists $x\in h_\beta(\cK_\gamma)\smallsetminus Z_\beta$. In both cases, $\bigl(\overline{\cO_\beta^*(1)} \cap H_\beta^\gamma\bigr) \smallsetminus Z_\beta$ is nonempty.
	
	(i) implies (iii): Assume that $\beta$ is emergent. Fix an arbitrary $A\in \overline{\cO^\sigma(\pi_\beta^*(1))}$. If $\pi_\beta^*(1)\in \overline{\cO^\sigma(A)}$, then $\overline{\cO^\sigma(\pi_\beta^*(1))}=\overline{\cO^\sigma(A)}$. So to show that $\bigl(\overline{\cO^\sigma(\pi_\beta^*(1))},\sigma \bigr)$ is minimal, it suffices to verify that $\pi_\beta^*(1)\in\overline{\cO^\sigma(A)}$.
	If on the contrary $\pi_\beta^*(1)\notin \overline{\cO^\sigma(A)}$, then $\overline{\cO^\sigma(A)}$ is a closed subset of $\cS_\beta$ disjoint from $\pi_\beta^*(1)$, which is the maximal element in $\cS_\beta$ (see (\ref{e_original_def_S_beta}) and (\ref{e_def_S_beta})). Hence, from the fact that $\lim_{\gamma\nearrow\beta}\pi_\gamma^*(1)=\pi_\beta^*(1) $ (see Proposition~\ref{p_relation_of_coding}~\ref{p_relation_of_coding__xiv} and (\ref{ix beta})), there exists $\gamma\in (1,\beta)$ with $\sigma^n(A)\preceq \pi_\gamma^*(1)$ for all $n\in \N_0$. Hence by (\ref{e_def_S_beta}), $A\in\overline{\cO^\sigma(\pi_\beta^*(1))}\cap \cS_\gamma$, which contradicts the assumption that $\beta$ is emergent.
	
	(iii) implies (i): Assume that $\beta$ is non-emergent. By Definition~\ref{d_emergent} there exists $\gamma\in (1,\beta)$ and some $B\in \overline{\cO^\sigma(\pi_\beta^*(1))}\cap \cS_\gamma$. Hence the closure of $\cO^\sigma (B)$ is contained in $\cS_\gamma$, and therefore $\bigl(\overline{\cO^\sigma(\pi_\beta^*(1))},\sigma \bigr)$ is not minimal.  
\end{proof}	

The following straightforward consequence of Proposition~\ref{cE} will be useful in the sequel,
in particular for the proofs of the Individual TPO and Joint TPO theorems.

\begin{cor}\label{c_relation_emergent&betanumbers}
    Simple beta-numbers are emergent, and non-simple beta-numbers are non-emergent.
\end{cor}

\begin{proof}
    Let $\beta>1$ be a simple beta-number. By Definition~\ref{d_classification_beta}~(i) and Proposition~\ref{p_relation_of_coding}~(i), $\pi_\beta^*(1)$ is a periodic sequence, so $\overline{ \cO^\sigma ( \pi_\beta^*(1) )  } = \cO^\sigma ( \pi_\beta^*(1) ) $, and $ \bigl( \overline{ \cO^\sigma ( \pi_\beta^*(1) )  } , \sigma \bigr)$ is minimal, therefore Proposition~\ref{cE} implies that $\beta$ is emergent.

    Now let $\beta>1$ be a non-simple beta-number. By Definition~\ref{d_classification_beta}~(ii) and Proposition~\ref{p_relation_of_coding}~(i), $\pi_\beta^*(1)$ is a preperiodic but nonperiodic sequence, so $\overline{ \cO^\sigma ( \pi_\beta^*(1) )  } = \cO^\sigma ( \pi_\beta^*(1) ) $, and $ \bigl( \overline{ \cO^\sigma ( \pi_\beta^*(1) )  } , \sigma \bigr)$ is not minimal, therefore Proposition~\ref{cE} implies that $\beta$ is non-emergent.
\end{proof}

\begin{ex}
	As a specific example of an emergent parameter that is not a simple beta-number, let $F_0=0$, $F_1 = 01$, and $F_{n+2} = F_{n+1}F_{n}$, $n\in \N_0$. The \emph{Fibonacci word} $F$ is then defined (see e.g.~\cite{Py02}) as 
    the sequence 
	whose prefixes are the $F_n$'s,
    and is a Sturmian word (see e.g.~\cite{Py02})
	of parameter $(3-\sqrt{5})/2$. There exists $\beta\in(1,2)$ such that $\pi_\beta(1)=1F$, i.e.,~the concatenation of $1$ and $F$ (cf.~e.g.~\cite[p.~404]{CK04}). Clearly, $\beta$ is not a simple beta-number. Since $\bigl( \overline{\cO^\sigma(1F)},\sigma \bigr)$ is minimal (cf.~\cite[pp.~397--398]{CK04}), $\beta$ is emergent. 
\end{ex}

\begin{lemma}\label{l_minimal_emergent}
	If $\beta>1$ is emergent then the restriction $U_\beta|_{\overline{\cO_\beta^*(1)}}$ is continuous, and
	the dynamical system $\bigl(\overline{\cO_\beta^*(1)},U_\beta\bigr)$ is minimal.
\end{lemma}

\begin{proof}
	Since $\beta$ is emergent, then $0\notin \overline{\cO_\beta^*(1)}$ by Proposition~\ref{cE}. Write $\delta\=d\bigl(0,\overline{\cO_\beta^*(1)}\bigr)$. By (\ref{e_def_U_beta}) and (\ref{e_D_beta}), $\overline{\cO_\beta^*(1)}\cap(y,y+\delta/\beta)=\emptyset$ for each $y\in D_\beta$. So for each pair of $x,\, y\in \overline{\cO_\beta^*(1)}$ with $\abs{x-y}<\delta/\beta$, we have $(x,y)\cap D_\beta=\emptyset$ and hence $U_\beta(x)-U_\beta(y)=\beta(x-y)$. The first part follows.
	
	Let $Y\subseteq \overline{\cO_\beta^*(1)}$ be an arbitrary nonempty closed subset of $\overline{\cO_\beta^*(1)}$ with $U_\beta(Y)=Y$. If $1\in Y$, then $Y=\overline{\cO_\beta^*(1)}$. So to show that $\bigl(\overline{\cO_\beta^*(1)},U_\beta\bigr)$ is minimal, it suffices to show that $1$ must belong to $Y$.
	If on the contrary $1\notin Y$, then by Lemma~\ref{H beta gamma} there exists $\gamma\in (1,\beta)$ such that $Y\subseteq H_\beta^\gamma$. 
	But $\beta$ is emergent, so $Y\subseteq  \overline{\cO_\beta^*(1)} \cap H_\beta^\gamma \subseteq Z_\beta$ by Proposition~\ref{cE}. By Lemma~\ref{l_properties_z_beta}~(i), for each $y\in Y$, there exists $n\in \N$ with $U_\beta^n(y)=1$, which contradicts the assumption that $1\notin Y$. 
	The lemma follows.
\end{proof}

By the following result, emergent parameters constitute a small subset of $(1,+\infty)$.

\begin{cor}\label{emergent_small_set}
The emergent set $\emergent$ has zero Lebesgue measure,
    and is a meagre subset of $(1,+\infty)$.
\end{cor}
\begin{proof}
Suppose $\beta\in\emergent$. 
By Proposition~\ref{cE}, $\bigl(\overline{\cO^\sigma(\pi_\beta^*(1))},\sigma \bigr)$ is minimal, so the fixed point $(0)^\infty$ does not belong to $\overline{\cO^\sigma(\pi_\beta^*(1))}$, and thus the beta-shift $(\cS_\beta,\sigma)$ is a specified system
(see \cite[Proposition~4.5]{Bl89}, \cite[Proposition~3.5]{Scj97}).
So if $\rm{Spec}$ denotes the set of those $\beta>1$ such that $(\cS_\beta,\sigma)$ is  specified, then $\emergent\subseteq \rm{Spec}$. 
But $\rm{Spec}$ is a meagre subset of $(1,+\infty)$ by \cite[Theorem~B]{Scj97}, and has zero Lebesgue measure by \cite[Theorem~E]{Scj97}, 
therefore $\emergent$ also has these properties.
\end{proof}

\subsection{The critical set $\Crit{\texorpdfstring{^\alpha(\beta)}{(β)}}$}\label{subsec_E_alpha_beta}

\begin{definition}
	For $\beta>1$ and $\alpha\in(0,1]$, define the \emph{critical set} $\Crit^\alpha(\beta)$ by
	\begin{equation}\label{e_alpha_defn_eqn}
		\Crit^\alpha(\beta)\=\bigl\{\phi\in \Holder{\alpha}(I): \cO_\beta^*(1) \text{ is a maximizing orbit for }(U_\beta,\phi)\bigr\}.
	\end{equation}
\end{definition}

This naturally prompts an investigation into those invariant measures that are generated by the orbit of the point 1.
It is convenient to first define the following notions in $X_\beta$:

\begin{definition}\label{qgdefn}
	For $\beta>1$, define the \emph{empirical measure} $\mu_n$
	on $X_\beta$
	by
	\begin{equation*}
	\mu_n \= \frac{1}{n}\sum_{i=0}^{n-1} \delta_{\sigma^i(\pi_\beta^*(1))} .
	\end{equation*}
	A measure $\mu\in\MMM(X_\beta,\sigma)$ is said to be \emph{quasi-generated} by $\cO^{\sigma}\bigl(\pi_\beta^*(1)\bigr)$
	if it is an accumulation point of the sequence $\{\mu_n\}_{n\in\N}$.
	Let $\QG(\beta)$ denote the set of measures that are quasi-generated by $\cO^{\sigma}(\pi_\beta^*(1))$
	and let $\CQG(\beta)$ denote the convex hull of $\QG(\beta)$.
\end{definition}

Clearly $\QG(\beta) \subseteq \MMM\bigl(\overline{\cO^{\sigma}(\pi_\beta^*(1))},\sigma\bigr)$,
hence  $\CQG(\beta) \subseteq \MMM\bigl(\overline{\cO^{\sigma}(\pi_\beta^*(1))},\sigma\bigr)$.
The set $\QG(\beta)$ is known to be weak$^*$ closed (see \cite[Proposition~3.8]{DGS76} and \cite[p.~98]{Gl03}), hence  
$\CQG(\beta)$ is weak$^*$ closed.

The above notions lead to an alternative expression for the critical set $\Crit^\alpha(\beta)$:

\begin{lemma}\label{e_alpha_alternative_defn_lemma}
	Suppose $\beta>1$ and $\alpha\in(0,1]$. Then
	\begin{equation}\label{e_alpha_alternative_defn_eqn}
		\Crit^\alpha(\beta)=\bigl\{\phi\in \Holder{\alpha}(I):
		H_\beta(\CQG(\beta))
		\subseteq
		\Mmax(U_\beta,\phi)
		\bigr\}.
	\end{equation}
    Moreover, $\Crit^\alpha(\beta)$ is a closed subset of $\Holder{\alpha}(I)$.
\end{lemma}

\begin{proof}
	We first verify (\ref{e_alpha_alternative_defn_eqn}). By (\ref{e_alpha_defn_eqn}), Proposition~\ref{p_relation_of_coding}~\ref{p_relation_of_coding__v}, and Proposition~\ref{p_coding_mpe_relation}~(iii), 
	\begin{equation}\label{in_E_alpha_iff}
		\phi\in \Crit^\alpha(\beta) 
        \quad \text{ if and only if } \quad
        \cO^{\sigma} \bigl(\pi_\beta^*(1) \bigr) \text{ is $(\sigma|_{X_\beta},\phi\circ h_\beta)$-maximizing}.
	\end{equation}	
	
	Assume that $\phi\in \Crit^\alpha(\beta)$. Then (\ref{in_E_alpha_iff}) implies that
	\begin{equation*}
		\lim_{n\to+\infty}\frac{1}{n}S_n^\sigma(\phi\circ h_\beta)\bigl(\pi_\beta^*(1)\bigr)=Q(\sigma|_{X_\beta},\phi\circ h_\beta),
	\end{equation*}
	and therefore $\QG(\beta)\subseteq \Mmax(\sigma|_{X_\beta},\phi\circ h_\beta)$. Hence
	$\CQG(\beta)\subseteq \Mmax(\sigma|_{X_\beta},\phi\circ h_\beta)$. By Proposition~\ref{p_coding_mpe_relation}~(i) and (iii), it follows that $H_\beta(\CQG(\beta))
	\subseteq
	\Mmax(U_\beta,\phi)$.
	
	Conversely, assume that $\phi\notin \Crit^\alpha(\beta)$. By \cite[Proposition~2.2]{Je19}, \begin{equation*}
		\limsup_{n\to+\infty}\frac{1}{n}S_n^\sigma(\phi\circ h_\beta)\bigl(\pi_\beta^*(1)\bigr)\leq Q(\sigma|_{X_\beta},\phi\circ h_\beta),
	\end{equation*}
	but the fact that $\phi\notin \Crit^\alpha(\beta)$ implies that there exists a subsequence $n_k\nearrow+\infty$ such that  
	\begin{equation*}
		\lim_{k\to+\infty}\frac{1}{n_k}S_{n_k}^\sigma(\phi\circ h_\beta)\bigl(\pi_\beta^*(1)\bigr)
		< Q(\sigma|_{X_\beta},\phi\circ h_\beta).
	\end{equation*}
	Hence any accumulation point of the sequence $\{\mu_{n_k}\}_{k\in\N}$, which belongs to $\QG(\beta)$
	by Definition~\ref{qgdefn}, cannot belong to $\Mmax(\sigma|_{X_\beta},\phi\circ h_\beta)$.  By Proposition~\ref{p_coding_mpe_relation}~(i) and (iii), it follows that $H_\beta(\CQG(\beta))
	$ is not contained in $
	\Mmax(U_\beta,\phi)$.

    	We next prove that $\Crit^\alpha(\beta)$ is a closed subset of $\Holder{\alpha}(I)$. If $\phi_n\in \Crit^\alpha(\beta)$, then
	\begin{equation}\label{phi_n_Q}
		\int \! \phi_n\, \mathrm{d}\mu = Q(U_\beta,\phi_n)
	\end{equation}
	for all $\mu\in H_\beta(\CQG(\beta))$,
	by (\ref{e_alpha_alternative_defn_eqn}).
	If $\phi_n\to\phi$ in $(\Holder{\alpha}(I),\Hnorm{\alpha,I}{\cdot})$,
	then 
    $\int \!\phi_n \, \mathrm{d}\mu \to \int\! \phi\, \mathrm{d}\mu$
	for all $\mu\in H_\beta(\CQG(\beta))$.
	But $Q(U_\beta, \cdot)$ is (1-Lipschitz) continuous,
	so $Q(U_\beta,\phi_n)\to Q(U_\beta,\phi)$,
	and therefore (\ref{phi_n_Q}) implies that $\int \!\phi\, d\mu=Q(U_\beta,\phi)$
	for all $\mu\in H_\beta(\CQG(\beta))$.
	So $H_\beta(\CQG(\beta)) \subseteq
	\Mmax(U_\beta,\phi)$. Hence $\phi\in \Crit^\alpha(\beta)$ by
	(\ref{e_alpha_alternative_defn_eqn}), and therefore $\Crit^\alpha(\beta)$ is closed.
\end{proof}

If $\beta>1$ is emergent, the following lemma gives another characterisation of $\Crit^\alpha(\beta)$.

\begin{lemma}\label{emergent_e_alpha_lemma}
	Assume that $\beta>1$ is emergent, $\alpha\in(0,1]$, and $\phi\in\Holder{\alpha}(I)$. Then the following are equivalent:
	\begin{enumerate}[label=\rm{(\roman*)}]	
		\smallskip
		\item $\phi\in \Crit^\alpha(\beta)$.
		
		\smallskip
		\item $H_\beta(\CQG(\beta)) \subseteq	\Mmax(U_\beta,\phi)$.
		
		\smallskip
		\item $1\in \supp\mu$ for some $\mu\in \Mmax(U_\beta,\phi)$.		
	\end{enumerate}
\end{lemma}

\begin{proof}
	By Lemma~\ref{e_alpha_alternative_defn_lemma}, (i) is equivalent to (ii), so it 
	suffices to prove that (i) is equivalent to (iii). 
	For this, first assume that $\phi\in \Crit^\alpha(\beta)$,
	so that $\cO_\beta^*(1)$ is $(U_\beta,\phi)$-maximizing.
	Suppose $\mu$ is any accumulation point of the sequence
	$\bigl\{\frac{1}{n}\sum_{i=0}^{n-1} \delta_{U_\beta^i(1)}\bigr\}_{n\in\N}$.
	Note that $\overline{\cO_\beta^*(1)}$ is compact, $U_\beta\bigl(\overline{\cO_\beta^*(1)}\bigr)=\overline{\cO_\beta^*(1)}$, and $U_\beta|_{\overline{\cO_\beta^*(1)}}$ is continuous (see Lemma~\ref{l_minimal_emergent}).
    We obtain that $\supp \mu\subseteq \overline{\cO_\beta^*(1)}$ and $\mu\in \MMM\bigl(\overline{\cO_\beta^*(1)},U_\beta|_{\overline{\cO_\beta^*(1)}}\bigr)$ (see \cite[Theorem~6.9]{Wal82}). Hence $\mu$ is also $U_\beta$-invariant as a measure on $I$.
    In addition, $\mu\in\Mmax(U_\beta,\phi)$
	since $\cO_\beta^*(1)$ is $(U_\beta,\phi)$-maximizing.
	But $\beta$ is emergent, so $\bigl(\overline{\cO_\beta^*(1)},U_\beta\bigr)$ is minimal by Lemma~\ref{l_minimal_emergent}.  
 Note that $\mu$ can be seen as a measure on $\overline{\cO_\beta^*(1)}$ and it follows from \cite[p.~156]{Ak93} that $U_\beta(\supp \mu)=\supp \mu$. Hence $\supp \mu$ must equal $\overline{\cO_\beta^*(1)}$, and in particular $1\in \supp \mu$.
	
	Conversely, if we assume that $1\in \supp\mu$ for some $\mu\in \Mmax(U_\beta,\phi)$, then (\ref{e_def_tphi^-}), Theorem~\ref{mane}, Theorem~\ref{l_subordination}~(ii), and
	Lemma~\ref{normalised_cohomologous}~(iii)
	together imply that $\cO^*_\beta(1)$ is a maximizing orbit for $(U_\beta,\phi)$,
	and hence that $\phi\in \Crit^\alpha(\beta)$.
\end{proof}

\begin{lemma}\label{l_case_two_one_optimal}
	Suppose $\beta>1$, $\alpha \in (0,1]$,
	and $\phi\in \Holder{\alpha}(I)$.
	If $1\notin \supp\mu$ for some
	$\mu\in\Mmax(U_\beta,\phi)$,
	then there exists $\beta'\in(1,\beta)$ such that
	$\mpe_{\beta,\gamma}(\phi) = \mpe ( U_\beta , \phi )$
	for all $\gamma\in(\beta',\beta)$.
\end{lemma}

\begin{proof}
	Suppose $\mu\in \Mmax(U_\beta,\phi)$ and $1\notin \supp \mu \eqqcolon\cK$.
	It follows that $\mu \in \cM(I,T_\beta)$ by Proposition~\ref{p_relation_T_beta_and_wt_T_beta}~(iii),
	and $T_\beta(\cK)=\cK$ by Lemma~\ref{support_invariant_set}. Therefore,
	by Lemma~\ref{H beta gamma},
	there exists $\beta' \in (1,\beta)$ such that $\cK \subseteq H_\beta^\gamma$
	for all $\gamma\in(\beta',\beta)$.
	In particular, $\mu$ can be considered as an element of $\cM\bigl(H_\beta^\gamma, T_\beta\bigr)$
	for all $\gamma\in(\beta',\beta)$, and thus
	$\mpe_{\beta, \gamma}(\phi) = \mpe ( U_\beta, \phi )$ by (\ref{e_def_beta_gamma}), as required.
\end{proof}

If $\beta$ is emergent, we have the following corollary:  

\begin{cor}\label{l_case_two_one_optimal_emergent}
	Suppose $\beta>1$ is emergent, $\alpha \in (0,1]$, and $\phi\in \Holder{\alpha}(I)\smallsetminus \Crit^\alpha(\beta)$. Then there exists $\beta'\in(1,\beta)$ such that
	$\mpe_{\beta,\gamma}(\phi) = \mpe ( U_\beta , \phi )$
	for all $\gamma\in(\beta',\beta)$.
\end{cor}
\begin{proof}
	Proposition~\ref{p_coding_mpe_relation}~(iii) implies $\Mmax(U_\beta,\phi)\neq \emptyset$. Since $\phi\notin \Crit^\alpha(\beta)$, Lemma~\ref{emergent_e_alpha_lemma} implies that
	if $\mu\in \Mmax(U_\beta,\phi)$ 
	then $1\notin \supp \mu$.
	The corollary now follows from Lemma~\ref{l_case_two_one_optimal}.
\end{proof}

\subsection{The regular set $\RLS{\texorpdfstring{^\alpha(\beta)}{(β)}}$}\label{subsec_U_alpha_beta}

In this subsection
we establish a key result, the Dense Regular Functions Theorem (Theorem~\ref{p_density_local_locking}), asserting the density of the so-called \emph{regular set} in $\Holder{\alpha}(I)$.

\begin{notation}\label{d_cU_beta}
	For each $\beta>1$, let
    $\Sim(\beta)$
    denote the set of those simple beta-numbers contained in the interval $(1,\beta)$.
    \end{notation}

    \begin{definition}
	For each $\beta>1$ and each $\alpha\in (0,1]$, define the \emph{regular set}
	$\RLS^\alpha(\beta)\subseteq \Holder{\alpha}(I)$ by
	\begin{equation*}\label{e_def_local_locking}
		\RLS^\alpha(\beta)
		\=
		\bigl\{ \phi\in \Holder{\alpha}(I) :
		\phi|_{H_\beta^{\gamma}}\in \Lock^\alpha\bigl(T_\beta|_{H_\beta^{\gamma}}\bigr)
		 \text{ for all }\gamma\in\Sim(\beta)\bigr\}.
	\end{equation*}
\end{definition}

To prove the Dense Regular Functions Theorem, we will first use that for each $\beta\in(1,+\infty)$, $\Sim(\beta)$ is dense in $(1,\beta)$ by \cite[Theorem~5]{Par60}. Then, by Definitions~\ref{beta shifts} and~\ref{d_classification_beta}, $\gamma\in\Sim(\beta)$ if and only if $\pi_\gamma(1)$ has finitely many nonzero terms and $\gamma\in (1,\beta)$. Note that the function $i_1$, as defined in Definition~\ref{d_X_beta_and_tilde_X_beta}, is 
by Proposition~\ref{p_relation_of_coding}~\ref{p_relation_of_coding__xiii}
a strictly increasing
map from $(1,\beta)$ to $\cB^\N$, hence $\Sim(\beta)$ is countable, and enumerated as $\Sim(\beta)=\{\gamma_1,\,\dots,\,\gamma_n,\,\dots\}$, say. Then for each $\phi \in \Holder{\alpha}(I)$, Theorem~\ref{co} can be used to obtain a sequence of functions $\{\phi_n\}_{n\in \N}$ in $\Holder{\alpha}(I)$ sufficiently close to $\phi$ such that $\phi_n |_{H_\beta^{\gamma_n}} \in \Lock^\alpha \bigl( T_\beta|_{H_\beta^{\gamma_n}} \bigr)$ for each $n \in \N$, and it can be verified that $\{\phi_n\}_{n\in \N}$ converges to a function $\phi_\infty \in \Holder{\alpha}(I)$ approximating $\phi$.

In the proof of Theorem~\ref{p_density_local_locking} below, we will need to extend a H\"older function, defined on a subset of $I$,
to the whole of $I$, without increasing its norm. The following lemma guarantees that this can be done.

\begin{lemma}\label{H extension}
	For $\alpha\in (0,1]$ and $\emptyset \neq \cK \subseteq I$, and $\phi \in C^{0, \alpha}(\cK)$, there exists $\psi \in C^{0, \alpha}(I)$ such that $\psi|_\cK = \phi$ and $\Hnorm{ \alpha, I }{\psi}= \Hnorm{\alpha, \cK}{\phi}$.
\end{lemma}
\begin{proof}
	This follows immediately from the McShane extension theorem (see e.g.~\cite[Theorem~1.33]{Wea18}).
\end{proof}

Now we are ready to state and prove Theorem~\ref{p_density_local_locking}:

\begin{theorem}[Dense Regular Functions]\label{p_density_local_locking}
	For all $\beta>1$, $\alpha\in(0,1]$, the
    regular set $\RLS^\alpha(\beta)$ is dense in $\Holder{\alpha}(I)$.
\end{theorem}
\begin{proof}
	Let $\myepsilon>0$ and $\phi \in C^{0, \alpha}(I)$ be arbitrary. Since $\Sim(\beta)$ is countable, we write
	\begin{equation*}
		\Sim(\beta)=\{\gamma_1,\,\dots,\,\gamma_n,\,\dots\}.
	\end{equation*}
	
	We will recursively construct a sequence of functions $\{\phi_n\}_{n\in\N}$ in $\Holder{\alpha}(I)$, 
	and two sequences of positive numbers $\{\delta_n\}_{n\in \N}$ and $\{e_n\}_{n\in\N}$ below:
	
	\smallskip
	\emph{Base step.} Define $\phi_0 \= \phi$ and $\delta_0 \= 1$.
	
	\smallskip
	\emph{Recursive step.} For $n \in \N$, assume that $\phi_0, \, \dots, \, \phi_{n-1}, \, \delta_0, \, \dots, \, \delta_{n-1}$ are defined. Define 
	\begin{equation}\label{e_e_n}
		e_n \= \min \bigl\{ 2^{-n}\myepsilon, \, 2^{-n}\delta_0, \, 2^{-n+1}\delta_1, \, \dots, \, 2^{-1}\delta_{n-1}  \bigr\}.
	\end{equation}
	Since $\gamma_n$ is a simple beta-number, Proposition~\ref{l_property_H_gamma} gives that $T_\beta|_{H_\beta^{\gamma_n}}$ is Lipschitz continuous, open, and distance-expanding, and then Contreras' Individual TPO theorem
	(Theorem~\ref{co}) guarantees that there exists $\psi_n\in \Holder{\alpha}\bigl(H_\beta^{\gamma_n}\bigr)$ satisfying
	\begin{equation*}\label{e_est_psi_n_}
		\psi_n \in \Lock^\alpha\bigl(T_\beta|_{H_\beta^{\gamma_n}} \bigr)
        \quad \text{ and } \quad
        \Hnormbig{\alpha,H_\beta^{\gamma_n}}{ \phi_{n-1}|_{H_\beta^{\gamma_n}} -\psi_n }<e_n.
	\end{equation*}
	Applying Lemma~\ref{H extension}, for the subset $H_\beta^{\gamma_n}\subseteq I$ and function $\phi_{n-1}|_{H_\beta^{\gamma_n}} - \psi_n$, we obtain an $\alpha$-H\"older extension function $\Phi_n$ defined on all of $I$, and with the same  $\alpha$-H\"older norm, in other words, there exists
	$\Phi_n\in \Holder{\alpha}(I)$ satisfying 
	$\Phi_n|_{H_\beta^{\gamma_n}} = \phi_{n-1}|_{H_\beta^{\gamma_n}} - \psi_n$ and $\Hnorm{\alpha,I}{\Phi_n }<e_n$. Defining $\phi_n\=\phi_{n-1}-\Phi_n$ then gives
	\begin{equation}\label{e_est_varphi_and_varphi_n}
		\phi_n|_{H_\beta^{\gamma_n}} = \psi_n
        \quad \text{ and } \quad
        \Hnorm{\alpha,I}{ \phi_{n-1} -\phi_n }<e_n.
	\end{equation}
	Finally, $\delta_n$ is defined to be any value such that	
	\begin{equation}\label{e_local_lock_neighbourhood}
		\bigl\{\Phi\in \Holder{\alpha}\bigl(H_\beta^{\gamma_n}\bigr): \Hnormbig{\alpha,H_\beta^{\gamma_n}}{\Phi-\phi_n|_{H_\beta^{\gamma_n}}}<\delta_n\bigr\} \subseteq \Lock^\alpha\bigl( T_\beta|_{H_\beta^{\gamma_n}}\bigr),
	\end{equation}
	where again Theorem~\ref{co} and Proposition~\ref{l_property_H_gamma} guarantee that such a $\delta_n$ exists. 
	Since each of $e_n$, $\phi_n$, and $\delta_n$ has been defined, the recursive step is complete.	
	
	By (\ref{e_est_varphi_and_varphi_n}) and (\ref{e_e_n}), $\{\phi_n\}_{n\in\N}$ converges uniformly to some $\phi_\infty\in \Holder{\alpha}(I)$, and
	\begin{equation}\label{e_est_varphi_and_varphi_infty}
		\Hnorm{\alpha,I}{\phi-\phi_\infty}\leq\sum_{n=1}^{+\infty}\Hnorm{\alpha,I}{ \phi_{n-1} -\phi_n }<\sum_{n=1}^{+\infty}e_n\leq \myepsilon.
	\end{equation}
	For each $n\in \N$,   from (\ref{e_e_n}) it follows that $\sum_{i=n+1}^{+\infty} e_i \le \delta_n$
	and then
	\begin{equation}\label{e_est_Hnorm_varphi_and_varphi_infty}
		\Hnormbig{\alpha,H_\beta^{\gamma_n}}{\phi_n|_{H_\beta^{\gamma_n}}-\phi_\infty|_{H_\beta^{\gamma_n}}}
        \leq\Hnorm{\alpha,I}{\phi_n-\phi_\infty}
        \leq\sum_{i=n}^{+\infty}\Hnorm{\alpha,I}{ \phi_{i} -\phi_{i+1} }
        <\sum_{i=n+1}^{+\infty}e_i
        \leq \delta_n.
	\end{equation} 
	But (\ref{e_local_lock_neighbourhood}) and (\ref{e_est_Hnorm_varphi_and_varphi_infty})
	imply that  $\phi_\infty|_{H_\beta^{\gamma_n}}  \in \Lock^\alpha(T_\beta|_{H_\beta^{\gamma_n}})$ for all $n\in\N$, in other words, $\phi_{\infty}\in \RLS^\alpha(\beta)$. But $\myepsilon>0$ was arbitrary, so the result follows.
\end{proof}

\begin{cor}\label{p_density_local_locking_corollary}
	For all $\beta>1$ and $\alpha\in(0,1]$, the set $\RLS^\alpha(\beta)\smallsetminus \Crit^\alpha(\beta)$ is a dense subset of $\Holder{\alpha}(I)\smallsetminus \Crit^\alpha(\beta)$.
\end{cor}
\begin{proof}
	 This follows immediately from the fact that
	$\RLS^\alpha(\beta)$ is dense in $\Holder{\alpha}(I)$
	by Theorem~\ref{p_density_local_locking},
	and the fact that $\Crit^\alpha(\beta)$ is a closed subset of  $\Holder{\alpha}(I)$
	by Lemma~\ref{e_alpha_alternative_defn_lemma}.
\end{proof}

\subsection{A structural theorem for emergent parameters}\label{subsec_proof_of_emergent}

In this subsection we shall be primarily concerned with those parameters $\beta$ that are emergent.
For such $\beta$ it remains an open question as to whether $U_\beta$
has the typical periodic optimization property, nevertheless we shall establish a structural theorem
(Theorem~\ref{t_TPO_thm_emergent}) that identifies the critical set $\Crit^\alpha(\beta)$ as the only possible obstacle to TPO.
We first require:

\begin{cor}\label{u_alpha_beta_subset_p_alpha_beta}
	If $\beta>1$ is emergent and $\alpha\in(0,1]$, then
	\begin{equation}\label{u_e_p}
		\RLS^\alpha(\beta) \smallsetminus \Crit^\alpha(\beta) \subseteq \sP^\alpha(U_\beta) \smallsetminus \Crit^\alpha(\beta).
	\end{equation}
\end{cor}
\begin{proof}
	Suppose $\phi\in \RLS^\alpha(\beta) \smallsetminus \Crit^\alpha(\beta)$.
        By 	Corollary~\ref{l_case_two_one_optimal_emergent},
        there is a simple beta-number $\gamma_0\in (1,\beta)$ such that
	\begin{equation}\label{gamma_N_beta_phi_emergent}
		\mpe_{\beta,\gamma_0}(\phi) = \mpe ( U_\beta , \phi ) .
	\end{equation}
	The fact that $\phi\in \RLS^\alpha(\beta)$ implies that
	$
	\phi|_{H_\beta^{\gamma_0}}\in \Lock^\alpha\bigl(T_\beta|_{H_\beta^{\gamma_0}}\bigr)
	\subseteq
	\sP^\alpha \bigl(T_\beta|_{H_\beta^{\gamma_0}}\bigr)$.
	So there exists a periodic measure $\mu\in \MMM(I,T_\beta)\subseteq \MMM(I,U_\beta)$ (see Proposition~\ref{p_relation_T_beta_and_wt_T_beta}~(iii)) such that
	\begin{equation}\label{integral_phi_mu_equal_emergent}
		\int\! \phi\, \mathrm{d}\mu = \mpe_{\beta,\gamma_0}(\phi) .
	\end{equation}
	So from (\ref{gamma_N_beta_phi_emergent}) and (\ref{integral_phi_mu_equal_emergent}) we see that the periodic measure $\mu$
	satisfies $\int\! \phi\, \mathrm{d}\mu= Q(U_\beta,\phi)$,
	and therefore $\phi\in \sP^\alpha(U_\beta)$.
	But  $\phi\notin  \Crit^\alpha(\beta)$, so 
    (\ref{u_e_p}) follows.
\end{proof}

In the sequel, it will be convenient to articulate the relationship between the set of functions where at least one maximizing measure is periodic, and its subset consisting of those functions whose maximizing measure is unique, periodic, and stably maximizing under perturbations.
For this, recall that
$\sP^\alpha(U_\beta)$
denotes the set of those $\phi \in \Holder{\alpha}(I)$ with a $(U_\beta,\phi)$-maximizing measure supported on a periodic orbit of $U_\beta$,
and that $\Lock^\alpha(U_\beta)$
    consists of those functions 
    $\phi \in \sP^\alpha(U_\beta)$
    with $\card \Mmax(U_\beta ,  \phi) = 1$ and $\Mmax(U_\beta,  \phi) = \Mmax(U_\beta,  \psi)$ for all $\psi \in  \Holder{\alpha}(I)$ sufficiently close to $\phi$ in $\Holder{\alpha}(I)$.    
We wish to  show that $\Lock^\alpha(U_\beta)$ is dense in $ \sP^\alpha(U_\beta)$ (see Theorem~\ref{3.1} below).
A statement analogous to this 
appeared as Remark~4.5 in \cite{YH99}
for maps with hyperbolicity,
and as Proposition~1 in the unpublished note \cite{BZ15} for continuous maps.
Our method of proof, using ideas from \cite{BZ15}, begins with the following lemma.

\begin{lemma}\label{3.3}
	Suppose $\beta>1$, $\alpha \in (0,1]$, and let $\mu\in\MMM ( I, U_\beta )$ be supported on a periodic orbit $\cO^*$ of $U_\beta$. Then there exists $C_\mu >0$ such that for all $\nu\in\MMM (I,U_\beta )$ and  
	$\phi\in \Holder{\alpha} (I )$, 
	\begin{equation}\label{eq 3.2}
		\int_I \! \phi \, \mathrm{d}\nu
		\le\int_I \! \phi \, \mathrm{d} \mu+C_\mu \Hseminorm{\alpha,I}{\phi} \int_I \! d (\cdot,\cO^* )^\alpha \, \mathrm{d} \nu.
	\end{equation}
\end{lemma}

\begin{proof}
	Let us write $n \= \card\cO^*$.
	
	\smallskip
	\emph{Case I.} If $n=1$ then $\cO^*= \{y \}$ for some $y\in I$, and 
	(\ref{eq 3.2}) holds with $C_\mu=1$ because
	$\int_I \! \phi\, \mathrm{d}\nu
	\le\int_I \! (\phi (y ) + \Hseminorm{\alpha,I}{\phi}  d (\cdot,y ) ^\alpha) \, \mathrm{d}\nu 
	=\int_I \! \phi\, \mathrm{d}\mu+ \Hseminorm{\alpha,I}{\phi} \int_I\! { d (\cdot,\cO^* ) }^\alpha\, \mathrm{d}\nu$.

    \smallskip
	
	\emph{Case II}. If $n>1$ then by ergodic decomposition, it suffices to prove (\ref{eq 3.2}) for every ergodic $\nu \in \MMM (I, U_\beta )$. Fixing an arbitrary ergodic $\nu\in \MMM (I, U_\beta )$, the ergodic theorem
	implies that there exists $a\in I$ with
	\begin{align}
		\int_I \! \phi \, \mathrm{d}\nu &= \lim_{k\to+\infty} \frac{1}{k} S^{U_\beta}_k\phi(a) \quad \text{ and} \label{eq 3..4} \\
		\int_I \! d(\cdot , \cO^*)^\alpha \, \mathrm{d}\nu&=\lim_{k\to+\infty}\frac{1}{k}\sum_{i=0}^{k-1} d \bigl( U_\beta^i (a), \cO^* \bigr)^\alpha .\label{eq 3.4}
	\end{align}
	\emph{Claim.} There exists $C_\mu >0$ 
	and a sequence $\underline{y}= \{y_i\}_{i=-1}^{+\infty}$ with entries from $\cO^*$ such that
	\begin{equation}\label{eq 3.7}
		\lim_{k\to+\infty}\frac{1}{k}\card \{i\in [0,k-1]\cap \N_0:y_i=y \}=\frac{1}{n} \quad \text{ for each }  y\in \cO^* \text{ and }
	\end{equation}
	\begin{equation}\label{eq 3.8}
		\Absbig{ U_\beta^i (a )-y_i }\le C_\mu^{1/\alpha} d \bigl(U_\beta^i (a ),\cO^*\bigr) \quad \text{ for each } i\in\N_0.
	\end{equation}
	Note that a consequence of this Claim is, by (\ref{eq 3..4}), (\ref{eq 3.8}), (\ref{eq 3.7}), and (\ref{eq 3.4}), that
	\begin{equation*}
		\begin{aligned}
			\int_I\! \phi\, \mathrm{d}\nu
			&=\lim_{k\to+\infty}\frac{1}{k}\sum_{i=0}^{k-1}\phi \bigl(U_\beta^i (a ) \bigr)
			\le\lim_{k\to+\infty}\frac{1}{k}\sum_{i=0}^{k-1} \bigl(\phi (y_i )+ \Hseminorm{\alpha,I}{\phi}  \Absbig{ U_\beta^i (a )-y_i }^\alpha \bigr)\\
			&\le\lim_{k\to+\infty}\frac{1}{k}\sum_{i=0}^{k-1} \bigl(\phi (y_i )+C_\mu \Hseminorm{\alpha,I}{\phi} d \bigl(U_\beta^i (a ), \cO^* \bigr)^\alpha \bigr) & \\
			&= \frac{1}{n}\sum_{y\in \cO^*}\phi (y )+C_\mu \Hseminorm{\alpha,I}{\phi} \lim_{k\to+\infty} \frac{1}{k}\sum_{i=0}^{k-1} d \bigl(U_\beta^i (a ), \cO^* \bigr)^\alpha \\
			&=\int_I \! \phi\, \mathrm{d}\mu+C_\mu \Hseminorm{\alpha,I}{\phi} \int_I \! d (\cdot , \cO^* )^\alpha \, \mathrm{d} \nu .
		\end{aligned}
	\end{equation*}
	So the required inequality (\ref{eq 3.2}) will hold, and the lemma will follow.
	
	\smallskip
	\emph{Proof of Claim.} Our discussion will be divided into $2$ cases.
    Recall (cf.~Subsection~\ref{sct_Notation}) that since
    $\card\cO^*>1$, the minimum interpoint distance
    is given by
	$\Delta (\O^*)  = \min\{ d(x, y) : x, \, y \in \O^*, \, x \neq y \}$.

	\smallskip
    
	\emph{Case 1.} Assume that $\cO^*$ is not the orbit of $1$ under $U_\beta$. By Proposition~\ref{p_relation_T_beta_and_wt_T_beta}~(i),  $D_\beta\cap \cO^*=\emptyset$. Set 
	\begin{equation}\label{eq 3.5}
		\delta\= ( 1/2 ) \Delta(\cO^*)
	\end{equation}
	and $\myepsilon^*\=(1/2)\beta^{-n}d(\cO^*,D_\beta)$, so that $U_\beta^k\big|_{B(\cO^*,\myepsilon^*)}$ is continuous for all $0\le k\le n-1$. For each $x\in\cO^*$, there exists $\myepsilon_x\in (0,\myepsilon^*)$ such that $\Absbig{U_\beta^i(x)-U_\beta^i(y)}<\delta$ for all $y\in(x-\myepsilon_x,x+\myepsilon_x)$ and $0\le i\le n-1$. Moreover, 
	if $\myepsilon\=\min  \{\myepsilon_x : x\in\cO^* \}$ then 
	\begin{equation}\label{eq 3.6}
		\Absbig{U^k_\beta(x)-U^k_\beta(y)}<\delta
	\end{equation}
	for all $x\in\cO^*$, $y\in B(x,\myepsilon)$, and $0\le k\le n-1$.
	
	The sequence $\underline{y}$ is constructed recursively as follows.
	
	\smallskip
	
	\emph{Base step.} Choose an arbitrary $y_{-1}\in\cO^*$, and mark $y_{-1}$ as a bad point.
	
	\smallskip
	\emph{Recursive step.} For some $t\in \N_0$, assume that $y_{-1}, \, y_0, \, \dots, \, y_{t-1}$ are defined. 
	
	If $d \bigl(U_\beta^t (a ),\cO^* \bigr)<\myepsilon$, choose $y_t\in\cO^*$ such that $d\bigl(U_\beta^t (a ),\cO^*\bigr)=\Absbig{ U_\beta^t (a )-y_t }$, set $y_{t+i} \= U_\beta^i(y_t)$ for each $1\le i\le n-1$, and mark $y_t, \, y_{t+1}, \, \dots, \, y_{t+n-1}$ as good points.
	
	If $d\bigl(U_\beta^t (a ),\cO^*\bigr)\ge\myepsilon$, let $p_t$ denote the number of bad points in the set $ \{y_{-1}, \, y_0, \, \dots, \, y_{t-1} \}$, then set $y_t\=U_\beta^{p_t} (y_{-1} )$ and mark $y_t$ as a bad point.
	
	\smallskip
	Note that the required (\ref{eq 3.7}) is immediate from the above construction. 
	To prove  (\ref{eq 3.8}), note that
	for each bad point $y_i$, $i\in \N_0$, we have $d\bigl(U_\beta^i (a ),\cO^*\bigr)\ge\myepsilon \ge \myepsilon d\bigl(U_\beta^i (a ),y_i\bigr)$, and for each good point $y_i$, 
	we obtain $d\bigl(U_\beta^i (a ),\cO^*\bigr)= d\bigl(U_\beta^i (a ),y_i\bigr)$
	by (\ref{eq 3.5}),  (\ref{eq 3.6}), and our construction.
	Therefore (\ref{eq 3.8}) holds by choosing $C_\mu \= \max\{1, \, \myepsilon^{-\alpha}\}$, so Claim is proved for Case~1.

	\smallskip
    
	\emph{Case 2.} Assume that $\cO^*$ is the orbit of $1$ under $U_\beta$. 
	By Remark~\ref{r_after_classification_of_beta}, $\beta$ is a simple beta-number. By Proposition~\ref{p_relation_T_beta_and_wt_T_beta}~(i), $D_\beta\cap \cO^*_\beta(1) = U_\beta^{-1}(1) \cap \cO^* =\bigl\{U^{n-1}_\beta(1)\bigr\}$ and $U_\beta^{-1}(0) = \{0\}$. Set 
	\begin{equation}\label{eq_delta_lemma_locking_simple}
		\delta'\=\min\{ (1/2) \Delta(\cO^*), \, (1/2) d(0,\cO^*)\}.
	\end{equation} 
	As $U_\beta^{-1}(0) = \{0\}$, then $0\notin \cO^*$, so $\delta'>0$. For each $0 \le i \le n-1$ and
	$z_i\=U^{i}_\beta(1)$, we 
	claim that
	there exists $\myepsilon_i>0$ satisfying the following properties:
	
	\smallskip
	(1) $\Absbig{U_\beta^j(z_i)-U_\beta^j(y)}<\delta'$ for all $y\in(z_i-\myepsilon_i,z_i]$, $0\le j\le n-1$. 
	
	\smallskip
	(2) $\Absbig{U_\beta^j(z_i)-U_\beta^j(y)}<\delta'$ for all $y\in(z_i,z_i+\myepsilon_i)$, $0\le j \le n-i-1$.
	
	\smallskip
	(3) $\Absbig{U_\beta^j(y)}<\delta'$ for all $y\in(z_i,z_i+\myepsilon_i)$, $n-i \le j\le n-1$.
	
	\smallskip
	Indeed, by Lemma~\ref{l_continuity_T_U_on_x}~(i), $\lim_{y\nearrow x} U_\beta^j(y) = U_\beta^j(x)^-$ for each $x\in (0,1]$ and each $j \in \N$. Thus, there exists $\myepsilon_{i,1}>0$ satisfying property~(1). 
	
	Now define $A_i\=\emptyset$ when $i=n-1$ 
	and $A_i\= \bigl\{ z_i,\, \dots,\, U_\beta^{n-i-2}(z_i) \bigr\}$ when $i<n-1$. Note that for every $0\leq i\leq n-1$, $U_\beta = T_\beta$ in a neighbourhood of $A_i$ (see Remark~\ref{r_after_def_beta_expansions}) since $A_i\cap D_\beta=\emptyset$, so
	if $0 \le j\le n-i-1$ then
	by Lemma~\ref{l_continuity_T_U_on_x}~(i), we have $\lim_{y\searrow z_i} U_\beta^{j}(y) = U_\beta^{j}(z_i)^+$.
	Hence, there exists $\myepsilon_{i,2}>0$ satisfying property~(2).

	Since $U_\beta^{n-1}(1) \in D_\beta$, by (\ref{e_def_U_beta}), we get that $\lim_{y\searrow z_i} U_\beta^{j}(y) = 0$ for each $n-i \le j\le n-1$. Thus, there exists $\myepsilon_{i,3}>0$ satisfying property~(3). 
	
	Defining $\myepsilon_i \= \min \{ \myepsilon_{i,1}, \, \myepsilon_{i,2}, \, \myepsilon_{i,3}\}$,
	we see that $\myepsilon_i$ satisfies properties~(1), (2), and~(3).

	Now define $\myepsilon'\=\min\{\myepsilon_i : 0\le i \le n-1 \}$,
	and construct $\underline{y}$ 
	as in Case~1, except that $\myepsilon$ and $\delta$
	are replaced, respectively, by $\myepsilon'$ and $\delta'$. Then 
	(\ref{eq 3.7}) holds immediately, while for each bad point $y_i$, $i\in \N_0$, we have $d\bigl(U_\beta^i (a ),\cO^*\bigr)\ge\myepsilon' \ge \myepsilon' d\bigl(U_\beta^i (a ),y_i\bigr)$, and for each good point $y_i$, by (\ref{eq_delta_lemma_locking_simple}) and properties~(1), (2), and (3), either $d\bigl(U_\beta^i (a ),\cO^*\bigr)= d\bigl(U_\beta^i (a ),y_i\bigr)$ or $d\bigl(U_\beta^i (a ),\cO^*\bigr)\ge\delta' \ge \delta' d\bigl(U_\beta^i (a ),y_i\bigr)$. Hence (\ref{eq 3.8}) holds if we take $C_\mu \= \min \{ (\myepsilon')^{-\alpha}, \, (\delta')^{-\alpha}, \, 1 \}$, so Claim is proved for Case~2.
\end{proof}

Having established Lemma~\ref{3.3}, we can now prove the following Theorem~\ref{3.1}.

\begin{theorem} \label{3.1}
	If $\beta>1$ and $\alpha\in (0,1]$, then the set $\Lock^\alpha(U_\beta)$ is an open and dense subset of $\sP^\alpha(U_\beta)$.
\end{theorem} 
\begin{proof}
	If $\phi\in\sP^\alpha(U_\beta)$, choose $\mu\in\Mmax (U_\beta, \phi )$ supported on a periodic orbit $\cO^*$ of $U_\beta$, and for each $t>0$ define
	\begin{equation*}
		\phi_t\=\phi-t d (\cdot,\cO^* ) ^\alpha \in \Holder{\alpha}(I).
	\end{equation*}
	Clearly $\phi_t$ belongs to $\Holder{\alpha}(I)$, and converges to $\phi$ as $t\to 0$.
	By Lemma~\ref{3.3}, if $t >0$, $\nu\in \MMM (I, U_\beta) \smallsetminus\{\mu\}$, and $\psi\in \Holder{\alpha} (I )$ with $\Hnorm{\alpha,I}{\psi} <t/C_\mu$, then
	\begin{equation*}\label{eq 3.12}
		\begin{aligned}
			\int_I\! (\phi_t+\psi )\, \mathrm{d}\nu
            &=\int_I \! \phi \, \mathrm{d}\nu+\int_I \! \psi\, \mathrm{d} \nu-t\int_I\! d (\cdot,\cO^* ) ^\alpha\,\mathrm{d}\nu\\
			&\le\int_I \! \phi \, \mathrm{d}\mu+\int_I \! \psi \, \mathrm{d}\mu+ (C_\mu \Hseminorm{\alpha, I}{\psi}-t )\int_I \!  d (\cdot,\cO^*)  ^\alpha\, \mathrm{d}\nu \\
			&=\int_I\! (\phi_t+\psi )\, \mathrm{d}\mu+ (C_\mu \Hseminorm{\alpha,I}{\psi}-t )\int_I \!  d (\cdot,\cO^* ) ^\alpha\, \mathrm{d}\nu \\
			&<\int_I \! (\phi_t+\psi )\, \mathrm{d}\mu,  
		\end{aligned}
	\end{equation*}
	so $\Mmax \bigl(U_\beta, \phi_t+\psi \bigr)=\{\mu\}$,
	and consequently, $\phi_t\in\Lock^\alpha(U_\beta)$. Hence, $\Lock^\alpha(U_\beta)$ is dense in $\sP^\alpha(U_\beta)$. 
	Moreover, from their definitions,
	$\Lock^\alpha(U_\beta)$ is an open subset of $\sP^\alpha(U_\beta)$, therefore the theorem is established.
\end{proof}
	
\begin{rem}\label{r_counter_example_simple_beta}
As mentioned in Section~\ref{s:introduction}, 
for certain parameters $\beta$
there exist continuous functions $\phi$ without a $(T_\beta,\phi)$-maximizing measure
(e.g.~for $\beta=2$, the function $\phi(x)\=x$ has no maximizing measure; indeed, there is 
an open neighbourhood of $\phi$ in 
$\Holder{1}(I)$  consisting of functions with no maximizing measure);
this absence of a maximizing measure is
due to the lack of compactness of $\MMM(I, T_\beta)$.
More generally, for $\beta$ a simple beta-number and $\alpha\in(0,1]$, the above proof
of Theorem~\ref{3.1} can be used to show that there is a nonempty open subset of $\Holder{\alpha}(I)$ consisting of functions with no maximizing measure.
Specifically,
the function
$\phi \= - d \bigl(\cdot, \cO_\beta^*(1) \bigr)^\alpha$ has the locking property (with respect to $U_\beta$) in $\Holder{\alpha}(I)$, with the unique $(U_\beta,\phi)$-maximizing measure $\mu_{\cO_\beta^*(1)}$; since $\mu_{\cO_\beta^*(1)}$ is not $T_\beta$-invariant, but
$\mpe(T_\beta, \phi) = \mpe (U_\beta, \phi)$
by Proposition~\ref{mpe=}~(ii), and $\MMM(I, T_\beta) \subseteq \MMM(I,U_\beta)$ by Proposition~\ref{p_relation_T_beta_and_wt_T_beta}~(iii), any function $\psi$ sufficiently close to $\phi$ has no $(T_\beta,\psi)$-maximizing measure.
The phenomenon of absence of maximizing measures when $\beta$ is a simple beta-number  means that for $T_\beta$ (as distinct from $U_\beta$),
the TPO property does not hold in the whole of $\Holder{\alpha}(I)$, for the straightforward reason that
\emph{optimization itself} is not a typical property in $\Holder{\alpha}(I)$. 
\end{rem}

We can now prove a structural theorem for emergent parameters $\beta$:

	\begin{theorem}[Structural theorem for emergent parameters]\label{t_TPO_thm_emergent}
If $\beta>1$  is emergent and $\alpha \in (0,1]$, then
$\Holder{\alpha}(I)$ is equal to the union of the critical set $\Crit^\alpha(\beta)$ and the closure of the open set
$\Lock^\alpha(U_\beta)$.
	\end{theorem}
\begin{proof}
	Fix an arbitrary $\beta>1$ that is emergent, and $\alpha\in (0,1]$. 
    We wish to show that $\Lock^\alpha(U_\beta)\smallsetminus \Crit^\alpha(\beta)$ is dense in
	$\Holder{\alpha}(I)\smallsetminus \Crit^\alpha(\beta)$.
	By Theorem~\ref{3.1} and the fact that $\Crit^\alpha(\beta)$ is closed (see Lemma~\ref{e_alpha_alternative_defn_lemma}), it suffices to prove that $\sP^\alpha(U_\beta)\smallsetminus \Crit^\alpha(\beta)$ is dense in 
	$\Holder{\alpha}(I)\smallsetminus \Crit^\alpha(\beta)$.
	The set $\RLS^\alpha(\beta) \smallsetminus \Crit^\alpha(\beta)$ is a subset of $\sP^\alpha(U_\beta) \smallsetminus \Crit^\alpha(\beta)$
	by Corollary~\ref{u_alpha_beta_subset_p_alpha_beta},
	and 
	is dense in $\Holder{\alpha}(I)\smallsetminus \Crit^\alpha(\beta)$,
	by Corollary~\ref{p_density_local_locking_corollary}. Therefore,
	$\sP^\alpha(U_\beta)\smallsetminus \Crit^\alpha(\beta)$ is itself dense in 
	$\Holder{\alpha}(I)\smallsetminus \Crit^\alpha(\beta)$,
	as required.
\end{proof}

\subsection{Proof of Individual TPO theorems}\label{subsec:proofofTPOtheorems}

In this subsection, we prove the individual typical periodic optimization theorems stated in Section
\ref{s:introduction}, namely for generic $\beta$
(Theorem~\ref{typical_beta}), for Lebesgue almost every $\beta$ (Theorem~\ref{almost_every_beta}),
whenever $\beta$ is a beta-number (Theorem~\ref{t_TPO_thm_beta_number}),
and whenever $\beta$ is non-emergent (Theorem~\ref{t_TPO_thm_non_emergent}).

Having considered emergent parameters $\beta$ in Subsection~\ref{subsec_proof_of_emergent}, we begin with a number of results about non-emergent parameters:

\begin{lemma}\label{l_Q_gamma_equal_Q_non_emergent}
	Suppose $\beta >1$ is non-emergent, $\alpha \in (0,1]$, and $\phi \in \Holder{\alpha}(I)$. Then there exists $\beta' \in (1,\beta)$ such that $\mpe_{\beta,\beta'}(\phi) = \mpe ( U_\beta , \phi )$.
\end{lemma}

\begin{proof}
	By Proposition~\ref{p_coding_mpe_relation}~(iii) there exists $\mu\in\Mmax ( U_\beta, \phi )$. Let us denote $\cK \= \supp\mu$.

	If $0 \in \cK$, then $\tphi^+(0) = 0=Q\bigl(U_\beta,\tphi^+\bigr)$ or $\tphi^-(0)=0=Q\bigl(U_\beta,\tphi^-\bigr)$ by Theorems~\ref{l_subordination}~(i) and~\ref{mane}~(ii),
	so 
	$\delta_0\in \MMM_{\max}(U_\beta,\phi)$ (by (\ref{e_def_tphi^-}), (\ref{e_def_tphi^+}), and  Lemma~\ref{normalised_cohomologous}~(i)),
	and therefore $\mpe_{\beta,\gamma}(\phi) = \mpe ( U_\beta , \phi )$ for every $\gamma\in(1,\beta)$.
	
	If $0 \notin \cK$ then 
    $U_\beta(\cK)=\cK$ by Lemma~\ref{support_invariant_set}.  
	Let us assume, for a contradiction, that the result is false, i.e.,~that 
	\begin{equation}\label{assume_false_gamma}
		\mpe_{\beta,\gamma}(\phi) < \mpe ( U_\beta , \phi )\ \text{ for all }\gamma\in(1,\beta).
	\end{equation}
	Lemma~\ref{l_case_two_one_optimal} then implies that $1\in\cK$.
	Since $\beta$ is non-emergent, Proposition~\ref{cE} implies that there exists $\gamma\in (1,\beta)$ such that
	$\overline{\cO_\beta^*(1)} \cap H_\beta^\gamma$ is not a subset of $Z_\beta$,
where we recall from (\ref{e_def_Z_beta}) that $Z_\beta=\bigl\{ x\in I:\pi_\beta(x)\neq \pi_\beta^*(x) \bigr\}$.
	But $1\in\cK$, so  $\overline{\cO_\beta^*(1)} \subseteq \cK$, and therefore
	$\cK \cap H_\beta^\gamma$ is not a subset of $Z_\beta$,
	in other words, there exists $x \in \bigl(\cK \cap H_\beta^\gamma\bigr) \smallsetminus Z_\beta$. 
	
	By Lemma~\ref{l_properties_z_beta}~(iii), the orbit $\cO_\beta(x)$ is equal to $\cO_\beta^*(x)$, and it
	is contained in $\cK \cap H_\beta^\gamma$ since $U_\beta(\cK)=\cK$ and $T_\beta\bigl(H_\beta^\gamma\bigr)\subseteq H_\beta^\gamma$. By Theorem~\ref{l_subordination}~(i), $\tphi^+|_{\cO_\beta(x)} \equiv 0$ or $\tphi^-|_{\cO_\beta(x)} \equiv 0$.
	In other words, $\cO_\beta(x)$ is contained in either $\bigl(\tphi^{-}\bigr)^{-1}(0)$ or $\bigl(\tphi^{+}\bigr)^{-1}(0)$. Since $\tphi^{-} \le 0$ and $\tphi^{+}\le 0$ (by Theorem~\ref{mane}~(ii)), Lemma~\ref{normalised_cohomologous}~(iii) implies that 
	$\cO_\beta(x)=\cO_\beta^*(x)$ is a $(T_\beta,\phi)$-maximizing orbit and a $(U_\beta,\phi)$-maximizing orbit (by Proposition~\ref{mpe=}~(ii)),
	so in particular, 
	\begin{equation}\label{time_Q}
		\lim_{n \to +\infty} \frac{1}{n} S_n \phi(x) = \mpe ( U_\beta, \phi ). 
	\end{equation}
	
	Now $\cO_\beta(x)\subseteq H_\beta^\gamma$ and $T_\beta|_{H_\beta^\gamma}$ is continuous (see Proposition~\ref{l_property_H_gamma}~(i)), so by \cite[Proposition~2.2]{Je19}, the corresponding time average is bounded above by the ergodic supremum, in other words,
	\begin{equation}\label{time_average_bounded_above}
		\mpe_{\beta,\gamma}(\phi) \ge \lim_{n \to +\infty} \frac{1}{n} S_n \phi(x).
	\end{equation}
	Now as an immediate consequence of (\ref{e_def_beta_gamma}),
	we have $\mpe ( U_\beta, \phi ) \ge \mpe_{\beta,\gamma}(\phi)$.
	So combining this inequality with (\ref{time_Q}) and (\ref{time_average_bounded_above}) gives
	$\mpe ( U_\beta, \phi ) = \mpe_{\beta,\gamma}(\phi)$, which gives the required
	contradiction to (\ref{assume_false_gamma}). The lemma follows.
\end{proof}

\begin{cor}\label{l_Q_gamma_equal_Q_non_emergent_corollary}
	Suppose $\beta >1$ is non-emergent, $\alpha \in (0,1]$, and $\phi \in \Holder{\alpha}(I)$. Then there exists $\beta' \in (1,\beta)$ such that
	\begin{equation}\label{beta_gamma_theta}
		\mpe_{\beta,\gamma}(\phi) = \mpe \bigl( U_\beta , \phi \bigr)\quad\text{ for all }\gamma\in[\beta',\beta).
	\end{equation}
\end{cor}
\begin{proof}
	Let $\beta'$ be as in Lemma~\ref{l_Q_gamma_equal_Q_non_emergent}. If $\gamma\in[\beta',\beta)$ then 
	\begin{equation}\label{beta_gamma_theta_inequalities}
		Q_{\beta,\beta'}(\phi) \le Q_{\beta,\gamma}(\phi) \le Q(U_\beta,\phi),
	\end{equation}
	an immediate consequence of (\ref{e_def_beta_gamma}), since $H_\beta^{\beta'}\subseteq H_\beta^\gamma\subseteq I$.
	But $Q(U_\beta,\phi)=Q_{\beta,\beta'}(\phi)$,
	by Lemma~\ref{l_Q_gamma_equal_Q_non_emergent},
	so (\ref{beta_gamma_theta_inequalities}) implies the required equality (\ref{beta_gamma_theta}).
\end{proof}

The proof of the following result is similar to that of Corollary~\ref{u_alpha_beta_subset_p_alpha_beta}:

\begin{cor}\label{u_alpha_beta_subset_p_alpha_beta_non_emergent}
	If $\beta>1$ is non-emergent, then $\RLS^\alpha(\beta)\subseteq \sP^\alpha(U_\beta)$.
\end{cor}
\begin{proof}
	Suppose $\phi\in \RLS^\alpha(\beta)$. By Corollary~\ref{l_Q_gamma_equal_Q_non_emergent_corollary},
	there exists $\beta' \in (1,\beta)$ such that
	$
	\mpe_{\beta,\gamma}(\phi) = \mpe \bigl( U_\beta , \phi \bigr)
	$
	for all $\gamma\in[\beta',\beta)$, so in particular there is a simple beta-number $\gamma_0\in (1,\beta)$ such that
	\begin{equation}\label{gamma_N_beta_phi}
		\mpe_{\beta,\gamma_0}(\phi) = \mpe ( U_\beta , \phi ).
	\end{equation}
	The fact that $\phi\in \RLS^\alpha(\beta)$ implies that
	$
	\phi|_{H_\beta^{\gamma_0}}\in \Lock^\alpha\bigl(T_\beta|_{H_\beta^{\gamma_0}}\bigr)
	\subseteq
	\sP^\alpha \bigl(T_\beta|_{H_\beta^{\gamma_0}}\bigr) $.
	So there exists a periodic measure $\mu\in \MMM(I,T_\beta)\subseteq \MMM(I,U_\beta)$ (see Proposition~\ref{p_relation_T_beta_and_wt_T_beta}~(iii))  such that
	\begin{equation}\label{integral_phi_mu_equal}
		\int \!\phi\, \mathrm{d}\mu = \mpe_{\beta,\gamma_0}(\phi) .
	\end{equation}
	Thus from (\ref{gamma_N_beta_phi}) and (\ref{integral_phi_mu_equal}) we see that the periodic measure $\mu$
	satisfies $\int\! \phi\, \mathrm{d}\mu= Q(U_\beta,\phi)$.
	Therefore $\phi\in \sP^\alpha(U_\beta)$, as required.
\end{proof}

We are now in a position to prove our Individual TPO theorems.
We establish the following slightly stronger version of Theorem~\ref{t_TPO_thm_non_emergent} (which in particular implies Theorem~\ref{t_TPO_thm_non_emergent}):

\setcounter{thml}{9}

\begin{thml}[Individual TPO for non-emergent parameters]\label{t_TPO_thm_non_emergent'}
    Fix $\alpha\in (0,1]$. If $\beta>1$ is non-emergent, then both $\Lock^\alpha(T_\beta)$ and $\Lock^\alpha(U_\beta)$ are open and dense subsets of $\Holder{\alpha}(I)$.
\end{thml}

\begin{proof}
	Let $\beta>1$ be non-emergent. By Corollary~\ref{c_relation_emergent&betanumbers}, $\beta$ cannot be a simple beta-number. So
    Theorem~\ref{p_relation_T_beta_and_wt_T_beta}~(iv) implies that $\Per(T_\beta) = \Per(U_\beta)$ and $\cM(I,T_\beta) = \cM(I,U_\beta)$. This implies that $\Lock^\alpha(T_\beta) = \Lock^\alpha(U_\beta)$. Thus it suffices to show that $\Lock^\alpha(U_\beta)$ is an open dense subset of $\Holder{\alpha}(I)$.
    
    Now $\Lock^\alpha(U_\beta)$ is by definition an open subset of $\Holder{\alpha}(I)$,
	and Theorem~\ref{3.1}
	asserts that  $\Lock^\alpha(U_\beta)$ is dense in $\sP^\alpha(U_\beta)$, so it suffices
	to prove that $\sP^\alpha(U_\beta)$ is dense in $\Holder{\alpha}(I)$.
	Since $\beta>1$ is non-emergent, Corollary~\ref{u_alpha_beta_subset_p_alpha_beta_non_emergent}
	gives
	$\RLS^\alpha(\beta)\subseteq \sP^\alpha(U_\beta)$,
	and $\RLS^\alpha(\beta)$ is dense in $\Holder{\alpha}(I)$
	by Theorem~\ref{p_density_local_locking}. Therefore, it follows
	that $\sP^\alpha(U_\beta)$ is dense in $\Holder{\alpha}(I)$, as required.
\end{proof}

The following is a slightly stronger version of Theorem~\ref{typical_beta}, which in particular implies Theorem~\ref{typical_beta}.

\setcounter{thml}{3}

\begin{thml}[Individual TPO for generic parameters $\beta$]\label{typical_beta'}
    Fix $\alpha\in (0,1]$. For a residual set of values $\beta>1$, $\Lock^\alpha(T_\beta)$ is an open and dense subset of $\Holder{\alpha}(I)$.
\end{thml}

\begin{proof}
By Corollary~\ref{emergent_small_set}, the set
$(1,+\infty)\smallsetminus\emergent$ of non-emergent parameters is a residual subset of $(1,+\infty)$.
By Theorem~\ref{t_TPO_thm_non_emergent'}, if
$\beta\in (1,+\infty)\smallsetminus\emergent$ then $\Lock^\alpha(T_\beta)$ is an open and dense subset of $\Holder{\alpha}(I)$, so the result follows. 
\end{proof}

The following is a slightly stronger version of Theorem~\ref{almost_every_beta}, which in particular implies Theorem~\ref{almost_every_beta}.

\begin{thml}[Individual TPO for almost every parameter $\beta$]\label{almost_every_beta'}
    Fix $\alpha\in (0,1]$. For Lebesgue almost every $\beta>1$, $\Lock^\alpha(T_\beta)$ is an open and dense subset of $\Holder{\alpha}(I)$.
\end{thml}

\begin{proof}
By Corollary~\ref{emergent_small_set}, the set
$(1,+\infty)\smallsetminus\emergent$ of non-emergent parameters has full Lebesgue measure.
By Theorem~\ref{t_TPO_thm_non_emergent'}, if
$\beta\in (1,+\infty)\smallsetminus\emergent$ then $\Lock^\alpha(T_\beta)$ is an open and dense subset of $\Holder{\alpha}(I)$, so the result follows. 
\end{proof}

We now prove a slightly stronger version of Theorem~\ref{t_TPO_thm_beta_number}, which in particular implies Theorem~\ref{t_TPO_thm_beta_number}.

\setcounter{thml}{8}

\begin{thml}[Individual TPO for beta-numbers]\label{t_TPO_thm_beta_number'}
    Fix $\alpha\in (0,1]$. If $\beta>1$ is a beta-number, then $\Lock^\alpha(U_\beta)$ is an open and dense subset of $\Holder{\alpha}(I)$.
\end{thml}

\begin{proof}
	First assume that $\beta$ is a non-simple beta-number. By Corollary~\ref{c_relation_emergent&betanumbers}, $\beta$ is not emergent, so the result follows from Theorem~\ref{t_TPO_thm_non_emergent'}.
	
	Now assume that $\beta$ is a simple beta-number. By Corollary~\ref{c_relation_emergent&betanumbers}, $\beta$ is emergent.
	Theorem~\ref{t_TPO_thm_emergent} gives that
	$\Holder{\alpha}(I)$ is equal to the union of $\Crit^\alpha(\beta)$ and the closure of $\Lock^\alpha(U_\beta)$,
	so it suffices to show that $\Crit^\alpha(\beta)$ is a subset of the closure of $\Lock^\alpha(U_\beta)$ (as $\Lock^\alpha(U_\beta)$ is by definition an open subset of $\Holder{\alpha}(I)$).
	If $\phi\in \Crit^\alpha(\beta)$ then
	from (\ref{e_alpha_alternative_defn_eqn}) we see that
	the periodic measure supported by
	$\cO_\beta^*(1)$ is $(U_\beta,\phi)$-maximizing, so
	$\phi \in \sP^\alpha (U_\beta)$, and therefore 
	$\phi$ belongs to the closure of $\Lock^\alpha(U_\beta)$ by Theorem~\ref{3.1}.
	Theorem~\ref{t_TPO_thm_beta_number'} follows.
\end{proof}

\section{Joint TPO: beta-transformations}\label{JTPO_beta_section}

In this section, we prove Theorem~\ref{jtpo_tbeta},
the Joint TPO theorem for beta-transformations, and then deduce Theorem~\ref{GPO_open_dense} (Individual TPO for generic potentials).
In fact Theorem~\ref{jtpo_tbeta} will follow from
a stronger
Theorem~\ref{t.product.typical.periodic},
that in particular establishes the joint typical periodic optimization property for both beta-transformations and upper beta-transformations. 
The proof of Theorem~\ref{t.product.typical.periodic}
comprises two steps.
The first step, consisting of the key joint perturbation result (Theorem~\ref{l_8_main_lemma}), is to prove that for any 
\emph{non-simple beta-number} $\beta$, any $\phi\in \Lock^\alpha(U_\beta)$ that is uniquely $U_\beta$-maximized on a periodic orbit $\cO_\beta$, and any $\gamma<\beta$ sufficiently close to $\beta$, we can perform a small perturbation of $\phi$ so as to render it uniquely $U_\gamma$-maximized by the $U_\gamma$-periodic orbit $(h_\gamma \circ \pi_\beta)(\cO_\beta)$. In this first step, the beta-transformations Ma\~n\'e lemma (Theorem~\ref{mane}) is key to making the perturbation, 
and the subsequent analysis is
inspired by ideas in \cite[Section~4]{Boc19}
(itself based on the preprint version of
\cite{HLMXZ25}), together with some more careful estimates exploiting various specific characteristics of the family of beta-transformations. 
The choice of $\beta$ as a non-simple beta-number is an essential feature of this step, since 
on the one hand the functions $u_{\beta,\phi}^\pm$ in the Ma\~n\'e lemma (Theorem~\ref{mane}) are required to be piecewise $\alpha$-H\"older (which is the case if $\beta$ is a beta-number, cf.~Remark~\ref{post_mane_holder_remark}(i)), and on the other hand the perturbative analysis 
(specifically, a shadowing result, Corollary~\ref{l8.shadow})
requires that the orbit $\cO_\beta$ does not contain the point $1$, thereby forcing $\beta$ to be non-simple. The second step
consists of combining the result from the first step with the fact that non-simple beta-numbers are dense in $(1,+\infty)$, and invoking Theorem~\ref{t_TPO_thm_beta_number'} (the strong version of Individual TPO for beta-numbers), in order to
deduce Theorem~\ref{t.product.typical.periodic}.

\subsection{Shadowing for beta-transformations}

As a preliminary step, we first give a simple bound on the distance between $h_\beta(\underline{a})$ and $h_\gamma(\underline{a})$ for $\beta> \gamma>1$ and $\underline{a}\in X_\gamma$. 

\begin{lemma}\label{l8.dist.h}
	If $1<\gamma<\beta$ then for each $\underline{a}\in X_\gamma$,  
	\begin{equation}\label{e8.dist.h}
		\abs{h_\gamma(\underline{a}) - h_\beta(\underline{a})} \le \frac{(\beta-\gamma) \gamma^2}{\beta(\gamma-1)^2}.
	\end{equation}
\end{lemma}

\begin{proof}
	Writing $\underline{a} = a_1 a_2 \dots$, since $\underline{a}\in X_\gamma$ then $a_n \le \gamma$ for each $n\in \N$, so from (\ref{hbeta}), 
	$		\abs{h_\gamma(\underline{a}) - h_\beta(\underline{a})} = \sum_{n=1}^{+\infty} a_n \frac{\beta^n-\gamma^n} {\beta^n \gamma^n} \le \sum_{n=1}^{+\infty} \frac{\beta^n-\gamma^n}{\beta^n \gamma^{n-1}}$,
	and therefore	
	$		\abs{h_\gamma(\underline{a}) - h_\beta(\underline{a})}  \le (\beta-\gamma) \sum_{n=1}^{+\infty} \frac{\sum_{i=0}^{n-1} \beta^i \gamma^{n-1-i} }{\beta^n \gamma^{n-1}}  \le (\beta-\gamma) \sum_{n=1}^{+\infty} \frac{n}{\beta \gamma^{n-1}} 
		= \frac{(\beta-\gamma) \gamma^2}{\beta(\gamma-1)^2}$.
\end{proof}

From Lemma~\ref{l8.dist.h}, we are able to estimate the distance between a $U_\beta$-orbit $\cO_\beta$ and its image $(h_\gamma \circ \pi_\beta)(\cO_\beta)$; this rather explicit version of a shadowing lemma will simplify our subsequent proof of Theorem~\ref{l_8_main_lemma}, where it is used systematically.

\begin{cor}\label{l8.shadow}
	Fix $\beta>1$, and let 
    $\cO_\beta\in\Per(U_\beta)$ be such that $1\notin \cO_\beta$. Then there exists $c \in (0,\beta-1)$ such that 
    if $\gamma\in (\beta-c,\beta)$,
    then
    $h_\gamma \circ \pi_\beta$ is well defined 
    on $\cO_\beta$, and:
    \begin{enumerate}[label=\rm{(\roman*)}]
	\smallskip
	\item
     $\cO_\gamma\=(h_\gamma \circ \pi_\beta)(\cO_\beta)$ is a $U_\gamma$-periodic orbit with $\card \cO_\gamma = \card \cO_\beta$.
    \smallskip
    \item
    There exists $M>0$ such that for each $x\in \cO_\beta$,  
	\begin{equation}\label{e8.shadow}
		\abs{(h_\gamma\circ \pi_\beta)(x) - x} \le M(\beta-\gamma).
	\end{equation}
	\item
    If $\cO_\beta\neq \{0\}$ then there exists $s>0$ such that for each $x\in \cO_\beta$,   
	\begin{equation}\label{e8.shadow2}
		s(\beta-\gamma) \le \abs{(h_\gamma \circ \pi_\beta)(x) - x}.
	\end{equation} 
    \end{enumerate}
\end{cor}

\begin{proof}
Since $1\notin \cO_\beta$,  Proposition~\ref{p_relation_T_beta_and_wt_T_beta}~(ii) implies that  $\cO_\beta$ is also $T_\beta$-periodic. Applying Lemma~\ref{H beta gamma} with $\cK= \cO_\beta$, there exists $c>0$ such that $\cO_\beta \subseteq H_\beta^\gamma$ for each $\gamma\in (\beta-c, \beta)$, and without loss of generality $c$ may be chosen to be strictly smaller than $\beta-1$.
For each $\gamma\in (\beta-c, \beta)$, Lemma~\ref{l_bi_lipschitz_S_gamma_H_gamma} implies that $\pi_\beta\bigl(H_\beta^\gamma\bigr) \subseteq \cS_\gamma \subseteq X_\gamma$,
and since $h_\gamma$ is defined on $X_\gamma$
(see (\ref{hbeta})),
it follows that $h_\gamma \circ \pi_\beta$
is well defined on $H_\beta^\gamma$, hence in particular on $\cO_\beta$, as required.

\smallskip

(i) By Proposition~\ref{p_relation_of_coding}~(iii) and (vi), $\pi_\beta(\cO_\beta)$ is a $\sigma$-periodic orbit with $\card  \pi_\beta(\cO_\beta)  = \card \cO_\beta$. So by Proposition~\ref{p_coding_mpe_relation}~(iv), $\cO_\gamma = (h_\gamma \circ \pi_\beta)(\cO_\beta)$ is a $U_\gamma$-periodic orbit with $\card \cO_\gamma = \card \pi_\beta(\cO_\beta) = \card \cO_\beta$. 

(ii) Now $c<\beta-1$, so 
for $\gamma\in(\beta-c, \beta)$ the expression
$\frac{\gamma^2}{\beta (\gamma-1)^2}$ has finite upper bound $M$ (equal to $\frac{(\beta-c)^2}{\beta(\beta-c-1)^2}$).  Moreover, applying Proposition~\ref{p_relation_of_coding}~(iv) and Lemma~\ref{l8.dist.h}, for each $x\in \cO_\beta$, we have 
\begin{equation} \label{e:Pf_l8.shadow}
\begin{split}
    |(h_\gamma \circ \pi_\beta)(x) - x| 
    &= |(h_\gamma \circ \pi_\beta)(x) - (h_\beta \circ \pi_\beta)(x)| \\
    &\le (\beta-\gamma) \gamma^2 \beta^{-1} (\gamma-1)^{-2} \\
    &\le M(\beta-\gamma).
\end{split}
\end{equation}
(iii) Since $\cO_\beta\neq\{0\}$, then $(0)^\infty \notin \pi_\beta(\cO_\beta)$. Since $\pi_\beta(\cO_\beta)$ is finite, there exists $N\in \N$ such that if $x\in \cO_\beta$ then $(0)^N 1 (0)^\infty \prec \pi_\beta(x)$. Let $x\in \cO_\beta$, and write $\pi_\beta(x) = a_1 a_2 \dots$, and let $n_x \le N+1$ be the smallest integer such that $a_{n_x} \neq 0$, so that by (\ref{e:Pf_l8.shadow}) and (\ref{hbeta}), 
	\begin{align*}
		\abs{(h_\gamma \circ \pi_\beta)(x) - x} 
         & = \abs{(h_\gamma \circ \pi_\beta)(x) - (h_\beta \circ \pi_\beta)(x)}\\
		   & = \sum_{i=n_x}^{+\infty} a_i \frac{\beta^i- \gamma^i}{\beta^i \gamma^i} 
         \ge \frac{\beta^{n_x} - \gamma^{n_x}}{ \beta^{n_x}  \gamma^{n_x}} 
         \ge \frac{\beta-\gamma}{ \beta^{n_x} \gamma} 
         \ge \frac{\beta-\gamma}{\beta^{N+2}} .
	\end{align*}
Thus (\ref{e8.shadow2}) holds with $s\= \beta^{-N-2}$.
\end{proof}

We will need the following expression for the ergodic supremum (the analogue for continuous maps is well known, see e.g.~\cite[Proposition~2.2]{Je19}):

\begin{lemma}\label{l8.Q=supliminf}
	Given any $\beta>1$ and any $\phi \in C(I)$, 
    \begin{equation*}
    \mpe(U_\beta, \phi) 
    = \sup_{x\in I} \liminf_{n\to +\infty} \frac{1}{n} S_n^{U_\beta}  \phi(x).
    \end{equation*} 
\end{lemma}

\begin{proof}
	Since $\sigma\: X_\beta\to X_\beta$ is a continuous map on a compact metric space, and $\phi \circ h_\beta$ is continuous,  \cite[Proposition~2.2]{Je19} gives 
	\begin{equation}\label{e8.Q=code}
		\mpe \bigl(\sigma|_{X_\beta}, \phi \circ h_\beta \bigr) 
        = \sup_{\underline{a}\in X_\beta} \liminf_{n\to +\infty} \frac{1}{n} S_n^{\sigma} (\phi \circ h_\beta) (\underline{a}).
	\end{equation}
	Applying Lemma~\ref{l_properties_z_beta}~(ii) with $W=I$ gives $X_\beta = \pi_\beta(Z_\beta) \cup \pi^*_\beta(I)$. By Proposition~\ref{p_relation_of_coding}~(ii) and (\ref{e_def_Z_beta}), for each $\underline{b} \in \pi_\beta(Z_\beta)$ there exists $m\in \N$ such that $\sigma^m(\underline{b}) = (0)^\infty$. Thus, noting that $(0)^\infty = \pi_\beta^*(0)$, we obtain that 
    $
    \liminf\limits_{n\to +\infty} \frac{1}{n} S_n^\sigma (\phi \circ h_\beta)(\underline{b}) 
    = \liminf\limits_{n\to +\infty} \frac{1}{n} S_n^{\sigma} (\phi \circ h_\beta) ((0)^\infty)  
    \le \sup\limits_{\underline{a}\in \pi^*_\beta(I)} \liminf\limits_{n\to +\infty} \frac{1}{n} S_n^{\sigma} (\phi \circ h_\beta) (\underline{a})
    $
    for each $\underline{b} \in \pi_\beta(Z_\beta)$. This, together with the fact that $X_\beta = \pi_\beta(Z_\beta) \cup \pi^*_\beta(I)$, implies that
	\begin{equation}\label{e8.Q=X}
		\sup_{\underline{a}\in X_\beta} \liminf_{n\to +\infty} \frac{1}{n} S_n^{\sigma} (\phi \circ h_\beta) (\underline{a}) = \sup_{\underline{a}\in \pi^*_\beta(I)} \liminf_{n\to +\infty} \frac{1}{n} S_n^{\sigma} (\phi \circ h_\beta) (\underline{a}).
	\end{equation}
	By Proposition~\ref{p_relation_of_coding}~(iv) and (v), 
	\begin{equation}\label{e8.Q=pi*}
    \begin{aligned}
      		\sup_{x\in I} \liminf_{n\to +\infty} \frac{1}{n} S_n^{U_\beta} \phi (x) 
        &= \sup_{x\in I} \liminf_{n\to +\infty} \frac{1}{n} S_n^{\sigma} (\phi \circ h_\beta) \bigl(\pi_\beta^*(x)\bigr) \\
        &= \sup_{\underline{a}\in \pi^*_\beta(I)} \liminf_{n\to +\infty} \frac{1}{n} S_n^{\sigma} (\phi \circ h_\beta) (\underline{a}) .  
    \end{aligned}
	\end{equation}
	So combining (\ref{e8.Q=code}), (\ref{e8.Q=X}), (\ref{e8.Q=pi*}), and using Proposition~\ref{p_coding_mpe_relation}~(iii), completes the proof.
\end{proof}

\subsection{Joint perturbation for beta-transformations}

The following key Theorem~\ref{l_8_main_lemma}, valid for \emph{non-simple beta-numbers}, is the first perturbative step (described at the start of this section) towards proving Theorem~\ref{t.product.typical.periodic}.
As noted previously, it has a similar character to the joint perturbation theorems for expanding maps
(Theorem~\ref{t.criticaltheorem})
and for Anosov diffeomorphisms
(Theorems~\ref{jointperturbationuniformhyp}
and \ref{anosovJTPO_C1}),
though here the analysis is more intricate, and the proof is lengthier.

\begin{theorem}[Joint Perturbation: beta-transformations]\label{l_8_main_lemma}
	Fix $\alpha\in (0,1]$, and suppose $\beta>1$ is a non-simple beta-number. 
    \begin{enumerate}[label=\rm{(\roman*)}]
		\smallskip
		\item
    If $\cO_\beta$ is a $T_\beta$-periodic orbit, then there exist $C_1, \, C_2 >0$ such that if $\gamma \in (\beta-C_2, \beta)$ and $\phi \in \Holder{\alpha}(I)$ with $\Mmax(T_\beta, \phi) = \{ \mu_{\cO_\beta} \}$, then the $T_\gamma$-periodic measure $\mu_{\cO_\gamma}$ supported by 
    $\cO_\gamma \= (h_\gamma \circ \pi_\beta)(\cO_\beta)$,
    is the unique $T_\gamma$-maximizing measure for the function 
\begin{equation}\label{perturbation_minus_2C1}
    \phi- 2C_1 \Hseminorm{\alpha}{\phi}  (\beta- \gamma)^{\alpha/2} d(\cdot, \cO_\gamma)^\alpha.
    \end{equation}
    \item
    If $\cO_\beta$ is a $U_\beta$-periodic orbit, then there exist $C_1, \, C_2 >0$ such that if $\gamma \in (\beta-C_2, \beta)$ and $\phi\in \Holder{\alpha}(I)$ with $\Mmax(U_\beta, \phi) = \{ \mu_{\cO_\beta} \}$, then the $U_\gamma$-periodic measure $\mu_{\cO_\gamma}$ supported by 
    $\cO_\gamma \= (h_\gamma \circ \pi_\beta)(\cO_\beta)$,
    is the unique $U_\gamma$-maximizing measure for the function defined by (\ref{perturbation_minus_2C1}).
    \end{enumerate}
\end{theorem}

\begin{proof} Fix a non-simple beta-number $\beta>1$.

(i)  First we show that part~(i) follows readily from part~(ii).
 Given a $T_\beta$-periodic orbit $\cO_\beta$,  Proposition~\ref{p_relation_T_beta_and_wt_T_beta}~(ii) implies that $\cO_\beta$ is also $U_\beta$-periodic.  
 Assuming that part~(ii) of this theorem has been proved, let $C_1, \, C_2 >0$ be as in (ii),  let $\gamma \in (\beta-C_2, \beta)$ and $\gamma' \in (\beta - C_2, \gamma)$, and recall that $\cO_\gamma = (h_\gamma \circ \pi_\beta)(\cO_\beta)$. By the definition of
 $h_\beta\: X_\beta\to I$ (see (\ref{hbeta})), $\max \{ x : x\in \cO_\gamma \} < \max \{ x : x\in \cO_{\gamma'} \} \le 1$, so $1\notin \cO_\gamma$. By (ii) and Proposition~\ref{p_relation_T_beta_and_wt_T_beta}~(ii), $\cO_\gamma$ is also a $T_\gamma$-periodic orbit, so $\mu_{\cO_\gamma} \in \cM(I,T_\gamma)$. By (ii) we know that $\mu_{\cO_\gamma}$ is the unique $U_\gamma$-maximizing measure for the function given by (\ref{perturbation_minus_2C1}), so the fact that $\cM(I,T_\gamma) \subseteq \cM(I, U_\gamma)$ (cf.~Proposition~\ref{p_relation_T_beta_and_wt_T_beta}~(iii)) implies that $\mu_{\cO_\gamma}$ is also the unique $T_\gamma$-maximizing measure for this function, as required.
    
 (ii)   To prove (ii), suppose that $\cO_\beta$ is a $U_\beta$-periodic orbit. Since $\beta$ is non-simple, the point $1$ is not $U_\beta$-periodic, and therefore not an element of $\cO_\beta$. 
	By Corollary~\ref{l8.shadow}~(ii) and (iii), it follows that there exist constants $M>0$ and $ c\in(0,\beta-1)$ such that if $\gamma\in(\beta-c,\beta)$ then (\ref{e8.shadow}) holds, and if 
	moreover $\cO_\beta\neq\{0\}$ then there also exists $s>0$ such that (\ref{e8.shadow2}) holds. Let us write $u\=u_{\beta,\phi}^-$ for the bounded left-continuous (cf.~Theorem~\ref{mane}~(i)) function defined by (\ref{e_def_u_phi_-}), and write $K_\beta \= K_{\alpha, \beta} = \frac{1}{\beta^\alpha -1}$ (cf.~(\ref{kalphabetadefn})). In the following proof we shall systematically exploit two inequalities. The first is that
	\begin{equation}\label{e8.psi<0}
		\psi_\beta \= \overline{\phi} + u - u \circ U_\beta \le 0, 
	\end{equation}
    where $\overline{\phi} \= \phi - \mpe(U_\beta, \phi)$,
    which is a consequence of (\ref{e_def_tphi^-}) and Theorem~\ref{mane}~(ii).
	The second is that
	\begin{equation}
		u(x) - u(y)  \le K_\beta \Hseminorm{\alpha}{\phi} \abs{x-y}^\alpha \label{e8.Hbound}
	\end{equation}
	holds if $x<y$ and $[x,y) \cap \cO_\beta^*(1) = \emptyset$, and also holds if $y<x$. 
	Note that if $x<y$ and $[x,y) \cap \cO_\beta^*(1) = \emptyset $ then (\ref{e8.Hbound}) follows from Theorem~\ref{mane}~(iii) and the left continuity of $u$, whereas if $y<x$ then (\ref{e8.Hbound}) follows from Lemma~\ref{l_Bousch_Op_preserve_space}, together with (\ref{e_calibrated_sub-action_exists}) and (\ref{e_def_u_phi_-}).
	
	Recall (cf.~Subsection~\ref{sct_Notation}) that if $F \subseteq I$ is a finite set, then $\Delta(F)$ denotes its minimum interpoint distance, i.e., $\Delta(F) = \min \{ \abs{x-y} : x, \, y \in F, \, x\neq y\}$ if $\card F\ge 2$ and $\Delta(F) = +\infty$ if $\card F=1$. 
	In the following, the finite set $F$ will be chosen as either a periodic orbit or a preperiodic orbit.
	
	We define the following notation for various constants\footnote{Note that the form of the presentation of $C_2$ in (\ref{d_8_C2}) is for ease of reference in subsequent calculations rather than for economy of notation, as certain elements of the set on the right-hand side are manifestly dominated by others.
    For future reference, we note that the bound $C_2\le c$ is needed so as to invoke Corollary~\ref{l8.shadow}, while $C_2\le \frac{1}{2}$ and $C_2\le \frac{1}{2}\Delta(\cO_\beta^*(1))$ are used in proving Claim~3, 
    the inequalities $C_2\le M^{-1}r$ and $C_2\le \beta r$ are used in proving Claim~1, and $C_2\le \frac{1}{2M} \Delta(\cO_\beta^*(1))$ is used in proving that $d(\cO_\gamma, \cO_\beta^*(1)) \ge s(\beta-\gamma)$ in Subcase~(ii) of Case~C.
    As for the constant $C_1$, the bound $C_1 \ge 1$ is assumed in order to simplify calculations, the second term in (\ref{d_8_C1}) is needed to check that $\rho \le r/\beta$ (cf.~(\ref{rho_r_beta})), the third term is used in deriving (\ref{e_8_Snpsi3}), and the fourth term in (\ref{d_8_C1'}) is used in deriving (\ref{e8.222}).} 
    that will be used in our perturbative arguments (recall that $D_\beta$ is the discontinuity set, defined in (\ref{e_D_beta})):
	\begin{align}
		p &\= \card \cO_\beta, \label{d_8_p}\\
		r &\= \begin{cases}
			\min \bigl\{ \frac{d(\cO_\beta, D_\beta)}{3}, \, \frac{\Delta(\cO_\beta)}{4} , \, \frac{d (\cO_\beta, \cO_\beta^*(1)) }{2} \bigr\} & \text{ if } \cO_\beta \cap \cO_\beta^*(1) =\emptyset, \\
			\min \bigl\{ \frac{d(\cO_\beta, D_\beta)}{3}, \,  \frac{\Delta(\cO_\beta) }{4} \bigr\} & \text{ if }  \cO_\beta \cap \cO_\beta^*(1) \neq \emptyset,
		\end{cases} \label{d_8_r}\\
		C_2 &\= \min \Bigl\{ c ,\, \frac{1}{2}, \, \frac{1}{2}\Delta \bigl(\cO_\beta^*(1) \bigr), \, \frac{1}{2M}\Delta \bigl(\cO_\beta^*(1) \bigr), \, M^{-1}r, \, \beta r  \Bigr\}, \label{d_8_C2}\\
		L_1 & \= 1+ \frac{1}{  (\beta-C_2)^\alpha -1} + 2 \beta^\alpha K_\beta, \label{d_8_L1}\\
		L_2 & \= 3K_\beta + 1+ M^\alpha, \label{d_8_L2}\\
		C_1 & \= \max \bigl\{1, \,  L_2 \cdot C_2^{\alpha/2} \beta^\alpha  r^{-\alpha}, \, r^{-\alpha} \beta^\alpha (p+1+L_1)L_2 \bigr\}. \label{d_8_C1}
	\end{align}   
	Note that $C_2 \le c < \beta-1$, so that $\beta-C_2>1$ and $\ln(\beta-C_2)>0$. If $\cO_\beta\neq\{0\}$, we define the additional constants
	\begin{align}
		L_3 &\= 1+\frac{1}{\ln(\beta-C_2)} \Bigl( \ln \beta + \frac{1}{\alpha} \ln L_2 - \ln s \Bigr), \label{d_8_L3}\\
		L_4 &\= \frac{s^\alpha}{\beta^\alpha}L_1 + pL_2+L_3L_2 + \frac{L_2 }{e \alpha \ln(\beta-C_2)}, \label{d_8_L4}
	\end{align}
	and make the additional assumption that
	$C_1\ge r^{-\alpha}\beta^\alpha (L_4+L_2)$, in other words, if $\cO_\beta\neq\{0\}$ we re-define $C_1$ by
	\begin{equation}\label{d_8_C1'}
		C_1  \= \max \bigl\{1, \,  L_2 \cdot C_2^{\alpha/2} \beta^\alpha r^{-\alpha}, \, r^{-\alpha} \beta^\alpha (p+1+L_1)L_2 , \, r^{-\alpha}\beta^\alpha (L_4+L_2)\bigr\} .
	\end{equation}
	Note that the constants $L_1$, $L_2$, $L_3$, and $L_4$ are defined purely for convenience, in order to lighten the notation.

	Now suppose $\gamma \in (\beta-C_2, \beta)$, and define
	\begin{align}
		r_\gamma &\= \begin{cases}
			\min \bigl\{ d(\cO_\gamma, D_\gamma), \, \frac{\Delta(\cO_\gamma)}{2} , \, d \bigl(\cO_\gamma, \cO_\beta^*(1) \bigr) \bigr\} & \text{if }  \cO_\beta \cap \cO_\beta^*(1) = \emptyset,\\
			\min \bigl\{ d(\cO_\gamma, D_\gamma), \,  \frac{\Delta(\cO_\gamma) }{2}  \bigr\}& \text{if }  \cO_\beta \cap \cO_\beta^*(1) \neq \emptyset,
		\end{cases} \label{d_8_r_gamma} \\
		\psi_\gamma &\= \overline{\phi} + u - u \circ U_\gamma, \label{d_8_psi}\\
		\tau &\= (3K_\beta + 1) \Hseminorm{\alpha}{\phi}(\beta-\gamma)^\alpha. \label{d.8.tau}
	\end{align}
	
	Before embarking on the perturbative part of the proof, we shall establish the following three preparatory claims:
	
	\smallskip
	\emph{Claim~1.} $r>0$, $C_2>0$, and $r_\gamma \ge r$.
	
	\smallskip
	\emph{Proof of Claim~1.} First we show that $r>0$. If there were some point $x\in \cO_\beta \cap D_\beta$, then $U_\beta(x)$ would belong to $\cO_\beta \cap \{1\}$ by (\ref{e_D_beta}), but this contradicts the fact that $1$ is not a $U_\beta$-periodic point, since $\beta$ is a non-simple beta-number; thus in fact $d(\cO_\beta, D_\beta)>0$. The strict positivity of the minimum interpoint distance $\Delta(\cO_\beta)$ follows directly from its definition, and the strict positivity of $d \bigl(\cO_\beta, \cO_\beta^*(1) \bigr)$ clearly does hold in the case where $\cO_\beta \cap \cO_\beta^*(1) = \emptyset$. Thus we have shown that $r>0$. It readily follows that $C_2>0$ as well (since the terms not involving $r$ on the right-hand side of (\ref{d_8_C2}) are clearly positive).
	
	It remains to show that $r_\gamma \ge r$, and for this we will use the definitions of the constants $r, \, r_\gamma, \, C_2$, and that of the discontinuity set $D_\beta$ (see (\ref{e_D_beta})). 
	Consider $x\in\cO_\gamma$ and $x'\in \cO_\beta$ with $x = (h_\gamma \circ \pi_\beta)(x')$. Suppose $z= i/\gamma \in D_\gamma$ for some nonnegative integer $i\le \gamma$, and fix $w\in \cO_\beta^*(1)$. By Corollary~\ref{l8.shadow}, and using that $C_2\le M^{-1}r$, we have $\abs{x-x'} \le M(\beta-\gamma) \le MC_2 \le r$. Now set $z' \= i/\beta \in D_\beta$, and using that $C_2\le \beta r$, we have $\abs{z-z'} = \frac{i (\beta-\gamma)}{\beta \gamma} \le \frac{C_2}{\beta}\le r$. Hence, if $\cO_\beta \cap \cO_\beta^*(1) \neq \emptyset$, the above inequalities, together with the triangle inequality, and the definition of $r$, yield
	\begin{equation}\label{xz_ineq} 
		\abs{x-z} \ge \abs{x'-z'} - \abs{x'-x} -\abs{z'-z} \ge d(\cO_\beta,D_\beta) -r-r \ge 3r-2r= r.     
	\end{equation}
	If $\cO_\beta \cap \cO_\beta^*(1) = \emptyset$, then (\ref{xz_ineq}) also holds, and the definition of $r$ in this case additionally yields
	\begin{equation}\label{xw_ineq}
		\abs{x-w} \ge \abs{x'-w} - \abs{x-x'} \ge d \bigl( \cO_\beta, \cO_\beta^*(1) \bigr) -r \ge 2r-r=r.
	\end{equation}
	From (\ref{xz_ineq}) we deduce that $d(\cO_\gamma, D_\gamma) \ge r$, and if $\cO_\beta \cap \cO_\beta^*(1) = \emptyset$
	then from (\ref{xw_ineq}) we deduce that
	$d \bigl(\cO_\gamma, \cO_\beta^*(1) \bigr) \ge r$. If $p=1$, by Corollary~\ref{l8.shadow}~(i), $\card \cO_\gamma=1$, so $\Delta(\cO_\gamma) = +\infty$. As $d(\cO_\gamma, D_\gamma) \ge r$ and if $\cO_\beta \cap \cO_\beta^*(1) = \emptyset$ then $d \bigl(\cO_\gamma, \cO_\beta^*(1) \bigr) \ge r$, we obtain $r_\gamma\ge r$. If $p \ge 2 $, by Corollary~\ref{l8.shadow}~(i), $\card \cO_\gamma \ge 2$. In this case, consider $y\in \cO_\gamma$ and $y'\in \cO_\beta$ with $x\neq y$ and $y=(h_\gamma \circ \pi_\beta)(y')$. Similarly, we have $\abs{y-y'} \le r$. Recalling that $\abs{x-x'}\le r$, we have
    \begin{equation}\label{xy_ineq}
        \abs{x-y} \ge \abs{x'-y'} - \abs{x'-x} -\abs{y'-y} \ge \Delta(\cO_\beta) - 2r \ge 4r-2r=2r, 
    \end{equation}
    which implies that $\frac{1}{2}\Delta(\cO_\gamma) \ge r$. Recalling that $d(\cO_\gamma, D_\gamma) \ge r$ and if $\cO_\beta \cap \cO_\beta^*(1) = \emptyset$ then $d \bigl(\cO_\gamma, \cO_\beta^*(1) \bigr) \ge r$, we obtain  that $r_\gamma\ge r$.
    Therefore, Claim~1 is proved.
	
	\smallskip
	\emph{Claim~2.} For each $x\in \cO_\gamma$, if $s, \, t \in B(x, r)$ then $U_\gamma(s) - U_\gamma(t) = \gamma(s-t)$, moreover, $d(s,\cO_\gamma) = \abs{s-x}$.
	
	\smallskip
	\emph{Proof of Claim~2.} Note that $r\le r_\gamma$ by Claim~1. Given $x\in \cO_\gamma$ and $s, \, t\in B(x,r) \subseteq B(x, r_\gamma)$, note first that since $r_\gamma \le d(\cO_\gamma, D_\gamma)$, then $B(x,r_\gamma) \cap D_\gamma = \emptyset$, thus $U_\gamma(s) - U_\gamma(t) = \gamma(s-t)$. Secondly, since $\abs{s-x} < r_\gamma \le \frac{1}{2} \Delta(\cO_\gamma)$, then $x$ is the closest point in $\cO_\gamma$ to $s$, so that $d(s,\cO_\gamma) = \abs{s-x}$. So Claim~2 is proved. 
	
	\smallskip
	
	The following\footnote{\label{Claim3remarks} Note that Claim~3 represents an analogue of the inequality (\ref{e.psi'<tau}), established in the context of expanding maps. While (\ref{e.psi'<tau}) is obtained relatively easily, the proof of Claim~3 is more delicate (in particular, it need not hold if either $\gamma>\beta$ or $\beta$ is not a beta-number), and exploits particular properties of beta-transformations.}
    Claim~3 means that the function $u$ can be regarded as a sub-action for $(U_\gamma, \phi)$: 
	
	\smallskip
	\emph{Claim~3.} $\psi_\gamma \le \tau$.
	
	\smallskip
	\emph{Proof of Claim~3.} Recall that $\lfloor x \rfloor' \= \max\{n\in \Z : n <x\}$ is a nondecreasing function that only takes integer values.
	
	Fix $x\in I$. First consider the case where $\lfloor \gamma x \rfloor' = \lfloor \beta x \rfloor'$. By the definition of $U_\beta$ (see (\ref{e_def_U_beta})), $U_\beta(x) = U_\gamma(x) + x(\beta-\gamma) > U_\gamma(x)$, so by (\ref{d_8_psi}), (\ref{e8.psi<0}), and (\ref{e8.Hbound}),
	\begin{equation}\label{e8_C31}
		\begin{aligned}
            \psi_\gamma(x) 
            &= \psi_\beta(x) + u ( U_\beta(x)) - u( U_\gamma(x)) \\
            &\le K_\beta \Hseminorm{\alpha}{\phi} (U_\beta(x) - U_\gamma(x))^\alpha  
            \le K_\beta \Hseminorm{\alpha}{\phi} (\beta-\gamma)^\alpha.		    
		\end{aligned}
	\end{equation}
	Next consider the case where $\lfloor \gamma x \rfloor' < \lfloor \beta x \rfloor'$, and set $y \= x- \frac{\beta-\gamma}{ \beta}$. Since $\beta-\gamma < C_2 \le \frac{1}{2}$, 
	\begin{equation}\label{e8.gamma}
    	\begin{aligned}
		\gamma y 
        = \gamma x - \gamma  \beta^{-1} (\beta-\gamma) 
        &= \beta x - \bigl( x+ \gamma \beta^{-1} \bigr)(\beta-\gamma)\\
        &> \beta x - 2(\beta-\gamma) 
        >\beta x -1.
        \end{aligned}
	\end{equation}
	Thus, $\lfloor \beta x  \rfloor' -1 \le \lfloor \gamma y \rfloor' \le \lfloor \gamma x \rfloor' < \lfloor \beta x \rfloor'$. So $\lfloor \gamma y  \rfloor' = \lfloor \gamma x \rfloor' = \lfloor \beta x \rfloor' -1$. On the other hand, we have $\gamma y < \beta y = \beta x - (\beta-\gamma) = \gamma x +(x-1)(\beta-\gamma) \le \gamma x$, thus 
	\begin{equation}\label{gamma_y_beta}
		\lfloor \gamma y \rfloor' = \lfloor \beta y \rfloor'.\end{equation} 
	In view of (\ref{gamma_y_beta}), the inequality
	(\ref{e8_C31}) holds (with $y$ replacing $x$), in other words,
	\begin{equation}\label{e8_C31_y}
		\psi_\gamma(y) \le K_\beta \Hseminorm{\alpha}{\phi} (\beta-\gamma)^\alpha.
	\end{equation}
	Now $y<x$, so by (\ref{e8.Hbound}) 
    we have
	$u(x) - u(y)  \le K_\beta \Hseminorm{\alpha}{\phi} \abs{x-y}^\alpha$, and combining this with (\ref{e8_C31_y}) and (\ref{d_8_psi}) gives
	\begin{equation}\label{e8_C32}
		\begin{aligned}
			&\psi_\gamma(x) - \abs{ u(U_\gamma(x)) - u ( U_\gamma(y)) } \\
            &\qquad\le \psi_\gamma(y) + \abs{ \phi(x) - \phi(y) } + u(x) - u(y)   \\ 
			&\qquad\le (2K_\beta +1) \Hseminorm{\alpha}{\phi} (x-y)^\alpha 
			\le (2K_\beta +1) \Hseminorm{\alpha}{\phi} (\beta-\gamma)^\alpha .
		\end{aligned}
	\end{equation}
	The last inequality above follows from $x-y = (\beta-\gamma)/\beta < \beta-\gamma$. Since $C_2 \le \frac{1}{2} \Delta\bigl(\cO_\beta^*(1)\bigr)$, and $1\in \cO_\beta^*(1)$, we have $\cO_\beta^*(1) \cap (1-2C_2, 1) =\emptyset$. By (\ref{e8.gamma}) and $\lfloor \beta x \rfloor'= \lfloor \gamma y \rfloor'+1$, we have $\gamma y > \beta x - 2C_2 \ge \lfloor \gamma y \rfloor'+1-2C_2$, which implies that $U_\gamma(y) \in (1-2C_2, 1]$. Since $\lfloor \gamma y \rfloor' =\lfloor \gamma x \rfloor'$, we have $1-2C_2 < U_\gamma(y) < U_\gamma(y) + \gamma(x-y) =  U_\gamma(x) \le 1$. Thus, applying (\ref{e8.Hbound}) (with the points $x$ and $y$ replaced by $U_\gamma(x)$ and $U_\gamma(y)$) gives
	\begin{equation}\label{e8_C33}
		\abs{ u( U_\gamma(x)) - u ( U_\gamma(y)) } 
        \le K_\beta \Hseminorm{\alpha}{\phi} (\gamma(x-y))^\alpha 
        < K_\beta \Hseminorm{\alpha}{\phi} (\beta-\gamma)^\alpha,
	\end{equation}
	where the last inequality follows from $x-y = (\beta-\gamma)/\beta <(\beta-\gamma)/\gamma$. Combining (\ref{e8_C31}), (\ref{e8_C32}),  (\ref{e8_C33}), and recalling that
	$\tau = (3K_\beta + 1) \Hseminorm{\alpha}{\phi}(\beta-\gamma)^\alpha$ (cf.~(\ref{d.8.tau})), we obtain that $\psi_\gamma \le \tau$, so Claim~3 is proved.

	\smallskip
	
	Having proved the above Claims 1, 2, and 3, we are now ready to begin the perturbative part of the proof, by defining the functions
	\begin{align}
		\phi_\gamma' &\= \overline{\phi}- C_1 \Hseminorm{\alpha}{\phi} (\beta- \gamma)^{\alpha/2} d(\cdot, \cO_\gamma)^\alpha \quad\text{ and} \label{d_8_phi_gamma'} \\
		\psi_\gamma' &\= \psi_\gamma - C_1 \Hseminorm{\alpha}{\phi} (\beta- \gamma)^{\alpha/2} d(\cdot, \cO_\gamma)^\alpha 
        = \phi_\gamma' + u - u \circ U_\gamma, \label{d_8_psi_gamma}
    \end{align}
	noting that the second equality in (\ref{d_8_psi_gamma}) follows from (\ref{d_8_psi}) and (\ref{d_8_phi_gamma'}). 

    Note that the function defined (cf.~(\ref{perturbation_minus_2C1})) in the statement of the theorem differs from $\phi_\gamma'$ by $Q(U_\beta,\phi)- C_1 \Hseminorm{\alpha}{\phi} (\beta- \gamma)^{\alpha/2} d(\cdot, \cO_\gamma)^\alpha$,
    and this latter function clearly has $\mu_{\cO_\gamma}$ as its unique $U_\gamma$-maximizing measure, since the maximum of the function is attained uniquely on the $U_\gamma$-periodic orbit $\cO_\gamma$.
    Consequently, if it can be shown that 
    the periodic measure
    $\mu_{\cO_\gamma}$
	is $\bigl(U_\gamma,\phi_\gamma'\bigr)$-maximizing,
    then it will follow that this measure is the \emph{unique} $U_\gamma$-maximizing measure for the function defined in (\ref{perturbation_minus_2C1}), and the theorem will be proved.

    We therefore aim to prove that
    $\mu_{\cO_\gamma}$
	is $\bigl(U_\gamma,\phi_\gamma'\bigr)$-maximizing.
    Using (\ref{d_8_phi_gamma'}),
	combined with Corollary~\ref{l8.shadow} and the assumption that the unique $(U_\beta,\phi)$-maximizing measure is supported by the period-$p$ orbit $\cO_\beta$, 
	we see that
	\begin{equation}\label{d_8_eta}
    \begin{aligned}	
        \eta 
        \= \int_I \! \phi_\gamma' \,\mathrm{d}\mu_{\cO_\gamma} 
        &= \int_I \! \overline{\phi} \,\mathrm{d}\mu_{\cO_\gamma} \\
        &\ge \int_I \! \overline{\phi} \,\mathrm{d}\mu_{\cO_\beta} - \frac{1}{p}\sum_{x\in \cO_\beta} \Absbig{ \bigl( \overline{\phi} \circ h_\gamma \circ \pi_\beta \bigr) (x) - \overline{\phi}(x)}  \\
        &\ge  -\Hseminorm{\alpha}{\phi} M^\alpha (\beta-\gamma)^\alpha,
    \end{aligned}
	\end{equation}
	and deduce from (\ref{d_8_eta}), using  (\ref{d_8_L2}) and (\ref{d.8.tau}),  that 
	\begin{equation}\label{e8.tau-eta}
		\tau-\eta \le L_2 \Hseminorm{\alpha}{\phi} (\beta-\gamma)^\alpha.
	\end{equation}
	Since $\psi_\gamma\le\tau$ by Claim~3, and $\int\! \psi_\gamma\, \mathrm{d}\mu_{\cO_\gamma}
	=\int\! \overline{\phi}\, \mathrm{d}\mu_{\cO_\gamma}
	=
	\int\! \phi_\gamma'\, \mathrm{d}\mu_{\cO_\gamma}=\eta$ (cf.~(\ref{d_8_psi}) and (\ref{d_8_eta})), then
	\begin{equation}\label{tau_minus_eta}
		\tau - \eta \ge 0.
	\end{equation}
	
	Now $\cO_\gamma$ is a $U_\gamma$-periodic orbit, by Corollary~\ref{l8.shadow}~(i),
	so $\mu_{\cO_\gamma}\in\MMM(I,U_\gamma)$, 
	and $\eta = \int_I \! \phi_\gamma' \,\mathrm{d}\mu_{\cO_\gamma}$ (cf.~(\ref{d_8_eta})),
	so to show that $\mu_{\cO_\gamma}$ is a $\bigl(U_\gamma,\phi_\gamma'\bigr)$-maximizing measure,
	it suffices to prove that
	$\mpe\bigl(U_\gamma,\phi_\gamma'\bigr)\le \eta$.
	But 
    \begin{equation*}
        \mpe\bigl(U_\gamma, \phi_\gamma'\bigr) = \sup_{x\in I} \liminf_{n\to +\infty} \frac{1}{n} S_n^{U_\gamma}  \phi_\gamma'(x)
    \end{equation*} by Lemma~\ref{l8.Q=supliminf},
	so in view of (\ref{d_8_psi_gamma}), and the fact that $u$ is bounded, if we can show that 
	\begin{equation}\label{e8.sufcon}
		\liminf_{n \to +\infty} \frac{1}{n} S_n^{U_\gamma} \psi_\gamma'(x) 
		\le \eta \quad \text{ for all }x\in I,
	\end{equation}
	then the required
    fact that $\mu_{\cO_\gamma}$ is  $\bigl(U_\gamma,\phi_\gamma'\bigr)$-maximizing 
     will follow. 
	
	We therefore aim to prove (\ref{e8.sufcon}).
	Defining
	\begin{equation}\label{e_8_rho}
		\rho \=  (\tau - \eta) ^{1/\alpha}  (C_1 \Hseminorm{\alpha}{\phi })^{-1/\alpha} (\beta-\gamma)^{-1/2},
	\end{equation}
	note that the definition of $\psi_\gamma'$ (cf.~(\ref{d_8_psi_gamma})),
	together with $\psi_\gamma\le \tau$ (by Claim~3), implies that 
	\begin{equation}\label{psi_gamma_eta}
		\psi_\gamma'(x) \le \eta \quad \text{ for all }x \notin B(\cO_\gamma, \rho).
	\end{equation}

	Now fix an arbitrary point $x\in I$. 
	In order to show that $\liminf\limits_{n \to +\infty} \frac{1}{n} S_n^{U_\gamma} \psi_\gamma'(x) 
	\le \eta $
	we will recursively construct a sequence $\{x_k\}_{k\in\N}$ in $\cO_\gamma^*(x)$ and a sequence $\{n_k\}_{k\in \N}$ in $\N$ satisfying 
    \begin{equation*}
    x_{k+1} = U_\gamma^{n_k} (x_k) \quad \text{ and } \quad S_{n_k}^{U_\gamma} \psi_\gamma' (x_k) \le n_k \eta.
    \end{equation*} 
	
	\smallskip
	\emph{Base step.} Define $x_1 \= x$.
	
	\smallskip
	\emph{Recursive step.} Assume that for some $t\in \N$, $\{x_k\}_{k=1}^t$ and $\{n_k\}_{k=1}^{t-1}$ are defined. We now divide our discussion into three cases, the third of which requires some delicate analysis.
	
	\smallskip
	\emph{Case~A.} Assume $x_t \in \cO_\gamma$. Then define $n_t \= p $ and $x_{t+1} \= U_\gamma^{n_t}(x_t) = x_t$. Thus, by (\ref{d_8_psi_gamma}) and (\ref{d_8_eta}), we have 
	\begin{equation}\label{e_8_Snpsi1}
		S_{n_t}^{U_\gamma} \psi_\gamma' (x_t) = n_t \eta.
	\end{equation}
	
	\smallskip
	\emph{Case~B.} Assume $x_t \notin B(\cO_\gamma, \rho)$. Then define $n_t \= 1$ and $x_{t+1} \= U_\gamma^{n_t}(x_t) = U_\gamma(x_t)$, so that (\ref{psi_gamma_eta}) gives 
	\begin{equation}\label{e_8_Snpsi2}
		S_{n_t}^{U_\gamma} \psi_\gamma' (x_t) = \psi_\gamma'(x_t) \le \eta.
	\end{equation}
	
	\smallskip
	\emph{Case~C.} Assume $x_t \in B( \cO_\gamma, \rho) \smallsetminus \cO_\gamma$. Note that $C_1\ge \frac{L_2 \cdot C_2^{\alpha/2} \beta^\alpha}{ r^\alpha}$ by (\ref{d_8_C1}),
	and $\beta-\gamma<C_2$,
	so using (\ref{e8.tau-eta}) and (\ref{e_8_rho}) it follows that 
	\begin{equation}\label{rho_r_beta}
		\rho \le r/\beta.
	\end{equation}
    Let $y$ be a point in $\cO_\gamma$ that is closest to $x_t$, in the sense that 
     \begin{equation}\label{y_defined_closest}
     \abs{x_t-y} = d( x_t, \cO_\gamma ),\end{equation}
     so $x_t\in B(y,\rho) \subseteq B(y,r)$ by the Case~C assumption together with (\ref{rho_r_beta}).
	Next, define 
    \begin{equation} \label{N_definition_as_largest_m_definition_as_smallest}
     \begin{aligned}
      N   &\= \min \bigl\{i\in\N_0: d \bigl(U_\gamma^{i+1}(x_t), U_\gamma^{i+1}(y) \bigr) \ge r \bigr\}  \quad \text{ and } \\
      m   &\= \min \bigl\{i\in\N_0: d \bigl(U_\gamma^i(x_t), U_\gamma^i(y) \bigr) \ge \rho  \bigr\}, 
      \end{aligned}
     \end{equation} 
     noting that such $N$ and $m$ exist by Claim~2 and the fact that $\rho < r $ (cf.~(\ref{rho_r_beta})). 
     So Claim~2 gives
	\begin{equation}\label{e8.Claim2result}
		\Absbig{U_\gamma^{i+1}(x_t)-U_\gamma^{i+1}(y)}=\gamma\Absbig{U_\gamma^i(x_t)-U_\gamma^i(y)}= \gamma d \bigl( U_\gamma^i(x_t), \cO_\gamma \bigr)
	\end{equation}
	for all $0\le i\le N$. Now $\gamma<\beta$, so combining (\ref{rho_r_beta}) and (\ref{e8.Claim2result}), we conclude that
	\begin{equation}\label{claim2_gives}
     \begin{aligned}
		d \bigl(U_\gamma^N(x_t), \cO_\gamma \bigr)  
        &= \gamma^{-1} \Absbig{ U_\gamma^{N+1}(x_t)- U_\gamma^{N+1}(y) }  \\
        &\ge \gamma^{-1} d\bigl( U_\gamma^{N+1}(x_t), \cO_\gamma \bigr) \ge r/\gamma > r/\beta.
     \end{aligned}
	\end{equation}
	Now $d(x_t, \cO_\gamma)<\rho$, so (\ref{claim2_gives}) and (\ref{rho_r_beta}) imply that
	\begin{equation}\label{e8.1<m<N}
		1 \le m \le N.
	\end{equation}
	In this case, we define 
    $n_t \= N+1$ and $x_{t+1} \= T_\gamma^{n_t}(x_t)$. 
    
    It is convenient to divide the following discussion into two subcases.
	
	\smallskip
	\emph{Subcase~(i).} Assume $B(\cO_\gamma, \gamma\rho) \cap \cO_\beta^*(1) = \emptyset$. In this subcase, by (\ref{e8.1<m<N}), (\ref{e8.Claim2result}), and the definition of $m$ (cf.~(\ref{N_definition_as_largest_m_definition_as_smallest})), we have 
    \begin{equation*}
    \abs{x_t-y} 
    < \Absbig{ U_\gamma^m(x_t) -  U_\gamma^m(y)}
    = \gamma \Absbig{ U_\gamma^{m-1}(x_t) -  U_\gamma^{m-1}(y)} 
    <\gamma \rho 
    \le d \bigl( \cO_\gamma, \cO_\beta^*(1) \bigr).
    \end{equation*}
    Then by (\ref{d_8_psi}), Theorem~\ref{mane}~(iii), and (\ref{d_8_L1}),
	\begin{align}\label{e8_Smd}
			&\Absbig{S_m^{U_\gamma} \psi_\gamma (x_t) - S_m^{U_\gamma} \psi_\gamma(y)} \\
            &\qquad \le \Absbig{S_m^{U_\gamma} \overline{\phi} (x_t) - S_m^{U_\gamma} \overline{\phi}(y)} + \abs{u(x_t)-u(y)} + \Absbig{u \bigl(U_\gamma^m(x) \bigr) - u \bigl( U_\gamma^m(y) \bigr)} \notag\\
			&\qquad\le \Hseminorm{\alpha}{\phi} \biggl( \sum_{i=0}^{m-1} \Absbig{U_\gamma^i(x_t) - U_\gamma^i(y)}^\alpha + 2K_\beta \rho^\alpha \gamma^\alpha \biggr)  \notag\\
			&\qquad\le \Hseminorm{\alpha}{\phi} \biggl( \Absbig{U_\gamma^{m-1}(x_t) - U_\gamma^{m-1}(y)}^\alpha \sum_{i=0}^{m-1} \gamma^{-i \alpha} + 2K_\beta \rho^\alpha \beta^\alpha \biggr)\notag\\
			&\qquad< \Hseminorm{\alpha}{\phi} \rho^\alpha \bigl( 1+   (\gamma^\alpha -1)^{-1} + 2K_\beta \beta^\alpha \bigr)  
			< \Hseminorm{\alpha}{\phi} \rho^\alpha L_1. \notag
	\end{align}
    Recalling that $p$ denotes the common cardinality of $\cO_\beta$ and $\cO_\gamma$
    (cf.~Corollary~\ref{l8.shadow}~(i) and (\ref{d_8_p})), let us
	write $m = qp+l$, where $q, \, l$ are integers with $q\ge 0$ and $l\in [0,p-1]$.  Thus, 
	since $\psi_\gamma\le\tau$ by Claim~3, and $\eta= \int\! \phi_\gamma'\, \mathrm{d}\mu_{\cO_\gamma}
	=
	\int\! \overline{\phi}\, \mathrm{d}\mu_{\cO_\gamma}
	=
	\int\! \psi_\gamma\, \mathrm{d}\mu_{\cO_\gamma}
	=
	p^{-1}S_p^{U_\gamma} \psi_\gamma(y)$,
	we have  
	\begin{equation}\label{e8_Sm<}
		S_m^{U_\gamma} \psi_\gamma(y)  
            = qS_p^{U_\gamma} \psi_\gamma(y) + S_l^{U_\gamma} \psi_\gamma(y) 
            \le qp \eta +l\tau
		= m\eta +l(\tau-\eta)
		\le m\eta + p(\tau-\eta),
	\end{equation}
	where the final inequality uses that $\tau-\eta\geq 0$ (cf.~(\ref{tau_minus_eta})).
	
	The two preceding inequalities (\ref{e8_Smd}) and (\ref{e8_Sm<}), together with the definition of $\rho$ (cf.~(\ref{e_8_rho})), and
	the fact that $C_1\ge 1$ and $\beta-\gamma < C_2 <1$, give
	\begin{equation}\label{e_8_Smpsi}
		\begin{aligned}
			S_m^{U_\gamma} \psi_\gamma (x_t) 
            &\le m \eta + p(\tau-\eta) +  \Hseminorm{\alpha}{\phi} \rho^\alpha L_1 \\
			 &= m \eta + (\tau-\eta) \bigl( p+  L_1 C_1^{-1}(\beta-\gamma)^{-\alpha/2} \bigr)\\
		& \le m \eta +  (\tau-\eta) (\beta-\gamma)^{-\alpha/2}    (p + L_1) \\
			& \le m \eta + (p  + L_1)L_2 \Hseminorm{\alpha}{\phi}(\beta-\gamma) ^{\alpha/2},
		\end{aligned}
	\end{equation}
	where in the final inequality we used that
	$\tau-\eta \le L_2 \Hseminorm{\alpha}{\phi} (\beta-\gamma)^\alpha$
	(cf.~(\ref{e8.tau-eta})).
	
	Now $n_t=N+1$, so
	$S_{n_t}^{U_\gamma} \psi_\gamma' (x_t)
	=
	S_m^{U_\gamma} \psi_\gamma'(x_t)
	+
	S_{N-m}^{U_\gamma} \psi_\gamma' \bigl(U_\gamma^m(x_t) \bigr)
	+
	\psi_\gamma' \bigl(U_\gamma^N (x_t) \bigr),
	$
	and since
	$\psi_\gamma' \le \psi_\gamma$ (cf.~(\ref{d_8_psi_gamma})) and $\psi_\gamma'\bigl( U_\gamma^i (x_t)\bigr) \le \eta$ for $m\le i <N$ by (\ref{psi_gamma_eta}), 
	then
	\begin{equation}\label{n_t_N_sum}
		S_{n_t}^{U_\gamma} \psi_\gamma' (x_t)
		\le
		S_m^{U_\gamma} \psi_\gamma(x_t)
		+
		(N-m)\eta
		+
		\psi_\gamma' \bigl(U_\gamma^N (x_t) \bigr).
	\end{equation}
	Now $\psi_\gamma'=\psi_\gamma - C_1 \Hseminorm{\alpha}{\phi} (\beta- \gamma)^{\alpha/2} d(\cdot, \cO_\gamma)^\alpha$ (cf.~(\ref{d_8_psi_gamma})),
	so 
	the fact that
	$d \bigl(U_\gamma^N(x_t), \cO_\gamma \bigr) > r/\beta$ by (\ref{claim2_gives}), and that $\psi_\gamma\le \tau$ by Claim~3, together give
	\begin{equation*}
    \psi_\gamma'\bigl(U_\gamma^N (x_t) \bigr) 
	\le \tau -	C_1 \Hseminorm{\alpha}{\phi} (\beta- \gamma)^{\alpha/2} (r/\beta)^\alpha,
    \end{equation*}
	and combining with (\ref{n_t_N_sum}) gives
	\begin{equation}\label{combined}
		S_{n_t}^{U_\gamma} \psi_\gamma' (x_t)
		\le
		S_m^{U_\gamma} \psi_\gamma(x_t)
		+
		(N-m)\eta
		+ \tau -
		C_1 \Hseminorm{\alpha}{\phi} (\beta- \gamma)^{\alpha/2} r^\alpha\beta^{-\alpha}.
	\end{equation}
	Combining
	(\ref{e_8_Smpsi})
	and (\ref{combined}) gives
	\begin{align*}
		S_{n_t}^{U_\gamma} \psi_\gamma' (x_t)
		&\le
		m \eta + (p  + L_1)L_2 \Hseminorm{\alpha}{\phi}(\beta-\gamma) ^{\alpha/2}  \\
		&\quad +
		(N-m)\eta
		+ \tau -
		C_1 \Hseminorm{\alpha}{\phi} (\beta- \gamma)^{\alpha/2} r^\alpha \beta^{-\alpha},
	\end{align*}
	and therefore
	\begin{equation}\label{e_8_Snpsi3}
		\begin{aligned}
			&S_{n_t}^{U_\gamma} \psi_\gamma' (x_t)- n_t \eta \\
			&\qquad \le (p  + L_1 ) L_2 \Hseminorm{\alpha}{\phi} (\beta-\gamma) ^{\alpha/2} + (\tau-\eta)  - C_1 \Hseminorm{\alpha}{ \phi}  r^\alpha \beta^{-\alpha} (\beta-\gamma)^{\alpha/2}\\
			&\qquad \le (p+1 + L_1 ) L_2 \Hseminorm{\alpha}{\phi} (\beta-\gamma) ^{\alpha/2} - C_1 \Hseminorm{\alpha}{ \phi}  r^\alpha \beta^{-\alpha} (\beta-\gamma)^{\alpha/2} 
			 \le 0,
		\end{aligned}
	\end{equation}
	where the final two inequalities follow from the fact that $\tau-\eta 
	\le L_2 \Hseminorm{\alpha}{\phi} (\beta-\gamma)^\alpha
	\le L_2 \Hseminorm{\alpha}{\phi} (\beta-\gamma)^{\alpha/2}
	$
	(by (\ref{e8.tau-eta}), and since $\beta-\gamma<C_2<1$), and 
	$C_1\ge r^{-\alpha}\beta^\alpha(p+1+L_1)L_2$
	(cf.~(\ref{d_8_C1})). 
	
	\smallskip
	\emph{Subcase~(ii).}  Assume that $B(\cO_\gamma, \gamma\rho) \cap \cO_\beta^*(1) \neq \emptyset$. Now $\gamma<\beta$, so (\ref{rho_r_beta}) and Claim~1 give
	\begin{equation}\label{gamma_rho_r_gamma}
		\gamma \rho <\beta\rho \le r \le r_\gamma,\end{equation}  
	which implies that 
	\begin{equation}\label{beta_periodic_critical_intersection}
		\cO_\beta \cap \cO_\beta^*(1) \neq \emptyset,
	\end{equation}
	since if (\ref{beta_periodic_critical_intersection}) were false then the definition of $r_\gamma$
	(cf.~(\ref{d_8_r_gamma}))
	would give $r_\gamma \le  d \bigl(\cO_\gamma, \cO_\beta^*(1) \bigr)$,
	and combining with (\ref{gamma_rho_r_gamma})
	would yield
	$d \bigl(\cO_\gamma, \cO_\beta^*(1) \bigr)>\gamma\rho$,
	contradicting the assumption of this subcase.
	
	Now      (\ref{beta_periodic_critical_intersection})
	means that
	\begin{equation*}
    \cO_\beta \subseteq \cO_\beta^*(1),
    \end{equation*}
	in other words,  $\cO_\beta$ is precisely
	the $U_\beta$-periodic orbit corresponding to the eventually periodic critical orbit
	$\cO_\beta^*(1)$.
	In particular, $\cO_\beta\neq\{0\}$, since $0$ is never in the critical orbit $\cO_\beta^*(1)$ (as the only $U_\beta$-preimage of $0$ is $0$ itself). Recalling that $s>0$ is the constant obtained from Corollary~\ref{l8.shadow}~(iii) when $\cO_\beta \neq \{0\}$, then for each $x\in \cO_\beta$, we obtain from Corollary~\ref{l8.shadow} that
	\begin{equation}\label{e8_dO}
		s(\beta-\gamma) 
        \le \abs{(h_\gamma \circ \pi_\beta)(x) - x}
        \le M(\beta-\gamma) < C_2M 
        \le \frac{1}{2} \Delta \bigl( \cO_\beta^*(1) \bigr),
	\end{equation}
	using that $\beta-\gamma<C_2$ and $C_2 \le \frac{1}{2M} \Delta \bigl(\cO_\beta^*(1)\bigr)$ (by (\ref{d_8_C2})).
	A consequence of the bound
	$\abs{(\pi_\gamma \circ h_\beta)(x) - x} \le \frac{1}{2} \Delta \bigl( \cO_\beta^*(1) \bigr)$ from
	(\ref{e8_dO}) is that if $x'\in\cO_\gamma$,	
	then the point $x\in\cO_\beta$ satisfying
	$x' = (h_\gamma \circ \pi_\beta)(x)\in \cO_\gamma$ must be the closest point in $\cO_\beta^*(1)$ to $x'$, so that $\abs{ x-x' } = d \bigl(x',\cO_\beta^*(1)\bigr)$. This, together with (\ref{e8_dO}), implies that  
\begin{equation}\label{gamma_periodic_beta_critical}
		d \bigl(\cO_\gamma, \cO_\beta^*(1) \bigr) \ge s(\beta- \gamma),
	\end{equation}
    which is a key estimate for this subcase, facilitating the following 
    analysis.\footnote{\label{subcase(ii)remarks} Some remarks on the differences between Subcases~(i) and~(ii) are in order; in particular, although the proof strategy for Subcase~(i) resembles arguments used in Section~\ref{sec_JTPO_distance_expanding}, Subcase~(ii) has no such analogue. While in Subcase~(i), $\psi_\gamma$
    is $\alpha$-H\"older on
    $B(\cO_\gamma, \rho)$, so that
    the length-$m$ orbit sums in (\ref{e8_Smd})
    can be bounded in terms of
     $\Absbig{U_\gamma^i( x_t ) - U_\gamma^i ( y ) }$,
     the same is not true in Subcase~(ii), 
       and since $m$ could be arbitrarily large (as the distance 
      $\abs{ x_t-y }$
      could be arbitrarily small), 
      the crude bound
      $\psi_\gamma \bigl( U_\gamma^i(x_t) \bigr) \le \tau$
      would then be insufficient
to obtain an effective bound analogous to (\ref{e_8_Smpsi}).
      However, (\ref{gamma_periodic_beta_critical})
      allows us to use estimates analogous to
       (\ref{e8_Smd}) for the first $k$ terms in the length-$m$ orbit sum, while the
       (proportionally small number of) remaining terms can be crudely bounded
      using $\psi_\gamma \bigl( U_\gamma^i(x_t) \bigr) \le \tau$, and the corresponding term
      $(m-k)(\tau -\eta)$
      in (\ref{e8.111_pre})
       can be estimated by a multiple of 
       $(\beta-\gamma)^{\alpha/2} d(\cdot, \cO_\gamma)^\alpha$.}
	
	Now define
	\begin{equation}\label{d_8_delta}
		\delta \= \beta^{-1}s(\beta-\gamma),
	\end{equation} 
	so 
	that (\ref{gamma_periodic_beta_critical}) and (\ref{d_8_delta}) give
	\begin{equation}\label{gamma_betadelta_critical_empty}
		B(\cO_\gamma, \beta\delta) \cap \cO_\beta^*(1) = \emptyset,
	\end{equation}
	and since the assumption of this subcase
	is that  $B(\cO_\gamma, \gamma\rho) \cap \cO_\beta^*(1) \neq \emptyset$, comparison with (\ref{gamma_betadelta_critical_empty})
	yields 
	\begin{equation}\label{delta_<_rho}
		\delta < (\gamma/\beta)\rho < \rho.
	\end{equation}
	
	Recalling 
    from (\ref{y_defined_closest})
    that $y\in\cO_\gamma$ is defined to satisfy $\abs{x_t-y}=d(x_t,\cO_\gamma)$, we now
	 define
	\begin{equation}\label{k_min_def}
		k\= \min\bigl\{i \in \N_0 : d \bigl( U_\gamma^i(x_t), U_\gamma^i(y) \bigr) \ge \delta  \bigr\}\in\N_0.
	\end{equation}
	In particular,
	(\ref{delta_<_rho}) and (\ref{k_min_def}) imply
	that $0 \le k\le m$,
	since $m$ was defined
    (cf.~(\ref{N_definition_as_largest_m_definition_as_smallest}))
    to be the smallest integer such that
	$d \bigl(U_\gamma^m(x_t), \cO_\gamma \bigr) \ge \rho$.
	Note that $\rho/\delta>1$ by (\ref{delta_<_rho}), $\beta-C_2 \ge 1$ since $C_2 \le c <\beta-1$, and $0\le k\le m\le N$ by (\ref{e8.1<m<N}). By (\ref{e8.Claim2result}) and the definitions of $m$ (cf.~(\ref{N_definition_as_largest_m_definition_as_smallest}))
    and $k$
    (cf.~(\ref{k_min_def})), we have 
	\begin{equation}\label{e8_ineqkm1}
			m \le k +  \lceil \log_\gamma ( \rho/\delta ) \rceil \le k +1 +\log_{\beta -C_2} ( \rho/\delta ).
	\end{equation}
	Combining (\ref{e8_ineqkm1})
    with the definitions $\rho =  (\tau - \eta) ^{1/\alpha} (C_1 \Hseminorm{\alpha}{\phi })^{-1/\alpha} (\beta-\gamma)^{-1/2}$ 
    (cf.~(\ref{e_8_rho})), and $\delta = \beta^{-1}s(\beta-\gamma)$
    (cf.~(\ref{d_8_delta})), and
    the bound $\tau-\eta \le L_2 \Hseminorm{\alpha}{\phi} (\beta-\gamma)^\alpha$
    (cf.~(\ref{e8.tau-eta})), we see that
	\begin{equation}\label{e8_ineqkm2}
			m  \le k+1 + \frac{1}{\ln (\beta-C_2)} \Bigl( \ln\beta + \frac{1}{\alpha} (\ln L_2 - \ln C_1) - \ln s -\frac{1}{2} \ln(\beta-\gamma) \Bigr).
	\end{equation}
	Now $C_1\ge 1$ (cf.~(\ref{d_8_C1})) and
    $L_3 = 
    1+\frac{1}{\ln(\beta-C_2)} 
    \bigl( \ln \beta + \frac{1}{\alpha} \ln L_2 - \ln s \bigr)$
    (cf.~(\ref{d_8_L3})), so (\ref{e8_ineqkm2}) implies that
	\begin{equation}\label{e8_ineqkm}
		m\le  k +L_3 - \frac{ \ln(\beta-\gamma)}{2\ln(\beta-C_2)}.
	\end{equation}
	In the case where $k\ge 1$, the definition of $k$ in (\ref{k_min_def}),
	together with $k\le m\le N$, (\ref{e8.Claim2result}), and (\ref{gamma_betadelta_critical_empty}),
	gives  
    $
    \abs{x_t-y} 
    < \delta 
    \le\Absbig{U_\gamma^k(x_t) - U_\gamma^k(y)} 
    = \gamma\Absbig{U_\gamma^{k-1}(x_t) - U_\gamma^{k-1}(y)} 
    < \gamma\delta 
    < \beta\delta 
    \le d \bigl( \cO_\gamma, \cO_\beta^*(1) \bigr)$.

	We now wish to estimate orbit sums for the function $\psi_\gamma$. By exactly the same reasoning as was used to derive 
	(\ref{e8_Smd}) in Subcase~(i) (except that here $\delta$ replaces the $\rho$ in (\ref{e8_Smd})), we obtain
	\begin{equation}\label{e8_Skd}
		\Absbig{S_{k}^{U_\gamma} \psi_\gamma (x_t) - S_{k}^{U_\gamma} \psi_\gamma(y)} \le \Hseminorm{\alpha}{\phi} \delta^\alpha L_1 .
	\end{equation}
	Similarly, 
	the arguments used to derive (\ref{e8_Sm<}) 
	in Subcase~(i) can be used here (with $k$ replacing the $m$ used in (\ref{e8_Sm<}))
	to obtain 
	\begin{equation} \label{e8_Sk<}
		S_{k}^{U_\gamma} \psi_\gamma(y)  \le k\eta + p(\tau-\eta).
	\end{equation}
	Obviously, (\ref{e8_Skd}) and (\ref{e8_Sk<}) also hold
	for $k=0$.

	Now $\psi_\gamma\le \tau$ by Claim~3, so
	\begin{equation}\label{e8_SK<}
		S_{m-k}^{U_\gamma} \psi_\gamma \bigl( U_\gamma^{k}(x_t) \bigr) \le (m-k)\tau.
	\end{equation}

	Combining (\ref{e8_Sk<}), (\ref{e8_Skd}), and (\ref{e8_SK<})
	gives
	\begin{equation}\label{e8.111_pre}
		\begin{aligned}
			S_m^{U_\gamma} \psi_\gamma(x_t) & \le 	S_{k}^{U_\gamma} \psi_\gamma(y) + \Absbig{S_{k}^{U_\gamma} \psi_\gamma (x_t) - S_{k}^{U_\gamma} \psi_\gamma(y)} + S_{m-k}^{U_\gamma} \psi_\gamma \bigl( U_\gamma^{k}(x_t) \bigr) \\
			&\le m\eta + \Hseminorm{\alpha}{\phi} \delta^\alpha L_1 + p(\tau-\eta) + (m-k)(\tau-\eta),
		\end{aligned}
	\end{equation}
	and using (\ref{d_8_delta}) and (\ref{e8_ineqkm}) to estimate the right-hand side of (\ref{e8.111_pre}) yields
	\begin{equation}\label{e8.111_pre_pre}
		S_m^{U_\gamma} \psi_\gamma(x_t)   
		\le m\eta + \Hseminorm{\alpha}{ \phi} \frac{s^\alpha}{\beta^\alpha} (\beta-\gamma)^\alpha L_1 + \biggl( p+L_3 - \frac{ \ln(\beta-\gamma)}{2\ln(\beta-C_2)} \biggr)(\tau-\eta).
	\end{equation}
	But $\tau-\eta \le L_2 \Hseminorm{\alpha}{\phi} (\beta-\gamma)^\alpha$
	by (\ref{e8.tau-eta}), 
	so (\ref{e8.111_pre_pre}) gives
	\begin{equation}\label{e8.111_pre_pre_pre}
		S_m^{U_\gamma} \psi_\gamma(x_t)   
		\le m \eta + \Hseminorm{\alpha}{\phi} (\beta-\gamma)^\alpha \biggl(\frac{s^\alpha}{\beta^\alpha}L_1 + pL_2+L_3L_2 - \frac{L_2 \ln(\beta-\gamma)}{2 \ln(\beta-C_2)}  \biggr).
	\end{equation}
	Now $(\beta-\gamma)^\alpha \le (\beta-\gamma)^{\alpha/2}$,
	and we can use the fact\footnote{It is readily shown that 
		$-\frac{2}{e\alpha}$
		is the minimum of the function $x\mapsto x^{\alpha/2}\ln x$.}
	that 
	$x^{\alpha/2}\ln x \ge  -\frac{2}{e\alpha}$ for all $x>0$ to derive $-(\beta-\gamma)^\alpha \ln(\beta-\gamma) \le \frac{2(\beta-\gamma)^{\alpha/2}}{e\alpha}$, so
	(\ref{e8.111_pre_pre_pre}) implies
	\begin{equation}\label{e8.111}
		\begin{aligned}
			S_m^{U_\gamma} \psi_\gamma(x_t)
			&\le m\eta +  \Hseminorm{\alpha}{\phi} (\beta-\gamma)^{\alpha/2} \biggl(\frac{s^\alpha}{\beta^\alpha}L_1 + pL_2+L_3L_2 + \frac{L_2 }{e \alpha \ln(\beta-C_2)}  \biggr) \\
			&= m\eta + \Hseminorm{\alpha}{\phi} (\beta-\gamma)^{\alpha/2}L_4,
		\end{aligned}
	\end{equation}
	where we use that
	$L_4=\frac{s^\alpha}{\beta^\alpha}L_1 + pL_2+L_3L_2 + \frac{L_2 }{e \alpha \ln(\beta-C_2)}$
	(cf.~(\ref{d_8_L4})).

	Now we wish to argue that
	$S_{n_t}^{U_\gamma} \psi_\gamma' (x_t)\le n_t \eta$, and will do so 
	in a way that is analogous to the derivation of (\ref{e_8_Snpsi3}) in Subcase~(i).
	Specifically, we use that $\psi_\gamma' \le \psi_\gamma$
	by (\ref{d_8_psi_gamma}), $\psi_\gamma'\bigl( U_\beta^i (x_t)\bigr) \le \eta$ for $m\le i <N$ by
	(\ref{psi_gamma_eta}) (cf.~(\ref{N_definition_as_largest_m_definition_as_smallest})), and $d \bigl(U_\gamma^N(x_t), \cO_\gamma \bigr) > r/\beta$ (cf.~(\ref{claim2_gives})), together 
	with (\ref{e8.111}) above, the definition (\ref{d_8_psi_gamma}) of $\psi_\gamma'$, Claim~3, 
	the bound
	$\tau-\eta\le L_2 \Hseminorm{\alpha}{\phi} (\beta-\gamma)^\alpha \le L_2 \Hseminorm{\alpha}{\phi} (\beta-\gamma)^{\alpha/2}$
	(cf.~(\ref{e8.tau-eta})), and the fact that
	$C_1\ge r^{-\alpha}\beta^\alpha(L_4+L_2)$ (which is valid since $\cO_\beta\neq\{0\}$, cf.~(\ref{d_8_C1'})), to see that
	\begin{equation}\label{e8.222}
		\begin{aligned}
			&S_{n_t}^{U_\gamma} \psi_\gamma' (x_t)- n_t \eta \\
            &\qquad \le S_m^{U_\gamma} \psi_\gamma(x_t) - m \eta + S_{N-m}^{U_\gamma} \psi_\gamma' \bigl(U_\gamma^m(x_t) \bigr) - (N-m)\eta +
			\psi_\gamma' \bigl(U_\gamma^N (x_t) \bigr) - \eta \\
			&\qquad \le L_4 \Hseminorm{\alpha}{\phi} (\beta-\gamma) ^{\alpha/2} + (\tau-\eta)  - C_1 \Hseminorm{\alpha}{ \phi}  r^\alpha \beta^{-\alpha} (\beta-\gamma)^{\alpha/2}\\
			&\qquad \le (L_4+L_2) \Hseminorm{\alpha}{\phi} (\beta-\gamma) ^{\alpha/2} - C_1 \Hseminorm{\alpha}{ \phi}  r^\alpha \beta^{-\alpha} (\beta-\gamma)^{\alpha/2} 
			 \le 0.
		\end{aligned}
	\end{equation}

	This completes Subcase~(ii), and therefore Case~C is complete. Having concluded each of Cases~A, B, and~C, the recursive step is now complete.
	
	\smallskip
	
	By (\ref{e_8_Snpsi1}), (\ref{e_8_Snpsi2}), (\ref{e_8_Snpsi3}), and (\ref{e8.222}), $S_{n_k}^{U_\gamma} \psi_\gamma'(x_k) \le n_k \eta$ for all $k\in \N$. Therefore, defining $N_k \= n_1 + \cdots + n_k$ for each $k\in \N$, we have
	\begin{equation*}
		\liminf_{n\to +\infty} \frac{1}{n}S_n^{U_\gamma} \psi_\gamma'(x) \le \liminf_{k \to +\infty} \frac{1}{N_k} \sum_{i=1}^{k} S_{n_i}^{U_\gamma} \psi_\gamma'(x_i) \le \liminf_{k\to +\infty}\frac{1}{ N_k} \sum_{i=1}^k n_i \eta = \eta.
	\end{equation*}
	But this is precisely the required inequality (\ref{e8.sufcon}), and therefore, as noted prior to the statement of (\ref{e8.sufcon}), we deduce that $\mu_{\cO_\gamma}$ is
	$(U_\gamma,\phi_\gamma')$-maximizing. 
	In other words, the periodic
    measure $\mu_{\cO_\gamma}$ is $U_\gamma$-maximizing
    for the function
\begin{equation}\label{phi_gamma_prime_restated}
    \phi_\gamma' = \overline{\phi}- C_1 \Hseminorm{\alpha}{\phi} (\beta- \gamma)^{\alpha/2} d(\cdot, \cO_\gamma)^\alpha.
    \end{equation}
Now the function 
\begin{equation}\label{distance_function_o_gamma}
    - C_1 \Hseminorm{\alpha}{\phi} (\beta- \gamma)^{\alpha/2} d(\cdot, \cO_\gamma)^\alpha
\end{equation}
    attains its maximum value precisely on the set $\cO_\gamma$, so its \emph{unique} $U_\gamma$-maxi\-mizing measure
    is $\mu_{\cO_\gamma}$.
    It follows that the sum of the functions
    (\ref{phi_gamma_prime_restated}) and (\ref{distance_function_o_gamma}),
    namely the function
\begin{equation*}
    \overline{\phi}- 2C_1 \Hseminorm{\alpha}{\phi} (\beta- \gamma)^{\alpha/2} d(\cdot, \cO_\gamma)^\alpha,
    \end{equation*}
    has  $\mu_{\cO_\gamma}$ as its
    unique $U_\gamma$-maximizing measure.
    Consequently, the function defined
    (cf.~(\ref{perturbation_minus_2C1}))
    in the statement of the theorem, namely
$
    \phi- 2C_1 \Hseminorm{\alpha}{\phi} (\beta- \gamma)^{\alpha/2} d(\cdot, \cO_\gamma)^\alpha$,
also has  $\mu_{\cO_\gamma}$ as its
    unique $U_\gamma$-maximizing measure,
    and this completes the proof of part~(ii) of the theorem.
\end{proof}

\subsection{Proof of Theorem~\ref{jtpo_tbeta} (Joint TPO for beta-transformations)}

In order to prove Theorem~\ref{t.product.typical.periodic} (and hence Theorem~\ref{jtpo_tbeta}), a
final ingredient is that
 the set of non-simple beta-numbers is dense in $(1,+\infty)$.
 This complement to Parry's result on the density of simple beta-numbers \cite[Theorem~5]{Par60} does not seem to be available in the literature, so we prove it here:

\begin{lemma}\label{l8_density_nonsimple}
	The set of non-simple beta-numbers is a dense subset of $(1,+\infty)$.
\end{lemma}
\begin{proof}
	Since the set of simple beta-numbers is countable, it suffices to
    show that if $1<\beta_1 <\beta_2$, where
    neither $\beta_1$ nor $\beta_2$ is a simple beta-number,
    then there is a non-simple beta-number $\gamma\in(\beta_1,\beta_2)$.
    
    Writing 
    \begin{equation*}
    \pi_{\beta_1}(1) = a_1 a_2 \dots \quad \text{ and } \quad \pi_{\beta_2}(1) = b_1 b_2 \dots,
    \end{equation*} 
    the fact that $\beta_1<\beta_2$ means, by Proposition~\ref{p_relation_of_coding}~(xiii), that $\pi_{\beta_1}(1) \prec \pi_{\beta_2}(1)$. Hence, there exists  $m\in\N$ with $a_1 \dots a_m \prec b_1 \dots b_m$, and $b_m\neq 0$.
    Since
     $\beta_2$ is not a simple beta-number, it follows that $b_{m+1} b_{m+2} \dots \neq (0)^\infty$, so there exists $n\in\N$ such that $b_1 \dots b_m (0)^n \prec b_1 \dots b_{m+n}$, and this $n$ can be chosen such that $n\ge m$. 
     Defining 
     \begin{equation*}
     \underline{c} \= b_1 \dots b_m (0)^n \bigr( a_1 (0)^{m-1} \bigl)^\infty,
     \end{equation*}
     we see that $\pi_{\beta_1}(1) \prec \underline{c} \prec \pi_{\beta_2}(1)$, and claim that there exists a non-simple beta-number $\gamma\in (\beta_1, \beta_2)$ such that $\pi_\gamma(1) = \underline{c}$.

     Recall that by \cite[Corollary~1]{Par60}, to prove that there exists $\gamma>1$ with $\pi_\gamma(1) = \underline{c}$, it suffices to show that $\sigma^t(\underline{c}) \prec \underline{c}$ for all $t\in \N$.
     For this, note first that $a_1 (0)^{m-1} \preceq a_1 \dots a_m \prec b_1 \dots b_m$, and $b_1 >0$, so if $t \ge m$
     then $\sigma^t(\underline{c}) \prec \underline{c}$. On the other hand, if $t <m$ then 
     \begin{equation}\label{double_lex_ineq}
     b_{t+1} \dots b_m (0)^{t} \preceq b_{t+1} \dots b_{t+m} \preceq b_1 \dots b_m,
     \end{equation}
     the second inequality in (\ref{double_lex_ineq}) following from the fact that $\sigma^t ( \pi_{\beta_2}(1)  ) \prec \pi_{\beta_2}(1)$ (cf.~\cite[Theorem~3]{Par60}). Now $b_m \neq 0$, so
     (\ref{double_lex_ineq}) implies the strict inequality
\begin{equation}\label{strict_lex_ineq}
     b_{t+1} \dots b_m (0)^{t} \prec b_1 \dots b_m.
     \end{equation}
      The fact that $n\ge m > t$ implies that $b_{t+1} \dots b_m (0)^{t}$ are the first $m$ terms of $\sigma^t(\underline{c})$, 
      and since moreover $b_1 \dots b_m$
      are the first $m$ terms of $\underline{c}$, (\ref{strict_lex_ineq}) implies that
       $\sigma^t(\underline{c}) \prec \underline{c}$. 
Having shown that $\sigma^t(\underline{c}) \prec \underline{c}$ for all $t\in \N$,
it follows that there exists $\gamma>1$ with $\pi_\gamma(1) = \underline{c}$.
       
     Now  $\underline{c}$ is pre-periodic under $\sigma$, so $\gamma$ is a beta-number, and
      $\underline{c}$ has nonzero tails
      since $a_1>0$, thus $\gamma$ is a non-simple beta-number. Finally, the fact that $\pi_{\beta_1}(1) \prec \underline{c} \prec \pi_{\beta_2}(1)$ implies, by Proposition~\ref{p_relation_of_coding}~(xiii), that $\gamma\in (\beta_1, \beta_2)$.
\end{proof}

We are now able to prove the joint typical periodic optimization theorem using the results above:
the following Theorem~\ref{t.product.typical.periodic}
represents a slightly stronger version of 
Theorem~\ref{jtpo_tbeta}:

\setcounter{thml}{2}

\begin{thml}[Joint TPO for beta-transformations and upper beta-transformations]\label{t.product.typical.periodic}
    Given $\alpha\in (0,1]$, 
the sets 
$\fL^\alpha_U \=\bigl\{ (\beta,\phi)\in (1,+\infty)\times \Holder{\alpha}(I) : \phi\in \Lock^\alpha(U_\beta) \bigr\}$ 
    and
    $\fL^\alpha_T \=\bigl\{ (\beta,\phi)\in (1,+\infty)\times \Holder{\alpha}(I) : \phi\in \Lock^\alpha(T_\beta) \bigr\}$ 
    both contain 
    an open and dense subset of $(1,+\infty) \times \Holder{\alpha}(I)$.
\end{thml}

\begin{proof}
    We first prove the result for $\fL^\alpha_U$. Suppose $\beta\in (1,+\infty)$, $\phi \in \Holder{\alpha}(I)$, and let $\myepsilon>0$ be arbitrary. Note that the set of non-simple beta-numbers is dense in $(1,+\infty)$, by Lemma~\ref{l8_density_nonsimple}, so there exists a non-simple beta-number $\theta>1$ with $\abs{\theta- \beta} < \myepsilon$. By Theorem~\ref{t_TPO_thm_beta_number'} (Individual TPO theorem for beta-numbers), applied to $\theta$, there exists $\Phi \in \Lock^\alpha(U_\theta)$ with $\Hnorm{\alpha}{\phi- \Phi} < \myepsilon$. If $\cO_\theta$ denotes the unique maximizing periodic orbit for $\Phi$, then there exists $\delta>0$ such that $\cO_\theta$ is the unique maximizing periodic orbit for each function in $B(\Phi, \delta)$ (with respect to $\Hnorm{\alpha}{\cdot}$), and we may assume that $\delta< (1/2) \Hseminorm{\alpha}{\Phi}$. Let $\psi\in B(\Phi,\delta)$. Thus, we have $\Hseminorm{\alpha}{\Phi-\psi} \le \Hnorm{\alpha}{\Phi-\psi} \le \delta \le (1/2)\Hseminorm{\alpha}{\Phi}$. As $\Hseminorm{\alpha}{\,\cdot\,}$ is sub-additive, we get 
	\begin{equation}\label{e8.bound.Hsn}
		\frac{1}{2} \Hseminorm{\alpha}{\Phi} \le \Hseminorm{\alpha}{\psi} \le \frac{3}{2} \Hseminorm{\alpha}{\Phi}.
	\end{equation}
	
	Now let $C_1, \, C_2$ denote the constants obtained by applying Theorem~\ref{l_8_main_lemma} (ii) to $\theta$ and $\cO_\theta$, and set 
	\begin{equation}\label{d_8_E}
		E \= \min \bigl\{C_2, \, ( \delta / (36 C_1 \Hseminorm{\alpha}{\Phi} ) )^{2/\alpha}  \bigr\}.
	\end{equation}
	For each $\gamma \in ( \theta- E , \theta )$ and each $\psi \in B(\Phi, \delta/2)$, we set $\cO_\gamma \= (h_\gamma \circ \pi_\theta)(\cO_\theta)$ and 
	\begin{equation}\label{d_8_psip}
		\psi' \= \psi + 6C_1 \Hseminorm{\alpha}{\psi}  (\theta- \gamma)^{\alpha/2 } d(\cdot, \cO_\gamma)^\alpha. 
	\end{equation}
	Note that 
    $
    \Hnorm{\alpha}{ d(\cdot,\cO_\gamma)^\alpha}
    = \Hseminorm{\alpha}{d(\cdot,\cO_\gamma)^\alpha} + \Hnorm{\infty}{d(\cdot,\cO_\gamma)^\alpha} 
    \le 2$.
    Then from $\abs{\theta- \gamma} < E$, (\ref{d_8_E}), and (\ref{d_8_psip}), we conclude that 
    \begin{equation*}
    \Hnorm{\alpha}{\psi' - \psi} 
    \le 12C_1 \Hseminorm{\alpha}{\psi} E^{\alpha/2} 
    \le 18C_1 \Hseminorm{\alpha}{\Phi} E^{\alpha/2} 
    \le  \delta / 2, 
    \end{equation*}
    so that $\psi'\in B(\Phi,\delta) \subseteq \Lock^\alpha(U_\theta)$. Therefore, applying Theorem~\ref{l_8_main_lemma} (ii) to $\theta$, $\cO_\theta$, and $\psi'$, we obtain that the measure $\mu_{\cO_\gamma}$ uniquely maximizes $\psi' - 2C_1 \Hseminorm{\alpha}{\psi'}(\theta-\gamma)^{\alpha/2} d(\cdot, \cO_\gamma)^\alpha$. As $\psi, \, \psi'\in B(\Phi,\delta)$, applying (\ref{e8.bound.Hsn}) gives $6C_1 \Hseminorm{\alpha}{\psi} \ge 3C_1 \Hseminorm{\alpha}{\Phi} \ge 2C_1 \Hseminorm{\alpha}{\psi'}$, so the measure $\mu_{\cO_\gamma}$ also uniquely maximizes the function $\psi = \psi' - 6C_1 \Hseminorm{\alpha}{\psi}  (\theta- \gamma)^{\alpha/2 } d(\cdot, \cO_\gamma)^\alpha$. Therefore, $(\theta-E, \theta) \times B(\Phi, \delta/2)$ is contained in the interior of the set $\fL^\alpha_U$. Since $\myepsilon>0$ was arbitrary, this completes the proof for the set $\fL^\alpha_U$.

It remains to prove the assertion about $\fL^\alpha_T$. Now non-simple beta-numbers are non-emergent by Corollary~\ref{c_relation_emergent&betanumbers}, and $\Lock^\alpha(T_\beta)$ is an open and dense subset of $\Holder{\alpha}(I)$ when $\beta>1$ is a non-simple beta-number
by Theorem~\ref{t_TPO_thm_non_emergent'}, so an argument analogous to the one above,
using part~(i) (rather than part~(ii)) of Theorem~\ref{l_8_main_lemma}, 
can be used to show that $\fL^\alpha_T$ contains
an open dense subset of $(1,+\infty) \times \Holder{\alpha}(I)$, as required.
\end{proof}

Having proved the Joint TPO result 
Theorem~\ref{t.product.typical.periodic}, we can now
deduce Theorem~\ref{GPO_open_dense}. We establish the following slightly stronger version of Theorem~\ref{GPO_open_dense} (which in particular implies Theorem~\ref{GPO_open_dense}):

\setcounter{thml}{5}

\begin{thml}[Individual TPO for generic potentials]\label{GPO_open_dense'}
    Fix $\alpha\in (0,1]$. There is a residual subset $R \subseteq \Holder{\alpha}(I)$ such that for all $\phi\in R$, there is an open and dense set of parameters $B_\phi \subseteq (1,+\infty)$ such that $\phi\in \Lock^\alpha(T_\beta)$ for all $\beta\in B_\phi$.
\end{thml}

\begin{proof}
The set  $\fL^\alpha_T =\bigl\{ (\beta,\phi)\in (1,+\infty)\times \Holder{\alpha}(I) : \phi\in \Lock^\alpha(T_\beta) \bigr\}$ 
   contains an open dense subset of 
   $(1,+\infty) \times \Holder{\alpha}(I)$, by
   Theorem~\ref{t.product.typical.periodic}, and the space
    $(1,+\infty)$ is second countable. 
    A standard result from topology (see e.g.~\cite[Lemma~8.42]{Ke95}) then implies that there is a residual subset $R\subseteq\Holder{\alpha}(I)$ such that if $\phi \in R$ then $\{ \beta\in (1,+\infty) : \phi \in \Lock^\alpha(T_\beta) \}$ contains an open and dense subset of $(1,+\infty)$, as required.
\end{proof}

\subsection{A Joint TPO question for symbolic dynamics}

Beyond the dynamical systems for which we have established Joint TPO,
and in view of the fact that all of these systems admit a \emph{symbolic} description
(e.g.~via Markov partitions, or in terms of beta-shifts),
it is natural to ask whether a variant Joint TPO phenomenon holds in a purely symbolic
setting, where the space of all subshifts is equipped with a suitable topology,
allowing the phase space itself to be varied (instead of varying the dynamics on a fixed phase space).
We note that there has been some investigation of
typical properties of subshifts, both from a 
topological (see e.g.\ \cite{Hoc08, PS23}) and
a probabilistic (see e.g.\ \cite{McG12}) viewpoint.

Specifically, for $(\Sigma_d,\sigma)$ the one-sided full shift on $d\ge 2$
symbols, equipped with a standard metric, consider
the space 
\begin{equation*}
\fS_d \= \{X\subseteq \Sigma_d : X\neq \emptyset
           \text{ is closed and } \sigma(X)=X\}
\end{equation*}
of all subshifts of $(\Sigma_d,\sigma)$.
Equipped with the Hausdorff metric, $\fS_d$ is compact, and we
associate each $X\in\fS_d$ with the dynamical system
$(X,\sigma|_X)$.
For $\alpha\in(0,1]$,
we say that a pair $(X,\phi)\in \fS_d \times \Holder{\alpha}(\Sigma_d)$
has the \emph{periodic optimization property} if the maximizing
measure for $(\sigma|_X,\phi|_X)$ is unique and supported on a periodic
orbit of~$\sigma|_X$.
This motivates the following question.

\begin{question}\label{question}
For $\alpha\in(0,1]$, let $S$ be a suitable subset of $\fS_d$, for example consisting of all transitive subshifts of finite
type, or of all transitive sofic subshifts.
Does there exist an open dense subset $U\subseteq S\times
\Holder{\alpha}(\Sigma_d)$ such that every $(X,\phi)\in U$ has the
periodic optimization property?
\end{question}

Some care is needed in selecting
an appropriate set $S$
for Question~\ref{question}:
choosing it to consist of all 
transitive\footnote{The transitivity condition is imposed to avoid trivialities,
noting that a generic (not necessarily transitive) subshift is  somewhat degenerate (cf.~\cite{PS23}), with only finitely many ergodic invariant measures, each of which is periodic.} 
subshifts of finite
type is dynamically natural, and already poses challenges going beyond the techniques
introduced in this article (e.g.\ this $S$ is countable, and not a Baire space).
Moreover, although every transitive subshift of finite type is a Lipschitz distance-expanding
open map on its phase space, and therefore has Individual TPO by Theorem \ref{co}, 
the fact that our Joint TPO result for expanding maps (Theorem~\ref{expandingJTPO}) treats a \emph{fixed} compact
 metric space~$X$ 
(which moreover is assumed to be locally connected, unlike $\Sigma_d$)
means that it has no direct bearing on Question~\ref{question}.
While the set of pairs with the periodic optimization property is readily seen to be dense in
$S\times\Holder{\alpha}(\Sigma_d)$,  establishing
openness
appears to be a difficult problem, requiring new ideas.
In particular, the joint perturbation mechanism, a crucial ingredient in the proofs
of Joint TPO for the various systems treated in this
article, does not apply in the setting of Question~\ref{question},
in view of a more radical non-persistence of periodic orbits
than was the case for our treatment of beta-transformations
(where the structural monotonicity provided a degree of persistence).

\appendix
\section{Beta-transformations and maximizing measures: proofs}\label{beta_proofs_section}

The purpose of this appendix is to prove the results stated
in Section~\ref{betashiftsection}.

\begin{proof}[\bf Proof of Lemma~\ref{l_continuity_T_U_on_x}]
	Define functions $f,g_1,g_2 \: \R \to \R$ by
	\begin{equation*}
		f(u)\=\beta u,\quad g_1(u)\=u-\lfloor u \rfloor,\quad g_2(u)\=u-\lfloor u \rfloor'.
	\end{equation*}	
	Then $f$ is continuous and strictly increasing, and for each $u\in \R$, we have 
	\begin{equation}\label{e_g_1_a_2_f}
		\lim_{x\nearrow u} g_1(x)= g_2(u)^{-},\quad\,\lim_{x\searrow u} g_1(x)= g_1(u)^{+},\quad\,\lim_{x\nearrow u} g_2(x)= g_2(u)^{-}.
	\end{equation} 
	
	(i) By (\ref{e_def_T_beta}), we have $T_\beta=g_1\circ f$ and $T^n_\beta=g_1\circ f\circ T^{n-1}_\beta$. By (\ref{e_g_1_a_2_f}), $\lim_{x\searrow a}T_\beta(x)=T_\beta(a)^+$. Assume that $\lim_{x\searrow a}T^{n}_\beta(x)=T^{n}_\beta(a)^+$ when $n=k$. When $n=k+1$,  
	\begin{equation*}
		\lim_{x\searrow a}T^{k+1}_\beta(x)=\lim_{x\searrow a}T_\beta \bigl(T^{k}_\beta(x)\bigr)= \lim_{y\searrow T_\beta^k(x)} T_\beta(y) =T_\beta^{k+1}(a)^+.
	\end{equation*} 
	By induction, for all $n\in \N$, $\lim_{x\searrow a}T^n_\beta(x)=T_\beta^n(a)^+$. Similarly, we can prove $\lim_{x\nearrow b}U^n_\beta(x)=U_\beta^n(b)^-$. 
	
	(ii) By (\ref{en}) and (\ref{en*}), we have $\ve_n(\cdot,\beta)=\lfloor\cdot\rfloor \circ f\circ T_\beta^{n-1}$ and
	$\ve^*_n(\cdot,\beta)=\lfloor\cdot\rfloor' \circ f\circ U_\beta^{n-1}$. Hence (ii) follows from (i) and the fact that $\lfloor \cdot \rfloor$ is right-continuous and $\lfloor \cdot \rfloor'$ is left-continuous.  
	\smallskip
	
	(iii) follows immediately from (\ref{e_def_T_beta}), (\ref{en}), (\ref{en*}), and $U_\beta(0)=0$.
	\smallskip
	
	(iv) Since $T_\beta=g_1\circ f$ and $U_\beta=g_2\circ f$ (see (\ref{e_def_T_beta}) and (\ref{e_def_U_beta})), by (\ref{e_g_1_a_2_f}),
	\begin{equation*}
		\lim_{x\nearrow b}T_\beta(x)
		= \lim_{x\nearrow b}(g_1\circ f)(x)
		=\lim_{u\nearrow f(b)}g_1(u)
		=(g_2\circ f)(b)^-=U_\beta(b)^-.
	\end{equation*} 
	Assume that $\lim_{x\nearrow b}T_\beta^n(x)= U_\beta^n(b)^{-}$ holds for $n=k$. When $n=k+1$,
	\begin{equation*}
		\lim_{x\nearrow b}T^{k+1}_\beta(x)
		= \lim_{x\nearrow b}T_\beta\bigl(T^{k}_\beta(x)\bigr)
		= \lim_{u\nearrow U^k_\beta(b)}T_\beta(u)
		=U_\beta\bigl(U^{k}_\beta(b)\bigr)^-
		=U^{k+1}_\beta(b)^- .
	\end{equation*}
	
	Hence the first part of (iv) follows by induction.
	
	By (\ref{en}) and the first part of (iv),
	\begin{equation*}
		\lim_{x\nearrow b}\ve_n(x,\beta)=\lim_{x\nearrow b}(\lfloor\cdot \rfloor\circ f)\bigl(T_\beta^{n-1}(x)\bigr)=\lim_{u\nearrow U_\beta^{n-1}(b)}(\lfloor\cdot \rfloor\circ f)(u)=\lim_{v\nearrow f(U_\beta^{n-1}(b))}\lfloor v \rfloor.
	\end{equation*}
	By the fact that $\lim_{x\nearrow u}\lfloor x \rfloor=\lfloor u\rfloor'$ for all $u\in \R$ and (\ref{en*}),
	\begin{equation*}
		\lim_{v\nearrow f(U_\beta^{n-1}(b))}\lfloor v \rfloor
		=\bigl\lfloor \beta U_\beta^{n-1}(b)) \bigr\rfloor'
		=\ve^*_n(b,\beta).
	\end{equation*}
	The second part of (iv) follows from the above two equalities.
\end{proof}

\begin{proof}[\bf Proof of Lemma~\ref{l_continuity_T_U_on_beta}]
	Without loss of generality we can assume that $x\neq 0$ (see Lemma~\ref{l_continuity_T_U_on_x}~(iii)). Define functions $f,g_1,g_2 \: \R \to \R$ by
	\begin{equation*}
		f(u)\=xu,\quad g_1(u)\=u-\lfloor u \rfloor,\quad g_2(u)\=u-\lfloor u \rfloor'.
	\end{equation*}	
	Note that $f$ is continuous and strictly increasing and
	\begin{equation}\label{e_g_1_a_2_f_2}
		\lim_{x\nearrow u} g_1(x)= g_2(u)^{-},\quad\lim_{x\searrow u} g_1(x)= g_1(u)^{+},\quad\lim_{x\nearrow u} g_2(x)= g_2(u)^{-}.
	\end{equation}
	
	(i) By (\ref{e_def_T_beta}), we have $T_\beta(x)=(g_1\circ f)(\beta)$ and $T^n_\beta(x)=g_1 \bigl(\beta T^{n-1}_\beta(x)\bigr)$. By (\ref{e_g_1_a_2_f_2}), $\lim_{\gamma\searrow \beta} T_\gamma(x)= T_\beta(x)^+$. Since $\beta>1$, if $\lim_{\gamma\searrow \beta} T^{k-1}_\gamma(x)= T_\beta^{k-1}(x)^+$ for some $k\in \N$, we have $\lim_{\gamma\searrow \beta}\gamma T^{k-1}_\gamma(x)=\beta T_\beta^{k-1}(x)^+$. Then by (\ref{e_g_1_a_2_f_2}), $\lim_{\gamma\searrow \beta} T^{k}_\gamma(x)= T_\beta^{k}(x)^+$. By induction, for each $\beta>1$, $\lim_{\gamma\searrow \beta}T^n_\gamma(x)=T_\beta^n(x)^+$. Similarly, we have $\lim_{\gamma\nearrow \beta}U^n_\beta(x)=U_\beta^n(x)^-$.
	
	(ii) By (i), and the fact that $\beta>1$, we have  $\lim_{\gamma\searrow \beta}\gamma T^{n-1}_\gamma(x)=\beta T_\beta^{n-1}(x)^+$ and $\lim_{\gamma\nearrow \beta}\gamma U^{n-1}_\gamma(x)= \beta U_\beta^{n-1}(x)^-$. Since $\lfloor\cdot \rfloor$ is right-continuous and $\lfloor\cdot \rfloor'$ is left-continuous, (ii) follows from (\ref{en}) and (\ref{en*}).   
	\smallskip
	
	(iii) Since $T_\beta(x)=(g_1\circ f)(\beta)$ and $U_\beta(x)=(g_2\circ f)(\beta)$, by (\ref{e_g_1_a_2_f_2}), 
	\begin{equation*}
		\lim_{\gamma\nearrow \beta}T_\gamma(x)
		= \lim_{\gamma\nearrow \beta}g_1(\gamma x)
		= \lim_{u\nearrow f(\beta)}g_1(u)=g_2(f(\beta))^-=U_\beta(x)^-.
	\end{equation*} 
	Assume that $\lim_{\gamma\nearrow \beta}T_\gamma^n(x)= U_\beta^n(x)^{-}$ holds for $n=k$. When $n=k+1$, by (\ref{e_g_1_a_2_f_2}),
	\begin{equation*}
		\lim_{\gamma\nearrow \beta}T^{k+1}_\gamma(x)
		= \lim_{\gamma\nearrow \beta}g_1\bigl(\gamma T^{k}_\gamma(x)\bigr)
		=\lim_{u\nearrow \beta U^k_\beta(x)}g_1(u)
		=g_2\bigl(\beta U^{k}_\beta(x)\bigr)^-=U^{k+1}_\beta(x)^-.
	\end{equation*}
	
	Hence the first part of (iii) follows from induction.
	
	By (\ref{en}), (\ref{en*}), the first part of (iii), and the fact that $\lim_{x\nearrow u}\lfloor x \rfloor=\lfloor u\rfloor'$ for all $u\in \R$,
	\begin{equation*}
		\lim_{\gamma\nearrow \beta}\ve_n(x,\gamma)=\lim_{\gamma\nearrow \beta}\bigl\lfloor \gamma T_\gamma^{n-1}(x) \bigr\rfloor=\lim_{u\nearrow \beta U_\beta^{n-1}(x)}\lfloor u \rfloor= \bigl\lfloor \beta U_\beta^{n-1}(x) \bigr\rfloor'=\ve_n^*(x,\beta).
	\end{equation*}
	Therefore, the second part of (iii) follows.
\end{proof}

\begin{proof}[\bf Proof of Lemma~\ref{l_equivalent_def_pi_*}]
	The first part follows from Lemma~\ref{l_continuity_T_U_on_x}~(iii), whereas the second part follows from Lemma~\ref{l_continuity_T_U_on_x}~(iv).
\end{proof}

\begin{proof}[\bf Proof of Proposition~\ref{p_relation_of_coding}]
	\smallskip
	Statements~\ref{p_relation_of_coding__i}, \ref{p_relation_of_coding__ii}, \ref{p_relation_of_coding__viii}, and~\ref{p_relation_of_coding__ix} follow from \cite[Lemma~1.2]{YT21}.
	
	\smallskip
	The statements about $\pi_\beta$ in \ref{p_relation_of_coding__iv} and~\ref{p_relation_of_coding__vi} follow from \cite[Proposition~3.2]{IT74}, while statements about $\pi^*_\beta$ in~\ref{p_relation_of_coding__iv} and~\ref{p_relation_of_coding__vi} follow from \cite[Lemma~1.2]{YT21}.
	
	\smallskip
	Statements~\ref{p_relation_of_coding__x} and~\ref{p_relation_of_coding__xi} follow from \cite[Proposition~3.2]{IT74}.
	
	\smallskip
	It remains to prove statements~\ref{p_relation_of_coding__iii}, \ref{p_relation_of_coding__v}, \ref{p_relation_of_coding__vii}, and~\ref{p_relation_of_coding__xii}--\ref{p_relation_of_coding__xiv}.
	
	\smallskip
	\ref{p_relation_of_coding__iii} follows from (\ref{en}), (\ref{en*}), and Definition~\ref{beta shifts}.
	
	\smallskip
	\ref{p_relation_of_coding__v} From~\ref{p_relation_of_coding__iii} and~\ref{p_relation_of_coding__iv}, we have $h_\beta \circ \sigma \circ \pi_\beta = h_\beta \circ \pi_\beta \circ T_\beta = T_\beta = T_\beta \circ h_\beta \circ \pi_\beta$. Thus, $h_\beta \circ \sigma = T_\beta \circ h_\beta$ on $\pi_\beta(I)$. Similarly, we have $h_\beta \circ \sigma = U_\beta \circ h_\beta$ on $\pi^*_\beta(I)$.
	
	\smallskip
	\ref{p_relation_of_coding__vii} Consider arbitrary $x,y\in I$ with $0\leq x<y\leq 1$. By Lemma~\ref{l_equivalent_def_pi_*}, we have $\lim_{z\nearrow y}\pi_\beta(z)=\pi_\beta^*(y)$. Combining this with the fact that $\pi_\beta$ is strictly increasing (see~\ref{p_relation_of_coding__vi}), we obtain $\pi_\beta(x)\prec \pi_\beta^*(y)$.
	
	\smallskip
	\ref{p_relation_of_coding__xii} Fix $\beta>1$. Consider a pair of sequences $A= a_1 a_2 \dots$ and $B= b_1 b_2 \dots$ in $X_\beta$. Assume that $d_\beta(A,B) = \beta^{-k}$ for some integer $k \in \N$. By (\ref{hbeta}), we have
	$
		\abs{h_\beta(A) - h_\beta(B)} \le \sum_{n=k}^{+\infty} \frac{\abs{a_n-b_n }}{\beta^n} \le \beta \sum_{n=k}^{+\infty} \frac{1}{\beta^n} = \frac{1}{\beta^{k-2} (\beta-1)} = \frac{\beta^2}{\beta-1} d_\beta(A,B)$.
	
	\ref{p_relation_of_coding__xiii} By (\ref{ix beta}) and Lemma~\ref{l_equivalent_def_pi_*}, $i_0(\beta)=i_0^*(\beta)=(0)^\infty$ for all $\beta>1$. 
	
	Fix an arbitrary $x\in (0,1]$. It is easy to see that $i_x(\beta_1)\neq i_x(\beta_2)$ when $\beta_1\neq \beta_2$. So it suffices to prove that $i_x$ is nondecreasing. Assume that there exist $1<\beta_1<\beta_2$ such that $i_x(\beta_2)\prec i_x(\beta_1)$. Then there exists $n\in \N$ such that
	$\ve_k(\beta_1,x)=\ve_k(\beta_2,x)$ for all $k\in\{1,\, \dots,\, n-1\}$ and $\ve_n(\beta_1,x)>\ve_n(\beta_2,x)$. Then we have $\ve_n(\beta_1,x)\geq \ve_n(\beta_2,x)+1\geq 1$. By~\ref{p_relation_of_coding__iii} and~\ref{p_relation_of_coding__iv}, for each $\beta>1$, we have
	$
		x=h_\beta(\pi_\beta(x))=\sum_{k=1}^{n-1}\frac{\ve_k(x,\beta)}{\beta^k}+\frac{\ve_{n}(x,\beta)}{\beta^{n}}+\frac{T_\beta^n(x)}{\beta^n}$.

	So we obtain
	$
			x
			=\sum_{k=1}^{n-1}\frac{\ve_k(x,\beta_2)}{\beta^k_2}+\frac{\ve_{n}(x,\beta_2)}{\beta_2^{n}}+\frac{T_{\beta_2}^n(x)}{\beta_2^n}
			\leq \sum_{k=1}^{n-1}\frac{\ve_k(x,\beta_2)}{\beta^k_2}+\frac{\ve_{n}(x,\beta_2)+1}{\beta_2^{n}}
			\leq\sum_{k=1}^{n}\frac{\ve_k(x,\beta_1)}{\beta^k_2} <\sum_{k=1}^{n}\frac{\ve_k(x,\beta_1)}{\beta^k_1}
			\leq x$,
	which leads to a contradiction. So $i_x$ is strictly increasing. The proof that $i_x^*$ is strictly increasing follows by using a similar argument.
	
	\smallskip 	
	\ref{p_relation_of_coding__xiv} Fix $x\in I$. By Lemma~\ref{l_continuity_T_U_on_beta}~(ii), we have that $\ve_n(x,\cdot)$ is right-continuous and $\ve^*_n(x,\cdot)$ is left-continuous for each $n\in\N$. By Definition~\ref{beta shifts} and~(\ref{ix beta}), \ref{p_relation_of_coding__xiv} follows.
\end{proof}

\begin{proof}[\bf Proof of Proposition~\ref{p_relation_T_beta_and_wt_T_beta}]
	(i) These properties follow immediately from the definitions of $T_\beta$, $U_\beta$, and $D_\beta$.
	
	\smallskip
	(ii) Fix a periodic orbit $\cO_\beta$ of $T_\beta$ that is not $\{0\}$. By (i) we have $1\notin \cO_\beta$ and $\cO_\beta\cap D_\beta=\emptyset$. Then $U_\beta(x)=T_\beta(x)$ for all $x\in \cO_\beta$. Hence $\cO_\beta$ is also a periodic orbit of $U_\beta$. 
	
	Fix a periodic orbit $\cO_\beta^*$ of $U_\beta$. If $1\in \cO_\beta^*$, by (i) we get that $\cO_\beta^*$ is not a periodic orbit of $T_\beta$. If $1\notin \cO_\beta^*$, by (i) we have $\cO_\beta^*\cap D_\beta=\emptyset$. Then we have $U_\beta(x)=T_\beta(x)$ for all $x\in \cO_\beta^*$. Hence $\cO_\beta^*$ is a periodic orbit for $T_\beta$.
	
	\smallskip
	(iii) Fix an arbitrary $\mu\in \MMM(I,T_\beta)$. By (i) we have $\mu(\{1\})=\mu\bigl(T_\beta^{-1}(1)\bigr)=\mu(\emptyset)=0$ and $\mu(D_\beta)=\mu\bigl(T_\beta^{-1}(0)\bigr)-\mu(\{0\})=0$. Then we have $\mu\bigl(U_\beta^{-1}(0)\bigr)=\mu(\{0\})$ and $\mu\bigl(U_\beta^{-1}(1)\bigr)=\mu(D_\beta)=0=\mu(\{1\})$. By definition we have $T_\beta^{-1}(Y)=U_\beta^{-1}(Y)$ for all Borel measurable subsets $Y\subseteq (0,1)$. Hence, we have $\mu\in \MMM(I,U_\beta)$. 
	
	Now fix an arbitrary $\nu\in \MMM(I,U_\beta)$. If $\nu(\{1\})=0$, then by (i) we have $\nu\bigl(T_\beta^{-1}(1)\bigr)=\nu(\emptyset) = 0 = \nu(\{1\})$, $\nu(D_\beta)=\nu\bigl(U_\beta^{-1}(1)\bigr)=\nu(\{1\})=0$, and $\nu\bigl(T_\beta^{-1} (0) \bigr) =\nu(\{0\})+\nu(D_\beta)=\nu(\{0\})$. By definition we have $T_\beta^{-1}(Y)=U_\beta^{-1}(Y)$ for all Borel measurable subsets $Y\subseteq (0,1)$. Hence, we have $\nu\in \MMM(I,T_\beta)$. If on the other hand $\nu(\{1\})>0$, since $T_\beta^{-1}(1)=\emptyset$ by (i), we have $\nu\notin\MMM(I,T_\beta)$.
	
	\smallskip
	(iv) and (v) Recall that $1$ is a periodic point of $U_\beta$ if and only if $\beta$ is a simple beta-number (see Remark~\ref{r_after_classification_of_beta}). Then the first part of (iv) and the first part of (v) follow immediately from (ii).
	
	Fix an arbitrary $\mu\in \MMM(I,U_\beta)$. If $\beta$ is not a simple beta-number, then $1$ is not a periodic point of $U_\beta$ (see Remark~\ref{r_after_classification_of_beta}) and it is straightforward to check that $\mu(\{1\})=0$. 
	By (iii), $\mu\in \MMM(I ,T_\beta)$. The second part of (iv) follows.
	
	Assume that $\beta$ is a simple beta-number. Since $\cO_\beta^*(1)$ is a periodic orbit of $U_\beta$, then it is easy to see that $\mu(\{x\})=\mu(\{y\})$ for all $x,\, y\in \cO_\beta^*(1)$. Write $t\=\mu \bigl( \cO_\beta^*(1) \bigr)$. When $t=1$, we have $\mu(\{x\})=1 \big/ \card  \cO_\beta^*(1)$ for all $x\in \cO_\beta^*(1)$. In this case, $\mu=\mu_{\cO_\beta^*(1)}$. When $t\in [0,1)$, let us write $\nu\=\frac{1}{1-t}\bigl(\mu-t\mu_{\cO_\beta^*(1)}\bigr)$. Then $\nu\in \MMM(I,U_\beta)$ and $\nu(\{1\})=0$. By (iii), $\nu\in \MMM(I ,T_\beta)$. In this case, $\mu=t\mu_{\cO_\beta^*(1)}+(1-t)\nu$. The second part of (v) follows.
	
	\smallskip
	(vi) Fix arbitrary $x, \, y \in I$ with $0 < y-x < 1/(2\beta)$. If there exists an integer $i$ such that $i/\beta \le x < y < (i+1)/\beta$, then $T_\beta(y) - T_\beta(x) = \beta(y-x)$. Otherwise, there exists an integer $i$ such that $(i-1)/\beta < x < i/\beta \le y < (i+1)/\beta$. Then we have $T_\beta(y) - T_\beta(x) = \beta(y-x) - 1 < -1/2 < -\beta\abs{y-x}$. Similarly, we can prove $\abs{U_\beta(y) - U_\beta(x)} \geq \beta \abs{y-x}$.
\end{proof}

\begin{proof}[\bf Proof of Lemma~\ref{support_invariant_set}]
	Let us write $\cK\=\supp \mu$.
	
	Assume that $0 \notin \cK$ and denote $\delta_1\=d(\cK,0)>0$. 
	By (\ref{e_def_U_beta}), for each $y\in D_\beta$, we obtain
	\begin{equation*}
		\mu((y,y+\delta_1/\beta)\cap I)\leq \mu\bigl(U_\beta^{-1}(0,\delta_1)\bigr)=\mu((0,\delta_1))=0.
	\end{equation*}
	So $\cK\cap(y,y+\delta_1/\beta)=\emptyset$ for each $y\in D_\beta$. Hence for each pair of $x,\, y\in \cK$ with $\abs{x-y}<\delta_1/\beta$, we have $(x,y)\cap D_\beta=\emptyset$ and $U_\beta(x)-U_\beta(y)=\beta(x-y)$. So $U_\beta|_{\cK}$ is continuous and $\mu$ can be seen as an invariant measure for $(\cK,U_\beta|_{\cK})$. Therefore, $U_\beta(\cK)=\cK$ 
	(\cite[p.~156]{Ak93}). 
	
	Assume that $1 \notin \cK$ and denote $\delta_2\=d(\cK,1)>0$. By Proposition~\ref{p_relation_T_beta_and_wt_T_beta}~(iii), $\mu\in \MMM(I,T_\beta)$. 
	By (\ref{e_def_T_beta}), for each $y\in D_\beta$, we obtain
	$
		\mu((y-\delta_2/\beta,y)\cap I)\leq \mu\bigl(T_\beta^{-1}(1-\delta_2, 1)\bigr)=\mu((1-\delta_2, 1))=0$.
    So $\cK\cap(y-\delta_2/\beta,y)=\emptyset$ for each $y\in D_\beta$. Hence for each pair of $x,\, y\in \cK$ with $\abs{x-y}<\delta_2/\beta$, we have $(x,y)\cap D_\beta=\emptyset$ and $T_\beta(x)-T_\beta(y)=\beta(x-y)$. So $T_\beta|_{\cK}$ is continuous and  $\mu$ can be seen as an invariant measure for $(\cK,T_\beta|_{\cK})$. Therefore, $T_\beta(\cK)=\cK$ (\cite[p.~156]{Ak93}). 
\end{proof}

\begin{proof}[\bf Proof of Proposition~\ref{2.11}]
	(i) follows from \cite[Theorem~3]{Par60} and Proposition~\ref{p_relation_of_coding}~\ref{p_relation_of_coding__viii}, while (ii) is exactly \cite[Lemma~4.4]{IT74}.
\end{proof}

\begin{proof}[\bf Proof of Lemma~\ref{l_apprioxiation_beta_shif_beta}]
	\smallskip
	(i) Assume that $1<\beta'<\beta$. By Proposition~\ref{p_relation_of_coding}~\ref{p_relation_of_coding__xiii} and (\ref{ix beta}), we have $\pi^*_{\beta'}(1)\prec \pi^*_\beta(1)$. By (\ref{e_def_S_beta}), $\cS_{\beta'}\subseteq \cS_\beta$.

	\smallskip
	(ii) Assume that $\pi^*_\beta(1) = a_1 a_2 \dots$. For each $n\in\N$, put $A_n \= a_1 \dots a_n 0 0 \dots$. By Proposition~\ref{p_relation_of_coding}~\ref{p_relation_of_coding__xiv}, we get that $\pi^*_\beta(1) = \lim_{\gamma \nearrow \beta} \pi^*_\gamma(1)$. Thus, for each $n\in \N$, there exists $\gamma_n \in (1,\beta)$ such that $ A_n\preceq \pi_{\gamma_n}^*(1)$. Fix arbitrary $B = b_1 b_2 \ldots\in \cS_\beta$. Put $B_n \= b_1 \dots b_n 0 0 \dots$ for each $n\in \N$. By (\ref{e_def_S_beta}), for each $k\in \N_0$ and $n\in \N$, we have $\sigma^k(B_n) \preceq A_n\preceq \pi_{\gamma_n}^*(1)$. Thus, $B_n \in \cS_{\gamma_n}$. Note that $\lim_{n\to +\infty} B_n = B$, so $B \in \overline{\bigcup_{\gamma\in(1,\beta)}\cS_\gamma}$. Since $B$ is chosen arbitrarily, we obtain from (i) that $\cS_\beta=\overline{\bigcup_{\gamma\in(1,\beta)}\cS_\gamma}$.
\end{proof}

\begin{proof}[\bf Proof of Lemma~\ref{H beta gamma}]
	If $\cK$ is a nonempty compact set with $1\notin \cK=T_\beta(\cK)$, then
	the largest point in $\cK$ is strictly smaller than 1.
	By Proposition~\ref{p_relation_of_coding}~\ref{p_relation_of_coding__vi},~\ref{p_relation_of_coding__xiv}, and (\ref{ix beta}), there exists $\beta' \in (1,\beta)$ such that $ \max \{\pi^*_\beta(x) : x\in \cK\}\prec \pi^*_{\beta'} (1) .$ Furthermore, by Proposition~\ref{p_relation_of_coding}~\ref{p_relation_of_coding__vi},~\ref{p_relation_of_coding__viii}, and the fact that $\pi^*_{\beta'}(1)\in X_\beta$ (see (\ref{e_def_S_beta}) and Lemma~\ref{l_apprioxiation_beta_shif_beta}~(i)), we have $\max \{\pi_\beta(x) : x\in \cK\}\preceq \pi^*_{\beta'} (1) .$ By Proposition~\ref{p_relation_of_coding}~\ref{p_relation_of_coding__iii}, we get $\sigma(\pi_\beta(\cK)) = \pi_\beta(\cK)$. So if $z\in \cK$ 
	then $\sigma^n(\pi_\beta(z)) \preceq \pi^*_{\beta'}(1)$ for all $n\in \N_0$, and thus by (\ref{e_def_S_beta}), $\pi_\beta(z)\in \cS_{\beta'}$. Hence, by the definition of $H_\beta^\gamma$ and Proposition~\ref{p_relation_of_coding}~\ref{p_relation_of_coding__iv}, we have $\cK = h_\beta(\pi_\beta(\cK)) \subseteq h_\beta(\cS_{\beta'}) = H_\beta^{\beta'}$. By (\ref{Hbetagamma}) and Lemma~\ref{l_apprioxiation_beta_shif_beta}~(i), we have $H_\beta^{\beta'} \subseteq H_\beta^{\gamma}$ for each $\gamma \in (\beta', \beta)$. Hence, $\cK \subseteq H_\beta^\gamma$ for each $\gamma \in (\beta', \beta)$.
	
	Now let $\cK^*$ be a nonempty compact set with $1 \notin \cK^* = U_\beta(\cK^*)$. Applying Proposition~\ref{p_relation_T_beta_and_wt_T_beta}~(i), we have $\cK^* \cap D_\beta =\emptyset$, so $T_\beta(x) = U_\beta(x)$ for each $x\in \cK^*$. Thus, $T_\beta(\cK^*)=\cK^*$. Therefore, there exists $\beta' \in (1,\beta)$ such that $\cK^* \subseteq H_\beta^\gamma$ for each $\gamma \in  (\beta', \beta)$.
\end{proof}

\begin{proof}[\bf Proof of Lemma~\ref{l_bi_lipschitz_S_gamma_H_gamma}]
	Define $\delta\=d\bigl(H_\beta^\gamma,1\bigr)$. By Proposition~\ref{p_relation_of_coding}~\ref{p_relation_of_coding__i} and (\ref{ix beta}) and~\ref{p_relation_of_coding__xiii}, we have $\pi_\gamma^*(1)\prec \pi_\beta^*(1)\preceq \pi_\beta(1)$. Hence $\pi_\beta(1),\,\pi_\beta^*(1)\notin \cS_\gamma$  (see (\ref{e_def_S_beta})). So by Proposition~\ref{p_relation_of_coding}~\ref{p_relation_of_coding__x},~\ref{p_relation_of_coding__xi}, and (\ref{Hbetagamma}), we have $1\notin H_\beta^\gamma$ and $0<\delta \leq 1$.
	
	Assume that $x,y\in H_\beta^\gamma$ satisfy $d_\beta(\pi_\beta(x),\pi_\beta(y))=\beta^{-n}$ and $x<y$, then
	\begin{equation*}
		\pi_\beta(x)=a_1\dots a_{n-1}b_n b_{n+1}\dots,\quad\pi_\beta(y)=a_1\dots a_{n-1}c_n c_{n+1}\dots,
	\end{equation*}
	where $b_n< c_n$. Then by Proposition~\ref{p_relation_of_coding}~\ref{p_relation_of_coding__iv} and the definition of $h_\beta$, 
	\begin{equation}\label{dxy}
     \begin{aligned}
		d(x,y)
        &= d(h_\beta(\pi_\beta(x)), h_\beta(\pi_\beta(y))) \\
        &= \beta^{-n+1} d(h_\beta(b_n b_{n+1} \dots), h_\beta(c_n c_{n+1} \dots)) 
        \le \beta^{-n+1}. 
     \end{aligned}
	\end{equation} 
	Moreover, by the definition of $h_\beta$ and $H_\beta^\gamma$, we have 
	\begin{align*}
		h_\beta(b_n b_{n+1} \dots) &= (b_n  + h_\beta(b_{n+1} \dots))/\beta \le (b_n+1-\delta)/\beta        \quad \text{ and} \\
		h_\beta(c_n c_{n+1} \dots) &= (c_n  + h_\beta(c_{n+1} \dots))/\beta \ge c_n/\beta \ge (b_n+1)/\beta.
	\end{align*}
	So we have $d(h_\beta(b_n b_{n+1} \dots), h_\beta(c_n c_{n+1} \dots)) \ge \delta/\beta$. Thus, by (\ref{dxy}), we have 
	\begin{equation}\label{dxy<}
		d(x,y)\geq \delta\beta^{-n}=\delta d_\beta(\pi_\beta(x),\pi_\beta(y)).
	\end{equation}	
	Let us write $C\=\max\{\beta,\, 1/\delta\}\geq 1$. Combining (\ref{dxy}) and (\ref{dxy<}), we have
	\begin{equation*}
		C^{-1}d(x,y)\leq d_\beta(\pi_\beta(x),\pi_\beta(y))\leq C d(x,y).\qedhere
	\end{equation*}
\end{proof}

\begin{proof}[\bf Proof of Lemma~\ref{l_property_H_gamma}]
	By (\ref{e_def_S_beta}), $\cS_\gamma$ is closed and $\sigma(\cS_\gamma)\subseteq \cS_\gamma$. By Proposition~\ref{p_relation_of_coding}~\ref{p_relation_of_coding__iv}, (\ref{Hbetagamma}), and Lemma~\ref{l_bi_lipschitz_S_gamma_H_gamma}, $\pi_\beta|_{H_\beta^\gamma}$ is bi-Lipschitz with inverse $h_\beta|_{\cS_\gamma}$ and $H_\beta^\gamma$ is closed. By  Proposition~\ref{p_relation_of_coding}~\ref{p_relation_of_coding__iii}, $T_\beta \bigl( H_\beta^\gamma \bigr)\subseteq H_\beta^\gamma$. Since $T_\beta|_{H_\beta^\gamma} =h_\beta|_{\cS_\gamma}  \circ \sigma|_{\cS_\gamma} \circ \pi_\beta|_{H_\beta^\gamma}$ (see Proposition~\ref{p_relation_of_coding}~\ref{p_relation_of_coding__iii} and~\ref{p_relation_of_coding__v}), we obtain (i).
	
	\smallskip
	(ii) follows from Proposition~\ref{p_relation_T_beta_and_wt_T_beta}~(vi).
	
	\smallskip
	To verify (iii), assume that $\gamma$ is a simple beta-number. Then $(\cS_\gamma,\sigma)$ is a subshift of finite type (see e.g.~\cite[Proposition~4.1]{Bl89}), and therefore $\sigma|_{\cS_\gamma}$ is open (see e.g.~\cite[Theorem~3.2.12]{URM22}).
	By Lemma~\ref{l_bi_lipschitz_S_gamma_H_gamma} and Proposition~\ref{p_relation_of_coding}~\ref{p_relation_of_coding__iv}, we know that $\pi_\beta|_{H_\beta^\gamma}$ and $h_\beta|_{\cS_\gamma}$ are homeomorphisms. It follows that $T_\beta|_{H_\beta^\gamma}=h_\beta|_{\cS_\gamma} \circ \sigma|_{\cS_\gamma} \circ \pi_\beta|_{H_\beta^\gamma}$ is open. 
\end{proof}

	\begin{proof}[\bf Proof of Proposition~\ref{p_cylinders}]
		Denote $\pi^*_\beta(1) = a_1 a_2 \dots$ in this proof. Recall that $(\pi_\beta \circ T_\beta)(x) =( \sigma \circ \pi_\beta)(x)$ and $(h_\beta \circ \pi_\beta)(x) = x$ for each $x\in I$ (see Proposition~\ref{p_relation_of_coding}~\ref{p_relation_of_coding__iii} and~\ref{p_relation_of_coding__iv}). By (\ref{e_def_n_cylinders}) and Proposition~\ref{p_relation_of_coding}~\ref{p_relation_of_coding__vi}, $x\in I^n$ if and only if
		\begin{equation}\label{e_equi_def_n_cylinders}
			\myepsilon_{1}\myepsilon_{2}\dots \myepsilon_n (0)^\infty\preceq \pi_\beta(x)\prec \myepsilon_{1}\myepsilon_{2}\dots (\myepsilon_n+1) (0)^\infty.
		\end{equation}
		
		\smallskip
		(i) By (\ref{e_def_n_cylinders}) and Proposition~\ref{p_relation_of_coding}~\ref{p_relation_of_coding__vi}, $I^n_1\cap I^n_2=\emptyset$ for each $I^n_1,\, I^n_2\in W^n$ with $I^n_1\neq I^n_2$. For each $x\in [0,1)$, assume $\pi_\beta(x)=x_1 x_2\dots x_n\dots$. Then $x\in I(x_1,\dots,x_n)$. So $[0,1) = \bigcup_{J^n \in W^n} J^n$. 
		
		\smallskip
		(ii) For arbitrary $x, \, y \in I^n$, assume that $\pi_\beta(x) = \myepsilon_1  \dots  \myepsilon_n  x_1  x_2  \dots$ and $\pi_\beta(y) = \myepsilon_1  \dots  \myepsilon_n  y_1  y_2  \dots$. So 
        $$y-x = h_\beta(\pi_\beta(y)) - h_\beta(\pi_\beta(x)) = \beta^{-n}\sum_{i=1}^{+\infty} \beta^{-i} (y_i-x_i)$$ and 
		$
			T_\beta^m(y) - T_\beta^m(x)
			= h_\beta (\sigma^m (\pi_\beta(y))) - h_\beta (\sigma^m (\pi_\beta(x))) 
			= \beta^{m-n}\sum_{i=1}^{+\infty} \beta^{-i} (y_i-x_i) 
			= \beta^m(y-x)$,
		as required.
		
		\smallskip
		(iii) Fix arbitrary $b \in \{0, \, \dots, \, \myepsilon_n-1\}$. Since $(\ve_1,  \dots,  \ve_n)$ is admissible, by Proposition~\ref{2.11}~(i), there exist $y_b, \, z_b$ with $\pi_\beta(y_b) = \ve_1 \ve_2 \dots \ve_{n-1} b (0)^\infty$ and $\pi_\beta(z_b) = \ve_1 \ve_2 \dots \ve_{n-1} (b+1) (0)^\infty$. Since $\pi_\beta$ is strictly increasing (see Proposition~\ref{p_relation_of_coding}~\ref{p_relation_of_coding__vi}), using (\ref{e_equi_def_n_cylinders}), $x\in I(\ve_1,  \dots,  \ve_{n-1}, b)$ if and only if $y_b\le x  < z_b$. Now using (\ref{hbeta}), $\diam I(\myepsilon_1,\dots,\myepsilon_{n-1},b)= z_b-y_b = h_\beta(\pi_\beta(z_b)) -h_\beta (\pi_\beta (y_b)) = \beta^{-n}$. Moreover, $T_\beta^n(I(\myepsilon_1,\dots,\myepsilon_{n-1},b))=[0,1)$ by (ii). Therefore $I(\myepsilon_1,\dots,\myepsilon_{n-1},b)\in W_0^n$ with the right endpoint $z_b = h_\beta (\ve_1 \dots \ve_{n-1} (b+1) (0)^\infty) = (b+1)\beta^{-n} + \sum_{i=1}^{n-1}\myepsilon_i \beta^{-i}$.   
		
		\smallskip
		(iv) By (ii), $T_{\beta}^n|_{I^n}$ is continuous and increasing. Write $I^n = [x,y)$ and assume $y<1$. As $T_\beta(I^n) = [0,1)$, so $T_\beta(x) =0 \le x$ and $\lim_{z\nearrow y} T_\beta(z) = 1>y$. By the intermediate value theorem, $T_{\beta}^n$ has a fixed point in $I^n$.
		
		\smallskip
		(v) 
		Denote
		\begin{equation*}
			m\=\max(\{j\in\N:\myepsilon_{n-j+i} = a_i\text{ for all }1\leq i\leq j \}\cup\{0\} ).
		\end{equation*} 
		Let $y$ be the left endpoint of $I^n$ and $z\=h_\beta(A)$ with			 
		\begin{equation}\label{e_max_I^n}
			A \= \myepsilon_1  \dots  \myepsilon_{n-m}\pi_\beta^*(1) 
               = \myepsilon_1  \dots  \myepsilon_n  a_{m+1}  a_{m+2}  \dots.
		\end{equation}		
		We first check that $A \in \cS_\beta$. Fix $k \in \N_0$ arbitrarily. If $k<n-m$, then $\myepsilon_{k+1}  \dots \myepsilon_n (0)^\infty \prec a_1  \dots  a_{n-k} (0)^\infty$ by the maximality of $m$ and Proposition~\ref{p_relation_of_coding}~\ref{p_relation_of_coding__vi}, and therefore $\sigma^k(A) \prec \pi^*_\beta(1)$. If $k \ge n-m$, then $\sigma^k(A) = \sigma^{k-(n-m)}\bigl(\pi^*_\beta(1)\bigr) \preceq \pi^*_\beta(1)$ by (\ref{e_max_I^n}) and Proposition~\ref{p_relation_of_coding}~\ref{p_relation_of_coding__vi}. We obtain $A \in \cS_\beta$ by (\ref{e_def_S_beta}). Since $h_\beta(A)=z$ and $A=\myepsilon_1  \dots  \myepsilon_{n-m}\pi_\beta^*(1)$, by Proposition~\ref{p_relation_of_coding}~\ref{p_relation_of_coding__ii} and~\ref{p_relation_of_coding__xi}, $A=\pi_\beta^*(z)$.
		
		Consider arbitrary $x\in I^n$, then $\pi_\beta(x)\in \cS_\beta$ by Definition~\ref{beta shifts}. So $\sigma^{n-m}(\pi_\beta(x)) \preceq \pi_\beta^*(1) = \sigma^{n-m}(A)$ by Proposition~\ref{2.11}~(ii) and (\ref{e_max_I^n}). Combining this with the fact that the first $n-m$ terms of $\pi_\beta(x)$ and $A$ coincide (see (\ref{e_equi_def_n_cylinders})), we get $\pi_\beta(x) \preceq A$. Since $h_\beta$ is nondecreasing (see Proposition~\ref{p_relation_of_coding}~\ref{p_relation_of_coding__x}), $x\leq z=h_\beta(A)$ for all $x\in I^n$ by Proposition~\ref{p_relation_of_coding}~\ref{p_relation_of_coding__iv}. 
		Consider arbitrary $w\in[y,z)$. By Proposition~\ref{p_relation_of_coding}~\ref{p_relation_of_coding__vii} and~\ref{p_relation_of_coding__ii}, $\pi_\beta(w)\preceq \pi_\beta^*(z)=A\preceq \pi_\beta(z)$. Hence $w\in I^n$ for all $w\in[y,z)$ by (\ref{e_equi_def_n_cylinders}) and (\ref{e_max_I^n}).
		
		By our discussion above, $z$ is the right endpoint of $I^n$. Moreover, by (ii) and (\ref{left_endpoint}), $T_{\beta}^n(I^n)$ is a left closed and right open interval with left endpoint $0$ and right endpoint $\beta^n(z-y)=h_\beta\bigl(\sigma^m \bigl(\pi^*_\beta(1)\bigr)\bigr) = U_\beta^m(1)$ (see Proposition~\ref{p_relation_of_coding}~\ref{p_relation_of_coding__iii} and~\ref{p_relation_of_coding__iv} for the last equality).
	\end{proof}

\begin{proof}[\bf Proof of Lemma~\ref{l_Distortion}]
	By Proposition~\ref{p_cylinders}~(ii), we have
    \begin{equation*}
\begin{split}
    |S_{n}\phi(x)-S_{n}\phi(y)|
    &\leq \Hseminorm{\alpha}{\phi} \sum_{i=0}^{n-1} \Absbig{T^i_{\beta}(x)-T^i_{\beta}(y)}^\alpha \\
    &= \Hseminorm{\alpha}{\phi} \sum_{i=0}^{n-1} \Absbig{T_{\beta}^n(x)-T_{\beta}^n(y)}^\alpha \beta^{-(n-i)\alpha} \\
    &\leq \Hseminorm{\alpha}{\phi} \Absbig{T_{\beta}^n(x)-T_{\beta}^n(y)}^{\alpha} \big/ (\beta^\alpha-1).\qedhere
\end{split}
\end{equation*}
\end{proof}

\begin{proof}[\bf Proof of Lemma~\ref{l_properties_z_beta}]
(i) Fix an arbitrary $x\in I$. By (\ref{e_def_Z_beta}), Proposition~\ref{p_relation_of_coding}~\ref{p_relation_of_coding__ii}, and Lemma~\ref{l_equivalent_def_pi_*}, $x\in Z_\beta$ if and only if there exists $n\in\N$ and $z_1,\dots,z_n\in \N$ with $z_n>0$ such that
$
	\pi_\beta(x)=z_1\dots z_n (0)^\infty$.
By Proposition~\ref{p_relation_of_coding}~\ref{p_relation_of_coding__iv} and~\ref{p_relation_of_coding__v}, $T_\beta^n(x)=\bigl(T_\beta^n\circ h_\beta\circ\pi_\beta\bigr)(x)=(h_\beta\circ \sigma^n\circ\pi_\beta)(x)$. So the condition above is equivalent to the condition that there exists $n\in\N$ such that $x\in T_\beta^{-n}(0)\smallsetminus T_\beta^{-(n-1)}(0)$. Hence
\begin{equation*}
	Z_\beta=\bigcup_{n\in\N}\bigl(T_\beta^{-n}(0)\smallsetminus T_\beta^{-(n-1)}(0)\bigr)=\Bigl(\bigcup_{n\in\N}T_\beta^{-n}(0)\Bigr)\smallsetminus\{0\}.
\end{equation*}
Similarly, we can prove that $Z_\beta=\bigcup_{n\in\N}U_\beta^{-n}(1)$. In particular, by Proposition~\ref{p_relation_T_beta_and_wt_T_beta}~(i), $D_\beta=U_\beta^{-1}(1)\subseteq Z_\beta$.

 (ii) Fix an arbitrary $W\subseteq I$. By Proposition~\ref{p_relation_of_coding}~\ref{p_relation_of_coding__xi}, we have $h_\beta^{-1}(W)=\pi_\beta(W)\cup\pi_\beta^*(W)$. By (\ref{e_def_Z_beta}), 
\begin{equation*}
 \begin{aligned}
 	h_\beta^{-1}(W)
 	&=\pi^*_\beta(W)\cup\pi_\beta(W)
 	=\pi_\beta^*(W)\cup(\pi_\beta(W\smallsetminus Z_\beta)\cup\pi_\beta(W\cap Z_\beta))\\
 	&=\bigl(\pi_\beta^*(W)\cup\pi^*_\beta(W\smallsetminus Z_\beta)\bigr)\cup\pi_\beta(W\cap Z_\beta)
 	=\pi_\beta^*(W)\cup\pi_\beta(W\cap Z_\beta).
 \end{aligned}
\end{equation*}	

For each $A\in \pi_\beta(Z_\beta)$, $A$ has finitely many nonzero terms and $(0)^\infty\notin \pi_\beta(Z_\beta)$  (see Proposition~\ref{p_relation_of_coding}~\ref{p_relation_of_coding__ii} and Lemma~\ref{l_equivalent_def_pi_*}). By Proposition~\ref{p_relation_of_coding}~\ref{p_relation_of_coding__i},~\ref{p_relation_of_coding__ii}, and Lemma~\ref{l_equivalent_def_pi_*}, $\pi_\beta^*(x)$ has infinitely many nonzero terms for all $x\in(0,1]$ and $\pi_\beta^*(0)=(0)^\infty$. The second part of (ii) follows.

\smallskip
(iii) By Proposition~\ref{p_relation_of_coding}~\ref{p_relation_of_coding__iv} and~\ref{p_relation_of_coding__v}, for each $x\in I$, 
$			T_\beta^n(x)=\bigl(T_\beta^n\circ h_\beta\circ\pi_\beta\bigr)(x)=(h_\beta\circ \sigma^n\circ\pi_\beta)(x)$ and $
			U_\beta^n(x)=\bigl(U_\beta^n\circ h_\beta\circ\pi^*_\beta\bigr)(x)=\bigl(h_\beta\circ \sigma^n\circ\pi^*_\beta\bigr)(x)$,
so these, together with (\ref{e_def_Z_beta}), give (iii).

(iv) Assume that $x\in (0,1]\smallsetminus Z_\beta$, so that $\pi_\beta(x)=\pi_\beta^*(x)$ by (\ref{e_def_Z_beta}). 
From the fact that
$\pi_\beta^*$ is left-continuous, the fact that $\pi_\beta^*(y)\preceq \pi_\beta(y)$ for all $y\in I$, and the fact that $\pi_\beta$ is strictly increasing (see Proposition~\ref{p_relation_of_coding}~\ref{p_relation_of_coding__ix},~\ref{p_relation_of_coding__ii}, and~\ref{p_relation_of_coding__vi}), we have 
$
	\pi_\beta^*(x)=\lim_{y\nearrow x}\pi^*_\beta(y)\preceq \lim_{y\nearrow x}\pi_\beta(y)\preceq \pi_\beta(x)$.
So $\lim_{y\nearrow x}\pi_\beta(y)= \pi_\beta(x)$. 
Combining this with the fact that $\pi_\beta$ is right-continuous on $[0,1)$, 
and that $x\in (0,1]\smallsetminus Z_\beta$ was arbitrary, we see that $\pi_\beta$ is continuous on $I\smallsetminus Z_\beta$.
The fact that $\pi_\beta^*$ is continuous on $I\smallsetminus Z_\beta$ can be proved similarly.

(v) If $x \in Z_\beta$ then $(0)^\infty \in \cO^\sigma(\pi_\beta(x))$ 
by Proposition~\ref{p_relation_of_coding}~\ref{p_relation_of_coding__ii} and Lemma~\ref{l_equivalent_def_pi_*}.
So 
$
	\pi_\beta (Z_{\beta}) \subseteq \bigl(\bigcup_{n=1}^{+\infty} \sigma^{-n}((0)^\infty)\bigr) 
	\smallsetminus \{ (0)^\infty\}$.
If $\mu\in \MMM(X_\beta,\sigma)$ and $n \in \N$, then $\mu(\sigma^{-n}((0)^\infty)) = \mu(\{(0)^\infty\})$ and $(0)^\infty\in \sigma^{-n}((0)^\infty)$. 
This implies that 
$\mu \bigl( \bigl( \bigcup_{n=1}^{+\infty} \sigma^{-n}((0)^\infty)\bigr) \smallsetminus \{(0)^\infty\} \bigr) =0$,
and therefore $\mu(\pi_\beta (Z_{\beta}))=0$. 
\end{proof}

\begin{proof}[\bf Proof of Proposition~\ref{p_coding_mpe_relation}]
	(i)	By Proposition~\ref{p_relation_of_coding}~\ref{p_relation_of_coding__iii}, $\sigma \bigl( \pi^*_\beta(I) \bigr) \subseteq \pi^*_\beta(I)$. By Lemma~\ref{l_properties_z_beta}~(ii) applied to $W=I$, we have $X_\beta = \pi^*_\beta(I) \cup \pi_\beta(Z_\beta)$. Thus, by Lemma~\ref{l_properties_z_beta}~(v), we have that $\MMM\bigl(\pi_\beta^*(I),\sigma\bigr)$ can be naturally identified with $\MMM(X_\beta,\sigma)$. More precisely, $\mu(\,\cdot\,) \mapsto \mu\bigl(\,\cdot \, \cap \pi^*_\beta(I)\bigr)$ is a bijection from $\cM\bigl(\pi^*_\beta(I), \sigma\bigl)$ to $\cM(X_\beta, \sigma) .$ So $H_\beta$ can be seen as the pushforward of $h_\beta|_{\pi_\beta^*(I)}$ from $\cM\bigl(\pi^*_\beta(I), \sigma\bigr)$ to $\cP(I)$. Proposition~\ref{p_relation_of_coding}~\ref{p_relation_of_coding__v} implies that $H_\beta(\MMM(X_\beta,\sigma))\subseteq \MMM(I,U_\beta)$.
	
	For each $\mu\in \MMM(I, U_\beta)$ and each Borel measurable subset $Y\subseteq X_\beta$, by (\ref{e_def_G}) and Proposition~\ref{p_relation_of_coding}~\ref{p_relation_of_coding__iv}, we have
	\begin{equation}\label{e_property_G}
		G_\beta(\mu)(Y)=\mu\bigl(\bigl(\pi_\beta^*\bigr)^{-1}(Y)\bigr)=\mu\bigl(\bigl(\pi_\beta^*\bigr)^{-1}(Y\cap \pi_\beta^*(I))\bigr)=\mu\bigl(h_\beta\bigl(Y\cap \pi_\beta^*(I)\bigr)\bigr).
	\end{equation}
	Hence we derive that $(H_\beta\circ G_\beta)(\mu)=\mu$ for all $\mu \in \MMM(I,U_\beta)$. More precisely, for each Borel measurable subset $W\subseteq I$, by (\ref{e_def_H}), (\ref{e_property_G}), Lemma~\ref{l_properties_z_beta}~(ii), and Proposition~\ref{p_relation_of_coding}~\ref{p_relation_of_coding__iv},
	\begin{equation*}
		\begin{aligned}
			(H_\beta\circ G_\beta)(\mu)(W)
            &=G_\beta(\mu)\bigl(h_\beta^{-1}(W)\bigr)
			=\mu\bigl(h_\beta \bigl(h_\beta^{-1}(W)\cap \pi_\beta^*(I)\bigr)\bigr)\\
			&= \mu\bigl(h_\beta \bigl(\bigl(\pi_\beta^*(W)\cup \pi_\beta(W\cap Z_\beta)\bigr)\cap \pi_\beta^*(I)\bigr)\bigr)\\
			&=\mu\bigl(h_\beta \bigl(\bigl(\pi_\beta^*(W)\cap \pi_\beta^*(I)\bigr)\cup\bigl(\pi_\beta(W\cap Z_\beta)\cap \pi_\beta^*(I)\bigr)\bigr)\bigr)\\
			&=\mu\bigl(h_\beta\bigl(\pi_\beta^*(W)\bigr)\bigr)
			=\mu(W) .
		\end{aligned}
	\end{equation*}
	For each $\nu\in \MMM(X_\beta, \sigma)$ and each Borel measurable subset $W\subseteq I$, by (\ref{e_def_H}) and  Lemma~\ref{l_properties_z_beta}~(ii) and (v), we have
	\begin{equation}\label{e_property_H}
		H_\beta(\nu)(W)=\nu\bigl(h_\beta^{-1}(W)\bigr)=\nu\bigl(\pi_\beta^*(W)\cup \pi_\beta(W\cap Z_\beta)\bigr)=\nu\bigl(\pi_\beta^*(W)\bigr).
	\end{equation}
	Hence we derive that $(G_\beta\circ H_\beta)(\nu)=\nu$ for all $\nu \in \MMM(X_\beta,\sigma)$. More precisely, for each Borel measurable subset $Y\subseteq X_\beta$, by (\ref{e_def_G}) and (\ref{e_property_H}),
	\begin{equation*}
		(G_\beta\circ H_\beta)(\nu)(Y)=H_\beta(\nu)\bigl(\bigl(\pi_\beta^*\bigr)^{-1}(Y)\bigr)=\nu\bigl(\pi_\beta^* \bigl(\bigl(\pi_\beta^*\bigr)^{-1}(Y)\bigr)\bigr)=\nu(Y).
	\end{equation*} 
	By the above, $H_\beta$ is a bijection from $\MMM(X_\beta,\sigma)$ to $\MMM(I,U_\beta)$. Moreover, by Proposition~\ref{p_relation_of_coding}~\ref{p_relation_of_coding__x}, $H_\beta$ is continuous from $\MMM(X_\beta,\sigma)$ to $\PPP(I)$.
		
	The weak$^*$ compactness of $\MMM(X_\beta,\sigma)$ follows immediately from the compactness of $X_\beta$ and the continuity of $\sigma$. 
	By \cite[Theorem~6.4]{Wal82}, the set of probability measures on $I$ is Hausdorff in the weak$^*$ topology, hence $\MMM(I,U_\beta)$ is Hausdorff, and therefore $H_\beta$ is a homeomorphism from $\MMM(X_\beta,\sigma)$ to $\MMM(I,U_\beta)$, with $G_\beta^{-1}=H_\beta$.

	\smallskip 
	(ii) follows immediately from (i), and the weak$^*$ compactness of $\MMM(X_\beta,\sigma)$. 
	
	\smallskip 
	(iii) Since $G_\beta$ is the pushforward of $\pi_\beta^*$, for each $\mu\in \MMM(X_\beta, \sigma)$, by Proposition~\ref{p_relation_of_coding}~\ref{p_relation_of_coding__iv} and statement~(i), we have
	\begin{equation*}
		\int_I\!\phi \,\mathrm{d} H_\beta(\mu)
		=\int_I\!\bigl(\phi\circ h_\beta\circ \pi_\beta^*\bigr) \,\mathrm{d} H_\beta(\mu)
		=\int_{X_\beta}\!(\phi\circ h_\beta) \,\mathrm{d} (G_\beta\circ H_\beta)\mu
        =\int_{X_\beta}\!(\phi\circ h_\beta) \,\mathrm{d} \mu.
	\end{equation*}
	By (i), we obtain the required identities
	\begin{equation*}
		\mpe(U_\beta,\phi)
        =\mpe \bigl(\sigma|_{X_\beta},\phi\circ h_\beta \bigr)
        \quad \text{ and } \quad 
        \MMM_{\max}( U_\beta, \phi )
        =H_\beta \bigl(\MMM_{\max}\bigl(\sigma|_{X_\beta},\phi\circ h_\beta\bigr)\bigr) .
	\end{equation*}
	Since $\MMM(X_\beta,\sigma)$ is weak$^*$ compact and $h_\beta$ is continuous (see Proposition~\ref{p_relation_of_coding}~\ref{p_relation_of_coding__x}),  $\MMM_{\max}(\sigma|_{X_\beta},\phi\circ h_\beta)$ is nonempty.

    \smallskip
    (iv) First note that $\pi_\beta(0) = \pi_\beta^*(0) = (0)^\infty$ (see Lemma~\ref{l_equivalent_def_pi_*}). 
    Next, by Proposition~\ref{p_relation_of_coding}~(ii), any $\underline{a} \in \pi_\beta(I) \smallsetminus \pi_\beta^*(I)$ is of the form $\underline{a} = a_1 a_2 \dots a_n (0)^{\infty}$ with $a_n>0$, and therefore not a $\sigma$-periodic point. Applying Lemma~\ref{l_properties_z_beta}~(ii) in the case $W= I$ gives 
    \begin{equation}\label{X_beta_subset}
    X_\beta = h_\beta^{-1}(I) \subseteq \pi_\beta(I) \cup \pi_\beta^*(I).
    \end{equation}
    Any $(X_\beta,\sigma)$-periodic orbit $\cO$ cannot intersect
    $\pi_\beta(I) \smallsetminus \pi_\beta^*(I)$, as noted above,
    therefore
    (\ref{X_beta_subset}) implies that
    $\cO \subseteq \pi_\beta^*(I)$. So the fact that
    $h_\beta\circ\sigma=U_\beta\circ h_\beta$ on $\pi_\beta^*(I)$
(by Proposition~\ref{p_relation_of_coding}~(v)) implies that $h_\beta(\cO)$ is an $(I,U_\beta)$-orbit and that $U_\beta(h_\beta(\cO)) = h_\beta(\cO)$. So $h_\beta(\cO)$ is an $(I,U_\beta)$-periodic orbit.
    
    To prove that $\card \cO = \card h_\beta(\cO)$ it suffices to show that $h_\beta$ is injective on $\cO$. Indeed, if $h_\beta$ were not injective on $\cO$, then by Proposition~\ref{p_relation_of_coding}~(ii) and (xi), there would exist a point in $\cO$ of the form $z_1 z_2 \dots z_n (0)^\infty$ for $z_n>0$, and this is not a $\sigma$-periodic point, a contradiction.
\end{proof}

\begin{proof}[\bf Proof of Proposition~\ref{mpe=}]
	Suppose $\beta$ is not a simple beta-number. Then 
	$\MMM(I, T_\beta)=\MMM(I,U_\beta)$ is weak$^*$ compact
	(see Proposition~\ref{p_relation_T_beta_and_wt_T_beta}~(iv) and Proposition~\ref{p_coding_mpe_relation}~(ii)). So (i) holds, and both (ii) and (iii) follow immediately from (i).
	
	If $\beta$ is a simple beta-number,
	by Proposition~\ref{p_relation_T_beta_and_wt_T_beta}~(v), it suffices to prove that $\mu_{\cO_\beta^*(1)}$ is contained in the weak$^*$ closure of $\MMM(I, T_\beta)$.
	Note that $\cO^*_\beta(1)$ and $\cO^\sigma\bigl(\pi_\beta^*(1)\bigr)$ are periodic orbits of $U_\beta$ and $\sigma$, respectively.
	By \cite[p.~249]{Si76}, the
	periodic measures 
	are weak$^*$ dense in $\MMM(X_\beta,\sigma)$, so
	there exists a sequence of periodic orbits $\{\cO_n\}$ of $(X_\beta,\sigma)$ satisfying: 
	
	(a) $\cO_n\neq\cO^\sigma(\pi_\beta^*(1))$ for all $n\in \N$, and
	
	(b) $\mu_{\cO_n}$ converges to $\mu_{\cO^\sigma(\pi_\beta^*(1))}$ in the weak$^*$ topology as $n$ tends to $+\infty$. 
	
	By Proposition~\ref{p_coding_mpe_relation}~(iv), $ \{ h_\beta(\cO_n) \}_{n\in \N}$ are $U_\beta$-periodic orbits. Note that $h_\beta \bigl(\cO^\sigma\bigl(\pi^*_\beta(1)\bigr)\bigr) = \cO_\beta^*(1)$, so that by Proposition~\ref{p_coding_mpe_relation}~(i), $H_\beta(\mu_{\cO_n})=\mu_{h_\beta(\cO_n)}$ converges to $H_\beta\bigl(\mu_{\cO^\sigma(\pi_\beta^*(1))}\bigr)=\mu_{\cO_\beta^*(1)}$ in the weak$^*$ topology, and $h_\beta(\cO_n) \neq \cO_\beta^*(1)$ for each $n\in \N$.
	But $H_\beta(\mu_{\cO_n})\in  \MMM(I,T_\beta)$ for each $n\in\N$,
	by Proposition~\ref{p_relation_T_beta_and_wt_T_beta}~(iii), 
	so (i) follows.	Both (ii) and (iii) follow immediately from (i).
\end{proof}

\end{document}